\documentclass{amsart}
\usepackage{xr}
\usepackage{docmute}
\usepackage{amsmath}
\usepackage{amssymb}
\usepackage{amsthm}
\usepackage{amsfonts}
\usepackage{etoolbox}
\usepackage{marginnote}
\usepackage[all]{xy}
\usepackage{xspace}
\usepackage{calc}
\usepackage{enumerate}
\usepackage[normalem]{ulem} 
\usepackage{array} 
\usepackage[colorlinks, linkcolor=black, citecolor=blue]{hyperref}
\usepackage[sort,numbers,square]{natbib}
\usepackage{graphicx}
\usepackage{ifthen}
\usepackage{theoremref}
\usepackage{tipa}
\usepackage{tikz}
\usetikzlibrary{matrix}
\usetikzlibrary{decorations.pathreplacing}
\usetikzlibrary{decorations.markings}
\usetikzlibrary{arrows}
\usetikzlibrary{calc}
\usetikzlibrary{shapes.misc}
\usetikzlibrary{fit}
\usepgflibrary{decorations.pathmorphing}
\usepgflibrary{shapes.geometric}
\usepackage{stmaryrd} 
\usepackage{varwidth} 

\usepackage{Cobordism} 

\usepackage{knots}

\makeatletter
\def\calign@preamble{%
   &\hfil\strut@
    \setboxz@h{\@lign$\m@th\displaystyle{##}$}%
    \ifmeasuring@\savefieldlength@\fi
    \set@field
    \hfil
    \tabskip\alignsep@
}
\let\cmeasure@\measure@
\patchcmd\cmeasure@{\divide\@tempcntb\tw@}{}{}{}
\patchcmd\cmeasure@{\divide\@tempcntb\tw@}{}{}{}
\patchcmd\cmeasure@{\ifodd\maxfields@
  \global\advance\maxfields@\@ne
  \fi}{}{}{}    
\newenvironment{calign}
{%
  \let\align@preamble\calign@preamble
  \let\measure@\cmeasure@
  \align
}
{%
  \endalign
}  
\makeatother

\newenvironment{tz}[1][]{\begin{tikzpicture}[baseline={([yshift=-.8ex]current bounding box.center)},#1]}{\end{tikzpicture}}
\newenvironment{tztiny}[1][]{\begin{aligned}\begin{tikzpicture}[#1] \tiny}{\end{tikzpicture}\end{aligned}}
\newcommand\redscalar[1]{{\color{red} #1}}

\newlength\xsep
\newlength\ysep
\newlength\xdelta

\definecolor{GluingColor}{rgb}{0.0,0.5,0.75}       

\def\inv{{\hspace{0.5pt}\text{-} \hspace{-0.5pt}1}}

\usepackage{aliascnt}

\theoremstyle{plain}
\newtheorem{theorem}{Theorem}[section]

\newaliascnt{lemma}{theorem}  
\newtheorem{lemma}[lemma]{Lemma}
\aliascntresetthe{lemma}

\newaliascnt{proposition}{theorem}  
\newtheorem{proposition}[proposition]{Proposition}
\aliascntresetthe{proposition}

\newaliascnt{corollary}{theorem}  
\newtheorem{corollary}[corollary]{Corollary}
\aliascntresetthe{corollary}

\newaliascnt{conjecture}{theorem}  

\aliascntresetthe{conjecture}

  \newtheorem{introthm}{Theorem}

\theoremstyle{definition}

\newaliascnt{defn}{theorem}  
\newtheorem{defn}[defn]{Definition}
\aliascntresetthe{defn}

\newaliascnt{remark}{theorem}  
\newtheorem{remark}[remark]{Remark}
\aliascntresetthe{remark}

\newaliascnt{example}{theorem}  
\newtheorem{example}[example]{Example}
\aliascntresetthe{example}

\makeatletter
\newcommand\Autoref[1]{\@first@ref#1,@}
\def\@throw@dot#1.#2@{#1}
\def\@set@refname#1{
    \edef\@tmp{\getrefbykeydefault{#1}{anchor}{}}%
    \def\@refname{\@nameuse{\expandafter\@throw@dot\@tmp.@autorefname}s}%
}
\def\@first@ref#1,#2{%
  \ifx#2@\autoref{#1}\let\@nextref\@gobble
  \else%
    \@set@refname{#1}
    \@refname~\ref{#1}
    \let\@nextref\@next@ref
  \fi%
  \@nextref#2%
}
\def\@next@ref#1,#2{%
   \ifx#2@ and~\ref{#1}\let\@nextref\@gobble
   \else, \ref{#1}
   \fi%
   \@nextref#2%
}
\makeatother

\newcounter{brucedefncountera}

\newcounter{brucedefncounterb}

\newcommand\relref[1]{\ensuremath{\text{(\hyperref[#1]{#1})}}{}}

\newcounter{Hequation}

\makeatletter
\g@addto@macro\equation{\stepcounter{Hequation}}
\makeatother

\allowdisplaybreaks

\ifx\ignore\undefined
\newcommand\ignore[1]{}
\fi

\renewcommand{\_}[0]{\nobreakdash--\hspace{0pt}}

\newcommand\vc[1]{\ensuremath{\begin{array}{@{}c@{}}#1\end{array}}}

\newcommand\red{\color{red}}
\newcommand\isd{\ensuremath{\smash{\widetilde Z}}\xspace}

\newcommand\uR{\mathrm{R}}
\renewcommand\d{\mathrm{d}}

\newcommand\fixboundingbox{\path [use as bounding box, draw=none] (current bounding box.north west) rectangle (current bounding box.south east);}
\newcommand\selectpart[2][\selectcolour]{\fixboundingbox\begin{pgfonlayer}{selectionbox}\node [draw, fit=#2, inner sep=0.8*\cobordismlinewidth, #1, line width=\cobordismlinewidth] {};\end{pgfonlayer}}

\newcommand*\arrowoffset{0.4pt}
\makeatletter
\pgfarrowsdeclare{new double arrowhead}{new double arrowhead}
{
  \pgfutil@tempdima=-0.84pt%
  \advance\pgfutil@tempdima by-1.3\pgflinewidth%
  \pgfutil@tempdimb=0.21pt%
  \advance\pgfutil@tempdimb by.625\pgflinewidth%
  \pgfarrowsleftextend{+\pgfutil@tempdima}
  \advance\pgfutil@tempdimb by \arrowoffset 
  \pgfarrowsrightextend{+\pgfutil@tempdimb}
}
{
  \pgfsetstrokecolor{black}
  \pgfsetlinewidth{0.6pt}
  \pgfutil@tempdima=0.28pt%
  \advance\pgfutil@tempdima by.3\pgflinewidth%
  \pgfsetlinewidth{0.8\pgflinewidth}
  \pgfsetdash{}{+0pt}
  \pgfsetroundcap
  \pgfsetroundjoin
  \pgfpathmoveto{\pgfpoint{-3\pgfutil@tempdima+\arrowoffset}{4\pgfutil@tempdima}}
  \pgfpathcurveto
                {\pgfpoint{-2.75\pgfutil@tempdima+\arrowoffset}{2.5\pgfutil@tempdima}}
                {\pgfpoint{\arrowoffset}{0.25\pgfutil@tempdima}}
                {\pgfpoint{0.75\pgfutil@tempdima+\arrowoffset}{0pt}}
  \pgfpathcurveto
                {\pgfpoint{\arrowoffset}{-0.25\pgfutil@tempdima}}
                {\pgfpoint{-2.75\pgfutil@tempdima+\arrowoffset}{-2.5\pgfutil@tempdima}}
                {\pgfpoint{-3\pgfutil@tempdima+\arrowoffset}{-4\pgfutil@tempdima}}
  \pgfusepathqstroke
}
\pgfarrowsdeclare{new single arrowhead}{new single arrowhead}
{
  \pgfutil@tempdima=-0.84pt%
  \advance\pgfutil@tempdima by-1.3\pgflinewidth%
  \pgfutil@tempdimb=0.21pt%
  \advance\pgfutil@tempdimb by.625\pgflinewidth%
  \pgfarrowsleftextend{+\pgfutil@tempdima}
  \advance\pgfutil@tempdimb by \arrowoffset 
  \pgfarrowsrightextend{+\pgfutil@tempdimb}
}
{
  \pgfsetstrokecolor{black}
  \pgfsetlinewidth{0.6pt}
  \pgfutil@tempdima=0.28pt%
  \advance\pgfutil@tempdima by.3\pgflinewidth%
  \pgfsetlinewidth{0.8\pgflinewidth}
  \pgfsetdash{}{+0pt}
  \pgfsetroundcap
  \pgfsetroundjoin
  \pgfpathmoveto{\pgfpoint{-3\pgfutil@tempdima}{4\pgfutil@tempdima}}
  \pgfpathcurveto
                {\pgfpoint{-2.75\pgfutil@tempdima}{2.5\pgfutil@tempdima}}
                {\pgfpoint{0pt}{0.25\pgfutil@tempdima}}
                {\pgfpoint{0.75\pgfutil@tempdima}{0pt}}
  \pgfpathcurveto
                {\pgfpoint{0pt}{-0.25\pgfutil@tempdima}}
                {\pgfpoint{-2.75\pgfutil@tempdima}{-2.5\pgfutil@tempdima}}
                {\pgfpoint{-3\pgfutil@tempdima}{-4\pgfutil@tempdima}}
  \pgfusepathqstroke
}
\makeatother
\def\proarrow{\relbar\joinrel\mapstochar\joinrel\rightarrow}
\makeatletter
\tikzset{nomorepostaction/.code={\let\tikz@postactions\pgfutil@empty}}
\makeatother
\tikzset{double arrow scope/.style={\ignore{every node/.style={font=\scriptsize},}every path/.style={
        double, -new double arrowhead}}}
\newcommand\doubleto[1][0.5]{\mathbin{\ensuremath{\begin{tikzpicture}
   \node (A) at (0,0) [inner xsep=1.2pt] {};
   \node (B) at (#1,0) [inner xsep=1pt] {};
   \draw [double, -new double arrowhead] (A) to (B);
\end{tikzpicture}}}}
\renewcommand\to[0]{\mathbin{\ensuremath{\begin{tikzpicture}
   \node (A) at (0,0) [inner xsep=1.2pt] {};
   \node (B) at (0.5,0) [inner xsep=1pt] {};
   \draw [-new single arrowhead] (A) to (B);
\end{tikzpicture}}}}
\newcommand\xdoubleto[1]{\mathbin{\begin{tikzpicture}[baseline={([yshift=-2pt]
current bounding box.south)}]\node (A) at (0,0) [inner xsep=0pt, inner ysep=2pt, minimum width=0.2cm, font=\scriptsize] {\ensuremath{\smash{#1\strut}}};\draw [double,-new double arrowhead] ([xshift=-2.5pt] A.south west) to ([xshift=3pt] A.south east);\end{tikzpicture}}}

\newcommand\longxdoubleto[2][0.5cm]{\mathbin{\begin{tikzpicture}[baseline={([yshift=1pt]current bounding box.south)}]\node (A) at (0,0) [inner xsep=0pt, inner ysep=2pt, minimum width=#1] {\ensuremath{#2}};\path [use as bounding box] ([xshift=-2.5pt, yshift=-3pt] A.south west) rectangle ([xshift=2.5pt] A.north east);\draw [double,-new double arrowhead] ([xshift=-2.5pt] A.south west) to ([xshift=3pt] A.south east);
\end{tikzpicture}}}

\newcommand\xleftdoubleto[1]{\mathbin{\begin{tikzpicture}[baseline={([yshift=-3pt]
current bounding box.south)}]
    \node (A) at (0,0) [inner xsep=0pt, inner ysep=-1pt, minimum width=0.2cm] {\ensuremath{\scriptstyle #1 \strut}};
    \draw [double,new double arrowhead-]
        ([xshift=-2.5pt] A.south west)
        to ([xshift=3pt] A.south east);
\end{tikzpicture}}}
\newcommand\xto[1]{\mathbin{\begin{tikzpicture}[baseline={([yshift=-3pt]
current bounding box.south)}]
    \node (A) at (0,0) [inner xsep=0pt, inner ysep=2pt, minimum width=0.2cm] {$\scriptstyle #1$};
    \draw [-new single arrowhead]
        ([xshift=-2.5pt] A.south west)
        to ([xshift=3pt] A.south east);
\end{tikzpicture}}}

\newcommand\xmapsto[1]{\mathbin{\begin{tikzpicture}[baseline={([yshift=-3pt]
current bounding box.south)}]
    \node (A) at (0,0) [inner xsep=0pt, inner ysep=2pt, minimum width=0.2cm] {$#1$};
    \draw [|-new single arrowhead]
        ([xshift=-2.5pt] A.south west)
        to ([xshift=3pt] A.south east);
\end{tikzpicture}}}

\newcommand\longxarrow[2]{\mathbin{\begin{tikzpicture}[baseline={([yshift=-2pt]
current bounding box.south)}]
   \node (1) at (0,0) [inner sep=1pt] {};
   \node (2) at (#2,0) [inner sep=1pt] {};
   \draw [-new single arrowhead] (1) to node [above, inner xsep=0pt, inner ysep=2pt]
{\ensuremath{\scriptstyle #1}} (2);
\end{tikzpicture}}}
\newcommand\longxmapsto[2]{\mathbin{\begin{tikzpicture}[baseline={([yshift=-2pt]
current bounding box.south)}]
   \node (1) at (0,0) [inner sep=1pt] {};
   \node (2) at (#2,0) [inner sep=1pt] {};
   \draw [|-new single arrowhead] (1) to node [above, inner xsep=0pt, inner ysep=2pt]
{\ensuremath{\scriptstyle #1}} (2);
\end{tikzpicture}}}
\renewcommand\xrightarrow[1]{\xto{#1}}

\newcommand\doublearrowspace{15pt}
\newcommand\rightleftdoublearrow[2]{\mathbin{\begin{array}{@{}c@{}}#1\\[-6pt]\,\xdoubleto{\hspace{\doublearrowspace}}\\[-10pt]\xleftdoubleto{\hspace{\doublearrowspace}}\,\\[-6pt]#2\end{array}}}

\newcommand{\cat}[1]{\ensuremath{\mathrm{#1}}}
\newcommand{\bicat}[1]{\ensuremath{\mathbf{#1}}}

\newcommand{\id}{\ensuremath{\mathrm{id}}}
\renewcommand{\dag}{\ensuremath{\dagger}}

\newcommand\Hom{\mathrm{Hom}}
\renewcommand\hom{\Hom}
\newcommand\fd{\mathrm{fd}}
\newcommand\fgp{\mathrm{fg,proj}}

\newcommand\Ob{\ensuremath{\mathrm{Ob}}}
\def\phiN{\ensuremath{\phi_{\mathrm{\scriptscriptstyle N}}}}
\def\phiM{\ensuremath{\phi_{\reflectbox{$\mathrm{\scriptscriptstyle N}$}}}}
\def\phiN{\ensuremath{\phi_1}}
\def\phiM{\ensuremath{\phi_2}}
\newcommand\II{\ensuremath{\mathrm{II}}}

\newcommand\III{\ensuremath{\mathrm{III}}}
\newcommand\op{\ensuremath{\mathrm{op}}}
\newcommand\Tau{\ensuremath{\mathrm{T}}}

\DeclareMathOperator{\Fun}{Fun}
\newcommand\set{\cat{Set}}

\newcommand\shrunkenf{\hspace{-1pt}f\hspace{-1pt}}
\newcommand\vect{\ensuremath{\cat{Vect}_k}}

\newcommand\alg{\ensuremath{\bicat{Alg}_k}}
\newcommand\ALG{\ensuremath{\bicat{ALG}_k}\xspace}

\newcommand\bimod{{\ensuremath{\bicat{Bimod}_k}}\xspace}
\newcommand\lincat{\ensuremath{\bicat{LinCat}_k}}
\newcommand\Rep{\mathrm{Rep}}
\newcommand\kvtwovect{\ensuremath{\bicat{K \hspace{-0.7pt} V \hspace{-0.7pt} 2 \hspace{-1pt} Vect}_k}\xspace}
\newcommand\twovect{{\ensuremath{\bicat{2 \hspace{-1pt} Vect}_k}}\xspace}

\newcommand\comp{\ensuremath{\mathrm{compose}}}

\def\L{\ensuremath{\mathrm{L}}}

\def\vect{\ensuremath{\cat{Vect}_k}}

\def\perspectivecircle{{\scalecobordisms{0.5}\raisebox{-3pt}{\ensuremath{\tikz{\node [Cobordism Top End 3D] at (0,0) {};}}}}}

\DeclareMathOperator{\Tr}{Tr}

\definecolor{CSPcolor}{rgb}{0.0,0.5,0.75}       
\definecolor{BBcolor}{rgb}{0.5,0.0,0.5}   
\definecolor{CDcolor}{rgb}{0.8,0.0,0.2}   
\definecolor{JVcolor}{rgb}{0.0,0.6,0.0}   

\def\NN{\mathbb N}
\def\RR{\mathbb R}

\def\ZZ{\mathbb Z}

\DeclareMathOperator{\End}{End}

\newcommand{\Bord}{\ensuremath{\bicat{Bord}^{\mathrm{or}}_{1,2,3}}\xspace}
\newcommand\BordMF{\ensuremath{\bicat{Bord}^{\mathrm{csig}}_{1,2,\sim}}\xspace}

\newcommand{\BordS}{\ensuremath{\smash{\bicat{Bord}^{\smash{\S}}_{1,2,3}}}\xspace}

\newcommand{\Bordsig}{\ensuremath{\smash{\bicat{Bord}^{\mathrm{sig}}_{1,2,3}}}\xspace}
\newcommand{\Bordcsig}{\ensuremath{\smash{\bicat{Bord}^{\mathrm{csig}}_{1,2,3}}}\xspace}

\newcommand{\Bordp}{\ensuremath{\smash{\bicat{Bord}^{\text{$p$}_1}_{1,2,3}}}\xspace}

\def\B{\ensuremath{\mathcal B}\xspace}
\def\F{\ensuremath{\mathcal F}\xspace}
\def\G{\ensuremath{\mathcal G}\xspace}

\def\M{\ensuremath{\mathcal M}\xspace}
\def\N{\ensuremath{\mathcal N}\xspace}
\def\O{\ensuremath{\mathcal O}\xspace}
\def\P{\ensuremath{\mathcal R}\xspace}
\def\Q{\ensuremath{\mathcal Q}\xspace}
\def\R{\ensuremath{\mathcal R}\xspace}
\def\S{\ensuremath{\mathcal S}\xspace}

\newcommand{\Vect}{\mathrm{Vect}}
\newcommand\qs{\mathrm{qs}}

\newcommand{\interchangor}{\varphi}

\newcommand{\yonsphere}{y'}

 \tikzset{box style/.style={draw,minimum width=0.5\cobwidth,minimum height=0.2\cobwidth,inner sep=1pt,black,fill=white, line width=\cobordismlinewidth}}

\tikzset{tempstyle/.style={draw, red, inner sep=0pt, line width=\cobordismlinewidth, font=\tiny, minimum width=13pt, minimum height=7pt, fill=white}}

\tikzset{dotted guideline/.style={green}}
\tikzset{helpingbox/.style={green}}
\tikzset{singularity/.style={red, thick}}
\tikzset{singularity color/.style={red, thick}}
\tikzset{little dot/.style={inner sep=0pt, shape=circle, blue, minimum size=0.13cm, fill=blue}}

\tikzset{comment_bb/.style={rectangle,minimum size=6mm,rounded corners=3mm,          very thick,draw=black!50, top color=white, bottom color=blue!20, execute at begin node={\begin{varwidth}{32em}}, execute at end node={\end{varwidth}}}}
\tikzset{comment_jv/.style={rectangle,minimum size=6mm,rounded corners=3mm,          very thick,draw=black!50, top color=white, bottom color=green!20, execute at begin node={\begin{varwidth}{32em}}, execute at end node={\end{varwidth}}}}
\tikzset{comment_csp/.style={rectangle,minimum size=6mm,rounded corners=3mm,          very thick,draw=black!50, top color=white, bottom color=yellow!20, execute at begin node={\begin{varwidth}{32em}}, execute at end node={\end{varwidth}}}}
\tikzset{comment_cd/.style={rectangle,minimum size=6mm,rounded corners=3mm,          very thick,draw=black!50, top color=white, bottom color=red!20, execute at begin node={\begin{varwidth}{32em}}, execute at end node={\end{varwidth}}}}

\definecolor{CSPcolor}{rgb}{0.0,0.5,0.75}       
\definecolor{CDcolor}{rgb}{0.8,0.0,0.2}         

\makeatletter
\pgfdeclareshape{arrow shape}
{
    \savedanchor\centerpoint
    {
        \pgf@x=0pt
        \pgf@y=0pt
    }
    \anchor{center}{\centerpoint}
    \anchor{south}{\centerpoint}
    \anchor{north}{\centerpoint}
    \anchorborder{\centerpoint}
    \backgroundpath
    {
        \begin{pgfonlayer}{foreground}
        \pgfsetstrokecolor{black}
        \pgfsetarrowsstart{>[\strarr]}
        \pgfsetlinewidth{\thickness}
        \pgfpathmoveto{\pgfpoint{0}{-2pt}}
        \pgfpathlineto{\pgfpoint{0}{1pt}}
        \pgfusepath{stroke}
        \end{pgfonlayer}
    }
}
\pgfdeclareshape{reverse arrow shape}
{
    \savedanchor\centerpoint
    {
        \pgf@x=0pt
        \pgf@y=0pt
    }
    \anchor{center}{\centerpoint}
    \anchor{south}{\centerpoint}
    \anchor{north}{\centerpoint}
    \anchorborder{\centerpoint}
    \backgroundpath
    {
        \begin{pgfonlayer}{foreground}
        \pgfsetstrokecolor{black}
        \pgfsetarrowsstart{<[\strarr]}
        \pgfsetlinewidth{\thickness}
        \pgfsetlinewidth{\thickness}
        \pgfpathmoveto{\pgfpoint{0}{-2pt}}
        \pgfpathlineto{\pgfpoint{0}{-1pt}}
        \pgfusepath{stroke}
        \end{pgfonlayer}
    }
}
\makeatother

\newtheorem{innercustomthm}{Theorem}
\newenvironment{customthm}[1]
  {\renewcommand\theinnercustomthm{#1}\innercustomthm\em}
  {\endinnercustomthm}
\newtheorem{innercustomprop}{Proposition}
\newenvironment{customprop}[1]
  {\renewcommand\theinnercustomprop{#1}\innercustomprop\em}
  {\endinnercustomprop}

\def\minus{\ensuremath{\text{-}}}

\def\microbordisms{\scalecobordisms{0.35}\setlength\obscurewidth{0pt}}
\def\smallbordisms{\scalecobordisms{0.5}\setlength\obscurewidth{0pt}}

\def\normalbordisms{\scalecobordisms{1}\setlength\obscurewidth{4pt}\def\toff{0.2cm}\def\boff{0.3cm}}
\def\mediumbordisms{\scalecobordisms{0.8}\setlength\obscurewidth{4pt}\def\toff{0.2cm}\def\boff{0.3cm}}

\def\gap{\hspace{5pt}}

\newcommand*\xcobpos[2]{#1*\the\cobwidth,#2*\the\cobheight}

\tikzset{knot/clip radius=\obscurewidth}
\tikzset{knot/clip width=0.1*\obscurewidth, knot/end tolerance=2pt}

\tikzset{xmorphismlabel/.style={draw=red, thin, circle, inner sep=-100pt, minimum width=5pt, fill=white, font=\tiny}}

\tikzset{morphlabel/.style={draw=black, thin, rectangle, minimum width=7pt, fill=white, font=\scriptsize}}

\tikzset{morphismlabel/.style={draw=red, thin, circle, inner sep=-1pt, minimum width=7pt, fill=white}}

\tikzset{above strand label/.style={yshift=0.4\cobwidth, anchor=mid, font=\footnotesize}}
\tikzset{below strand label/.style={yshift=-0.4\cobwidth, anchor=mid, font=\footnotesize}}

\def\myred{red}
\def\myblue{blue}
\def\mygreen{green}
\def\mypurple{purple}
\def\mybrown{brown}
\def\myblack{black}

\def\colourA{\myred}
\def\colourB{\myblue}
\def\colourC{\mygreen}
\def\colourD{\mypurple}
\def\colourE{\mybrown}
\def\colourF{\myblack}

\tikzset{style 1/.style={draw=white, double distance=\cobordismlinewidth, line width=0.5\obscurewidth, double=\colourA}}
\tikzset{style 2/.style={draw=white, double distance=\cobordismlinewidth, line width=0.5\obscurewidth, double=\colourB}}
\tikzset{style 3/.style={draw=white, double distance=\cobordismlinewidth, line width=0.5\obscurewidth, double=\colourC}}
\tikzset{style 4/.style={draw=white, double distance=\cobordismlinewidth, line width=0.5\obscurewidth, double=\colourD}}
\tikzset{style 5/.style={draw=white, double distance=\cobordismlinewidth, line width=0.5\obscurewidth, double=\colourE}}

\tikzset{red strand/.style={draw=white, double distance=\cobordismlinewidth, line width=0.5\obscurewidth, double=\colourA}}
\tikzset{red strand blue back/.style={draw=\outermorphismcolor, double distance=\cobordismlinewidth, line width=0.5\obscurewidth, double=\colourA}}
\tikzset{red strand no back/.style={draw=\colourA}}
\tikzset{blue strand/.style={draw=white, double distance=\cobordismlinewidth, line width=0.5\obscurewidth, double=\colourB}}
\tikzset{green strand/.style={draw=white, double distance=\cobordismlinewidth, line width=0.5\obscurewidth, double=\colourC}}
\tikzset{purple strand/.style={draw=white, double distance=\cobordismlinewidth, line width=0.5\obscurewidth, double=\colourD}}
\tikzset{brown strand/.style={draw=white, double distance=\cobordismlinewidth, line width=0.5\obscurewidth, double=\colourE}}
\tikzset{black strand/.style={draw=white, double distance=\cobordismlinewidth, line width=0.5\obscurewidth, double=\colourF}}
\tikzset{invisible strand/.style={draw=none}}

\newcommand\templabel[1]{}

\usepackage{enumitem}
\setlist[itemize,1]{leftmargin=\dimexpr 26pt-.1in}

\hypersetup{bookmarksdepth=3} 

\renewcommand\cat[1]{\ensuremath{\mathit{#1}}}

\renewcommand\P{\ensuremath{\mathcal P}}
\newcommand\Rib{\ensuremath{\mathcal R}}

\tikzset{bot=true}
\begin{document}

\normalem 

\title[Modular categories as representations of the 3d bordism 2-category]{Modular categories as representations of the 3-dimensional bordism 2-category}

\author[B. Bartlett]{Bruce Bartlett}
\address{Department of Mathematics, University of Stellenbosch}
\curraddr{Mathematical Institute, University of Oxford}
\email{brucehbartlett@gmail.com}

\author[C. L. Douglas]{Christopher L. Douglas}
\address{Mathematical Institute, University of Oxford}
\email{cdouglas@maths.ox.ac.uk}

\author[C. Schommer-Pries]{Christopher J. Schommer-Pries}
\address{Max Planck Institute for Mathematics, Bonn}
\email{schommerpries.chris@gmail.com}

\author[J. Vicary]{Jamie Vicary}
\address{Department of Computer Science, University of Oxford}
\email{jamie.vicary@cs.ox.ac.uk}

\begin{abstract}
We show that once-extended anomalous 3\-dimensional topological quantum field theories valued in the 2\-category of $k$-linear categories are in canonical bijection with modular tensor categories equipped with a square root of the global dimension in each factor.
\end{abstract}

\maketitle

\setcounter{tocdepth}{1}
\tableofcontents

\section{Introduction}

\subsection{Modular tensor categories and TQFTs}

The 3\-dimensional topological quantum field theories (TQFTs) developed in the late 1980s and the 1990s launched the new field of Quantum Topology, connecting such diverse subjects as the representation theory of quantum groups, the representation theory of loop groups, von Neumann algebras, and link invariants.  At the center of this development was the construction by Reshetikhin and Turaev \cite{RT91} of a 3\-dimensional TQFT from a modular tensor category.

An \textit{extended 3\-dimensional topological quantum field theory} is a symmetric monoidal functor
$$Z : \BordS \to \twovect.$$
Here $\BordS$ is a symmetric monoidal 2\-category of closed 1\-dimensional manifolds, 2\-dimensional bordisms with boundary, and diffeomorphism classes of 3\-dimensional bordisms with boundaries and corners, equipped with a geometrical structure $\S$; we consider the cases where this structure is an orientation, signature\footnote{A bordism has \textit{signature structure} when it is equipped with a bounding manifold.}, componentwise signature\footnote{A bordism has \textit{componentwise signature structure} when each connected component is equipped separately with a bounding manifold.}, or $p_1$\-structure\footnote{A bordism has \textit{$p_1$-structure} when it is equipped with a trivialization of the first Pontryagin class.}. We say that $Z$ is a \textit{linear representation} of $\BordS$. The target $\twovect$ is the symmetric monoidal 2\-category of Cauchy-complete $k$-linear categories\footnote{A category enriched in $k$-vector spaces is \emph{Cauchy-complete} if it is additive and idempotent complete.}, functors, and natural transformations, with the Deligne tensor product. Here and throughout $k$ is an algebraically closed field, and all categorical structures are fully weak.

In this paper we classify linear representations of these structured bordism categories, as follows.\footnote{We fix some notation, much of which is standard~\cite{bk01-ltc}.  We do not require the unit object $I$ of an MTC to be simple; let $I \cong \oplus s$ be a decomposition of the unit into simple objects, which we refer to as \textit{factors} (these are necessarily pairwise non-isomorphic). We say an object $T$ is \textit{preserved} by a factor $s$ if  $s \otimes T \cong T$, and we write $[s]$ for a maximal set of non-isomorphic simple objects preserved by $s$. For each factor $s$ the \textit{Gauss sums} $p_s^+,p_s^- \in k$ are defined by $p_s^+ = \sum_{i \in [s]} \theta _i ^{} d_i ^2$ and $\smash{p_s^- = \sum _{i \in [s]} \theta _i ^\inv d_i ^2}$, where $\theta$ is the twist of the MTC, and $d_i$ is the quantum dimension of the simple object $i \in [s]$.  The ratio $\smash{p _s^+/p_s ^-}$ is the \emph{anomaly} of the factor, and the product $p _s^+ p _s^-$ is the \emph{global dimension} of the factor.}
\begin{introthm}
\label{thm:introcsig}
Linear representations of the componentwise-signature bordism 2\-category $\Bordcsig$ are classified by modular tensor categories (MTCs) equipped with a square root of the anomaly of each factor.
\end{introthm}

\begin{introthm}
\label{thm:introoriented}
Linear representations of the oriented bordism 2\-category $\Bord$ are classified by MTCs whose anomaly is 1 in each factor.
\end{introthm}

\begin{introthm}
\label{thm:introsig}
Linear representations of the signature bordism 2\-category $\Bordsig$ are classified by MTCs such that the anomaly is the same in each factor, equipped with a single choice of square root of this anomaly.
\end{introthm}

\begin{introthm}
\label{thm:introp1}
Linear representations of the $p_1$\-structure bordism 2\-category $\Bordp$ are classified by MTCs equipped with a sixth root of the anomaly of each factor.
\end{introthm}

\noindent Note that our classification applies in arbitrary, in particular finite, characteristic. For each theorem, `classified' means that we provide explicit canonical bijections between equivalence classes of linear representations,  and equivalence classes of modular tensor categories equipped with the extra data. Also, rather than roots of the anomaly, it would be equivalent in Theorems~\ref{thm:introcsig} and \ref{thm:introsig} (but not~\ref{thm:introp1}) to ask instead for roots of the global dimension of the factor. 

There are a number of 2\-categories that are used in the literature as targets for TQFTs, and it may seem that this fragments the classification question. However, we show in the Appendix that for a wide variety of target 2\-categories $\bicat T$, for all the structures $\S$ that we consider, every symmetric monoidal functor $\BordS \to \bicat T$ factors through $\twovect$---in fact, through the full sub-2\-category given by the finite $k$\-linear semisimple categories. Thus, for the family of targets we consider, the classification question can be answered once and for all.

\subsection{TQFTs and modular functors}

There is a large literature on 3\-dimensional TQFTs extended to 1\-manifolds, including substantial texts by Bakalov and Kirillov~\cite{bk01-ltc}, Kerler and Lyubashenko~\cite{kl01}, Turaev~\cite{t94-qik}, and Walker~\cite{walker-notes}.\footnote{For work on the rather different and heretofore rarely considered problem of classifying {\em non}-extended 3-dimensional TQFTs, see the recent paper of Juhasz~\cite{juhasz2014defining}.} Our approach uses higher category theory, which is fast becoming the dominant tool for tackling structural questions in TQFT: developed to great effect by Baez--Dolan, Lurie, and many others~\cite{bd95-hda,lurie-tft,freed-higher,freed-expository,feshbach,kapustin,kl01,CSPthesis,segal2010locality}, it builds on Freed's and Lawrence's notion of an extended TQFT~\cite{freed-higher,Lawrence}, which in turn builds on the original compositional approach of Atiyah and Segal~\cite{a88-tqft, segal}.

Kerler and Lyubashenko~\cite{kl01} have a functorial setup similar to ours, based on double categories rather than 2\-categories. However, their bordism double category is more restricted, as it includes only the connected surfaces; for us, the disconnected surfaces play a crucial role.  In all other treatments, for example Walker's notion of `TQFT with corners' \cite{walker-notes}, the extension to 1\-manifolds is controlled by the auxiliary notion of a \emph{modular functor}. There are a number of definitions of modular functor in the literature and the relationships between these various notions are not clear.  In its simplest incarnation, a modular functor is a symmetric monoidal functor $\BordMF \to \twovect$, where $\BordMF$ is the 2\-category of (componentwise-signature structured) 1\-manifolds, 2\-manifold bordisms, and isotopy classes of diffeomorphisms between these bordisms. This notion of modular functor is studied by Tillmann~\cite{t98-sskl} and Kerler-Lyubashenko~\cite{kl01}, and advocated by Segal \cite{segal}. 

Many previous `modular-functor approaches' to extended TQFTs avoid dealing with 2-categories directly by introducing the category of \textit{$C$\-extended surfaces}, meaning that the surfaces are equipped with marked points, labeled by objects of $C$. Depending on the particular axiomatization, $C$ is a pre-specified linear category, a tensor category, or simply a labeling set. One notable difference between the `$C$\-extended approach' and the `2-categorical functorial approach' to modular functors, is that in the former no distinction is made between in-coming and out-going points. A modular functor in the $C$\-extended sense includes the data of a symmetric monoidal functor from the 1\-category of $C$\-extended surfaces and diffeomorphisms, to vector spaces and linear maps, and the structure of composition of surfaces with boundary is encoded by a collection of `gluing isomorphisms'. Such modular functors are required to satisfy a host of axioms, but the particulars of each axiomatization vary widely from author to author; we mention here in particular the notions of Andersen~\cite{a06-mfgz} (a version of Walker's~\cite{walker-notes} topological reformulation of Segal's~\cite{segal} complex-analytic definition; see also Grove~\cite{g99-mf}), Bakalov and Kirillov~\cite{bk01-ltc}, and Turaev~\cite{t94-qik}.  As far as we know there has been no attempt in the literature to compare these various notions. 

Modular functors (in whichever formulation) are weaker structures than TQFTs because they assign data only to invertible 3\-dimensional bordisms (viewed as mapping cylinders of diffeomorphisms). However, in some cases it is possible to construct a TQFT out of a modular functor. This was first claimed in an unpublished preprint by Kontsevich~\cite{k-88} which also provides a proof sketch. Walker gave an independent proof of a result of this type in his highly influential unpublished 1991 TQFT notes~\cite{walker-notes}. Later Grove~\cite{MR1871220} gave a different argument based on Kontsevich's original sketch (and using Walker's notion of modular functor).  In all of these treatments the modular functor must be equipped with the additional structure of a non-degenerate duality pairing between the vector spaces assigned to a given $C$-extended surface and to the same surface with `reversed' labels, and must satisfy further conditions. This additional duality pairing is crucial to the construction of the TQFT. 

In Grove's construction he requires that the modular functor with duality pairing satisfies a single additional condition (that the $(1,1)$-entry of the $S$-matrix is non-vanishing), and from this he produces a 1\-parameter family of non-extended TQFTs. Subsequently, Andersen-Ueno \cite{a06-mfgz} established the remarkable fact that modular functors in Walker's sense are determined by their genus zero data, and in the same paper they show (Proposition~7.1) that all modular functors satisfy Grove's condition. In contrast, Walker's construction, which produces what he calls a `TQFT with corners', requires that the modular functor satisfy a different additional condition (relating entries of the $S$-matrix via decompositions of a genus-two handlebody~\cite[equation~(4.7)]{walker-notes}. He speculates that this condition is always satisfied, but to our knowledge that remains an open question.   

These results all give (various versions of) an association
\begin{equation*}
        \{ \textrm{MF + duality pairing + conditions} \}  \to \{ \textrm{TQFT} \}
\end{equation*}
Note that the duality pairing required in this association may not exist, and in particular the association does not provide a bijection between modular functors and field theories---indeed there exist modular functors which do not give rise to any extended 3d TQFTs.

\subsection{Modular functors and modular tensor categories}

The main contribution of this paper is the first complete proof that extended 3\-dimensional TQFTs can be classified in terms of modular tensor categories; more precisely, that up to equivalence these TQFTs are in canonical bijective correspondence with isomorphism classes of modular tensor categories equipped with a choice of square root of the global dimension. While direct constructions of a TQFT from an MTC have been described in several places~\cite{bk01-ltc, t94-qik, kl01,walker-notes}, to our knowledge the seemingly easier direction, that the circle sector of an extended TQFT is an MTC, has not been previously established. 
Precise results relating MTCs to modular functors are given by each of the authors Bakalov and Kirillov~\cite{bk01-ltc}, Kerler and Lyubashenko~\cite{kl01}, Turaev~\cite{t94-qik}, and Walker~\cite{walker-notes}, but contrary to common perception, none of these authors claim to prove a bijection between the two notions.

The situation was recently described by Huang and Lepowsky \cite{MR3146015} in their construction of the modular tensor category associated to a conformal field theory (vertex operator algebra):   
\begin{quote}
        ``Many mathematicians have believed for a long time (at least twenty years) that these (rigid and) modular tensor categories must have been constructed either by using the works of Tsuchiya-Ueno-Yamada, and/or Beilinson-Feigin-Mazur and Bakalov-Kirillov, or by using the works of Kazhdan-Lusztig and Finkelberg. ...
        
        But unfortunately, it turns out that this belief has recently been shown to be wrong. ... it has now been known, and acknowledged, for a while, that [these works] cannot in fact be used to prove the rigidity of such tensor categories." 
\end{quote} 

Bakalov and Kirillov~\cite{bk01-ltc} give an approach to constructing a modular tensor category from their notion of modular functor, but run into and emphasize the following difficulty: they can only show that the circle sector is \textit{weakly rigid}, meaning that there is a contravariant monoidal involution $(-)^*$ with natural isomorphisms $\Hom_{\cat C} (S, T \otimes U) \simeq \Hom _\cat C (T^* \otimes S, U) \simeq \Hom_\cat C(S \otimes U^*, T)$. Weak rigidity is necessary but not sufficient for rigidity of a semisimple tensor category~\cite{crystals}.\footnote{Kuperberg's example~\cite{crystals} of a weakly-rigid non-rigid semisimple tensor categories has infinitely many simple objects; we do not know if there are examples of weakly-rigid non-rigid semisimple tensor categories with finitely many simples.} The same problem is encountered by Kerler and Lyubashenko~\cite[page~304]{kl01}.

Turaev obtains a concise result relating modular tensor categories and modular functors, but his notion of modular functor includes an additional axiom \cite[Axiom~1.5.8\ignore{,~page~245}]{t94-qik} which seems unmotivated from the functorial perspective, and indeed this axiom effectively asserts that the weakly rigid tensor category underlying the modular functor is in fact rigid. Assuming this axiom, Turaev constructs the following maps, where $\sqrt{\text{GD}}$ refers to a choice of square root of the global dimension:
\begin{multline*}
\{ C = \textrm{MTC + }\sqrt{\text{GD}} \}
\to \{C\textrm{-labeled MF + duality pairing + rigidity} \}
\\
\to \{ C = \textrm{MTC + }\sqrt{\text{GD}} \}
\end{multline*}
Concerning these he says~\cite[page~268]{t94-qik}:
\begin{quote}
        ``The composition of these arrows (in this order) is the identity. It seems natural to conjecture that the composition in the other order is also the identity."
\end{quote}

Walker~\cite[Theorem~6.4]{walker-notes}, inspired by Moore and Seiberg~\cite{ms89-cqcft}, classifies his notion of modular functor in terms of `basic data' satisfying a series of relations. It is not clear to us what relationship this data bears to the notion of modular tensor category.

Though as noted above they could not rely on the results of Bakalov--Kirillov, Turaev, or Walker to conclude rigidity, Huang and Lepwosky proved that, given not just a modular functor but a conformal field theory associated to a vertex operator algebra, the tensor category (associated to the circle) is indeed rigid.  They relate this rigidity proof back to Andersen--Ueno's work on genus-zero determination of modular functors:
 \begin{quote}
        ``In \cite{MR2568402} it was pointed out that while the \emph{statement} of rigidity in fact  involves only the genus-\emph{zero} conformal field theory, the \emph{proof} of rigidity in \cite{MR2468370} needs genus-\emph{one} conformal field theory."
 \end{quote}
In one of our main results, we are able to establish the rigidity of the tensor category (associated to a circle by an extended TQFT) by exploiting not a conformal structure but instead utilizing non-invertible 3\-dimensional bordisms.  Like Huang and Lepowsky's construction in the conformal setting, our rigidity proof passes through the genus-one sector.

It is not true that modular functors (in the 2-categorical functorial sense) all arise from modular tensor categories. The easiest counter-example is Dijkgraaf-Witten theory in finite characteristic dividing the order of the group.  This theory is well-defined on 1-manifolds, surfaces, and invertible 3-dimensional bordisms, but cannot be extended to all 3-dimensional bordisms.\footnote{In Dijkgraaf-Witten theory (see \cite{BBthesis, freed-higher, MR3348259}), the partition function of the circle is the representation category of the Drinfeld double of a finite group.  When the characteristic of the field divides the order of the group, this representation category is not semisimple (in particular not modular) and so by Theorem~\ref{thm:introcsig} the modular functor of this theory cannot be expanded to an extended TQFT.}

\subsection{Universal skein theory for TQFTs}
\label{sec:intro-isd}

In the early 1990s, Blanchet, Habegger, Masbaum, and Vogel~\cite{bhmv-tmi, bhmv-tqft, b00-hamc} developed a novel approach to the construction of 3d TQFTs based on the skein theory associated to the Kauffman bracket. In this approach the vector space associated to a surface, equipped with a choice of bounding handlebody, consists of linear combinations of links in the handlebody, modulo certain local relations. Variations of this construction were given by Lickorish~\cite{lsm92-tsm} and Roberts~\cite{r-st,robertsthesis}. It was extended to 3\-manifolds with corners by Gelca~\cite{g97-tqft}, and Andersen--Ueno~\cite{au14-cwrt} identified the resulting modular functor with one arising from conformal field theory.   
Walker \cite{walker-note-2006} described a method for computing the partition function of a field theory as a vector space of equivalence classes of local fields, and a connection between 4- and 3-dimensional field theory under which the 4d local-field viewpoint would explain the 3d skein-theoretic picture; motivated by this connection, he suggested that every 3d TQFT coming from an MTC should admit a `skein theoretic' description.  

In the course of establishing our classification theorem we will show that indeed every extended 3d TQFT admits an analogous `skein theoretic' description. Specifically we will see that the vector space assigned to a surface, by the extended TQFT associated to the MTC $\cat{S}$, can be identified with the vector space of ``internal string diagrams" (labeled by the objects and morphisms of $\cat{S}$) that can be drawn inside a handlebody bounding the surface. Unlike earlier approaches that take the skein theoretic description as an input, in our approach the skein description emerges from an analysis of the representation theory of the bordism 2\-category.

\subsection{Overview}

This is the fourth paper in a series. In the first paper~\cite{PaperI} we give a finite presentation of $\Bord$ in terms of generators and relations, and in the second paper~\cite{PaperII} we substantially simplify this presentation. In the third paper~\cite{PaperIII} we obtain finite presentations of $\Bordsig$, $\Bordcsig$, and $\Bordp$.

Here, in \autoref{sec:presentationsrepresentations} we outline the theory of presentations of symmetric monoidal 2\-categories, and discuss aspects of their representation theory in $k$\-linear categories. In \autoref{sec:geometricalpresentations} we recall the presentations of the bordism 2\-categories whose representations we will classify. In \autoref{sec:internal} we present the calculus of internal string diagrams, and use it to analyze linear representations of ribbon structures, showing that such representations correspond to ribbon categories. In \autoref{sec:modcatmodobj} we examine the representation theory of modular structures, and show that for every linear modular structure, the associated ribbon category is a modular tensor category. With this result, we then complete the proofs of the classification Theorems 1, 2, 3, and 4. In \autoref{sec:examples} we give examples of the internal string diagram calculus for calculations involving Dehn twists, mapping class groups, and lens spaces. In the Appendix we consider a variety of different target 2\-categories for TQFTs, and show that in each case the TQFTs factor through the sub-2\-category of \twovect consisting of the finite $k$-linear semisimple categories.

\subsection{Conventions}

Throughout this paper we adhere to the following conventions: 
\begin{itemize}
        \item \textbf{Base field.} We fix $k$ to be an algebraically-closed field of arbitrary characteristic. Whenever we say that a structure is linear, we mean with respect to $k$.
        \item \textbf{Cardinality issues.} To address potential  set-theoretical complications, we fix a chosen inaccessible cardinal $\kappa$. We will implicitly be utilizing $\kappa$\-small sets throughout this paper. Thus the objects of $\vect$ are the $\kappa$\-small $k$\-vector spaces, the objects of \ALG are the $\kappa$-small algebras, and the objects of $\twovect$ are the $\kappa$-small Cauchy-complete $\vect$-enriched categories, and so on. 
        \item \textbf{Weak structure.} All categorical structures are in their fully weak form, unless explicitly stated. In particular, this applies to the terms 2\-category, functor, natural transformation, modification, and monoid.

\item
\textbf{Graphical calculus.} We draw diagrams to represent composite 1\-morphisms in symmetric monoidal 2\-categories. That these diagrams are unambiguous follows from the coherence results of~\cite[Chapter~2]{CSPthesis}, which we will implicitly be using. We refer the reader to~\cite{b14-wd} and~\cite[Appendix~C]{PaperII} for a further discussion of diagram calculus in symmetric monoidal 2\-categories. 
\end{itemize}

\subsection{Acknowledgments}

We are indebted to John Baez, who encouraged and supported this project from the beginning, and provided key insights and inspiration.  We are grateful to Andr\'e Henriques, who provided crucial ideas, particularly regarding how to extend the classification results to the anomalous case.  We would like to thank Dan Freed for his sustained interest and for many informative conversations, and we also give a special acknowledgment to Kevin Walker, whose work has been highly influential to us.   We also appreciate helpful comments from Justin Roberts and Peter Teichner.

\section{Presentations and representations}
\label{sec:presentationsrepresentations}

In this section we review the theory of presentations of symmetric monoidal 2\-categories, and their linear representations in Cauchy-complete 2\_vector spaces.  Readers more interested in the geometry of 3-dimensional bordisms, who are willing to take on faith more technical aspects of the categorical algebra, may wish to skim or skip this section.

\subsection{Presentations of symmetric monoidal 2-categories}

We give a short account of the theory of presentations of symmetric monoidal 2\-categories, following the approach of the third author~\cite{CSPthesis}.


\tikzset{bot=true}
\setlength\obscurewidth{0pt}

We will be concerned in this paper with representations of certain higher algebraic structures: {\em finitely-presented symmetric monoidal 2\-categories}. These are the analogs of finite CW-complexes in the world of symmetric monoidal 2\-categories. A finite CW-complex is built inductively by specifying in each dimension a finite set of $n$-disks $D^n$ together with attaching maps from $\partial D^n$ to the $(n-1)$-skeleton (the space that has previously been constructed). If we are only interested in the 2\-type of the CW-complex, then we need only remember the 3-skeleton, and moreover the effect of the 3-disks is only to add relations to the second fundamental group.  A finitely-presented symmetric monoidal 2\-category is likewise built inductively by specifying in each dimension $n \in \{0,1,2\}$ a finite set of generating cells $C_n$, together with attaching maps from $\partial C_n$ to what has previously been constructed, and with a set of relations among the 2\-morphisms.

In more detail, a presentation $\mathcal G$ for a symmetric monoidal 2\-category consists of the following finite collection of data:
\begin{itemize}
\item generating objects;
\item generating 1\-morphisms, whose sources and targets are composites of generating objects;
\item generating 2\-morphisms, whose sources and targets are composites  of generating 1\-morphisms;
\item relations, which are equations between composites of generating 2\-morphisms.
\end{itemize}
At each level any composites are allowed that can be formed using the symmetric monoidal 2\-category structure. Examples of such presentations are given in \Autoref{sec:rightadjointpsmonoiddef,sec:modularobjectdefinition}.

Given a presentation $\mathcal G$, we write $\bicat F(\mathcal{G})$ for the fully-weak symmetric monoidal 2\-category generated by $\mathcal G$, in the sense of~\cite[page 161]{CSPthesis}. Such presented symmetric monoidal 2\-categories were called {\em computadic} in~\cite{CSPthesis}; another acceptable term might be {\em cofibrant}. Note that this differs somewhat from the notation used in~\cite{PaperII}, where $\bicat F(\mathcal{G})$ denotes a stricter object, the \textit{quasistrict} symmetric monoidal 2-category generated by $\mathcal{G}$, which we will denote by $\bicat F_\text{qs}(\mathcal{G})$ to avoid confusion. The symmetric monoidal 2\-categories $\bicat F(\mathcal{G})$  and $\bicat F_\qs(\mathcal{G})$ are related by a strict symmetric monoidal equivalence  \mbox{$\bicat F(\mathcal{G}) \to \bicat F_\qs(\mathcal{G})$}~\cite[Theorem~2.96]{CSPthesis}. This implies that we may directly employ the coherence results of~\cite[Chapter~2]{CSPthesis} to the 2\-category $\bicat F(\mathcal{G})$ and to functors out of it; in particular we are free to use the wire diagram calculus developed in~\cite{b14-wd} (see also~\cite[Appendix~C]{PaperII}.)

It is straightforward to describe strict symmetric monoidal functors \mbox{$Z :\bicat F(\mathcal{G}) \to \bicat C$}.  They are uniquely specified by the images of the generating objects, 1\-morphisms and 2\-morphisms; such an assignment gives rise to a strict functor precisely when the relations of $\mathcal{G}$ are satisfied.
\begin{defn}
        Given a symmetric monoidal 2\-category $\bicat C$ and a presentation $\mathcal G$, the 2\-category of {\em $\mathcal{G}$-structures in \bicat C}, denoted $\Rep _{\bicat C} (\mathcal{G})$, is defined as follows: 
\begin{itemize}
        \item an object is a strict symmetric monoidal functor $\bicat F(\mathcal G) \to \bicat C$;
        \item a 1\-morphism is a strict symmetric monoidal natural transformation;
        \item a 2\-morphism is a symmetric monoidal modification.
\end{itemize}
\end{defn}

\noindent
This definition of $\Rep _\bicat C (\G)$ can be spelled out explicitly as follows (see~\cite{CSPthesis} for details):
        \begin{itemize}
                \item An object $Z$ is an assignment of an object $Z(X)$ of $\bicat C$ for each generating object $X$ in $\mathcal G$, a 1\-morphism $Z(p)$ of $\bicat C$ for each generating 1\-morphism $p$ in $\mathcal G$, and a 2\-morphism $Z(\sigma)$ of $\bicat C$ for each generating 2\-morphism $\sigma$ in $\mathcal G$, subject to the obvious source-target matching conditions and the relations of $\mathcal G$.
                \item A 1\-morphism $\zeta : Z \rightarrow Z'$ is an assignment of a 1\-morphism \mbox{$\zeta_X : Z(X) \rightarrow Z'(X)$} of $\bicat C$ for each generating object $X$ of $\bicat{C}$, and a 2\-isomorphism {$\zeta_p : Z'(p) \circ \zeta_{s(p)} \xdoubleto{} \zeta_{t(p)} \circ Z_p$ }for each generating 1\-morphism $p$, where $\zeta_{s(p)}$ and $\zeta_{t(f)}$ refer to the tensor product of copies of the 1\-morphisms $\zeta_X$, appropriately parenthesized to match the source and target of $p$ respectively. This assignment $X \mapsto \zeta_X$ and $p \mapsto \zeta_p$ must be such that for each generating 2\-morphism $\sigma$ in $\mathcal{G}$, the corresponding diagram of 2\-morphisms in $\bicat{C}$ commutes.  

                \item A 2\-morphism $\omega : \zeta \xdoubleto{} \zeta'$ is an assignment of a 2\-morphism $\omega_X : \zeta_X \xdoubleto{} \zeta'_X$ for each generating object $X$ of $\bicat{C}$, such that for each generating 1\-morphism in $\mathcal{G}$ the corresponding diagram of 2\-morphisms in $\bicat{C}$ commutes. 
        \end{itemize}
        
Given two symmetric monoidal 2\-categories $\bicat{C}$ and $\bicat{D}$, we write $\bicat{SymBicat}(\bicat{C}, \bicat{D})$ for the 2\-category of (weak) symmetric monoidal homomorphisms from $\bicat{C}$ into $\bicat{D}$, (weak) symmetric monoidal transformations between them, and symmetric monoidal modifications, in the sense of~\cite{CSPthesis}; in particular, everything is as weak as possible.
\begin{theorem}[\cite{CSPthesis}, Theorems 2.96 and 2.78]
\label{thm:equivpres}
For a presentation $\G$, there is a natural equivalence of 2\-categories $\Rep _\bicat C (\G) \simeq \bicat{SymBicat}(\bicat F(\mathcal{G}), \bicat C)$. \qed
\end{theorem}
\noindent
The conclusion is that symmetric monoidal homomorphisms from a symmetric monoidal 2\-category generated by a presentation $\mathcal G$ into another symmetric monoidal 2\-category $\bicat C$ can equivalently be described in terms of \emph{strict} symmetric monoidal functors, i.e. $\mathcal G$-structures in $\bicat{C}$.

There is a naive notion of morphism of presentation $\G \to \G'$ which consists of maps taking the generating objects, 1\-morphisms, and 2\-morphisms of $\G$ to  generating objects, 1\-morphisms, and 2\-morphisms of $\G'$, provided the relations in $\G$ hold. Such a map induces a strict symmetric monoidal functor $\bicat F(\G) \to \bicat F(\G')$, but there are many such functors which do not arise from maps of presentations. For a fixed 2\-category $\bicat{C}$ we also obtain a strict functor of representation 2\-categories: $\Rep_{\G'}(\bicat C) \to \Rep_\G(\bicat C)$. This permits us to study representations of 2\-categories generated by complicated presentations by first studying representations of 2\-categories generated by simpler presentations. In particular we will often be in the situation of the following definition.

\begin{defn} \label{def:2extend}
A presentation \emph{k\-extends} a presentation $\mathcal T$ when it has just the same $j$\-morphism generators as $\mathcal T$ for $j<k$, and $j$\-morphism generators that include those of $\mathcal T$ for $j \geq k$.
\end{defn}

An algebra structure on a vector space $V$ can be transported across an isomorphism $V \simeq W$ to an algebra structure on the vector space $W$.  Similarly, higher-dimensional algebraic structures can be transported across isomorphisms.  One kind of such transportation is recorded in the following result.
\def\t{\mathrm{t}}
\begin{lemma}[{\cite[Lemma~A.19]{pstragowski-thesis}}]
\label{uniqueconjugate}
For a \G-structure $Z$ in a symmetric monoidal 2\-category \bicat C, given for each 1\-generator $p \in \G_1$ an invertible 2\-morphism \mbox{$\zeta_p : Z(p) \Rightarrow S_p$} for some 1\-cell $S_p$, there is a unique \G-structure $Z'$ in \bicat C such that the family $\zeta _p$ gives an identity-on-objects morphism of \G-structures. \qed
\end{lemma}

\noindent
Concretely, $Z'$ is defined as follows:
\begin{itemize}
\item On objects, $Z'$ agrees with $Z$.
\item On a 1\-morphism generator $p$, we have $Z'(p) := S_p$.
\item On a 2\-morphism generator $\mu$, we transport $Z(\mu)$ using the 2\-morphisms $\zeta _{-}$, as follows: $Z'(\mu) := \zeta_Q \circ Z(\mu) \circ \zeta_P ^\inv$, where $\mu : P \Rightarrow Q$ in $\bicat F (\G)$, and where $\zeta_P: Z(P) \Rightarrow Z'(P)$ and $\zeta_Q: Z(Q) \Rightarrow Z'(Q)$ are composites of the transport maps.
\end{itemize}

\renewcommand\hat[1]{\smash{\widehat{#1}}}
\label{sec:Cauchycomplete}

\subsection{Linear categories and bimodules}
\label{sec:linearcategories}

We will take representations of our presentations in a symmetric monoidal 2\-category \twovect of Cauchy-complete $k$\-linear categories, functors and natural transformations. We will also be interested in \bimod, the symmetric monoidal 2\-category of $k$\-linear categories, bimodules and natural transformations. In this section we introduce these 2\-categories and examine their relationship. We write $\vect$ for the symmetric monoidal category of vector spaces over $k$.

\begin{defn}[Linear category, enriched tensor product]
 A {\em linear category} is a category enriched over $\vect$, meaning its hom-sets are $k$\-vector spaces, and composition is $k$\-linear. The {\em enriched tensor product} $C \boxtimes D$ of two linear categories $C$ and $D$ has objects given by pairs $(c,d) \in \Ob (\cat C) \times \Ob (\cat D)$, and morphism vector spaces given by $\cat (C \otimes \cat D) \big( (c,d), (c', d') \big) = \cat C(c,c') \otimes_k \cat D(d,d')$.
 \end{defn}

\begin{defn}
A linear category is \emph{Cauchy complete} if it has all \textit{absolute colimits}, meaning that it has all those colimits which are preserved by all linear functors. (See \autoref{pro:CauchyComplete} for alternative characterizations of linear Cauchy-complete categories.)
\end{defn}

We can now define our main 2\-categories of interest.
\begin{defn}[See~\cite{t98-sskl}]
The symmetric monoidal 2\-category $\twovect$ is defined in the following way:
\begin{itemize}
\item objects are Cauchy-complete linear categories;
\item 1\-morphisms are linear functors;
\item 2\-morphisms are natural transformations;
\item the monoidal structure $\hat{\otimes}$ is the Cauchy completion of the enriched tensor product.
\end{itemize}
\end{defn}

\begin{defn}
The symmetric monoidal 2\-category $\bimod$ of linear categories and bimodules\footnote{What we are calling `bimodules' here are sometimes also called `profunctors' or `distributors'.} is defined as follows:
\begin{itemize}
\item objects are linear categories;
\item 1\-morphisms $P: \cat C \proarrow \cat D$ are \textit{bimodules}, defined as linear functors \mbox{$P : \cat D ^\op \boxtimes C \to \vect$};
\item 2\-morphisms are natural transformations between the associated linear functors;
\item the identity 1\-morphisms are given by the Hom-functor \mbox{$\cat{C}^\text{op} \boxtimes \cat{C} \rightarrow \vect$};
\item the monoidal structure is the enriched tensor product.
\end{itemize}
Composition of bimodules $P: \cat C \proarrow \cat D$ and $Q: \cat D \proarrow \cat E$ is defined as follows:
\begin{equation}
\label{bimodulecomposition}
(Q\circ P)(e,c) = \coprod_{d \in \cat D} \big(Q(e,d) \otimes P(d,c) \big) \, / \, {\sim},
\end{equation}
Here $c$, $d$ and $e$ are objects of \cat C, \cat D and \cat E respectively, and we quotient by the least equivalence relation generated by the relation
\begin{equation}
\label{eq:bimodulecomposition}
(q \cdot f, p) \sim (q,f \cdot p)
\end{equation}
for all $q \in Q(e,d)$, $p \in P(d',c)$, and $f: d \to d'$ in \cat D.
\end{defn}

There are interesting relationships between \twovect and \bimod, which we now explore.
\begin{defn}
Let $F \colon C \rightarrow D$ be a linear functor between linear categories. Then there are associated bimodules
\begin{align}
    F_* :  C & \proarrow D  \\
    F^* : D & \proarrow C
\end{align}
defined as $F_*(d,c) = \Hom(d, F(c))$ and $F^*(c,d) = \Hom(F(c), d)$. A bimodule which is isomorphic to $F_*$ (resp. $F^*$) for some linear functor $F$ is called {\em representable} (resp. {\em corepresentable}). 
\end{defn}

\noindent Sending a linear functor $F : \cat{C} \rightarrow D$ to its associated bimodule $F_* \colon \cat{C} \proarrow \cat{D}$ induces a covariant functor
\[
 (-)_* : \twovect \rightarrow \bimod \,.
\]

\begin{lemma}
\label{adj_lemma}
For a linear functor $F : C \rightarrow D$, we have $F_* \dashv F^*$ in $\bimod$.
\end{lemma}
\begin{proof} The bimodule composite $F^* \circ F_*$ is given by $(F^* \circ F_*)(c,c') \cong \Hom (F(c), F(c'))$. The unit and counit natural transformations
        \begin{align*}
                \eta _{c,c'} &: \Hom (c, c') \to \Hom (F(c), F(c')) \\
                \epsilon _{d,d'} &: (F_* \circ F^*) (d,d') = \oplus_{c \in C} \Hom (d, F(c)) \otimes \Hom ( F(c), d') \,/\sim  \\
                & \hspace{1cm} \to \Hom (d,d')
        \end{align*}
are induced by the functor $F$ and by composition respectively. 
\end{proof}

The following proposition, which is proved in \autoref{sec:cauchycompletionresult}, gives alternative characterizations of Cauchy completeness which are more useful in practice.

\begin{proposition}\label{pro:CauchyComplete}
For a linear category $\cat C$, the following are equivalent:
        \begin{enumerate}
                \item $\cat C$ is Cauchy-complete;
                \item $\cat C$ admits finite direct sums and all idempotents split;
\label{item:idempotentssplit}
                \item every bimodule $P: k \proarrow \cat C$ that is a left adjoint, is representable;
                \item for all small $\vect$-enriched categories $\cat B$, every bimodule $S: \cat B \proarrow \cat C$ that is a left adjoint, is representable. 
        \end{enumerate} 
\end{proposition}

Given a $\vect$-enriched category $\cat C$ we may form its \textit{biproduct completion} $\cat{Mat(C)}$, whose objects are finite tuples of objects of $\cat C$ and whose morphisms are matrices of morphisms from $\cat C$. The \textit{idempotent completion} of a category $\cat C$ has as object pairs $(x, p)$ where $x \in \cat C$ and $p: x \to x$ is an idempotent. A morphism from $(x,p)$ to $(x',p')$ is a morphism $f: x \to x'$ such that $f = p' f = f p$, with composition induced by composition in $\cat C$. The \emph{Cauchy completion} of $\cat C$, denoted $\hat{\cat C}$,  is the idempotent completion of $\cat{Mat(C)}$. Any functor from $\cat C$ to a Cauchy-complete category factors, up to essentially unique natural isomorphism, through the Cauchy completion of $\cat C$, via the canonical fully-faithful inclusion $i: \cat C \to \hat{\cat C}$. 

\begin{example}
        If $A$ is an algebra, viewed as a one-object $\vect$-enriched category, then the Cauchy completion $\hat{A}$ is equivalent to the category of finitely generated projective $A$-modules. The inclusion $i$ sends the unique object of the category $A$ to $A$ itself viewed as a free $A$-module.  
\end{example}

\begin{proposition}\label{pro:Cauchycompletionisequiv}
        Let $C$ be a $\vect$-enriched category and let $i: C \to \hat{C}$ be the natural inclusion into its Cauchy completion. Then the bimodule $i_*: C \proarrow \hat{C}$ is an equivalence in $\bimod$. 
\end{proposition}

\begin{proof}
        The inverse equivalence, if it exists, is necessarily given by the right adjoint $i^*$. In the proof of \autoref{adj_lemma}, we saw that the unit $\eta$ for the adjunction is the action of $i$ on the hom vector-spaces. Since $i$ is fully faithful, $\eta$ is an isomorphism. Similarly, the counit $\epsilon$ was induced by composition:
        \begin{equation*}
                \varepsilon_{x,x'}: (i_* \circ i^*) (x,x') = \bigoplus_{c \in C} \Hom_{\hat C}(x, i(c)) \otimes \Hom_{\hat C}(i(c), x') / \sim \to \Hom_{\hat C}(x, x').
        \end{equation*}
        Moreover, $\varepsilon_{x,x'}$ is certainly an isomorphism when $x$ and $x'$ are both in the image of $C$.  
        
        Let $\O_1$ be the collection of all objects $x \in \hat C$ such that $\varepsilon_{x,c}$ is an isomorphism for all $c \in C$. Then we see that $\O_1$
        \begin{itemize}
                \item contains $C$;
                \item is closed under direct sums;
                \item is closed under retracts. 
        \end{itemize} 
        Hence we conclude that $\O_1 = \hat C$. Now let $\O_2$ be the collection of all $x' \in \hat C$ such that $\varepsilon_{x,x'}$ is an isomorphism for all $x \in \hat C$. Again we have that $\O_2$
        \begin{itemize}
                \item contains $C$;
                \item is closed under direct sums;
                \item is closed under retracts. 
        \end{itemize}
        We conclude that $\O_2 = \hat C$, and hence that $i_*: C \proarrow \hat C$ is an equivalence. 
\end{proof}

\subsection{Extending presentations}
Given a presentation $\mathcal{G}$, we write $\mathcal{G}^\L$ for the extension of $\mathcal{G}$ obtained by freely adding right adjoints for all the generating 1\-morphisms. For each generating 1-morphism $f: x \to y$ of $\mathcal{G}$, we add to $\mathcal{G}^\L$ a generating 1-morphism $f^\uR: y \to x$ and generating 2-morphisms $\varepsilon_f: f \circ f^\uR \Rightarrow \id_y$ and $\eta_f: \id_x \Rightarrow f^\uR \circ f$ such that the following `zig-zag' relations hold:
\begin{align*}
\id_f &= (\varepsilon_f * \id_f) \circ (\id_f * \eta_f)
\\
\id_{f^\uR} &= (\id_{f^R} * \varepsilon_f) \circ (\eta_f * \id_{f^\uR})
\end{align*}

In this subsection we investigate the relationship between $\Rep_\bicat C(\G^\L)$ and $\Rep _{\bicat C ^\L} (\G)$, where $\bicat C^\L$ is the full subcategory of $\bicat C$ containing all objects, and all 1-morphisms that are left adjoints.  (This relationship justifies the notation $\mathcal{G}^\L$ for the addition of right adjoints to the presentation $\mathcal{G}$.) In particular, we demonstrate that the cores (ie maximal sub-2-groupoids) of these two representation 2-categories are equivalent. (This equivalence will play an important role in our proof of \autoref{equivalentMstructures}, where we show that a balanced braided equivalence of modular tensor categories can be promoted to an equivalence of linear modular structures.)
\begin{defn}
For a 2\-category $\bicat C$, its \emph{core} $\overline{\bicat{C}}$ is the sub-2-category with the same objects as $\bicat C$, with 1\-morphisms being all 1-morphisms of $\bicat C$ that are equivalences, and with 2\-morphisms being all invertible 2\-morphisms between these equivalences.
\end{defn}
\begin{proposition}
\label{pro:addingrightadjointsforrepinprof}
Let $\mathcal G$ be a presentation, and let $\bicat C$ be a symmetric monoidal 2\-category. Let $\mathcal G^\L$ denote the extension of $\mathcal G$ obtained by freely adding right adjoints for all the generating 1\-morphisms. Then the forgetful functor between the cores of the representation categories $\overline{\Rep_\bicat{C} (  \mathcal G^\L )} \to \overline{\Rep_{\bicat C ^\L} (\mathcal G)}$ is an equivalence.
\end{proposition}
\begin{proof}
We will construct an inverse equivalence $(-)' : \overline{\Rep _{\bicat C^\L} (\G)} \to \overline{\Rep_{\bicat C} (\G^\L)}$ for the forgetful functor. For this proof, $L:A \to B$ is any 1\-morphism generator of \G, and the corresponding adjunction in $\bicat F(\G^\L)$ is $L \dashv R$, witnessed by 2\-morphism generators $\eta$ and $\epsilon$ of $\G^\L$.

On an object $Z$ in $\overline{\Rep _{\bicat C^\L} (\G)}$, define $Z'$ as follows. On generating objects, 1\-morphisms and 2\-morphisms of $\G^\L$ which arise from $\G$, $Z'$ agrees with $Z$.  For the other generating 1\-morphisms, we choose $Z'(R) = Z'(L)^*$, a right adjoint of $Z'(L)$ in $\bicat C$. For the other generating 2\-morphisms, we choose $\Phi'(\eta)$ and $\Phi'(\epsilon)$ to be witnesses for $Z'(L) \dashv Z'(L)^*$ in $\bicat C$. Since the choice of an adjoint is a contractible choice.

On a 1\-morphism $\mu: Y \to Z$ in $\overline{\Rep _{\bicat C^\L} (\G)}$, define $\mu' : Y' \to Z'$ as follows. On generating objects and 1\-morphisms of $\G^\L$ that arise from $\G$, $\mu'$ agrees with $\mu$. On a right-adjoint generator $R$ of $\G^\L$, we define $\mu'(R)$ as follows, using string diagram notation in the 2\-category \bicat C:
\tikzset{box/.style={draw, font=\tiny, node on layer=foreground, fill=white, inner sep=0pt, minimum height=0.34cm, minimum width=1cm}}
\tikzset{halfbox/.style={box, minimum width=0.3cm}}
\tikzset{every node/.style={font=\tiny}}
\tikzset{every picture/.style={xscale=0.7, yscale=0.5}}
\def\p{{\vphantom{(}}}
\begin{equation}
\label{muRassignment}
\begin{tz}
\node (b) [box] at (0,0) {$\mu_R$};
\draw (-0.5,1) node [above] {$Z(R)$} to (-0.5,-1) node [below] {$\mu_A\p$};
\draw (0.5,1) node [above] {$\mu_B$} to (0.5,-1) node [below] {$Y(R)\p$};
\end{tz}
\quad=\quad
\begin{tz}
\node (1) [box] at (0,0) {$\mu_L^\inv$};
\node (2) [box] at (-1,-1) {$Z(\eta)$};
\node (3) [box] at (1,1) {$Y(\epsilon)$};
\draw (-1.5,2) node [above, font=\tiny] {$Z(R)\p$} to +(0,-3);
\draw (1.5,-2) node [below] {$Y(R)$} to +(0,3);
\draw (-0.5,2) node [above] {$\mu_B\p$} to +(0,-3);
\draw (0.5,-2) node [below] {$\mu_A\p$} to +(0,3);
\end{tz}
\end{equation}
Note we are using here that $\mu$ has invertible components. The generators $\eta$ and $\epsilon$ give rise to equations that this 2\-morphism must satisfy, as follows:
\begin{align}
\begin{tz}
\node [box] at (0,0) {$\mu_R$};
\node [box] at (1,-1) {$Y(\eta)$};
\node [box] at (1,1) {$\mu_L$};
\draw (-0.5,2) node [above] {$Z(R)$} to +(0,-4) node [below] {$\mu_A\p$};
\draw (0.5,2) node [above] {$Z(L)$} to +(0,-3);
\draw (1.5,2) node [above] {$\mu_A$} to +(0,-3);
\end{tz}
&=
\begin{tz}
\node [box] at (0,0) {$Z(\eta)$};
\draw (-0.5,2) node [above] {$Z(R)$} to +(0,-2);
\draw (0.5,2) node [above] {$Z(L)$} to +(0,-2);
\draw (1.5,2) node [above] {$\mu_A$} to +(0,-4) node [below] {$\mu_A\p$};
\end{tz}
&
\begin{tz}[xscale=-1, yscale=-1]
\node [box] at (0,0) {$\mu_L$};
\node [box] at (1,-1) {$Z(\epsilon)$};
\node [box] at (1,1) {$\mu_R$};
\draw (-0.5,2) node [below] {$Y(L)$} to +(0,-4) node [above] {$\mu_A$};
\draw (0.5,2) node [below] {$Y(R)$} to +(0,-3);
\draw (1.5,2) node [below] {$\mu_A\p$} to +(0,-3);
\end{tz}
&=
\begin{tz}[xscale=-1, yscale=-1]
\node [box] at (0,0) {$Y(\epsilon)$};
\draw (-0.5,2) node [below] {$Y(L)$} to +(0,-2);
\draw (0.5,2) node [below] {$Y(R)$} to +(0,-2);
\draw (1.5,2) node [below] {$\mu_A\p$} to +(0,-4) node [above] {$\mu_A$};
\end{tz}
\end{align}
We verify the first of these as follows:
\begin{equation*}
\begin{tz}
\node [box] at (0,0) {$\mu_R$};
\node [box] at (1,-1) {$Y(\eta)$};
\node [box] at (1,1) {$\mu_L$};
\draw (-0.5,2) node [above] {$Z(R)$} to +(0,-4) node [below] {$\mu_A\p$};
\draw (0.5,2) node [above] {$Z(L)$} to +(0,-3);
\draw (1.5,2) node [above] {$\mu_A$} to +(0,-3);
\end{tz}
=
\begin{tz}
\node [box] (1) at (1,0) {$Y(\eta)$};
\node [box] (2) at (-1,1) {$\mu_L ^\inv$};
\node [box] (4) at (0,2) {$Y(\epsilon)$};
\node [box] (5) at (-2,0) {$Z(\eta)$};
\node [box] (3) at (0,3) {$\mu_L$};
\draw (-2.5,4) node [above] {$Z(R)$} to +(0,-4);
\draw (-0.5,4) node [above] {$Z(L)$} to +(0,-1);
\draw (0.5,4) node [above] {$\mu_A\p$} to +(0,-1);
\draw (-1.5,1) to +(0,-1);
\draw (-0.5,2) to +(0,-3) node [below] {$\mu_A\p$};
\draw (0.5,2) to +(0,-2);
\draw (0.5,2.72) to [out=down, in=up] (1.5,2) to +(0,-2);
\draw (-0.5,2.72) to [out=down, in=up] (-1.5,2) to (-1.5,0);
\end{tz}
=
\begin{tz}
\node [box] at (0,0) {$Z(\eta)$};
\node [box] at (1,1) {$\mu_L ^\inv$};
\node [box] at (1,2) {$\mu_L$};
\draw (-0.5,3) node [above] {$Z(R)$} to +(0,-3);
\draw (0.5,3) node [above] {$Z(L)$} to +(0,-3);
\draw (1.5,3) node [above] {$\mu_A$} to +(0,-4) node [below] {$\mu_A\p$};
\end{tz}
=
\begin{tz}
\node [box] at (0,0) {$Z(\eta)$};
\draw (-0.5,2) node [above] {$Z(R)$} to +(0,-2);
\draw (0.5,2) node [above] {$Z(L)$} to +(0,-2);
\draw (1.5,2) node [above] {$\mu_A$} to +(0,-4) node [below] {$\mu_A\p$};
\end{tz}
\end{equation*}
The second follows similarly. These equations are seen to be equivalent to the assignment~\eqref{muRassignment} in the case that $\mu$ is invertible, so this definition of $\mu'$ is unique given the constraint that $(-)'$ is to be an equivalence.

We must show that $\mu_R$ is invertible. Its inverse takes the following form, where we use cusps to represent the adjoint equivalence structure between $\mu_A$ and its inverse, and between $\mu_B$ and its inverse:
\begin{equation}
\begin{tz}
\node (b) [box] at (0,0) {$\mu_R ^\inv$};
\draw (-0.5,1) node [above] {$\mu_A\p$} to (-0.5,-1) node [below] {$Z(R)\p$};
\draw (0.5,1) node [above] {$Y(R)$} to (0.5,-1) node [below] {$\mu(B)\p$};
\end{tz}
\quad=\quad
\begin{tz}
\node [box] at (0,0) {$\mu_L$};
\draw (-0.5,0) to (-0.5,1) to [out=up, in=down] (0.5,2) to (1.5,2) to [out=down, in=up] (2.5,1) to (2.5,-3.5) node [below] {$Z(R)\p$};
\draw (0.5,0) to (0.5,-1) to [out=down, in=up] (-0.5,-2) to (-1.5,-2) to [out=up, in=down] (-2.5,-1) to (-2.5,3) node [above] {$Y(R)$};
\draw (0.5,0) to (0.5,0.4) to [out=up, in=-145] (1,1) to [out=-35, in=up] (1.5,0.4) to (1.5,-2) to [out=down, in=35] (-1,-3.5) to [out=145, in=down] (-3.5,-2) to (-3.5,3) node [above] {$\mu_A$};
\node [box] at (-1,-2.3) {$Y(\eta)$};
\node [box] at (1,2.3) {$Z(\epsilon)$};
\draw (-0.5,0) to (-0.5,-0.4) to [out=down, in=35] (-1,-1) to [out=145, in=down] (-1.5,-0.4) to (-1.5,2) to [out=up, in=-145] (1,3.5) to [out=-35, in=up] (3.5,2) to (3.5,-3.5) node [below] {$\mu_B\p$};
\end{tz}
\end{equation}
Composing this with $\mu_R$ on either side, it can be shown that the identity is obtained, using the duality equations for $\eta$ and $\epsilon$ under the images of $Y$ and $Z$, and the adjoint equivalence structures.

Finally, for a 2\-morphism $\zeta: \mu \doubleto \nu$ in $\overline{\Rep_{\bicat C^\L} (\G)}$, define $\zeta' = \zeta$. We must check the following equation:
\begin{equation}
\begin{tz}
\node (1) [box] at (0,0) {$\mu_R$};
\node (2) [halfbox] at (0.5,1) {$\zeta$};
\draw (-0.5,2) node [above] {$Z(R)$} to +(0,-3) node [below] {$\mu_A\p$};
\draw (0.5,2) node [above] {$\nu_B\p$} to +(0,-3) node [below] {$Y(R)$};
\end{tz}
=
\begin{tz}
\node (1) [box] at (0,1) {$\nu_R$};
\node (2) [halfbox] at (-0.5,0) {$\zeta$};
\draw (-0.5,2) node [above] {$Z(R)$} to +(0,-3) node [below] {$\mu_A\p$};
\draw (0.5,2) node [above] {$\nu_B\p$} to +(0,-3) node [below] {$Y(R)$};
\end{tz}
\end{equation}
We verify this as follows:
\begin{align*}
\nonumber
&\begin{tz}
\node (1) [box] at (0,0) {$\mu_R$};
\node [halfbox] at (0.5,1) {$\zeta$};
\draw (-0.5,2) node [above] {$Z(R)$} to +(0,-3) node [below] {$\mu_A\p$};
\draw (0.5,2) node [above] {$\nu_B\p$} to +(0,-3) node [below] {$Y(R)$};
\end{tz}
=
\begin{tz}
\node (1) [box] at (0,0) {$\mu_L^\inv$};
\node (2) [box] at (-1,-1) {$Z(\eta)$};
\node (3) [box] at (1,1) {$Y(\epsilon)$};
\node [halfbox] at (-0.5,1) {$\zeta$};
\draw (-1.5,2) node [above, font=\tiny] {$Z(R)\p$} to +(0,-3);
\draw (1.5,-2) node [below] {$Y(R)$} to +(0,3);
\draw (-0.5,2) node [above] {$\nu_B\p$} to +(0,-3);
\draw (0.5,-2) node [below] {$\mu_A\p$} to +(0,3);
\end{tz}
=
\begin{tz}
\node [box] at (0,0) {$\mu_L^\inv$};
\node [box] at (0,2) {$\nu_L$};
\node [box] at (0,3) {$\nu_L^\inv$};
\node [box] at (-1,-1) {$Z(\eta)$};
\node [box] at (1,4) {$Y(\epsilon)$};
\node [halfbox] at (-0.5,1) {$\zeta$};
\draw (-1.5,5) node [above, font=\tiny] {$Z(R)\p$} to +(0,-6);
\draw (1.5,-2) node [below] {$Y(R)$} to +(0,6);
\draw (-0.5,5) node [above] {$\nu_B\p$} to +(0,-6);
\draw (0.5,-2) node [below] {$\mu_A\p$} to +(0,6);
\end{tz}
\\[-0.5cm]
&=
\begin{tz}
\node [box] at (0,0) {$\mu_L^\inv$};
\node [box] at (0,1) {$\mu_L$};
\node [box] at (0,3) {$\nu_L^\inv$};
\node [box] at (-1,-1) {$Z(\eta)$};
\node [box] at (1,4) {$Y(\epsilon)$};
\node [halfbox] at (0.5,2) {$\zeta$};
\draw (-1.5,5) node [above, font=\tiny] {$Z(R)\p$} to +(0,-6);
\draw (1.5,-2) node [below] {$Y(R)$} to +(0,6);
\draw (-0.5,5) node [above] {$\nu_B\p$} to +(0,-6);
\draw (0.5,-2) node [below] {$\mu_A\p$} to +(0,6);
\end{tz}
=
\begin{tz}
\node [box] at (0,3) {$\nu_L^\inv$};
\node [box] at (-1,2) {$Z(\eta)$};
\node [box] at (1,4) {$Y(\epsilon)$};
\node [halfbox] at (0.5,2) {$\zeta$};
\draw (-1.5,5) node [above, font=\tiny] {$Z(R)\p$} to +(0,-3);
\draw (1.5,1) node [below] {$Y(R)$} to +(0,3);
\draw (-0.5,5) node [above] {$\nu_B\p$} to +(0,-3);
\draw (0.5,1) node [below] {$\mu_A\p$} to +(0,3);
\end{tz}
=
\begin{tz}
\node (1) [box] at (0,1) {$\nu_R$};
\node (2) [halfbox] at (-0.5,0) {$\zeta$};
\draw (-0.5,2) node [above] {$Z(R)$} to +(0,-3) node [below] {$\mu_A\p$};
\draw (0.5,2) node [above] {$\nu_B\p$} to +(0,-3) node [below] {$Y(R)$};
\end{tz}
\end{align*}
This completes the proof.
\end{proof}

\section{Presentations of the bordism categories}
\label{sec:geometricalpresentations}

Here we review the presentations given in~\cite{PaperIII} of the geometrical 2\-categories \Bord, \Bordcsig, \Bordsig, and \Bordp. These are the presentations whose linear representations we will classify in \autoref{sec:classification}. Throughout we make use of a cobordism-style notation, and also manifold terminology such as circle, surface, genus and so on, to discuss the structures given below. This is a useful notation, but we emphasize that there is no cobordism theory as such in this paper; technically we are studying representations of particular algebraic structures. The connection to cobordism theory arises entirely from the papers~\cite{PaperI, PaperII, PaperIII}. When we refer to rotations and daggers of the equations below, this is in reference to~\cite[Appendix~A]{PaperII}.

\subsection{The monoid and balanced presentations}
\label{sec:rightadjointpsmonoiddef}

In this section we will consider the monoid presentation, the balanced presentation, and variations thereof.
 Following our weakness convention, by a `monoid' in a monoidal 2\-category we mean the fully weak structure, which is sometimes called a `pseudomonoid'~\cite{s04-fmp}.
\smallbordisms
 \setlength\obscurewidth{0pt}
 \begin{defn}
 \label{monoidpresentation}
 The \emph{monoid presentation} $\P$ is defined as follows:
 \setlength\obscurewidth{0pt}
 \begin{itemize}
  \item Generating object:
 \begin{equation}
 \begin{tz}
         \node[Cyl, top, height scale=0]  at (0,0) {};
 \end{tz}
 \end{equation}
  \item Generating 1\-morphisms:
  \begin{equation}
  \label{onecellgeneratorsmonoid}
  \begin{tz}
         \node[Pants, top, bot] (A) at (0,0) {};
         \node[Cup, top] (C) at (4,0.1) {};
  \end{tz}
  \end{equation}
  \item Invertible generating 2\-morphisms:
 \begin{calign}
 \label{MCmonoid}
 \begin{tz}
         \node[Pants, top, bot, wide] (A) at (0,0) {};
     \node[Pants,  bot, anchor=belt] (B) at (A.leftleg) {};    
     \node[Cyl, bot, anchor=top] at (A.rightleg) {}; 
 \end{tz}    
 \rightleftdoublearrow{\alpha}{\alpha^\inv}
 \begin{tz}
         \node[Pants, top, bot, wide] (A) at (0,0) {};
     \node[Pants,  bot, anchor=belt] (B) at (A.rightleg) {};    
     \node[Cyl, bot, anchor=top] at (A.leftleg) {}; 
 \end{tz} 
&
 \begin{tz}
         \node[Pants, top, bot] (A) at (0,0) {};
         \node[Cyl, bot, anchor=top] at (A.leftleg) {};
         \node[Cup] at (A.rightleg) {};  
 \end{tz}
 \rightleftdoublearrow{\rho}{\rho^\inv}
 \begin{tz}
         \node[Cyl, bot, top, tall] at (0,0) {};
 \end{tz}
 \rightleftdoublearrow{\lambda^\inv}{\lambda}
 \begin{tz}
         \node[Pants, top, bot] (A) at (0,0) {};
         \node[Cyl, bot, anchor=top] at (A.rightleg) {};
         \node[Cup] at (A.leftleg) {};   
 \end{tz}
 \end{calign}
 \end{itemize}
 This data must satisfy the following relations:
 \begin{itemize}
 \item (Inverses) Each of the invertible generating 2\-morphisms $\omega$ satisfies \mbox{$\omega \omega^\inv = \id$} and $\omega^\inv \omega = \id$. 

\item (Monoidal) The generators in \eqref{MCmonoid} obey the pentagon and triangle equations:
\begin{equation} \label{eq:pentagon}
\begin{tz}[xscale=2.1, yscale=1.2]
\node (1) at (0,0)
{
$\begin{tikzpicture}
        \node [Pants, wider, top] (A) at (0,0) {};
        \node [Pants, wide, anchor=belt] (B) at (A.leftleg) {};
        \node [Cyl, tall, anchor=top] (C) at (A.rightleg) {};
        \node[Pants, anchor=belt] (D) at (B.leftleg) {};
        \node[Cyl, anchor=top] (E) at (B.rightleg) {};
        \selectpart[green] {(A-belt) (B-leftleg) (A-rightleg)};
        \selectpart[red] {(A-leftleg) (D-leftleg) (B-rightleg)};
\end{tikzpicture}$
};
\node (2) at (1,1)
{
$\begin{tikzpicture}
        \node [Pants, verywide, top] (A) at (0,0) {};
        \node [Pants, anchor=belt] (B) at (A.rightleg) {};
        \node [Cyl, anchor=top] (C) at (A.leftleg) {};
        \node[Pants, anchor=belt] (D) at (C.bottom) {};
        \node[Cyl, anchor=top] (E) at (B.leftleg) {};
        \node[Cyl, anchor=top] (F) at (B.rightleg) {};
        \selectpart[green] {(A-leftleg) (A-rightleg) (D-leftleg) (F-bottom)};
\end{tikzpicture}$
};
\node (3) at (2,1)
{
$\begin{tikzpicture}
        \node [Pants, verywide, top] (A) at (0,0) {};
        \node [Pants, anchor=belt] (B) at (A.leftleg) {};
        \node[Cyl, anchor=top] (C) at (A.rightleg) {};
        \node[Cyl, anchor=top] (D) at (B.leftleg) {};
        \node[Cyl, anchor=top] (E) at (B.rightleg) {};
        \node[Pants, anchor=belt] (F) at (C.bottom) {};
        \selectpart[green] {(A-belt) (B-leftleg) (A-rightleg)};
\end{tikzpicture}$
};
\node (4) at (3,0)
{
$\begin{tikzpicture}
        \node [Pants, wider, top] (A) at (0,0) {};
        \node [Cyl, tall, anchor=top] (B) at (A.leftleg) {};
        \node[Pants, anchor=belt, wide] (C) at (A.rightleg) {};
        \node[Cyl, anchor=top] (D) at (C.leftleg) {};
        \node[Pants, anchor=belt] (E) at (C.rightleg) {};
\end{tikzpicture}$
};
\node (5) at (1,-1)
{
$\begin{tikzpicture}
        \node [Pants, veryverywide, top] (A) at (0,0) {};
        \node [Pants, wide, anchor=belt] (B) at (A.leftleg) {};
        \node [Cyl, tall, anchor=top] (C) at (A.rightleg) {};
        \node[Pants, anchor=belt] (D) at (B.rightleg) {};
        \node[Cyl, anchor=top] (E) at (B.leftleg) {};
                \selectpart[green] {(A-belt) (B-leftleg) (A-rightleg)};
\end{tikzpicture}$
};
\node (6) at (2,-1)
{
$\begin{tikzpicture}
        \node [Pants, veryverywide, top] (A) at (0,0) {};
        \node [Pants, wide, anchor=belt] (B) at (A.rightleg) {};
        \node [Cyl, tall, anchor=top] (C) at (A.leftleg) {};
        \node[Pants, anchor=belt] (D) at (B.leftleg) {};
        \node[Cyl, anchor=top] (E) at (B.rightleg) {};
        \selectpart[green] {(A-rightleg) (D-leftleg) (B-rightleg)};
\end{tikzpicture}$
};
\begin{scope}[double arrow scope]
    \draw (1) --  node[above left, green, pos=0.65]{$\alpha$} (2);
    \draw (2) --  node[above]{$\interchangor$} (3);
    \draw (3) --  node[above right, pos=0.35]{$\alpha$} (4);
    \draw (1) --  node[below left, red, pos=0.65]{$\alpha$} (5);
    \draw (5) --  node[below]{$\alpha$} (6);
    \draw (6) --  node[below right, pos=0.35]{$\alpha$} (4);
\end{scope}
\end{tz}
\end{equation}
\begin{equation}
\label{eq:triangle}
\begin{tz}[xscale=1, yscale=2, every to/.style={out=down,in=up}]
\node (1) at (-1,1)
{
$\begin{tikzpicture}
    \node (A) [Pants] at (0,0) {};
    \node (B) [Pants, wide, top, anchor=leftleg] at (A.belt) {};
    \node (C) [Cup] at (A.rightleg) {};
    \node [Cyl, anchor=top] (D) at (A.leftleg) {};
    \node [Cyl, tall, anchor=top] (E) at (B.rightleg) {};
    \selectpart[green] {(B-belt) (A-leftleg) (B-rightleg)};
    \selectpart[red] {(B-leftleg) (D-bottom) (A-rightleg)};
\end{tikzpicture}$
};
\node (2) at (1,1)
{
$\begin{tikzpicture}
    \node (A) [Pants] at (0,0) {};
    \node (B) [Pants, wide, top, anchor=rightleg] at (A.belt) {};
    \node (C) [Cup] at (A.leftleg) {};
    \node (D) [Cyl, anchor=top] at (A.rightleg) {};
    \node [Cyl, tall, anchor=top] at (B.leftleg) {};
    \selectpart[green] {(B-rightleg) (A-leftleg) (D-bottom)};
\end{tikzpicture}$
};
\node (3) at (0,0)
{
$\begin{tikzpicture}
    \node (A) [Pants, top] at (0,0) {};
\end{tikzpicture}$
};
\begin{scope}[double arrow scope]
    \draw (1) --  node[above, green]{$\alpha$} (2);
    \draw (2) --  node[below right, pos=0.35]{$\lambda$} (3);
    \draw (1) --  node[below left, red, pos=0.35]{$\rho$} (3);
\end{scope}
\end{tz}
\end{equation}
\end{itemize}
\end{defn}

\noindent
The 2-morphism labelled $\varphi$ is an \textit{interchanger}, one of a canonical family of 2-morphisms in any monoidal 2-category that switches the 1-morphism composition order of two 1-morphisms that have been tensored together (see~\cite{b14-wd} for details).

A representation of $\P$ in a symmetric monoidal 2\-category $\bicat C$ is the same thing as a monoid in $\bicat C$. Let us record this explicitly for the case $\bicat C = \twovect$.

\begin{lemma}
\label{lem:monoidrep}
A representation of the monoid presentation $\P$ in $\twovect$ is the same thing as a Cauchy-complete linear monoidal category. A morphism between representations of $\P$ is the same thing as a monoidal functor, and a 2-morphism is the same thing as a monoidal transformation.
\end{lemma}

This can be generalized to include braidings, twists, and other familiar categorical structures.

\begin{defn}
\label{balancedpresentation}        
The \emph{balanced presentation} $\B$ is a 2\-extension of the monoid presentation $\P$, with the following additional generators and relations:
\begin{itemize}
\item  Additional invertible generating 2\-morphisms:
\begin{align}
\begin{tz}
        \node[Pants, top, bot] (A) at (0,0) {};
\end{tz}
&\rightleftdoublearrow{\beta}{\beta^\inv}
\begin{tz}
        \node[Pants, top, bot] (A) at (0,0) {};
        \node[BraidB, anchor=topleft, bot] at (A.leftleg) {};
\end{tz}
&
\begin{tz}
        \node[Cyl, top, bot] (A) at (0,0) {};
\end{tz}
&\rightleftdoublearrow{\theta}{\theta^\inv}
\begin{tz}
        \node[Cyl, top, bot] (A) at (0,0) {};
\end{tz} \label{RC}
\end{align}

\item (Additional inverses) Each of the additional invertible generating 2\-morphisms $\omega$ satisfies \mbox{$\omega \omega^\inv = \id$} and $\omega^\inv \omega = \id$. 
 \item (Balanced) The generating data gives a braided monoidal object equipped with a compatible twist:
\begin{equation}
\label{eq:hexagon}
\begin{tz}[xscale=2.4, yscale=1.2]
\node (1) at (0,0)
{
$\begin{tikzpicture}
    \node [Pants, wide, top] (A) at (0,0) {};
    \node [Pants, anchor=belt] (B) at (A.leftleg) {};
    \node [Cyl, anchor=top] (C) at (A.rightleg) {};
    \selectpart[green]{(A-leftleg) (B-leftleg) (B-rightleg)};
\end{tikzpicture}$
};

\node (2) at (0.75,1)
{
$\begin{tikzpicture}
    \node [Pants, wide, top] (A) at (0,0) {};
    \node [Pants, anchor=belt] (B) at (A.rightleg) {};
    \node [Cyl, anchor=top] (C) at (A.leftleg) {};
    \selectpart[green]{(A-belt) (A-leftleg) (A-rightleg)};
\end{tikzpicture}$
};

\node (3) at (1.5,1)
{
$\begin{tikzpicture}
    \node [Pants, wide, top] (A) at (0,0) {};
    \node[BraidB, wide, anchor=topleft] (B) at (A.leftleg) {};
    \node [Pants, anchor=belt] (C) at (B.bottomright) {};
    \node [Cyl, anchor=top] (D) at (B.bottomleft) {};
\end{tikzpicture}$
};

\node (4) at (2.25,1)
{
$\begin{tikzpicture}
    \node [Pants, wide, top] (A) at (0,0) {};
    \node [Pants, anchor=belt] (C) at (A.leftleg) {};
    \node [Cyl, anchor=top] (D) at (A.rightleg) {};
    \node [BraidB, anchor=topleft] (E) at (C.rightleg) {};
    \node[Cyl, anchor=top] (F) at (C.leftleg) {};
    \node[BraidB, anchor=topleft] (G) at (F.bottom) {};
    \node[Cyl, anchor=top] (H) at (G.bottomright) {};
    \node[Cyl, anchor=top] (I) at (G.bottomleft) {};
    \node[Cyl, anchor=top, tall] (J) at (E.bottomright) {};
    \selectpart[green]{(A-belt) (C-leftleg) (D-bottom)};
\end{tikzpicture}$
};

\node (5) at (3,0)
{
$\begin{tikzpicture}
    \node [Pants, wide, top] (A) at (0,0) {};
    \node [Pants, anchor=belt] (C) at (A.rightleg) {};
    \node [Cyl, anchor=top] (D) at (A.leftleg) {};
    \node [BraidB, anchor=topleft] (E) at (C.leftleg) {};
    \node[Cyl, tall, anchor=top] (F) at (A.leftleg) {};
    \node[BraidB, anchor=topleft] (G) at (F.bottom) {};
    \node[Cyl, anchor=top] (H) at (E.bottomright) {};
\end{tikzpicture}$
};

\node (6) at (1,-1)
{
$\begin{tikzpicture}
    \node [Pants, wide, top] (A) at (0,0) {};
    \node [Pants, anchor=belt] (B) at (A.leftleg) {};
    \node [Cyl, anchor=top, tall] (C) at (A.rightleg) {};
    \node [BraidB, anchor=topleft] (D) at (B.leftleg) {};
    \selectpart[green]{(A-belt) (B-leftleg) (A-rightleg)};
\end{tikzpicture}$
};

\node (7) at (2.0,-1)
{
$\begin{tikzpicture}
    \node [Pants, wide, top] (A) at (0,0) {};
    \node [Pants, anchor=belt] (B) at (A.rightleg) {};
    \node [Cyl, anchor=top] (C) at (A.leftleg) {};
    \node[BraidB, anchor=topleft] (D) at (C.bottom) {};
    \node[Cyl, anchor=top] (E) at (B.rightleg) {};
    \selectpart[green]{(A-rightleg) (B-leftleg) (B-rightleg)};
\end{tikzpicture}$
};

\begin{scope}[double arrow scope]
    \draw (1) --  node[above left]{$\alpha$} (2);
    \draw (2) --  node[above]{$\beta$} (3);
    \draw[-=, shorten <=-4pt] (3) --  (4);
    \draw (4) --  node[above right]{$\alpha$} (5);
    \draw (1) --  node[below left]{$\beta$} (6);
    \draw (6) --  node[below]{$\alpha$} (7);
    \draw (7) --  node[below right]{$\beta$} (5);
\end{scope}
\end{tz}
\end{equation}
\begin{equation}
\label{balanced1}
\begin{tz}[xscale=1.4, yscale=1.5]

\node (1) at (0,0)
{
$\begin{tikzpicture}
        \node[Pants, top] (A) at (0,0) {};
        \selectpart[green, inner sep=1pt]{(A-belt)};
        \selectpart[red] {(A-leftleg) (A-rightleg) (A-belt)};
\end{tikzpicture}$
};
\node (2) at (1,0)
{
$\begin{tikzpicture}
        \node[Pants, top] (A) at (0,0) {};
\end{tikzpicture}$
};

\node (3) at (0,-1)
{
$\begin{tikzpicture}
        \node[Pants, top] (A) at (0,0) {};
        \selectpart[green, inner sep=1pt]{(A-leftleg)};
\end{tikzpicture}$
};

\node (4) at (1,-1)
{
$\begin{tikzpicture}
        \node[Pants, top] (A) at (0,0) {};
        \selectpart[green, inner sep=1pt]{(A-rightleg)};
\end{tikzpicture}$
};

\begin{scope}[double arrow scope]
    \draw (1) -- node[above, green] {$\theta$} (2);
    \draw (1) -- node[left, red] {$\beta^2$} (3);
    \draw (3) -- node[below] {$\theta$} (4);
    \draw (4) -- node[right] {$\theta$} (2);
\end{scope}
\end{tz}
\end{equation}
\begin{equation}
\label{balanced2}
\begin{tz}[xscale=1.6, yscale=2]
\node (1) at (0,0) {$
\begin{tikzpicture}
        \node[Cup, top] (C) at (0,0) {};
        \selectpart[green, inner sep=1pt]{(C-center)};
\end{tikzpicture}$};
\node (2) at (1,0) {$
\begin{tikzpicture}
        \node[Cup, top] (C) at (0,0) {};
        \selectpart[green, inner sep=1pt]{(C-center)};
\end{tikzpicture}$};
\begin{scope}[double arrow scope]
    \draw (1) --  node[above]{$\theta$} (2);
\end{scope}
\end{tz}
\quad = \quad
\id
\end{equation}
\end{itemize}
Note that the second hexagon axiom is redundant in the presence of a twist~\cite{js91-gtc}.
\end{defn}
Recall that a \textit{balanced braided monoidal category} is a braided monoidal category $C$ that is  equipped with a natural isomorphism $\theta \colon \id_C \doubleto \id_C$ called the \textit{twist}, satisfying the following equations:
\tikzset{morphism/.style={draw, fill=white}}
\normalbordisms
\begin{align}
\label{eq:balanced}
\begin{tz}[xscale=0.8]
\draw (0,0) to (0,3);
\draw (1,0) to (1,3);
\node [morphism, minimum width=0.8cm] at (0.5,1.5) {$\theta_{A \otimes B}$};
\end{tz}
\gap&=\gap
\begin{tz}
\draw (1,0) to [out=up, in=down] (0,1);
\draw [black strand] (0,0) to [out=up, in=down] (1,1);
\draw [black strand] (1,1) to [out=up, in=down] (0,2);
\draw [black strand] (0,2) to [out=up, in=down] (0,3);
\draw [black strand] (0,1) to [out=up, in=down] (1,2);
\draw [black strand] (1,2) to [out=up, in=down] (1,3);
\node [morphism, anchor=south] at (0,2) {$\theta_A$};
\node [morphism, anchor=south] at (1,2) {$\theta_B$};
\end{tz}
&
\begin{tz}
\draw [black strand, dotted] (0,-1) to (0,1);
\node [morphism] at (0,0) {$\theta _I$};
\end{tz}
\gap&=\gap
\begin{tz}
\draw [black strand, dotted] (0,-1) to (0,1);
\end{tz}
\end{align}
The second equation here says $\theta _I = \id_I$. We can classify representations of the balanced presentation $\B$ as follows.

\begin{lemma}
\label{lem:balancedrep}
A representation of the balanced presentation $\B$ in $\twovect$ is the same thing as a Cauchy-complete linear balanced braided monoidal category. A morphism of representations of $\B$ is the same thing as a braided monoidal functor that preserves the twist, and a 2-morphism is the same thing as a monoidal natural transformation.
\end{lemma}

Because they play important roles, we give explicit descriptions of the right-adjoint balanced presentation $\B^\L$ and the right-adjoint monoid presentation~$\P^\L$.
\smallbordisms
 \begin{defn}
 \label{rightadjointmonoidpresentation}
 The \emph{right-adjoint balanced presentation} $\B^\L$ (resp. \emph{right-adjoint monoid presentation} $\P^\L$) is a 1\-extension of the balanced presentation $\B$ (resp. monoid presentation $\P$), with the following additional generators and relations:
\setlength\obscurewidth{0pt}
\begin{itemize}
\item Additional generating 1\-morphisms:
\begin{equation}
        \label{onecellgeneratorsmonoid}
        \begin{tz}
                \node[Copants, top, bot] (B) at (2,0) {};
                \node[Cap, bot] (D) at (6,-0.1) {};
        \end{tz}
\end{equation}
\item 
Additional generating 2\-morphisms:
 \begin{align} 
 \begin{tz} 
  \node[Cyl, tall, top, bot] (A) at (0,0) {};
  \node[Cyl, tall, top, bot] (B) at (2*\cobwidth, 0) {};
 \end{tz}
 &\longxdoubleto{\eta}
 \begin{tz} 
  \node[Pants, bot] (A) at (0,0) {};
  \node[Copants, top, bot, anchor=belt] at (A.belt) {};
 \end{tz}
&
 \begin{tz} 
  \node[Pants, top, bot] (A) at (0,0) {};
  \node[Copants, bot, anchor=leftleg] at (A.leftleg) {};
 \end{tz}
 &
 \longxdoubleto{\epsilon}
 \begin{tz} 
  \node[Cyl, top, bot, tall] (A) at (0,0) {};
 \end{tz} \label{A1monoid}
\\[5pt]
 \begin{tz}
  \draw[green] (0,0) rectangle (0.6, -0.6);  
 \end{tz}
 &\longxdoubleto{\nu}
 \begin{tz}
         \node[Cap, bot] (A) at (0,0) {};
         \node[Cup] at (0,0) {};
 \end{tz}
&
 \begin{tz}
         \node[Cup, top] (A) at (0,0) {};
         \node[Cap, bot] (B) at (0,-2*\cobheight) {};
 \end{tz}
&\longxdoubleto{\mu}
 \begin{tz}
         \node[Cyl, top, bot, tall] (A) at (0,0) {};
 \end{tz}
 \label{A2monoid}
 \end{align}
 \item (Adjunction) The following composites are the identity, which means that the data \eqref{A1monoid} witnesses $\tikztinypants \dashv \tikztinycopants$ and the data \eqref{A2monoid} witnesses $\tikztinycup \dashv \tikztinycap$:
\begin{calign}
\label{adj_nu_mu_monoid}
\begin{tz}
     \node (1) [Cup, top] at (0,0) {};
     \node (2) [Cup, invisible] at (0,-1.5\cobheight) {};
     \node (3) [Cap, invisible] at (0,-1.5\cobheight) {};
     \selectpart[green]{(3) (2)}
\end{tz}
\longxdoubleto{\nu}
\begin{tz}
     \node (1) [Cup, top] at (0,0) {};
     \node (2) [Cap] at (0,-1.5\cobheight) {};
     \node (3) [Cup] at (0,-1.5\cobheight) {};
     \selectpart[green]{(1-center) (2-center)}
\end{tz}
\longxdoubleto{\mu}
\begin{tz}
    \node[Cyl, top, height scale=1.5] (X) at (0,0) {};
    \node[Cup] at (X.bottom) {};
\end{tz}
&
\begin{tz}
     \node (1) [Cap] at (0,-1.5\cobheight) {};
     \node (2) [Cup, invisible] at (0,0) {};
     \node (3) [Cap, invisible] at (0,0) {};
     \selectpart[green]{(3) (2)}
\end{tz}
\longxdoubleto{\nu}
\begin{tz}
     \node (1) [Cap] at (0,-1.5\cobheight) {};
     \node (2) [Cap] at (0,0) {};
     \node (3) [Cup] at (0,0) {};
     \selectpart[green]{(1-center) (2-center)}
\end{tz}
\longxdoubleto{\mu}
\begin{tz}
    \node[Cyl, bot, height scale=1.5] (X) at (0,0) {};
    \node[Cap] at (X.top) {};
\end{tz}
\\[5pt]
\label{adj_eta_epsilon_monoid}
\begin{tz}
     \node [Pants, top] (A) at (0,0) {};
     \node[Cyl, anchor=top, height scale=2] (B) at (A.leftleg) {};
     \node[Cyl, anchor=top, height scale=2] (C) at (A.rightleg) {};
     \selectpart[green]{(A-leftleg) (A-rightleg) (B-bottom) (C-bottom)}
\end{tz}
\longxdoubleto{\eta}
\begin{tz}
     \node [Pants, bot, top] (A) at (0,0) {};
     \node [Copants, anchor=leftleg] (B) at (A.leftleg) {};
     \node [Pants, anchor=belt, bot] at (B.belt) {};
     \selectpart [green] {(A-leftleg) (A-belt) (A-rightleg) (B-belt)};
\end{tz}
\longxdoubleto{\epsilon}
\begin{tz}
     \node [Pants] (A) at (0,0) {};
     \node [Cyl, tall, bot, anchor=bot, top] at (A.belt) {};
\end{tz}
&
\begin{tz}
     \node [Copants] (A) at (0,0) {};
     \node[Cyl, anchor=bottom, height scale=2, top] (B) at (A.leftleg) {};
     \node[Cyl, anchor=bottom, height scale=2, top] (C) at (A.rightleg) {};
     \selectpart[green]{(C-bottom) (C-top) (B-bottom) (B-top)}
\end{tz}
\longxdoubleto{\eta}
\begin{tz}
     \node [Copants, bot] (A) at (0,0) {};
     \node [Pants, anchor=leftleg] (B) at (A.leftleg) {};
     \node [Copants, anchor=belt, top] (C) at (B.belt) {};
     \selectpart [green]{(B-leftleg) (C-belt) (B-rightleg) (A-belt)};
\end{tz}
\longxdoubleto{\epsilon}
\begin{tz}
     \node [Copants, top] (A) at (0,0) {};
     \node [Cyl, tall, bot, anchor=top] at (A.belt) {};
\end{tz}
\end{calign}
\end{itemize}
\end{defn}
We can classify the core of the 2-category of representations of $\B^\L$ in $\bimod$ as follows.
\begin{lemma} There is an equivalence of 2-categories
\[
 \overline{\Rep_{\bimod}(\B^\L)} \stackrel{\sim}{\rightarrow} \overline{\Rep_{\twovect}(\B)}.
\]
A similar equivalence holds for representations of $\P^\L$. 
\end{lemma}
\begin{proof}
This is a direct application of \autoref{pro:addingrightadjointsforrepinprof} and \autoref{pro:leftadjointbimodule}.
\end{proof}

\subsection{The ribbon presentation}

Some categorical structures can be described as representations of a presentation in $\bimod$, but not in $\twovect$. This is the case with Cauchy-complete ribbon categories. 

\begin{defn}
\label{defn_ribbonpresentation}
The \emph{ribbon presentation} $\Rib$ is a 2\-extension of the right-adjoint balanced presentation $\B^\L$, with the following additional generators and relations:
\begin{itemize}  
\item (Inverse Frobeniusators) Additional generating 2\-morphisms:
\begin{calign}
\begin{tz}
                \node[Copants, top] (A) at (0,0) {};
                \node[Pants, anchor=belt] (B) at (A.belt) {};
\end{tz} 
\longxdoubleto{\phiN^\inv}
\begin{tz}
 \node[Pants, top, bot] (A) at (0,0) {};
 \node[Cyl, bot, anchor=top] (B) at (A.leftleg) {};
 \node[Copants, bot, anchor=leftleg] (C) at (A.rightleg) {};
 \node[Cyl, top, bot, anchor=bottom] (D) at (C.rightleg) {}; 
\end{tz}
&
\begin{tz}
                \node[Copants, top] (A) at (0,0) {};
                \node[Pants, anchor=belt] (B) at (A.belt) {};
\end{tz} 
\longxdoubleto{\phiM^\inv}
\begin{tz}
 \node[Pants, top, bot] (A) at (0,0) {};
 \node[Cyl, bot, anchor=top] (B) at (A.rightleg) {};
 \node[Copants, bot, anchor=rightleg] (C) at (A.leftleg) {};
 \node[Cyl, top, bot, anchor=bottom] (D) at (C.leftleg) {}; 
\end{tz}
\end{calign}
         \item (Rigidity) Write $\phiN$ for the following composite (`left Frobeniusator'):
         \begin{align} \label{defn_of_phileft}
         \phiN \quad&:=\quad
        \begin{tz}
         \node[Pants, top, bot] (A) at (0,0) {};
         \node[Cyl, bot, anchor=top] (B) at (A.leftleg) {};
         \node[Copants, bot, anchor=leftleg] (C) at (A.rightleg) {};
         \node[Cyl, top, bot, anchor=bottom] (D) at (C.rightleg) {}; 
         \selectpart[green, inner sep=1pt] {(A-belt) (D-top)};
        \end{tz}
        \longxdoubleto{\eta}
        \begin{tz}
         \node[Copants, top, wide, bot] (F) at (0,0) {};
         \node[Pants, bot, wide, anchor=belt] (G) at (F.belt) {};
         \node[Pants, bot, anchor=belt] (A) at (G.leftleg) {};
         \node[Cyl, bot, anchor=top] (B) at (A.leftleg) {};
         \node[Copants, bot, anchor=leftleg] (C) at (A.rightleg) {};
         \node[Cyl, bot, anchor=bottom] (X) at (C.rightleg) {}; 
         \selectpart[green] {(F-belt) (A-leftleg) (X-bottom)};
        \end{tz}
        \longxdoubleto{\alpha}
        \begin{tz}
         \node[Copants, top, wide, bot] (F) at (0,0) {};
         \node[Pants, bot, wide, anchor=belt] (G) at (F.belt) {};
         \node[Pants, bot, anchor=belt] (A) at (G.rightleg) {};
         \node[Cyl, tall, bot, anchor=top] (B) at (G.leftleg) {};
         \node[Copants, bot, anchor=leftleg] (C) at (A.leftleg) {};
         \selectpart[green] {(G-rightleg) (A-leftleg) (A-rightleg) (C-belt)};
        \end{tz}
        \longxdoubleto{\epsilon}
        \begin{tz}
                \node[Copants, top, bot] (A) at (0,0) {};
                \node[Pants, bot, anchor=belt] (B) at (A.belt) {};
        \end{tz}
        \intertext{Similarly, write $\phiM$ for $\phiN$ rotated about the $z$-axis (`right Frobeniusator'):}
         \label{defn_of_phiright}
         \phiM \quad&:=\quad
        \begin{tz}
         \node[Pants, top, bot] (A) at (0,0) {};
         \node[Cyl, bot, anchor=top] (B) at (A.rightleg) {};
         \node[Copants, bot, anchor=rightleg] (C) at (A.leftleg) {};
         \node[Cyl, top, bot, anchor=bottom] (D) at (C.leftleg) {}; 
         \selectpart[inner sep=1pt, green] {(D-top) (A-belt)};
         \end{tz}
        \longxdoubleto{\eta}
        \begin{tz}
         \node[Copants, top, wide, bot] (F) at (0,0) {};
         \node[Pants, bot, wide, anchor=belt] (G) at (F.belt) {};
         \node[Pants, bot, anchor=belt] (A) at (G.rightleg) {};
         \node[Cyl, bot, anchor=top] (B) at (A.rightleg) {};
         \node[Copants, bot, anchor=rightleg] (C) at (A.leftleg) {};
         \node[Cyl, bot, anchor=bottom] (X) at (C.leftleg) {}; 
         \selectpart[green] {(F-belt) (A-rightleg) (X-bottom)};
        \end{tz}
        \longxdoubleto{\alpha^\inv}
        \begin{tz}
         \node[Copants, top, wide, bot] (F) at (0,0) {};
         \node[Pants, bot, wide, anchor=belt] (G) at (F.belt) {};
         \node[Pants, bot, anchor=belt] (A) at (G.leftleg) {};
         \node[Cyl, tall, bot, anchor=top] (B) at (G.rightleg) {};
         \node[Copants, bot, anchor=rightleg] (C) at (A.rightleg) {};
         \selectpart[green] {(G-leftleg) (A-leftleg) (A-rightleg) (C-belt)};
        \end{tz}
        \longxdoubleto{\epsilon}
        \begin{tz}
                \node[Copants, top, bot] (A) at (0,0) {};
                \node[Pants, bot, anchor=belt] (B) at (A.belt) {};
        \end{tz}
        \end{align}
The additional relations say that $\phiN \phiN^\inv = \id$, $\phiN^\inv \phiN = \id$, $\phiM \phiM^\inv = \id$ and $\phiM^\inv \phiM = \id$.

        \item (Ribbon) The twist satisfies the following equation:
\begin{equation}
        \label{tortile}
        \begin{tz}[xscale=1.4, yscale=2]
        \node (1) at (0,0)
        {
        $\begin{tikzpicture}
                \node[Pants] (A) at (0,0) {};
                \node[Cap] at (A.belt) {};
                \selectpart[green, inner sep=1pt]{(A-leftleg)};
        \end{tikzpicture}$
        };
        \node (2) at (1,0)
        {
        $\begin{tikzpicture}
                \node[Pants] (A) at (0,0) {};
                \node[Cap] at (A.belt) {};
        \end{tikzpicture}$
        };
        \begin{scope}[double arrow scope]
            \draw (1) -- node[above] {$\theta$} (2);
        \end{scope}
        \end{tz}
        \quad = \quad
        \begin{tz}[xscale=1.4, yscale=2]
        \node (1) at (0,0)
        {
        $\begin{tikzpicture}
                \node[Pants] (A) at (0,0) {};
                \node[Cap] at (A.belt) {};
                \selectpart[green, inner sep=1pt]{(A-rightleg)};
        \end{tikzpicture}$
        };
        \node (2) at (1,0)
        {
        $\begin{tikzpicture}
                \node[Pants] (A) at (0,0) {};
                \node[Cap] at (A.belt) {};
        \end{tikzpicture}$
        };
        \begin{scope}[double arrow scope]
            \draw (1) -- node[above] {$\theta$} (2);
        \end{scope}
        \end{tz} 
        \end{equation}
\end{itemize}
\end{defn}

\begin{defn}
\label{def:ribboncategory}
A \textit{ribbon category} is a balanced braided  monoidal category which is rigid, and where the twist satisfies the ribbon condition:\footnote{By `rigid' we mean there exist duals for each object, not that there are chosen duals.  Any condition, such as \eqref{eq:ribboncondition}, that refers to particular duals, units, and/or counits, should be interpreted to mean that the condition holds for any (equivalently, for all) choices of duals.}
\tikzset{morphism/.style={draw, fill=white}}
\normalbordisms
\begin{align}
\label{eq:ribboncondition}
\begin{tz}
\draw (0,0) to (0,-1) to [out=down, in=down, looseness=2] (1,-1) to (1,0);
\node [morphism, anchor=south] at (0,-1) {$\theta_A$};
\end{tz}
\gap&=\gap
\begin{tz}
\draw (0,0) to (0,-1) to [out=down, in=down, looseness=2] (1,-1) to (1,0);
\node [morphism, anchor=south] at (1,-1) {$\theta_{A^*}$};
\end{tz}
\end{align}
\end{defn}

\noindent It is not obvious that representations of $\R$ in $\bimod$ correspond precisely to Cauchy-complete linear ribbon categories; in \autoref{thm:repofribisribbon}, we show at least that given a representations of \R in \bimod, the image of the circle is a Cauchy-complete linear ribbon category. 

\subsection{The modular presentation}
\label{sec:modularobjectdefinition}
 
\begin{defn} \label{defn_presentation_of_B}  The {\em modular presentation} $\M$ is a 2\-extension of the ribbon presentation $\R$, with the following additional generators and relations:
\begin{itemize}
\item Additional generating 2\-morphisms:
\begin{align}
\label{A1}
\begin{tz} 
 \node[Pants, bot] (A) at (0,0) {};
 \node[Copants, top, bot, anchor=belt] at (A.belt) {};
\end{tz}
&\longxdoubleto{\eta ^\dag}
\begin{tz} 
 \node[Cyl, tall, top, bot] (A) at (0,0) {};
 \node[Cyl, tall, top, bot] (B) at (2*\cobwidth, 0) {};
\end{tz}
&
\begin{tz} 
 \node[Cyl, top, bot, tall] (A) at (0,0) {};
\end{tz}
&\longxdoubleto{\epsilon ^\dag}
\begin{tz} 
 \node[Pants, top, bot] (A) at (0,0) {};
 \node[Copants, bot, anchor=leftleg] at (A.leftleg) {};
\end{tz}
\\[5pt]
\label{A2}
\begin{tz}
        \node[Cap, bot] (A) at (0,0) {};
        \node[Cup] at (0,0) {};
\end{tz}
&\longxdoubleto{\nu ^\dag}{}
&
\begin{tz}
        \node[Cyl, top, bot, tall] (A) at (0,0) {};
\end{tz}
&\longxdoubleto{\mu ^\dag}
\begin{tz}
        \node[Cup, top] (A) at (0,0) {};
        \node[Cap, bot] (B) at (0,-2*\cobheight) {};
\end{tz}
\end{align}
\item (Additional rigidity) The inverse of $\phiN$ has the following explicit formula:
\begin{align}
\label{explicit_phileft_inverse}
\phiN^\inv \quad&=\quad  \begin{tz}
                \node[Copants, top] (A) at (0,0) {};
                \node[Pants, anchor=belt] (B) at (A.belt) {};
                \selectpart[green, inner sep=1pt] {(B-rightleg)};
\end{tz} 
\longxdoubleto{\epsilon^\dagger}
\begin{tz}
 \node[Copants, top, wide, bot] (F) at (0,0) {};
 \node[Pants, bot, wide, anchor=belt] (G) at (F.belt) {};
 \node[Pants, bot, anchor=belt] (A) at (G.rightleg) {};
 \node[Cyl, tall, bot, anchor=top] (B) at (G.leftleg) {};
 \node[Copants, bot, anchor=leftleg] (C) at (A.leftleg) {};
 \selectpart[green]{(F-belt) (G-leftleg) (A-rightleg)};
\end{tz}
\longxdoubleto{\alpha^\inv}
\begin{tz}
 \node[Copants, top, wide, bot] (F) at (0,0) {};
 \node[Pants, bot, wide, anchor=belt] (G) at (F.belt) {};
 \node[Pants, bot, anchor=belt] (A) at (G.leftleg) {};
 \node[Cyl, bot, anchor=top] (B) at (A.leftleg) {};
 \node[Copants, bot, anchor=leftleg] (C) at (A.rightleg) {};
 \node[Cyl, bot, anchor=bottom] (X) at (C.rightleg) {}; 
 \selectpart[green]{(F-leftleg) (F-rightleg) (G-leftleg) (G-rightleg)};
\end{tz}
\longxdoubleto{\eta^\dagger}
\begin{tz}
 \node[Pants, top, bot] (A) at (0,0) {};
 \node[Cyl, bot, anchor=top] (B) at (A.leftleg) {};
 \node[Copants, bot, anchor=leftleg] (C) at (A.rightleg) {};
 \node[Cyl, top, bot, anchor=bottom] (D) at (C.rightleg) {}; 
\end{tz}
\intertext{Similarly $\phiM$ has the following explicit inverse:}
\label{explicit_phiright_inverse}
\phiM^\inv \quad&=\quad
\begin{tz}
                \node[Copants, top] (A) at (0,0) {};
                \node[Pants, anchor=belt] (B) at (A.belt) {};
                \selectpart[green, inner sep=1pt] {(B-leftleg)};
\end{tz} 
\longxdoubleto{\epsilon^\dagger}
\begin{tz}
 \node[Copants, top, wide, bot] (F) at (0,0) {};
 \node[Pants, bot, wide, anchor=belt] (G) at (F.belt) {};
 \node[Pants, bot, anchor=belt] (A) at (G.leftleg) {};
 \node[Cyl, tall, bot, anchor=top] (B) at (G.rightleg) {};
 \node[Copants, bot, anchor=rightleg] (C) at (A.rightleg) {};
 \selectpart[green]{(F-belt) (A-leftleg) (G-rightleg)};
\end{tz}
\longxdoubleto{\alpha}
\begin{tz}
 \node[Copants, top, wide, bot] (F) at (0,0) {};
 \node[Pants, bot, wide, anchor=belt] (G) at (F.belt) {};
 \node[Pants, bot, anchor=belt] (A) at (G.rightleg) {};
 \node[Cyl, bot, anchor=top] (B) at (A.rightleg) {};
 \node[Copants, bot, anchor=rightleg] (C) at (A.leftleg) {};
 \node[Cyl, top, bot, anchor=bottom] (X) at (C.leftleg) {}; 
 \selectpart[green]{(F-leftleg) (F-rightleg) (G-leftleg) (G-rightleg)};
\end{tz}
\longxdoubleto{\eta^\dagger}
\begin{tz}
 \node[Pants, top, bot] (A) at (0,0) {};
 \node[Cyl, bot, anchor=top] (B) at (A.rightleg) {};
 \node[Copants, bot, anchor=rightleg] (C) at (A.leftleg) {};
 \node[Cyl, top, bot, anchor=bottom] (D) at (C.leftleg) {}; 
\end{tz}
\end{align}

\item  (Additional adjunction) The following composites are the identity, which means that the data~\eqref{A1} expresses $\tikztinycopants \dashv \tikztinypants$ and the data~\eqref{A2} expresses $\tikztinycap \dashv \tikztinycup$:
\begin{calign}
\label{adj_nu_mu_monoid}
\begin{tz}
        \node[Cyl, top, height scale=1.5] (X) at (0,0) {};
        \node[Cup] at (X.bottom) {};
        \selectpart[green]{(X-top) (X-bottom)}
\end{tz}
\longxdoubleto{\mu^\dag}
\begin{tz}
        \node (1) [Cup, top] at (0,0) {};
        \node (2) [Cap] at (0,-1.5\cobheight) {};
        \node (3) [Cup] at (0,-1.5\cobheight) {};
        \selectpart[green]{(3) (2)}
\end{tz}
\longxdoubleto{\nu^\dag}
\begin{tz}
        \node (1) [Cup, top] at (0,0) {};
        \node (2) [Cup, invisible] at (0,-1.5\cobheight) {};
        \node (3) [Cap, invisible] at (0,-1.5\cobheight) {};
\end{tz}
&
\begin{tz}
        \node[Cyl, bot, height scale=1.5] (X) at (0,0) {};
        \node[Cap] at (X.top) {};
        \selectpart[green]{(X-top) (X-bottom)}
\end{tz}
\longxdoubleto{\mu^\dag}
\begin{tz}
        \node (1) [Cap] at (0,-1.5\cobheight) {};
        \node (2) [Cap] at (0,0) {};
        \node (3) [Cup] at (0,0) {};
        \selectpart[green]{(3) (2)}
\end{tz}
\longxdoubleto{\nu^\dag}
\begin{tz}
        \node (1) [Cap] at (0,-1.5\cobheight) {};
        \node (2) [Cup, invisible] at (0,0) {};
        \node (3) [Cap, invisible] at (0,0) {};
\end{tz}
\\[5pt]
\label{adj_eta_epsilon_monoid}
\begin{tz}
        \node [Pants] (A) at (0,0) {};
        \node (X) [Cyl, tall, bot, anchor=bot, top] at (A.belt) {};
        \selectpart[green]{(X-top) (X-bottom)}
\end{tz}
\longxdoubleto{\epsilon^\dag}
\begin{tz}
     \node [Pants, bot, top] (A) at (0,0) {};
     \node [Copants, anchor=leftleg] (B) at (A.leftleg) {};
     \node [Pants, anchor=belt, bot] (C) at (B.belt) {};
     \selectpart [green] {(A-leftleg) (A-rightleg) (C-leftleg)};
\end{tz}
\longxdoubleto{\eta^\dag}
\begin{tz}
        \node [Pants, top] (A) at (0,0) {};
        \node[Cyl, anchor=top, height scale=2] (B) at (A.leftleg) {};
        \node[Cyl, anchor=top, height scale=2] (C) at (A.rightleg) {};
\end{tz}
&
\begin{tz}
        \node [Copants, top] (A) at (0,0) {};
        \node (X) [Cyl, tall, bot, anchor=top] at (A.belt) {};
        \selectpart[green]{(A-belt) (X-bottom)}
\end{tz}
\longxdoubleto{\epsilon^\dag}
\begin{tz}
        \node [Copants, bot] (A) at (0,0) {};
        \node [Pants, anchor=leftleg] (B) at (A.leftleg) {};
        \node [Copants, anchor=belt, top] (C) at (B.belt) {};
        \selectpart [green]{(C-leftleg) (C-rightleg) (B-rightleg) (B-leftleg)};
\end{tz}
\longxdoubleto{\eta^\dag}
\begin{tz}
     \node [Copants] (A) at (0,0) {};
     \node[Cyl, anchor=bottom, height scale=2, top] (B) at (A.leftleg) {};
     \node[Cyl, anchor=bottom, height scale=2, top] (C) at (A.rightleg) {};
\end{tz}
\end{calign}

 \item (Pivotality) The following equation holds, together with its rotated form $\eqref{piv_on_sphere}{}^z$ (see~\cite[Appendix~B]{PaperII} for our rotation notation):
\begin{equation}
\label{piv_on_sphere}
\begin{tz}
 \node[Cap] (A) at (0,0) {};
 \node[Cup, bot=false] (B) at (0,0) {};
 \node[Cobordism Bottom End 3D] (AA) at (0,0) {};              
 \selectpart[green, inner sep=1pt] {(AA)};
\end{tz}
\longxdoubleto{\epsilon ^\dagger}
\begin{tz}
        \node [Pants] (A) at (0,0) {};
        \node [Cap] at (A.belt) {};
        \node [Copants, anchor=leftleg] (B) at (A.leftleg) {};
        \node [Cup] at (B.belt) {};
        \selectpart[green, inner sep=1pt] {(A-rightleg)};
\end{tz}
\longxdoubleto{\mu ^\dagger}
\begin{tz}
        \node [Pants] (A) at (0,0) {};
        \node [Cap] at (A.belt) {};
        \node [Cyl, height scale=1.5, anchor=top] (B) at (A.leftleg) {};
        \node [Cup] (C) at (A.rightleg) {};
        \node [Copants, anchor=leftleg] (D) at (B.bottom) {};
        \node [Cap, bot=false] (E) at (D.rightleg) {};
        \node [Cobordism Bottom End 3D] (B) at (D.rightleg) {};
        \node [Cup] at (D.belt) {};
        \selectpart[green] {(A-rightleg) (B)};
\end{tz}
\longxdoubleto{\mu}
\begin{tz}
    \node [Pants] (A) at (0,0) {};
        \node[Cap] (X) at (A.belt) {};
    \node[Copants, anchor=leftleg] (B) at (A.leftleg) {};
    \node[Cup] at (B.belt) {};
    \selectpart[green] {(X-center) (A-leftleg) (A-rightleg) (B-belt)};
\end{tz}
\longxdoubleto{\epsilon}
\begin{tz}
    \node [Cap] at (0,0) {};
    \node[Cup] at (0,0) {};
\end{tz}
\quad = \quad \id
\end{equation}

\item (Modularity) The following equation holds, together with $\eqref{MOD}{}^z$:
\begin{equation}
\label{MOD}
\begin{tz}[xscale=2, yscale=1.5]
\node (1) at (0,-0.5)
{$\begin{tikzpicture}
        \node[Cyl, top, tall] (A) at (0,0) {};
\end{tikzpicture}$};

\node (2) at (1,0)
{$\begin{tikzpicture}
        \node[Pants, top] (A) at (0,0) {};
        \node[Copants, anchor=leftleg] (B) at (A.leftleg) {};
        \selectpart[green, inner sep=1pt]{(A-leftleg)};
        \selectpart[red, inner sep=1pt] {(A-rightleg)};
\end{tikzpicture}$};

\node (3) at (2,0)
{$\begin{tikzpicture}
        \node[Pants, top] (A) at (0,0) {};
        \node[Copants, anchor=leftleg] (B) at (A.leftleg) {};
\end{tikzpicture}$};

\node (4) at (3,-0.5)
{$\begin{tikzpicture}
        \node[Cyl, top, tall] (A) at (0,0) {};
\end{tikzpicture}$};

\node(5) at (1.5, -1)
{$\begin{tikzpicture}
        \node[Cup, top] (A) at (0,0) {};
        \node[Cap] (B) at (0, -1.5*\cobheight) {};
\end{tikzpicture}$};
    
\begin{scope}[double arrow scope]
    \draw (1) -- node [above] {$\epsilon^\dagger$} (2);
    \draw[] ([xshift=0pt] 2.0) to node [above, inner sep=1pt] {${\color{green}\theta}, \color{red}{\theta^\inv}$} (3);
    \draw (3) -- node [above] {$\epsilon$} (4);
    \draw (1) -- node [below left] {$\mu^\dagger$} (5);
    \draw (5) -- node [below right] {$\mu$} (4);
\end{scope}
\end{tz}
\end{equation}
\end{itemize}
\end{defn}

\begin{theorem}[See \cite{PaperIII}]
\label{Miscsig}
There is a symmetric monoidal equivalence $\Bordcsig \simeq \bicat F (\M)$ between the componentwise signature 3-dimensional bordism 2\-category and the 2\-category generated by the modular presentation.
\end{theorem}

\subsection{The oriented, signature, and $p_1$ presentations}

\begin{defn}
\label{afpresentation}
The {\em anomaly-free modular presentation} $\O$ is a 3\-extension of the modular presentation $\M$, with the following extra relation:
\begin{itemize}
\item (Anomaly-freeness) The following equation holds:
\begin{equation} \label{AF}
\begin{tz}
        \node[Cap] (A) at (0,0) {};
        \node[Cup] (B) at (0,0) {};
        \node [Cobordism Bottom End 3D] (C) at (0,0) {};
        \selectpart[green, inner sep=1pt] {(C)};
\end{tz}
\longxdoubleto{\epsilon^\dagger}
\begin{tz}
        \node[Cap] (A) at (0,0) {};
        \node[Pants, anchor=belt] (B) at (A.center) {};
        \node[Copants, anchor=leftleg] (C) at (B.leftleg) {};
        \node[Cup] (D) at (C.belt) {};
        \selectpart[green, inner sep=1pt] {(B-leftleg)};
\end{tz}
\longxdoubleto{\theta}
\begin{tz}
        \node[Cap] (A) at (0,0) {};
        \node[Pants, anchor=belt] (B) at (A.center) {};
        \node[Copants, anchor=leftleg] (C) at (B.leftleg) {};
        \node[Cup] (D) at (C.belt) {};
        \selectpart[green] {(A-center) (B-leftleg) (B-rightleg) (C-belt)};
\end{tz}
\longxdoubleto{\epsilon}
\begin{tz}
        \node[Cap] (A) at (0,0) {};
        \node[Cup] (B) at (0,0) {};
\end{tz}
\quad=\quad
\id
\end{equation}
\end{itemize}
\end{defn}

\begin{theorem}[See \cite{PaperII}, Corollary~\ref{PII_mainthm}]
There is a symmetric monoidal equivalence $\Bord \simeq \bicat F(\O)$ between the oriented 3-dimensional bordism 2\-category and the 2\-category generated by the anomaly-free modular presentation.
\end{theorem}

\begin{defn}
\label{globalpresentation}
The {\em degree-$k$ global modular presentation $\N_k$}  is a 2\-extension of the modular presentation $\M$, with the following extra generators and relations:
\begin{itemize}
\item Additional invertible 2\-morphism generator:
\[
\begin{tz}
        \draw[green] (0,0) rectangle (\cobwidth, \cobwidth);
\end{tz}
\rightleftdoublearrow{\zeta}{\zeta ^\inv}
\begin{tz}
        \draw[green] (0,0) rectangle (\cobwidth, \cobwidth);
\end{tz}
\]
\item Invertibility relations: $\zeta \zeta^\inv = \id = \zeta^\inv \zeta$.
\item Anomaly factor relation: 
\begin{equation}
\label{eq:globalmodularaxiom}
\begin{tz}
    \node[Cap] (A) at (0,0) {};
    \node[Cup] (B) at (0,0) {};
    \draw[green] (1*\cobwidth, -\cupheight) rectangle +(\cobwidth, 2*\cupheight);
\end{tz}
\longxdoubleto{\zeta^k}
\begin{tz}
    \node[Cap] (A) at (0,0) {};
    \node[Cup] at (0,0) {};
    \draw[green] (1*\cobwidth, -\cupheight) rectangle +(\cobwidth, 2*\cupheight);
\end{tz}
\quad = \quad
\begin{tz}
        \node[Cap] (A) at (0,0) {};
        \node[Cup] (B) at (0,0) {};
        \selectpart[green, inner sep=1pt] {(A-center)};
\end{tz}
\longxdoubleto{\epsilon ^\dag}
\begin{tz}
        \node[Cap] (A) at (0,0) {};
        \node[Pants, anchor=belt] (B) at (A) {};
        \node[Copants, anchor=leftleg] (C) at (B.leftleg) {};
        \node[Cup] (D) at (C.belt) {};
        \selectpart[green, inner sep=1pt] {(B-rightleg)};
\end{tz}
\longxdoubleto{\theta}
\begin{tz}
        \node[Cap] (A) at (0,0) {};
        \node[Pants, anchor=belt] (B) at (A) {};
        \node[Copants, anchor=leftleg] (C) at (B.leftleg) {};
        \node[Cup] (D) at (C.belt) {};
        \selectpart[green] {(B-rightleg) (B-leftleg) (A-center) (C-belt)};
\end{tz}
\longxdoubleto{\epsilon}
\begin{tz}
        \node[Cap] (A) at (0,0) {};
        \node[Cup] (B) at (0,0) {};
\end{tz}
\end{equation}
\end{itemize}
\end{defn}

\begin{theorem}[See \cite{PaperIII}]
\label{thm:sigpresentation}
There is a symmetric monoidal equivalence $\Bordsig \simeq \F(\N_1)$ between the signature 3-dimensional bordism 2\-category and the 2\-category generated by the global modular presentation of degree~1.
\end{theorem}

\begin{defn}
\label{p1presentation}
The {\em degree-$k$ modular presentation $\M_k$} is defined as a 2\-extension of the modular presentation $\M$, with the following extra generators and relations:
\begin{itemize}
\item Invertible 2\-morphism generator:
\[
\begin{tz}
        \node[Cap] (A) at (0,0) {};
        \node[Cup] (B) at (0,0) {};
\end{tz}
\rightleftdoublearrow{y'}{y' {}^\inv}
\begin{tz}
        \node[Cap] (A) at (0,0) {};
        \node[Cup] (B) at (0,0) {};
\end{tz}
\]
\item Invertibility relations: $\yonsphere \yonsphere^\inv = \id = \yonsphere^\inv \yonsphere$
\item Additional relations: 
\begin{gather}
\label{eq:p1axiom}
\begin{tz}
        \node[Cap] (A) at (0,0) {};
        \node[Cup] (B) at (0,0) {};
\end{tz}
\xdoubleto{y' {}^k}
\begin{tz}
        \node[Cap] (A) at (0,0) {};
                \node[Cup] at (0,0) {};
\end{tz}
\quad = \quad
\begin{tz}
        \node[Cap] (A) at (0,0) {};
        \node[Cup] (B) at (0,0) {};
        \selectpart[green, inner sep=1pt] {(A-center)};
\end{tz}
\longxdoubleto{\epsilon ^\dag}
\begin{tz}
        \node[Cap] (A) at (0,0) {};
        \node[Pants, anchor=belt] (B) at (A) {};
        \node[Copants, anchor=leftleg] (C) at (B.leftleg) {};
        \node[Cup] (D) at (C.belt) {};
        \selectpart[green, inner sep=1pt] {(B-rightleg)};
\end{tz}
\longxdoubleto{\theta}
\begin{tz}
        \node[Cap] (A) at (0,0) {};
        \node[Pants, anchor=belt] (B) at (A) {};
        \node[Copants, anchor=leftleg] (C) at (B.leftleg) {};
        \node[Cup] (D) at (C.belt) {};
        \selectpart[green] {(B-rightleg) (B-leftleg) (A-center) (C-belt)};
\end{tz}
\longxdoubleto{\epsilon}
\begin{tz}
        \node[Cap] (A) at (0,0) {};
        \node[Cup] (B) at (0,0) {};
\end{tz}
\\
\label{eq:p1centrality}
\begin{tz}[xscale=1.5, yscale=1.5]
\node (1) at (0,0) {$\begin{tz}
    \node (A) [Cap] at (0,0) {};
    \node (B) [Cup] at (0,0) {};
    \node (C) [Cap] at (0,-1.5\cobheight) {};
    \node [Cup] at (0,-1.5\cobheight) {};
    \node (b1) [Cobordism Bottom End 3D] at (0,0) {};
    \node (b2) [Cobordism Bottom End 3D] at (0,-1.5\cobheight) {};
    \selectpart[green, inner sep=0pt, inner ysep=0.7pt]{(A) (B)};
    \selectpart[red, inner sep=0.7pt]{(B) (C) (b1) (b2)};
\end{tz}$};
\node (2) at (1,0) {$\begin{tz}
    \node [Cap] at (0,0) {};
    \node (B) [Cup] at (0,0) {};
    \node (C) [Cap] at (0,-1.5\cobheight) {};
    \node [Cup] at (0,-1.5\cobheight) {};
    \node (b1) [Cobordism Bottom End 3D, invisible] at (0,0) {};
    \node (b2) [Cobordism Bottom End 3D, invisible] at (0,-1.5\cobheight) {};
    \selectpart[green]{(B) (C) (b1) (b2)};
\end{tz}$};
\node (3) at (1,-1) {$\begin{tz}
    \node [Cap] at (0,0) {};
    \node [Cup] at (0,0) {};
\end{tz}$};
\node (4) at (0,-1) {$\begin{tz}
    \node [Cap] at (0,0) {};
    \node [Cup] at (0,0) {};
\end{tz}$};
\begin{scope}[double arrow scope]
\draw (1) to node [above, green] {$y'$} (2);
\draw (2) to node [right] {$\mu$} (3);
\draw (1) to node [left, red] {$\mu$} (4);
\draw (4) to node [below] {$y'$} (3);
\end{scope}
\end{tz}
\end{gather}
\end{itemize}
The case $k=3$ is called the {\em $p_1$-presentation}.
\end{defn}

\begin{theorem}[See \cite{PaperIII}]
There is a symmetric monoidal equivalence $\Bordp \simeq \bicat F (\M_3)$ between the componentwise signature 3-dimensional bordism 2\-category and the 2\-category generated by the $p_1$-presentation.
\end{theorem}

\section{Internal string diagrams}
\label{sec:internal}

\noindent
For any presentation that 2\-extends the right-adjoint monoid presentation $\P^\L$ (see Definition~\ref{rightadjointmonoidpresentation}), internal string diagrams give a notation for computing the invariants of a linear representation in terms of string diagrams drawn in the ``interior" of a surface, with respect to a canonical embedding into $\mathbb R ^3$. Note that the geometrical presentations defined in \autoref{sec:geometricalpresentations} all 2\-extend $\P^\L$, so internal string diagrams will be applicable for all of them.  The technique of internal string diagrams is closely related to previous work of many people, as discussed in the introduction, Section~\ref{sec:intro-isd}.

For this section, let \Q be a presentation that 2\-extends $\P^\L$, and let \mbox{$Z: \bicat F(\Q) \to \twovect$} be a linear representation.

\subsection{Motivation and definition} 

In the presentation $\P^\L$, and therefore in any presentation \Q that 2\-extends $\P^\L$, we have the adjunction $\tikztinypants \dashv \tikztinycopants$.  Given a linear representation \mbox{$Z: \bicat F(\Q) \to \twovect$}, we therefore have the adjunction $Z(\tikztinypants) \dashv Z(\tikztinycopants)$.  Because adjoints are unique, the value $Z(\tikztinycopants)$ is therefore already determined by $Z(\tikztinypants)$; moreover, by composing the representation \mbox{$Z: \bicat F(\Q) \to \twovect$} with the inclusion $(-)_* : \twovect \to \bimod$, we can easily identify and compute $Z(\tikztinycopants)$ as the adjoint $Z(\tikztinypants)^*$ in $\bimod$.  This approach is often much more convenient and computable as it provides an explicit algebraic method for computing $Z(\tikztinycopants)$ that makes no reference to the `geometric' adjunction $\tikztinypants \dashv \tikztinycopants$.  Internal string diagrams formalize this approach to algebraic computations of linear representations.

As above, given a functor \mbox{$Z: \bicat F(\Q) \to \twovect$}, by \autoref{adj_lemma} (that for any functor $F$, we have an adjunction $F_* \dashv F^*$ in $\bimod$), there is an adjunction $Z(\tikztinypants)_* \dashv Z(\tikztinypants) ^*$; by \autoref{rightadjointmonoidpresentation} (giving the adjunction $\tikztinypants \dashv \tikztinycopants$), there is an adjunction $Z(\tikztinypants)_* \dashv Z(\tikztinycopants)_*$.  There is therefore a canonical isomorphism
\begin{align}
\label{eq:zeta1}
\zeta (\tikztinycopants) &: Z(\tikztinycopants)_* \to Z(\tikztinypants) ^*
\intertext{Similarly, using the adjunction $\tikztinycup \dashv \tikztinycap$, there is a canonical isomorphism}
\zeta(\tikztinycap) &: Z(\tikztinycap)_* \to Z(\tikztinycup) ^*
\intertext{We can extend the family of isomorphisms $\zeta$ to all generating 1-morphisms of the presentation \Q by letting it be the identity on  $\tikztinypants$ and $\tikztinycup$:}
\zeta(\tikztinypants) &:  Z(\tikztinypants) _* \to Z(\tikztinypants)_*
\\
\label{eq:zeta4}
\zeta(\tikztinycup) &:  Z(\tikztinycup) _* \to Z(\tikztinycup)_*
\end{align}
The internal string diagram construction simply transports a representation $Z$ across these canonical isomorphisms $\zeta$, providing a new representation $\isd$.
\begin{defn}
\label{def:stringfunctor}
Given a linear representation $Z$ of a presentation $\Q$ that 2\-extends $\P^\L$, the \textit{internal string diagram construction} $\isd$  is defined as the unique transport (see \autoref{uniqueconjugate}) of $Z$ across the bimodule isomorphisms $\zeta$:
\begin{equation*}
\begin{aligned}
\begin{tikzpicture}
\node (1) at (0,1) {$\bicat F (\Q)$};
\node (2) at (3,1) {\twovect};
\node (3) at (3,-1) {\bimod};
\draw [->] (1) to node [auto] {$Z$} (2);
\draw [->] (2) to node [auto] {$(-)_*$} (3);
\draw [->] (1) to node [auto, swap, pos=0.5, inner sep=9pt] {$\isd$} (3);
\draw [double, ->] (2.5,0.5) to
 node [auto, inner sep=2pt] {$\zeta$}
 (2,0);
\end{tikzpicture}
\end{aligned}
\end{equation*}
\end{defn}

Note that $\isd$ is specified on the generating 1-morphisms by the images of the isomorphisms $\zeta$; that is
\begin{align}
\label{eq:isdbimodules}
\isd (\tikztinypants) &:= Z(\tikztinypants)_*
&
\isd (\tikztinycup) &:= Z(\tikztinycup) _*
&
\isd (\tikztinycopants) &:= Z(\tikztinypants)^*
&
\isd (\tikztinycap) &:= Z(\tikztinycup)^*
\end{align}

\noindent
Again, the internal string diagram construction is a representation \isd of \Q into categories and bimodules, isomorphic to the original representation $Z$ into categories and functors, such that the adjunctions $\isd(\tikztinypants) \dashv \isd (\tikztinycopants)$ and $\isd(\tikztinycup) \dashv \isd (\tikztinycap)$, which a priori would be witnessed geometrically by $\isd(\epsilon)$, $\isd(\eta)$, $\isd(\mu)$, and $\isd(\nu)$, are instead witnessed algebraically in the target bimodule category and are therefore easily and explicitly computable. The value of the internal string diagram construction functor on generating 2\-cells is given explicitly in Section~\ref{sec:actionof2morphisms}.

\subsection{Internal string notation}

Recall that a bimodule $M$ from the category $X$ to the category $Y$ is a functor $M: Y^\op \boxtimes X \to \vect$; we will use the notation $M^A_B$, where $B \in X$ and $A \in Y$, for the value of the bimodule on the object $(A,B)$. For each 1\-morphism generator $\Sigma$ of \Q, the bimodule $\isd(\Sigma)$ is given on objects as follows.  Abbreviating $S:= Z (\begin{tz}
         \node[Cyl, top, height scale=0]  at (0,0) {};
 \end{tz})$, we have the following, for all objects $A,B,C \in S$:
\begin{align}
\label{eq:string1}
\isd (\tikztinypants) ^A _{B,C} &= (Z(\tikztinypants)_*)^A _{B,C} = \Hom _\cat S (A, B \otimes C)
\\
\isd (\tikztinycopants) _C ^{A,B} &= (Z(\tikztinypants)^*) {}_C ^{A,B} =  \Hom _\cat S (A \otimes B, C)
\\
\isd (\tikztinycup) ^A &= (Z(\tikztinycup)_*) {}^A = \Hom _\cat S (A, I)
\\
\label{eq:string4}
\isd (\tikztinycap) _A &= (Z(\tikztinycup) ^*)_A = \Hom _\cat S (I,A)
\end{align}
On each line, the last equality uses the definition of the constructions $(-)_*$ and $(-)^*$, and uses the monoidal structure on $\cat S$ given by the image of the pants and cup under the functor $Z$.

Observe that in each of these four cases, the vector space of the bimodule is the space of string diagrams (drawn in the Joyal-Street graphical calculus for the associated monoidal category) internal to the volume of that particular surface embedded in $\RR^3$:
\normalbordisms
\begin{calign}
\label{eq:embeddingR3}
\begin{tz}
    \node[Pants, bot, top, height scale=1.0] (A) at (0,0) {};
    \begin{scope}[internal string scope]
        \node (i) at ([yshift=\toff] A.belt) [above] {$A$};
        \node (j) at ([yshift=-\boff] A.leftleg) [below] {$B$};
        \node (k) at ([yshift=-\boff] A.rightleg) [below] {$C$};
        \node [tiny label] (g) at (0,0.02\cobheight) {$f$};
        \draw (i.south)
            to (g.center)
            to [out=-140, in=up] (j.north);
        \draw (g.center)
            to [out=-40, in=up] (k.north);
    \end{scope}
\end{tz}
&
\begin{tz}
    \node[Copants, bot, top, height scale=1.0] (A) at (0,0) {};
    \begin{scope}[internal string scope]
        \node (i) at ([yshift=-\boff] A.belt) [below] {$C$};
        \node (j) at ([yshift=\toff] A.leftleg) [above] {$A$};
        \node (k) at ([yshift=\toff] A.rightleg) [above] {$B$};
        \node [tiny label] (g) at (0,-0.1\cobheight) {$g$};
        \draw (i.north)
            to (g.center)
            to [out=140, in=down] (j.south);
        \draw (g.center)
            to [out=40, in=down] (k.south);
    \end{scope}
\end{tz}
&
\begin{aligned}
\begin{tikzpicture}
\setlength\cupheight{1.5\cupheight}
\node (i) at (0,0) [Cup, top] {};
\node (j) at (0,-\cobheight) [Bot3D, invisible] {};
\node (g) [tiny label] at (0,-0.4\cobheight) {$h$};
\draw [internal string] (g.center) to ([yshift=\toff] i.center) node [above, red] {$A$};
\node at ([yshift=-\boff] j.center) [below, white] {$A$};
\end{tikzpicture}
\end{aligned}
&
\begin{aligned}
\begin{tikzpicture}
\setlength\cupheight{1.5\cupheight}
\node (i) at (0,0) [Cap, top] {};
\node (j) at (0,\cobheight) [Bot3D, invisible] {};
\node (g) [tiny label] at (0,0.4\cobheight) {$j$};
\draw [internal string] (g.center) to ([yshift=-\boff] i.center) node [below, red] {$A$};
\node at ([yshift=\toff] j.center) [above, white] {$A$};
\end{tikzpicture}
\end{aligned}
\\*
\nonumber
f \in \isd \big( \tikztinypants \big) {}^A _{B,C}
&
g \in \isd \big( \tikztinycopants \big) {}^{A,B} _{C}
&
h \in \isd \big( \tikztinycup \big) {}^A
&
j \in \isd \big( \tikztinycap \big) {}_A
\end{calign}
We do not need to, and do not, make this notion of an `interior morphism' geometrically precise; the equalities~(\ref{eq:string1}\_\ref{eq:string4}) provide the foundation and formalism for the internal string diagram calculations, and the pictures~\eqref{eq:embeddingR3} are merely a convenient graphical mnemonic notation.  Note that the interior morphisms are drawn with their source at the top and their target at the bottom, while the cobordisms are drawn in the reverse direction. This countervalence of conventions appears to be essential.

    \newcommand\cobpartup[2]{
        \begin{scope}[xshift=#1,yshift=#2, xscale=0.8]
        \node (b)[Bot3D, scale=0.8] at (0,0) {};
        \node (left1) at ([xshift=-0.5\cobwidth] b.center) [coordinate] {};
        \node (left2) at ([shift={(-0.2\cobwidth,0.6*\ymult*\cobwidth)}] left1) [coordinate] {};
        \node (left3) at ([shift={(-0.5\cobwidth,\ymult*\cobwidth)}] left1) [coordinate] {};
        \node (right1) at ([xshift=0.5\cobwidth] b.center) [coordinate] {};
        \node (right2) at ([shift={(0.2\cobwidth,0.6*\ymult*\cobwidth)}] right1) [coordinate] {};
        \node (right3) at ([shift={(0.5\cobwidth,1.0*\ymult*\cobwidth)}] right1) [coordinate] {};
        \begin{pgfonlayer}{top}
            \draw (left1) to [out=90, in =-55] (left2);
        \end{pgfonlayer}
        \draw [dash pattern=on 1pt off 1pt] (left2) to (left3);
        \begin{scope}[xscale=-1]
            \draw (right1) to [out=90, in =-55] (right2);
            \draw [dash pattern=on 1pt off 1pt] (right2) to (right3);
        \end{scope}
        \end{scope}
    }
    \newcommand\cobpartdown[2]{
        \begin{scope}[yscale=-1]
            \cobpartup{#1}{#2}
        \end{scope}
    }

We can extend this notation to describe the vector spaces determining $\isd(\Sigma)$ for an arbitrary 1\-morphism $\Sigma$ in $\bicat F(\Q)$: it is the vector space of string diagrams (up to embedded homotopy) that can be drawn in the interior of the entire surface $\Sigma$ (embedded componentwise in $\RR^3$ in the same manner as~\eqref{eq:embeddingR3}). To see this, suppose that $\Sigma$ is built from generating 1\-morphisms $\sigma_i$ of \Q, and suppose we label the upper and lower boundary circles of $\Sigma$ by $A_1, \ldots, A_m$ and $B_1, \ldots, B_n$ respectively.   Then by Definition~\ref{bimodulecomposition} of bimodule composition we have
\begin{equation}
\label{eq:isdcomposition}
\isd (\Sigma) {}^{A_1, \ldots, A_m} _{B_1, \ldots, B_n} := \bigoplus _{L} \bigotimes _{\sigma \in \Sigma} \isd ( \sigma ) _L \, /\sim,
\end{equation}
where the tensor product is taken over all the generating 1\-cell components $\sigma$ of $\Sigma$, and the direct sum is taken over all possible labelings $L$ of the internal boundary circles. Each generating bimodule $\isd (\sigma)$ is given the labeling induced by $L$ and the fixed labeling on the external boundary circles. We give the equivalence relation, from equation~\eqref{eq:bimodulecomposition}, the following graphical representation: 
\begin{equation}
\label{eq:pullthrough}
\newlength\ymult
\setlength\ymult{1pt}
\setlength{\cobwidth}{2.3\cobwidth}
\setlength{\ellipseheight}{1.5\ellipseheight}
\begin{tz}
    \cobpartup 0 0
    \cobpartdown 0 0
    \node (A) at (0cm,1.5cm) [coordinate] {};
    \node (B) at (0cm,-1.5cm) [coordinate] {};
    \begin{scope}[curvein]
        \draw [dash pattern=on 1pt off 1pt] (A) to (0,1.2cm);
        \draw [dash pattern=on 1pt off 1pt] (B) to (0,-1.2cm);
        \draw (0,1.2cm) to (0,-1.2cm);
        \node at (A) [above] {$C$};
        \node at (B) [below] {$D$};
        \node [small label] at (0,0.6cm) {$f$};
    \end{scope}
    \draw [->] (-2,-0.6\cobwidth) to (-2,-0.1);
    \draw [dash pattern = on 1pt off 1pt] (-2,-0.6\cobwidth) to (-2,-1.0\cobwidth);
    \node at (-2.4,-0.6) {$\sigma_{i}$};
    \draw (-2,0.6\cobwidth) to (-2,0.1);
    \draw [->,dash pattern = on 1pt off 1pt] (-2,0.6\cobwidth) to (-2,1.0\cobwidth);
    \node at (-2.4,0.6) {$\sigma _{j}$};
\end{tz}
\hspace{20pt}
\sim
\hspace{20pt}
\begin{tz}
    \cobpartup 0 0
    \cobpartdown 0 0
    \node (A) at (0cm,1.5cm) [coordinate] {};
    \node (B) at (0cm,-1.5cm) [coordinate] {};
    \begin{scope}[curvein]
        \draw [dash pattern=on 1pt off 1pt] (A) to (0,1.2cm);
        \draw [dash pattern=on 1pt off 1pt] (B) to (0,-1.2cm);
        \draw (0,1.2cm) to (0,-1.2cm);
        \node at (A) [above] {$C$};
        \node at (B) [below] {$D$};
        \node [small label] at (0,-0.6cm) {$f$};
    \end{scope}
    \draw [->] (-2,-0.6\cobwidth) to (-2,-0.1);
    \draw [dash pattern = on 1pt off 1pt] (-2,-0.6\cobwidth) to (-2,-1.0\cobwidth);
    \node at (-2.4,-0.6) {$\sigma_{i}$};
    \draw (-2,0.6\cobwidth) to (-2,0.1);
    \draw [->,dash pattern = on 1pt off 1pt] (-2,0.6\cobwidth) to (-2,1.0\cobwidth);
    \node at (-2.4,0.6) {$\sigma _{j}$};
\end{tz}
\end{equation}
Here the generating 1-cell $\sigma_i$ has been attached to the generating 1-cell $\sigma_j$ along an internal boundary circle to form $\sigma_j \circ \sigma_i$. Two elements of the composite bimodule vector space for $\Sigma$ will be related by the least equivalence relation generated by this congruence exactly when the internal string pictures for the elements can be transformed into one another by equations of string diagrams within the interior of generating 1\-morphisms, and movement of internal strings through boundary circles.

As an example, consider the following 1\-cell in $\bicat F(\P^\L)$:
$$\Sigma = \tikztinypants \circ \big( \tikztinypants \sqcup \tikztinycyl \big) \circ \big( \tikztinycyl \sqcup \tikztinycopants \big)$$
Then $\isd(\Sigma) {} ^{X} _{Y,Z}$ is spanned by diagrams of the following form:
\normalbordisms
\begin{equation}
\begin{tz}[every to/.style={out=down,in=up}]
    \node (A) [Pants, belt scale=1.0] at (0,0) {};
    \node (B) [Copants, anchor=leftleg] at (A.rightleg) {};
    \node (C) [Cyl, anchor=top] at (A.leftleg) {};
    \node (D) [Cyl, anchor=bottom] at (B.rightleg) {};
    \node (E) [Pants, bot, anchor=leftleg, wide, top, left leg scale=1.0] at (A.belt) {};
        \node (f) [tiny label] at ([yshift=0.1\cobwidth] A.center) {$f$};
        \node (g) [tiny label] at ([yshift=-0.2\cobwidth] B.center) {$g$};
        \node (h) [tiny label] at ([yshift=0.10\cobwidth] E.center) {$h$};
        \strand[red strand] ([yshift=\toff] E.belt) node[above strand label, red] {$X$} to (h.center);
\strand [red strand] (h.center)
  to [out=-20] (D.top)
  to
    (D.bottom)
  to[out=down, in=30, out looseness=1] (g.center);
        \strand[red strand] (f.center) to[out=-150, in=up] (A.leftleg)
                        to  ([yshift=-\boff] C.bottom) node[below strand label, red] {$Y$};
        \strand[red strand] (f.center) to[out=-30, in=up] (A.rightleg) to[out=down, in=135] (g.150);  
        \strand[red strand] (g.center) to[out=down, in=up] ([yshift=-\boff] B.belt) node[below strand label, red] {$Z$};
\strand [red strand] (h.center) to [out=-160]
  (E.leftleg) to (f.center);
\end{tz}
\end{equation}
We say that such an internal string diagram is in \textit{elementary form}, meaning that it is in the image (under the map $\bigotimes _{\sigma \in \Sigma} \isd ( \sigma ) \to \isd ( \Sigma )$) of an elementary tensor of a string diagram in each geometrical component.

For a surface of genus zero, every element of the internal string diagram vector space is equal to one of elementary form. We prove this for the following specific surface, which will play an important role later in the paper:
\mediumbordisms
\begin{equation}
\label{eq:snakesurface}
N :=
\begin{tz}[xscale=-1]
    \node[Pants, bot] (A) at (0,0) {};
    \node[Cyl, bot, height scale=1.7] [anchor=top] at (A.rightleg) (B) {};
    \node[Copants, bot, anchor=rightleg] at (A.leftleg) (D) {};
    \node[Cyl, bot, top, height scale=1.7, anchor=bottom] at (D.leftleg) (C) {};
    \node [Cap, bot] at (A.belt) {};
    \node [Cup] at (D.belt) {};
\end{tz}
\end{equation}
\begin{proposition}
\label{genericformsnake}
For objects $A,B \in S$, every element of the internal string diagram vector space $\isd(N)^A_B$, where $N$ is the surface \eqref{eq:snakesurface}, is represented by an internal string diagram in elementary form.
\end{proposition}
\begin{proof}
Given a linear representation of a presentation that 2\-extends $\P^\L$, in the associated  Cauchy-complete linear monoidal category, let $X_i$ be a finite family of objects, and let $f_i: I \to X_i \otimes B$ and $g_i: A \otimes X_i \to I$ be a family of morphisms with the same index set $S$. Define $X = \bigoplus_i X_i$, and denote the projection maps $X \to X_i$ and injection maps $X_i \to X$ as $i$. Then we define $f: I \to X \otimes B$ and $g: A \otimes X \to I$ as follows:
\begin{align}
\begin{tz}[xscale=0.5, yscale=-0.7]
\node [tiny label] at (0,0) {$f$};
\draw [red strand] (0,0) to [out=155, in=down, in looseness=2] (-1,2) node [below] {$\red X$};
\draw [red strand] (0,0) to [out=25, in=down, in looseness=2] (1,2) node [below] {$\red B$};
\end{tz}
&=\,\,
{\red\sum_i}
\begin{tz}[xscale=0.5, yscale=-0.7]
\node [tiny label] at (0,0) {$f_i$};
\node [tiny label] at (-1,1) {$i$};
\draw [red strand] (0,0)
  to [out=155, in=down]
    node [left=-1pt] {$\red X_i$} (-1,1)
  to +(0,1)
    node [below] {$\red X$};
\draw [red strand] (0,0) to [out=25, in=down, in looseness=2] (1,2) node [below] {$\red B$};
\end{tz}
&
\begin{tz}[xscale=0.5, yscale=0.7]
\node [tiny label] at (0,0) {$g$};
\draw [red strand] (0,0) to [out=155, in=down, in looseness=2] (-1,2) node [above] {$\red A$};
\draw [red strand] (0,0) to [out=25, in=down, in looseness=2] (1,2) node [above] {$\red X$};
\end{tz}
&=\,\,
{\red\sum_i}
\begin{tz}[xscale=0.5, yscale=0.7]
\node [tiny label] at (0,0) {$g_i$};
\node [tiny label] at (1,1) {$i$};
\draw [red strand] (0,0) to [out=25, in=down] node [right] {\red$X_i$} (1,1) to +(0,1) node [above] {$\red X$};
\draw [red strand] (0,0) to [out=155, in=down, in looseness=2] (-1,2) node [above] {$\red A$};
\end{tz}
\end{align}
Observe:
\normalbordisms
\begin{equation}
\begin{tz}[xscale=-1]
    \node[Pants, bot] (A) at (0,0) {};
    \node[Cyl, bot, height scale=1.7] [anchor=top] at (A.rightleg) (B) {};
    \node[Copants, bot, anchor=rightleg] at (A.leftleg) (D) {};
    \node[Cyl, bot, top, height scale=1.7, anchor=bottom] at (D.leftleg) (C) {};
    \node (f) [tiny label] at ([yshift=2pt] A.center) {$f$};
    \node (g) [tiny label] at ([yshift=-2pt] D.center) {$g$};
    \begin{scope}[internal string scope]
        \node (i) at ([yshift=\toff] A.belt) [above] {};
        \node (j) at ([yshift=\toff] C.top) [above] {$A$};
        \node (k) at ([yshift=-\boff] B.bottom) [below] {$B$};
        \node (l) at ([yshift=-\boff] D.belt) [below] {};
        \draw [out=210, in=up] (f.center) to (B.top);
        \draw (f.center)
            to [out=-45, in=100] node [left] {} (D.rightleg)
            to [out=-80, in=135] (g.center);
        \draw [out=45, in=down] (g.center) to (C.bottom);
        \draw [out=up, in=down] (C.bottom) to (j.south);
        \draw [out=down, in=up] (A.rightleg) to (k.north);
    \end{scope}
    \node [Cap, bot] at (A.belt) {};
    \node [Cup] at (D.belt) {};
\end{tz}
\quad=\quad
{\red \sum_{i,j}}
\,\,
\begin{tz}[xscale=-1]
    \node[Pants, bot] (A) at (0,0) {};
    \node[Cyl, bot, height scale=1.7] [anchor=top] at (A.rightleg) (B) {};
    \node[Copants, bot, anchor=rightleg] at (A.leftleg) (D) {};
    \node[Cyl, bot, top, height scale=1.7, anchor=bottom] at (D.leftleg) (C) {};
    \node (f) [tiny label] at ([yshift=4pt] A.center) {$f_i$};
    \node (g) [tiny label] at ([yshift=-8pt] D.center) {$g_j$};
    \node (j) [tiny label] at ([yshift=-9pt, xshift=4pt] A.leftleg) {$j$};
    \node (i) [tiny label] at ([yshift=5pt, xshift=-2pt] A.leftleg) {$i$};
    \begin{scope}[internal string scope]
        \node (i) at ([yshift=\toff] A.belt) [above] {};
        \node (j) at ([yshift=\toff] C.top) [above] {$A$};
        \node (k) at ([yshift=-\boff] B.bottom) [below] {$B$};
        \node (l) at ([yshift=-\boff] D.belt) [below] {};
        \draw [out=210, in=up] (f.center) to (B.top);
        \draw (f.center)
            to [out=-45, in=100] node [left] {} (D.rightleg)
            to [out=-80, in=135] (g.center);
        \draw [out=45, in=down] (g.center) to (C.bottom);
        \draw [out=up, in=down] (C.bottom) to (j.south);
        \draw [out=down, in=up] (A.rightleg) to (k.north);
    \end{scope}
    \node [Cap, bot] at (A.belt) {};
    \node [Cup] at (D.belt) {};
\end{tz}
\quad=\quad
{\red\sum_i}\,\,
\begin{tz}[xscale=-1]
    \node[Pants, bot] (A) at (0,0) {};
    \node[Cyl, bot, height scale=1.7] [anchor=top] at (A.rightleg) (B) {};
    \node[Copants, bot, anchor=rightleg] at (A.leftleg) (D) {};
    \node[Cyl, bot, top, height scale=1.7, anchor=bottom] at (D.leftleg) (C) {};
    \node (f) [tiny label] at ([yshift=2pt] A.center) {$\tiny f_i$};
    \node (g) [tiny label] at ([yshift=-2pt] D.center) {$g_i$};
    \begin{scope}[internal string scope]
        \node (i) at ([yshift=\toff] A.belt) [above] {};
        \node (j) at ([yshift=\toff] C.top) [above] {$A$};
        \node (k) at ([yshift=-\boff] B.bottom) [below] {$B$};
        \node (l) at ([yshift=-\boff] D.belt) [below] {};
        \draw [out=210, in=up] (f.center) to (B.top);
        \draw (f.center)
            to [out=-45, in=100] node [left] {} (D.rightleg)
            to [out=-80, in=135] (g.center);
        \draw [out=45, in=down] (g.center) to (C.bottom);
        \draw [out=up, in=down] (C.bottom) to (j.south);
        \draw [out=down, in=up] (A.rightleg) to (k.north);
    \end{scope}
    \node [Cap, bot] at (A.belt) {};
    \node [Cup] at (D.belt) {};
\end{tz}
\end{equation}
The first equity uses linearity of the tensor product.  The second equality picks out the nonzero terms of the summation.  The last expression is an arbitrary element of the internal string diagram vector space, while the first is in elementary form as required.
\end{proof}

\subsection{Actions of the generating 2-morphisms}
\label{sec:actionof2morphisms}

Let $\kappa: \Sigma \doubleto \Tau$ be a generating 2\-morphism in $\bicat F(\Q)$, where as above $\Q$ is a presentation that 2\-extends $\P^\L$, and let $Z$ be a linear representation of \Q.  Then we can interpret $\isd(\kappa)$ as a linear map between the spaces of internal string diagrams $\isd (\Sigma)$ and $\isd (\Tau)$ associated to the source and target of $\kappa$. In this section, we analyze the linear maps associated to the generators $\alpha$, $\lambda$, $\rho$ and their inverses, along with $\eta$, $\epsilon$, $\mu$, and $\nu$, and also $\beta$ and $\theta$ in the case that \Q 2\-extends $\B^\L$. For these generators, we will see that the associated linear maps between spaces of internal string diagrams have intuitive geometrical interpretations, which furthermore are independent of the particular chosen representation $Z$.

These linear maps $\isd(\kappa):\isd(\Sigma)^A_B \to \isd(\Tau)^A_B$, which we informally call `actions', are determined in Definition~\ref{def:stringfunctor} by the following conjugation:
\begin{align}
\begin{aligned}
\begin{tikzpicture}[xscale=4,yscale=2]
\node (A) at (0,0) {$\isd (\Sigma )_B ^A$};
\node (B) at (1,0) {$(Z (\Sigma)_*) ^A _B$};
\node (C) at (0,-1) {$\isd (\Tau )_B ^A$};
\node (D) at (1,-1) {$(Z (\Tau) _*) ^A _B$};
\draw [->] (A) to node [above] {$\zeta _\Sigma$} node [below] {$\simeq$} (B);
\draw [->] (B) to node [right] {$(Z (\kappa) _*) _B ^A$} (D);
\draw [->] (A) to node [left] {$\isd(\kappa )^A _B$} (C);
\draw [->] (C) to node [below] {$\zeta _\Tau$} node [above] {$\simeq$} (D);
\end{tikzpicture}
\end{aligned}
\end{align}
We will often neglect the labelings $A,B$ when they can be easily inferred from context.

While $Z$ is a strict 2\-functor, $\isd$ and $(-)_*$ are \textit{not} strict. This introduces a subtlety in the definition of $\zeta_\Sigma$, in the case that $\Sigma$ is a composite 1\-morphism: in this case, the 2-morphism $\zeta_\Sigma$ is given by first decomposing (using the weak structure morphisms of the functor $\isd$) the 1-morphism $\isd(\Sigma)$ into a composite of morphisms $\isd(\sigma)$ for generators $\sigma$, then applying the 2-morphisms $\zeta(\sigma)$, then recomposing (using the weak structure morphisms of the functor $(-)_*$) the pieces $Z(\sigma)_*$ into the morphism $Z(\Sigma)_*$. For example, for $\Sigma = \smash{\raisebox{-8pt}{$\tikz{
  \scalecobordisms{0.23}
  \node [Pants, top, anchor=leftleg] at (0,0) (A) {};
  \node [Copants, anchor=leftleg] at (0,0) {};
}$}}$, we compute $\zeta_\Sigma$ as follows, writing `compose' for the weak composition structure of the 2-functors:
\begin{equation}
\isd \big( \smash{\raisebox{-8.5pt}{$\tikz{
  \scalecobordisms{0.23}
  \node [Pants, top, anchor=leftleg] at (0,0) (A) {};
  \node [Copants, anchor=leftleg] at (0,0) {};
}$}} \big)
\,\xrightarrow{\text{compose${}^\inv$}}
\isd(\tikztinypants) \circ \isd (\tikztinycopants)
\,\xrightarrow{\zeta(\tikzverytinypants), \zeta(\tikzverytinycopants)} \,
Z(\tikztinypants)_* \circ Z(\tikztinycopants) _* \xrightarrow{\text{compose}}
Z\big( \smash{\raisebox{-8.5pt}{$\tikz{
  \scalecobordisms{0.23}
  \node [Pants, top, anchor=leftleg] at (0,0) (A) {};
  \node [Copants, anchor=leftleg] at (0,0) {};
}$}} \big)_*
\end{equation}
We will use this decomposition and recomposition method in several of the proofs in this section.

\subsubsection{Actions of the generators $\alpha$, $\lambda$, and $\rho$}
\def\arrlen{55pt}
We begin by considering the actions on internal string diagrams of the generating 2\-morphisms $\alpha$, $\lambda$, and $\rho$. These generators are the simplest to consider because they come directly from the monoid presentation~$\P$. 
\begin{proposition}
\label{thm:alrinternalstring}
For a linear representation $Z$ of a presentation that 2\-extends $\P^\L$, the generators $\alpha$, $\lambda$, and $\rho$ act in the following ways on internal string diagrams:
\normalbordisms
\setlength\ellipseheight{0.8\ellipseheight}
\setlength\cupheight{1.2\cupheight}
\begin{equation}
\label{associator-image}
    \begin{tz}
    \node[Pants, bot, top] (B) at (0,0) {};
    \node[Pants, bot, anchor=belt] (A) at (B.leftleg) {};
    \node[SwishL, bot, anchor=top] (C) at (B.rightleg) {};
    \begin{scope}[internal string scope]
        \node (i) at (B.belt) [above=\toff] {};
        \node (j) at (A.leftleg) [below=\boff] {};
        \node (k) at (A.rightleg) [below=\boff] {};
        \node (l) at (C.bot) [below=\boff] {};
        \node [tiny label] (f) at (0,0.13) {$f$};
        \node [tiny label] (g) at (A.center) {$g$};
        \node [tiny label] (h) at (C.center) {$h$};
        \draw (f.center) to (i.south);
        \draw (j.north) to [out=up, in=-135] (g.center);
        \draw (k.north) to [out=up, in=-35] (g.center);
        \draw (g.center) to [out=90, in=-135] node [left=-4pt] {} (f.center);
        \draw (l.north)
            to [out=90, in=-80] (h.center) 
            to [out=up, in=down, out looseness=0.7]
                (B.rightleg)
            to [out=up, in=-45]
                node [right=-4pt, pos=0.11] {}
 (f.center);
    \end{scope}
    \end{tz}
    \xmapsto{\textstyle\isd(\alpha)}
    \begin{tz}
    \node[Pants, bot, top] (B) at (0,0) {};
    \node[Pants, bot, anchor=belt] (A) at (B.rightleg) {};
    \node[SwishR, bot, anchor=top] (C) at (B.leftleg) {};
    \begin{scope}[internal string scope]
        \node (i) at (B.belt) [above=\toff] {};
        \node (j) at (C.bot) [below=\boff] {};
        \node (k) at (A.leftleg) [below=\boff] {};
        \node (l) at (A.rightleg) [below=\boff] {};
        \node [tiny label] (f) at (0.05\cobwidth,0.25\cobheight) {$f$};
        \node [tiny label] (g)  at (-0.25\cobwidth,-0.1\cobheight) {$g$};
        \draw (j.north)
            to [out=up, in=-130] (g.center);
        \draw (k.north)
            to [out=up, in=down] (B-rightleg.in-leftthird)
            to [out=up, in=-40] (g.center);
        \draw (l.north) to [out=up, in=down] (B-rightleg.in-rightthird)
            to [out=up, in=-60] (f.center);
        \draw (f.center) to [out=90, in=-90, looseness=2] (i.south);
        \draw (f.center) to [in=90, out=-120] (g.center);
        \node [tiny label] (h) at (0.8\cobwidth,-0.25\cobheight) {$h$};
    \end{scope}
    \end{tz}
\end{equation}
\begin{equation}
\label{units-image}
\begin{tz}
    \node[Pants, top, bot] (A) at (0,0) {};
    \node[Cup] (B) at (A.leftleg) {};
    \node[Cyl, bot, anchor=top] (C) at (A.rightleg) {};
    \begin{scope}[internal string scope]
        \node [tiny label] (f) at (0,0) {$f$};
        \node [tiny label] (g) at ([yshift=-0.33\cobheight] B) {$g$};
        \node [tiny label] (h) at (C) {$h$};
        \node (i) at ([yshift=\toff] A.belt) [above] {};
        \node (i2) at ([yshift=-\boff] C.bottom) [below] {};
        \draw (f.center) to [out=-140, in=90] node [left=-3pt] {} (g.center);
        \draw (f.center) to (i);
        \draw (f.center)
            to [out=-40, in=90, looseness=0.9]
                node [right=-2pt, pos=0.4] {} (h.center);
        \draw (h.center) to (i2);
    \end{scope}
\end{tz}
\,\,\xmapsto{\textstyle\isd(\lambda)}\,
\begin{tz}
    \node[Cyl, tall, bot, top] (A) at (0,0) {};
    \begin{scope}[internal string scope]
        \node [tiny label] (f) at (0,0.4\cobheight) {$f$};
        \node [tiny label] (g) at (-0.20\cobwidth, -0.1\cobheight) {$g$};
        \node [tiny label] (h) at (0.20\cobwidth,-0.5\cobheight) {$h$};
        \node (i) at ([yshift=\toff] A.top) [above] {};
        \node (i2) at ([yshift=-\boff] A.bot) [below] {};
        \draw (f.center) to [out=-120, in=90] (g.center);
        \draw (f.center) to (i);
        \draw (f.center)
            to [out=-60, in=90, looseness=0.9] (h.center);
        \draw (h.center) to [out=-90, in=up, looseness=1.2] (i2.north);
    \end{scope}
\end{tz}
\qquad \qquad
\begin{tz}
    \node[Pants, top, bot] (A) at (0,0) {};
    \node[Cup] (B) at (A.rightleg) {};
    \node[Cyl, bot, anchor=top] (C) at (A.leftleg) {};
    \begin{scope}[internal string scope]
        \node [tiny label] (f) at (0,0) {$f$};
        \node [tiny label] (g) at ([yshift=-0.33\cobheight] B) {$g$};
        \node [tiny label] (h) at (C) {$h$};
        \node (i) at ([yshift=\toff] A.belt) [above] {};
        \node (i2) at ([yshift=-\boff] C.bottom) [below] {};
        \draw (f.center) to [out=-40, in=90] node [right=-1pt] {} (g.center);
        \draw (f.center) to (i);
        \draw (f.center)
            to [out=-140, in=90, looseness=0.9]
                (h.center);
        \draw (h.center) to (i2);
    \end{scope}
\end{tz}
\,\,\xmapsto{\textstyle\isd(\rho)}\,
\begin{tz}
    \node[Cyl, tall, bot, top] (A) at (0,0) {};
    \begin{scope}[internal string scope]
        \node [tiny label] (f) at (0,0.4\cobheight) {$f$};
        \node [tiny label] (g) at (0.2\cobwidth, -0.1\cobheight) {$g$};
        \node [tiny label] (h) at (-0.2\cobwidth,-0.5\cobheight) {$h$};
        \node (i) at ([yshift=\toff] A.top) [above] {};
        \node (i2) at ([yshift=-\boff] A.bot) [below] {};
        \draw (f.center) to [out=-60, in=90] (g.center);
        \draw (f.center) to (i);
        \draw (f.center)
            to [out=-120, in=90, looseness=0.9] (h.center);
        \draw (h.center) to [out=-90, in=90, looseness=1.2] (i2);
    \end{scope}
\end{tz}
\end{equation}
\end{proposition}

\begin{proof}
The action of $\isd(\alpha)$ on the space of string diagrams has the following form, where we label the upper boundary circle by $A$, the intermediate boundary circles by $P, Q$, and the lower boundary circles by $B, C, D$:
\begin{align}
\nonumber
& \hspace{-5pt} \textstyle \bigoplus _{P,Q} \Hom _\cat{S} \big( A, Z(\tikztinypants)(P \boxtimes Q) \big)
\\*
\nonumber
& \hspace{32pt} \boxtimes \Hom _{\cat{S} \boxtimes \cat{S}} \big( P \boxtimes Q, Z(\raisebox{-3pt}
{$\begin{tikzpicture}
\microbordisms
    \node [Pants, bot, top] at (0,0) {};
    \node [SwishL, bot, top] at (1.25\cobwidth+1.25\cobgap,0) {};
\end{tikzpicture}$})
(B \boxtimes C \boxtimes D) \big) \,/ \sim
\\[-0pt]
\nonumber
& \longxarrow{\mathrm{compose}}{\arrlen} \,
\Hom _\cat{S} \big( A,
 Z(\tikztinypants) Z(\raisebox{-3pt}
{$\begin{tikzpicture}
\microbordisms
    \setlength\obscurewidth{0pt}
    \node [Pants, bot, top] at (0,0) {};
    \node [SwishL, bot, top] at (1.25\cobwidth+1.25\cobgap,0) {};
\end{tikzpicture}$})
(B \boxtimes C \boxtimes D) \big)
\\[-4pt]
\nonumber
&
\longxarrow{Z(\alpha) _{B \boxtimes C \boxtimes D}}{\arrlen}
\,
\Hom _\cat{S} \big( A, Z(\tikztinypants) Z(\raisebox{-3pt}{$\begin{tikzpicture}
\microbordisms
\node [Pants, bot, top] at (0,0) {}; \node [SwishR, bot, top] at (-1.25\cobwidth-1.25\cobgap,0) {}; \end{tikzpicture}$})
(B \boxtimes C \boxtimes D) \big)
\\
\nonumber
& \longxarrow{\smash{\mathrm{compose} ^{\inv}}}{\arrlen} \,
\textstyle \bigoplus _{P,Q} \Hom _\cat{S} \big( A, Z(\tikztinypants)(P \boxtimes Q) \big)
\\
\label{eq:associator-isos}
& \hspace{100pt} \boxtimes \Hom _{\cat{S} \boxtimes \cat{S}} \big( P \boxtimes Q, Z(\raisebox{-3pt}
{$\begin{tikzpicture}
\microbordisms
    \node [Pants, bot, top] at (0,0) {};
    \node [SwishR, bot, top] at (-1.25\cobwidth-1.25\cobgap,0) {};
\end{tikzpicture}$})
(B \boxtimes C \boxtimes D) \big) \,/ \sim
\end{align}
Suppose we begin with the element of the direct sum defined by the string diagram on the left-hand side of~\eqref{associator-image}:
\begin{equation}
\nonumber
\big[ P = X, \, Q = Y\big] \,\, \Big\{ A \xrightarrow{f} X \otimes Y \Big\} \boxtimes \Big\{ X \boxtimes Y \xrightarrow{g \boxtimes h} (B \otimes C) \boxtimes D \Big\}
\end{equation}
The data in square brackets indicates the element of the coproduct $\bigoplus_{P,Q}$ in which our term is supported. This evolves in the following manner under the chain of isomorphisms described in~\eqref{eq:associator-isos}:
\begin{align}
\nonumber
& \longxmapsto{\comp}{\arrlen} \Big\{ A \xrightarrow{(g \otimes h) \circ f} (B \otimes C) \otimes D \Big\}
\\
\nonumber
& \longxmapsto{Z(\alpha)_{B \boxtimes C \boxtimes D}}{\arrlen}  \Big\{ A \xrightarrow{(g \otimes h) \circ f} (B \otimes C) \otimes D \xrightarrow{Z(\alpha) _{B \boxtimes C \boxtimes D}} B \otimes (C \otimes D) \Big\}
\\
\nonumber
& \longxmapsto{\comp ^{\inv}}{\arrlen} \big[ P=B,\, Q=C \otimes D\big] \\
\nonumber
& \hspace{90pt} \Big\{ A \xrightarrow{(g \otimes h) \circ f} (B \otimes C) \otimes D \xrightarrow{Z(\alpha) _{B \boxtimes C \boxtimes D}} B \otimes (C \otimes D) \Big\}
\\
& \hspace{90pt} {} \boxtimes \Big\{ B \boxtimes (C \otimes D) \xrightarrow{\id} B \boxtimes (C \otimes D) \Big\}
\end{align}
This final element has the graphical interpretation given by the right-hand side of~\eqref{associator-image}.

We now consider the action of $\isd(\lambda)$. Its action on the space of string diagrams is as follows:
\begin{align}
\nonumber
& \hspace{-20pt} \textstyle \bigoplus _{P,Q} \Hom \big( A, Z(\tikztinypants)(P \boxtimes Q) \big) \boxtimes \Hom (P \boxtimes Q, Z(\raisebox{-3pt}{$\begin{tikzpicture}
\microbordisms
    \node [Cup, bot, top] at (0,0) {};
    \node [Cyl, bot, top, anchor=top] at (\cobwidth+1\cobgap,0) {};
\end{tikzpicture}$})
(B) \big) \,/ \sim
\\[-0pt]
\nonumber
& \longxarrow{\comp}{\arrlen}
\Hom _\cat{S} \big( A, Z(\tikztinypants) Z(
\raisebox{-3pt}
{$\begin{tikzpicture}
\microbordisms
    \node [Cup, bot, top] at (0,0) {};
    \node [Cyl, bot, top, anchor=top] at (\cobwidth+1\cobgap,0) {};
\end{tikzpicture}$})
(B) \big)
\\[-2pt]
& \longxarrow{Z(\lambda) _B}{\arrlen}
\Hom _\cat{S} \big( A, Z(\raisebox{-3pt}
{$\begin{tikzpicture}
\microbordisms
    \node [Cyl, bot, top, anchor=top] at (\cobwidth+1\cobgap,0) {};
\end{tikzpicture}$})
(B) \big) = \Hom _\cat{S} \big( A, B \big)
\intertext{This has the following effect on our chosen string diagram:}
\nonumber
& \hspace{-20pt} \big[ P=X, \, Q=Y \big] \,\, \Big\{ A \xrightarrow{f} X \otimes Y \Big\} \boxtimes \Big\{ X \boxtimes Y \xrightarrow{ g \boxtimes h } I \boxtimes B\Big\}
\\
\nonumber
& \longxmapsto{\comp}{\arrlen} \Big\{ A \xrightarrow{(g \otimes h) \circ f} I \otimes B \Big\}
\\
& \longxmapsto{Z(\lambda) _B}{\arrlen} \Big\{ A \xrightarrow{(g \otimes h) \circ f} I \otimes B \xrightarrow{Z(\lambda) _B} B \Big\}
\end{align}
The action of $\isd(\rho)$ can be verified similarly.
\end{proof}

Note that the actions given in this proposition can be deduced from the following actions on a smaller class of string diagrams, involving no nontrivial internal morphisms:
\setlength\ellipseheight{0.8\ellipseheight}
\setlength\cupheight{1.2\cupheight}
\mediumbordisms
\begin{align}
\label{associator-simpleimage}
    \begin{tz}
    \node[Pants, bot, top, belt scale=1.5] (B) at (0,0) {};
    \node[Pants, bot, anchor=belt] (A) at (B.leftleg) {};
    \node[SwishL, bot, anchor=top] (C) at (B.rightleg) {};
    \begin{scope}[internal string scope]
        \node (i) at (B.belt) [above=\toff] {};
        \node (j) at (A.leftleg) [below=\boff] {};
        \node (k) at (A.rightleg) [below=\boff] {};
        \node (l) at (C.bot) [below=\boff] {};
        \draw (j.north)
            to (A.leftleg)
            to [out=up, in=-90] (B-leftleg.in-leftthird)
            to [out=up, in=-90] (B-belt.in-leftquarter)
            to [out=up, in=-90] ([yshift=\toff] B-belt.in-leftquarter);
        \draw (k.north)
            to (A.rightleg)
            to [out=up, in=-90] (B-leftleg.in-rightthird)
            to [out=up, in=-90] (B-belt)
            to [out=up, in=-90] ([yshift=\toff] B.belt);
        \draw (l.north)
            to (C.bot)
            to [out=up, in=-90] (B.rightleg)
            to [out=up, in=-90] (B-belt.in-rightquarter)
            to [out=up, in=-90] ([yshift=\toff] B-belt.in-rightquarter);
    \end{scope}
    \end{tz}
    &\xmapsto{\textstyle\isd(\alpha)}
    \begin{tz}
    \node[Pants, bot, top, belt scale=1.5] (B) at (0,0) {};
    \node[Pants, bot, anchor=belt] (A) at (B.rightleg) {};
    \node[SwishR, bot, anchor=top] (C) at (B.leftleg) {};
    \begin{scope}[internal string scope]
        \node (i) at (B.belt) [above=\toff] {};
        \node (j) at (C.bot) [below=\boff] {};
        \node (k) at (A.leftleg) [below=\boff] {};
        \node (l) at (A.rightleg) [below=\boff] {};
        \draw (j.north)
            to (C.bot)
            to [out=up, in=-90] (B-leftleg)
            to [out=up, in=-90] (B-belt.in-leftquarter)
            to [out=up, in=-90] ([yshift=\toff] B-belt.in-leftquarter);
        \draw (k.north)
            to (A.leftleg)
            to [out=up, in=-90] (B-rightleg.in-leftthird)
            to [out=up, in=-90] (B-belt)
            to [out=up, in=-90] ([yshift=\toff] B.belt);
        \draw (l.north)
            to (A.rightleg)
            to [out=up, in=-90] (B-rightleg.in-rightthird)
            to [out=up, in=-90] (B-belt.in-rightquarter)
            to [out=up, in=-90] ([yshift=\toff] B-belt.in-rightquarter);
    \end{scope}
    \end{tz}
&
\begin{tz}
    \node[Pants, top, bot] (A) at (0,0) {};
    \node[Cup] (B) at (A.leftleg) {};
    \node[Cyl, bot, anchor=top] (C) at (A.rightleg) {};
    \begin{scope}[internal string scope]
        \node (i) at ([yshift=\toff] A.belt) [above] {};
        \node (i2) at ([yshift=-\boff] C.bottom) [below] {};
        \draw (i2) to (C.top) to [out=up, in=-90] (A.belt) to (i);
    \end{scope}
\end{tz}
&\xmapsto{\textstyle\isd(\lambda)}
\begin{tz}
    \node[Cyl, tall, bot, top] (A) at (0,0) {};
    \begin{scope}[internal string scope]
        \node (i) at ([yshift=\toff] A.top) [above] {};
        \node (i2) at ([yshift=-\boff] A.bot) [below] {};
        \draw (i2) to (i);
    \end{scope}
\end{tz}
&
\begin{tz}
    \node[Pants, top, bot] (A) at (0,0) {};
    \node[Cup] (B) at (A.rightleg) {};
    \node[Cyl, bot, anchor=top] (C) at (A.leftleg) {};
    \begin{scope}[internal string scope]
        \node (i) at ([yshift=\toff] A.belt) [above] {};
        \node (i2) at ([yshift=-\boff] C.bottom) [below] {};
        \draw (i2) to (C.top) to [out=up, in=-90] (A.belt) to (i);
    \end{scope}
\end{tz}
&\xmapsto{\textstyle\isd(\rho)}
\begin{tz}
    \node[Cyl, tall, bot, top] (A) at (0,0) {};
    \begin{scope}[internal string scope]
        \node (i) at ([yshift=\toff] A.top) [above] {};
        \node (i2) at ([yshift=-\boff] A.bot) [below] {};
        \draw (i2) to (i);
    \end{scope}
\end{tz}
\end{align}
For example, the general action~\eqref{associator-image} of $\isd (\alpha)$ can be recovered as follows, by composing with an identity cylinder:
\normalbordisms
\begin{equation}
    \begin{tz}
    \node[Pants, bot, top] (B) at (0,0) {};
    \node[Pants, bot, anchor=belt] (A) at (B.leftleg) {};
    \node[SwishL, bot, anchor=top] (C) at (B.rightleg) {};
    \begin{scope}[internal string scope]
        \node (i) at (B.belt) [above=\toff] {};
        \node (j) at (A.leftleg) [below=\boff] {};
        \node (k) at (A.rightleg) [below=\boff] {};
        \node (l) at (C.bot) [below=\boff] {};
        \node [tiny label] (f) at (0,0.0) {$f$};
        \node [tiny label] (g) at (A.center) {$g$};
        \node [tiny label] (h) at (C.center) {$h$};
        \draw (f.center) to (i.south);
        \draw (j.north) to [out=up, in=-135] (g.center);
        \draw (k.north) to [out=up, in=-35] (g.center);
        \draw (g.center) to [out=90, in=-135] node [left=-4pt, pos=0.65] {} (f.center);
        \draw (l.north)
            to [out=90, in=-80] (h.center) 
            to [out=up, in=down, out looseness=0.7]
                (B.rightleg)
            to [out=up, in=-45]
                node [right=-4pt, pos=0.251] {}
 (f.center);
    \end{scope}
    \end{tz}
\ignore{=
\begin{tz}
    \node[Pants, bot] (B) at (0,0) {};
    \node[Pants, bot, anchor=belt] (A) at (B.leftleg) {};
    \node[SwishL, bot, anchor=top] (C) at (B.rightleg) {};
    \node (D) [Cyl, anchor=bot, bot, top] at (B.belt) {};
    \begin{scope}[internal string scope]
        \node (i) at (D.top) [above=\toff] {};
        \node (j) at (A.leftleg) [below=\boff] {};
        \node (k) at (A.rightleg) [below=\boff] {};
        \node (l) at (C.bot) [below=\boff] {};
        \begin{pgfonlayer}{label}
        \node [tiny label] (f) at (0,0.13) {$f$};
        \node [tiny label] (g) at (A.center) {$g$};
        \node [tiny label] (h) at (C.center) {$h$};
        \end{pgfonlayer}
        \draw (f) to (i.south);
        \draw (j.north) to [out=up, in=-140] (g);
        \draw (k.north) to [out=up, in=-40] (g);
        \draw (g) to [out=90, in=-120] node [left=-2pt] {} (f);
        \draw (l.north)
            to [out=90, in=-80] (h) 
            to [out=100, in=-50, looseness=0.8]
                node [right=-2pt, pos=0.6] {$Y$} (f);
    \end{scope}
\end{tz}}
=
\begin{tz}
    \node[Pants, bot, belt scale=1.5] (B) at (0,0) {};
    \node[Pants, bot, anchor=belt] (A) at (B.leftleg) {};
    \node[SwishL, bot, anchor=top] (C) at (B.rightleg) {};
    \node [Cyl, anchor=bot, bot, top, top scale=1.5, bottom scale=1.5] (D) at (B.belt) {};
    \begin{scope}[internal string scope]
        \node (i) at (D.top) [above=\toff] {};
        \node (j) at (A.leftleg) [below=\boff] {};
        \node (k) at (A.rightleg) [below=\boff] {};
        \node (l) at (C.bot) [below=\boff] {};
        \node [tiny label] (f) at (0\cobwidth,1.10\cobheight) {$f$};
        \node [tiny label] (g) at (-0.35\cobwidth,0.72\cobheight) {$g$};
        \node [tiny label] (h) at (0.35\cobwidth,0.72\cobheight) {$h$};
        \draw (j.north)
            to (A.leftleg)
            to [out=up, in=-90] (B-leftleg.in-leftthird)
            to [out=up, in=-90] ([xshift=-3pt] B-belt.in-leftquarter)
            to [out=90, in=-110] (g.center);
        \draw (k.north)
            to (A.rightleg)
            to [out=up, in=-90] (B-leftleg.in-rightthird)
            to [out=up, in=-90] (B-belt);
        \draw (l.north)
            to (C.bot)
            to [out=up, in=-90] (B.rightleg)
            to [out=up, in=-90] (h.center);
        \draw (B-belt) to [out=90, in=-70] (g.center);
        \draw (g.center) to [out=90, in=-110] (f.center);
        \draw (h.center) to [out=90, in=-50] (f.center);
        \draw (f.center) to (i.south);
    \end{scope}
    \end{tz}
\hspace{-5pt}\xmapsto{\textstyle\isd(\alpha)}\hspace{-5pt}
\begin{tz}
    \node[Pants, bot, belt scale=1.5] (B) at (0,0) {};
    \node[Pants, bot, anchor=belt] (A) at (B.rightleg) {};
    \node[SwishR, bot, anchor=top] (C) at (B.leftleg) {};
    \node [Cyl, anchor=bot, bot, top, top scale=1.5, bottom scale=1.5] (D) at (B.belt) {};
    \begin{scope}[internal string scope]
        \node (i) at (D.top) [above=\toff] {};
        \node (j) at (C.bot) [below=\boff] {};
        \node (k) at (A.leftleg) [below=\boff] {};
        \node (l) at (A.rightleg) [below=\boff] {};
        \node [tiny label] (f) at (0\cobwidth,1.1\cobheight) {$f$};
        \node [tiny label] (g) at (-0.35\cobwidth,0.72\cobheight) {$g$};
        \node [tiny label] (h) at (0.35\cobwidth,0.72\cobheight) {$h$};
        \draw (j.north)
            to (C.bot)
            to [out=up, in=-90] (B-leftleg)
            to [out=up, in=-90] ([xshift=-3pt] B-belt.in-leftquarter)
            to [out=up, in=-135] (g.center);
        \draw (k.north)
            to (A.leftleg)
            to [out=up, in=-90] (B-rightleg.in-leftthird)
            to [out=up, in=-90] (B-belt)
            to [out=up, in=-45] (g.center);
        \draw (l.north)
            to (A.rightleg)
            to [out=up, in=-90] (B-rightleg.in-rightthird)
            to [out=up, in=-90] ([xshift=2pt] B-belt.in-rightquarter)
            to (h.center);
        \draw (g.center) to [out=90, in=-110] (f.center);
        \draw (h.center) to [out=90, in=-50] (f.center);
        \draw (f.center) to (i.south);
    \end{scope}
    \end{tz}
=
    \begin{tz}
    \node[Pants, bot, top, height scale=1.2] (B) at (0,0) {};
    \node[Pants, bot, anchor=belt] (A) at (B.rightleg) {};
    \node[SwishR, bot, anchor=top] (C) at (B.leftleg) {};
    \begin{scope}[internal string scope]
        \node (i) at (B.belt) [above=\toff] {};
        \node (j) at (C.bot) [below=\boff] {};
        \node (k) at (A.leftleg) [below=\boff] {};
        \node (l) at (A.rightleg) [below=\boff] {};
        \node [tiny label] (f) at (0.0\cobwidth,0.20\cobheight) {$f$};
        \node [tiny label] (g)  at (-0.3\cobwidth,-0.15\cobheight) {$g$};
        \node [tiny label] (h) at (0.8\cobwidth,-0.25\cobheight) {$h$};
        \draw (j.north)
            to [out=up, in=-130] (g.center);
        \draw (k.north)
            to [out=up, in=down] (B-rightleg.in-leftthird)
            to [out=up, in=-40] (g.center);
        \draw (l.north) to [out=up, in=down] (B-rightleg.in-rightthird)
            to [out=up, in=-60] (f.center);
        \draw (f.center) to [out=90, in=-90, looseness=2] (i.south);
        \draw (f.center) to [in=90, out=-120] (g.center);
    \end{scope}
    \end{tz}
\end{equation}
The general actions~\eqref{units-image} of $\isd(\lambda)$ and $\isd(\rho)$ can be recovered in a similar way. We will exploit this feature for some of the other generators, and only give their actions on simple string diagrams which are sufficient to recover the full action.

\subsubsection{Actions of the merging generators $\epsilon$, $\eta$, $\mu$, and $\nu$}

\begin{proposition}
\label{thm:eemninternalstring}
For a linear representation $Z$ of a presentation that 2\-extends $\P^\L$, the generators $\epsilon$, $\eta$, $\mu$, and $\nu$ act in the following ways on internal string diagrams:
\normalbordisms
\begin{align}
\label{nu-and-mu-image}
\begin{tz}
\draw[dotted] (0,0) rectangle (2\cupheight, 2\cupheight);
\end{tz}
\,\,&\xmapsto{\textstyle\isd (\nu)}\,\,
\begin{tz}
    \node[Cup] (A) at (0,0) {};
    \node[Cap, bot] (B) at (0,0) {};
\end{tz}
&
\begin{tz}
    \setlength{\cupheight}{1.3\cupheight}
    \node[Cup, top] (A) at (0,0) {};
    \node[Cap, bot] (B) at (0,-1.5\cobheight) {};
    \node [tiny label, text=red] (f) at ([yshift=-0.56\cupheight] A) {$\shrunkenf$};
    \node [tiny label, text=red] (g) at ([yshift=0.56\cupheight] B) {$g$};
    \begin{scope}[internal string scope]
        \node (i) at (A) [above=\toff] {};
        \node (j) at (B) [below=\boff] {};
        \draw (f.center) to (i);
        \draw (g.center) to (j);
    \end{scope}
\end{tz}
\,\,&\xmapsto{\textstyle\isd (\mu)} \,\,
\begin{tz}
    \setlength\cupheight{1.3\cupheight}
    \node[Cyl, top, bot=false] (A) at (0,0) {};
    \node[Cyl, bot] (B) at (0,-0.5\cobheight) {};
    \node [tiny label, text=red] (f) at ([yshift=-0.56\cupheight] A.top) {$f$};
    \node [tiny label, text=red] (g) at ([yshift=0.56\cupheight] B.bot) {$g$};
    \begin{scope}[internal string scope]
        \node (i) at (A.top) [above=\toff] {};
        \node (j) at (B.bot) [below=\boff] {};
        \draw (f.center) to (i);
        \draw (g.center) to (j);
    \end{scope}
\end{tz}
\\
\label{eta-and-epsilon-image}
\begin{tz}
        \node[Cyl, tall, top, bot] (A) at (0,0) {};
        \node[Cyl, tall, top, bot] (B) at (\cobgap + \cobwidth, 0) {};
        \begin{scope}[internal string scope]
                \draw ([yshift=\toff] A.top) node[above]{} -- ([yshift=-\boff] A.bot) node[below]{};
                \draw ([yshift=\toff] B.top) node[above]{} -- ([yshift=-\boff] B.bot) node[below]{};
        \end{scope}
\end{tz}
\,\,&\xmapsto{\textstyle\isd (\eta)}
\begin{tz}
        \node[Pants, bot, belt scale=1.5] (A) at (0,0) {};
        \node[Copants, bot, anchor=belt, belt scale=1.5, top] (B) at (A.belt) {}; 
        \begin{scope}[internal string scope]
                \node (i) at ([yshift=\toff] B.leftleg) [above] {};
                \node (j) at ([yshift=\toff] B.rightleg) [above] {};
                \node (i2) at ([yshift=-\boff] A.leftleg) [below] {};
                \node (j2) at ([yshift=-\boff] A.rightleg) [below] {};
                \draw (i.south) to[out=down, in=up] (B-belt.in-leftthird) to[out=down, in=up] (i2.north);
                \draw (j.south) to[out=down, in=up] (B-belt.in-rightthird) to[out=down, in=up] (j2.north);
        \end{scope}
\end{tz}
&
\begin{tz}
        \node[Pants, bot, top] (A) at (0,0) {};
        \node[Copants, bot, anchor=leftleg] (B) at (A.leftleg) {};
        \begin{scope}[internal string scope]
                \node (i) at ([yshift=\toff] A.belt) [above] {};
                \node (i2) at ([yshift=-\boff] B.belt) [below] {};
                \node[tiny label] (p) at ([yshift=0.05\cobheight] A.center) {$f$};
                \node[tiny label] (q) at ([yshift=-0.1\cobheight] B.center) {$g$};
                \draw (i.south)
                    to (p.center)
                    to [out=-120, in=up]
                        (A.leftleg)
                        to[out=down, in=120] (q.center);
                \draw (p.center) to[out=-60, in=up]
                (A.rightleg) to[out=down, in=60] (q.center);
                \draw (q.center) -- (i2.north);
        \end{scope}
\end{tz}
&\,\,\xmapsto{\textstyle\isd (\epsilon)}\,\,
\begin{tz}
        \node[Cyl, top, bot=false] (A) at (0,0) {};
        \node[Cyl, bot] (A2) at (0,-\cobheight) {};
        \begin{scope}[internal string scope]
                \node (i) at ([yshift=\toff] A.top) [above] {};
                \node (i2) at ([yshift=-\boff] A2.bottom) [below] {};
                \node[tiny label] (p) at (A.center) {$f$};
                \node[tiny label] (q) at (A2.center) {$g$};
                \draw (i.south) -- (p.center) to[out=-120, in=120] node[left, xshift=0.1cm] {} (q.center);
                \draw (p.center) to[out=-60, in=60] node[right, xshift=-0.1cm] {} (q.center);
                \draw (q.center) -- (i2.north);
        \end{scope}
\end{tz}
\end{align}
\end{proposition}

\begin{proof}
These generators are the units and counits  expressing the adjunctions $\tikztinypants \dashv \tikztinycopants$ and $\tikztinycup \dashv \tikztinycap$. Under the image of $\isd$, they are the units and counits expressing the adjunctions $Z(\tikztinypants) _* \dashv Z(\tikztinypants) ^*$ and $Z(\tikztinycup) _* \dashv Z(\tikztinycup) ^*$. \autoref{adj_lemma} provides an explicit formula for the unit and counit of the adjunction $F_* \dashv F^*$, which specializes to the above formulas. 
\end{proof}

\subsubsection{Actions of dual generators}

If we have a 2\-morphism $\sigma: A \doubleto B$ such that $A$ and $B$ have right duals, then we can construct the right dual 2\-morphism $\sigma^* : B ^* \doubleto A ^*$  in the following way:
\begin{equation}
B ^*
\xdoubleto{{\eta \circ B ^*}}
A^* \circ A \circ B ^*
\xdoubleto{A^* \circ \sigma \circ B ^*}
A^* \circ B \circ B^*
\xdoubleto{A^* \circ \epsilon'}
A^*
\end{equation}
Here $\eta : \id \doubleto A^* \circ A$ is the unit for the adjunction $A \dashv A^*$, and $\epsilon': B \circ B^* \doubleto \id$ is the counit for the adjunction $B \dashv B^*$
For a presentation that 2\-extends $\P^\L$, the generating 1\-morphisms $\tikztinypants$ and $\tikztinycup$ have right adjoints, and so we can employ this technique to construct duals to our standard generators, and calculate their actions on internal string diagrams.
\begin{proposition}
For a linear representation $Z$ of a presentation that 2\-extends $\P^\L$, the right duals $\rho ^*$, $\lambda ^*$, and $\alpha ^*$ act in the following way on internal string diagrams:
\mediumbordisms
\begin{align}
\begin{tz}
    \node[Cyl, tall, bot, top] (A) at (0,0) {};
    \node (i) at ([yshift=-\boff] A.bot) [below] {};
    \node (i2) at ([yshift=\toff] A.top) [above] {};
    \draw[internal string] (i2) -- (i);
\end{tz}
&\xmapsto{\textstyle\isd(\rho^*)}
\begin{tz}
    \node[Copants, bot] (A) at (0,0) {};
    \node[Cap, bot] (B) at (A.rightleg) {};
    \node[Cyl, bot, top, anchor=bottom] (C) at (A.leftleg) {};
    \node (i) at ([yshift=-\boff] A.belt) [below] {};
    \node (i2) at ([yshift=\toff] C.top) [above] {};
    \begin{scope}[internal string scope]
        \draw (i.north) to (A.belt) to [out=up, in=down] (C.bottom) to (i2.south);
    \end{scope}
\end{tz}
&
\begin{tz}
        \node[Cyl, tall, bot, top] (A) at (0,0) {};
        \node (i) at ([yshift=-\boff] A.bot) [below] {};
        \node (i2) at ([yshift=\toff] A.top) [above] {};
        \draw[internal string] (i2) -- (i);
\end{tz}
&\xmapsto{\textstyle\isd(\lambda^*)}
\begin{tz}
    \node[Copants, bot] (A) at (0,0) {};
    \node[Cap, bot] (B) at (A.leftleg) {};
    \node[Cyl, bot, top, anchor=bottom] (C) at (A.rightleg) {};
    \node (i) at ([yshift=-\boff] A.belt) [below] {};
    \node (i2) at ([yshift=\toff] C.top) [above] {};
    \begin{scope}[internal string scope]
        \draw (i.north) to (A.belt) to [out=up, in=down] (C.bottom) to (i2.south);
    \end{scope}
\end{tz}
&
\begin{tz}
    \node[Copants, bot, belt scale=1.5] (B) at (0,0) {};
    \node[Copants, bot, top, anchor=belt] (A) at (B.leftleg) {};
    \node[SwishR, bot, top, anchor=bot] (C) at (B.rightleg) {};
    \begin{scope}[internal string scope, out=down, in=up]
        \draw ([yshift=\toff] A.leftleg)
            to (A.leftleg)
            to (A-belt.in-leftthird)
            to (B-belt.in-leftquarter)
            to ([yshift=-\boff] B-belt.in-leftquarter);
        \draw ([yshift=\toff] A.rightleg)
            to (A.rightleg)
            to (A-belt.in-rightthird)
            to (B.belt)
            to ([yshift=-\boff] B.belt);
        \draw ([yshift=\toff] C.top)
            to (C.top)
            to (B.rightleg)
            to (B-belt.in-rightquarter)
            to ([yshift=-\boff] B-belt.in-rightquarter);
    \end{scope}
\end{tz}
&\xmapsto{\textstyle\isd(\alpha^*)}
\begin{tz}
    \node[Copants, bot, belt scale=1.5] (B) at (0,0) {};
    \node[Copants, bot, top, anchor=belt] (A) at (B.rightleg) {};
    \node[SwishL, bot, top, anchor=bot] (C) at (B.leftleg) {};
    \begin{scope}[internal string scope, out=down, in=up]
        \draw ([yshift=\toff] A.leftleg)
            to (A.leftleg)
            to (A-belt.in-leftthird)
            to (B.belt)
            to ([yshift=-\boff] B.belt);
        \draw ([yshift=\toff] A.rightleg)
            to (A.rightleg)
            to (A-belt.in-rightthird)
            to (B-belt.in-rightquarter)
            to ([yshift=-\boff] B-belt.in-rightquarter);
        \draw ([yshift=\toff] C.top)
            to (C.top)
            to (C.bot)
            to (B-belt.in-leftquarter)
            to ([yshift=-\boff] B-belt.in-leftquarter);
    \end{scope}
\end{tz}
\end{align}
\end{proposition}
\begin{proof}
We apply the definition of the dual of a 2\-morphism as given above. For $\rho ^*$, we obtain the following:
\mediumbordisms
\begin{equation}
\begin{tz}
        \node[Cyl, tall, bot, top] (A) at (0,0) {};
        \node (i) at ([yshift=-\boff] A.bot) [below] {};
        \node (i2) at ([yshift=\toff] A.top) [above] {};
        \draw[internal string] (i2) -- (i);
\end{tz}
\, \xmapsto{\textstyle\isd(\nu)} \,
\begin{tz}
        \node[Cyl, tall, bot, top] (A) at (0,0) {};
        \node (i) at ([yshift=-\boff] A.bot) [below] {};
        \node (i2) at ([yshift=\toff] A.top) [above] {};
        \draw[internal string] (i2) to (i);
        \node[Cap, bot] (B) at (1.5\cobwidth,0) {};
        \node[Cup] (C) at (B.center) {};
\end{tz}
\xmapsto{\textstyle\isd(\eta)}
\begin{tz}
        \node (T) [Cyl, bot, top] (A) at (0,0) {};
        \node[Copants, bot, anchor=leftleg, belt scale=1.5] (B) at (A.bottom) {};
        \node[Cap, bot] (C) at (B.rightleg) {};
        \node[Pants, bot, anchor=belt, belt scale=1.5] (D) at (B.belt) {};
        \node[Cup] (E) at (D.rightleg) {};
        \node (B) [Cyl, bot, anchor=top] (F) at (D.leftleg) {};
        \begin{scope}[internal string scope]
        \draw ([yshift=\toff] A.top) -- (A.bottom) to[out=down, in=up] (B-belt.in-leftthird) to[out=down, in=up] (D.leftleg) to[out=down, in=up] ([yshift=-\boff] F.bottom);
        \end{scope}
\end{tz}
\, \xmapsto{\textstyle \isd(\rho)} \,
\begin{tz}
    \node[Copants, bot] (A) at (0,0) {};
    \node[Cap, bot] (B) at (A.rightleg) {};
    \node[Cyl, bot, top, anchor=bottom] (C) at (A.leftleg) {};
    \node (i) at ([yshift=-\boff] A.belt) [below] {};
    \node (i2) at ([yshift=\toff] C.top) [above] {};
    \begin{scope}[internal string scope]
        \draw (i.north) to (A.belt) to [out=up, in=down] (C.bottom) to (i2.south);
    \end{scope}
\end{tz}
\end{equation}
The actions of $\lambda ^*$ and $\alpha ^*$ can be obtained in a similar way.
\end{proof}

\subsection{Balanced braided monoidal categories}

A linear representation of the balanced presentation $\B$ is the same as a linear braided balanced monoidal category. The classical graphical string calculus for these categories includes the braiding and the twist. We can use this to describe the action on internal string diagrams of the 2\-morphism generators $\beta$ and $\theta$ of any presentation that 2\-extends $\B^\L$. 
\begin{proposition}
\label{thm:btinternalstring}
For a linear representation $Z$ of a presentation that 2\-extends $\B^\L$, the generators $\beta$ and $\theta$ act in the following way on internal string diagrams:
\begin{align}
\label{braiding-and-twist-image}
\begin{tz}
    \node[Pants, bot, top, height scale=1.3] (A) at (0,0) {};
    \begin{scope}[internal string scope]
        \draw ([yshift=\toff] A-belt.in-leftthird)
                to +(0,-0.3)
            to [out=down, in=up] (A.leftleg)
            to +(0,-0.3);
        \draw ([yshift=\toff] A-belt.in-rightthird)
                to +(0,-0.3)
            to [out=down, in=up] (A.rightleg)
            to +(0,-0.3);
    \end{scope}
\end{tz}
&\gap\xmapsto{\textstyle\isd(\beta)}\gap
\begin{tz}
    \node[Pants, bot, top, height scale=1.3] (A) at (0,0) {};
    \node[BraidA, bot, anchor=topleft] (B) at (A.leftleg) {};
    \draw[internal string] ([yshift=\toff] A-belt.in-rightthird)
        to (A-belt.in-rightthird)
        to [out=down, in=up, out looseness=1.7, in looseness=0.7] (B.topleft);
    \draw [internal string] (B.bottomright) to +(0,-0.3);
    \draw[internal string] ([yshift=\toff] A-belt.in-leftthird)
        to +(0,-0.3)
        to [out=down, in=up, out looseness=1.7, in looseness=0.7] (A.rightleg)
        to [out=down, in=up, looseness=0.6] (B.bottomleft)
        to +(0,-0.3);
    \obscureA{B}
    {
\draw [solid, red] (B.topleft) to [out=down, in=up, looseness=0.6] (B.bottomright);
    }
\end{tz}
&
\begin{tz}
    \node[Cyl, bot, top, height scale=1.3] (A) at (0,0) {};
    \begin{scope}[internal string scope]
        \node (i) at ([yshift=\toff] A.top) [above] {};
        \node (i2) at ([yshift=-\boff] A.bot) [below] {};
        \draw (i.south) -- (i2.north);
    \end{scope}
\end{tz}
&\gap\xmapsto{\textstyle\isd(\theta)}\gap
\begin{tz}
    \node[Cyl, bot, top, height scale=1.3] (A) at (0,0) {};
    \begin{scope}[internal string scope]
        \node (i) at ([yshift=\toff] A.top) [above] {};
        \node (i2) at ([yshift=-\boff] A.bot) [below] {};
        \draw (i.south) -- (i2.north);
        \node [tiny label] at (A.center) {$\theta$};
    \end{scope}
\end{tz}
\end{align}
\end{proposition}

\noindent
Before proving the proposition, let us examine in detail what these pictures represent. Under the equivalence $\bicat F(\M) \simeq \Bordcsig$, $\beta$ corresponds to an invertible bordism whose source is the pants surface, and whose target is the pants surface composed with the crossing~\cite{PaperII}. Since the symmetric monoidal 2\-category $\bicat F(\M)$ is strictly symmetric, there is no distinction to be made between an over- or under-crossing at the level of the bordism. In red, we draw the internal string diagrams, which for the case of a presentation that 2\-extends $\B^\L$ is drawn in the language of a balanced braided monoidal category. In the target of $\isd(\beta)$ drawn above, the upper crossing of red strands represents the braiding morphism, but the lower crossing does not. In the target of $\isd(\theta)$, the red $\theta$ represents the balanced structure of the monoidal category.

\begin{proof}
\def\arrlen{48pt}
The action of $\isd(\beta)$ has the following form:
\begin{align}
\nonumber
& \hspace{-20pt} \Hom _\cat{S} \big( A, Z(\tikztinypants) (B \boxtimes C) \big)
\\*
\nonumber
& \longxarrow{Z(\beta)_{B \boxtimes C}}{\arrlen} \Hom _\cat{S} \big( A, Z(\tikztinypants) Z(\raisebox{-3pt}
{$\begin{tikzpicture}
\microbordisms
    \node [BraidA, bot, top] at (0,0) {};
\end{tikzpicture}$})
(B \boxtimes C) \big)
\\*
& \longxarrow{\comp ^{\inv}}{\arrlen} \Hom_\cat{S} \big( A, Z(\tikztinypants) (C \boxtimes B) \big)
\end{align}
This establishes the action of $\isd(\beta)$. The action of $\isd(\theta)$ can be shown similarly.
\end{proof}

\subsection{Ribbon categories}
\label{sec:rigidity}

We now consider a presentation that 2\-extends the ribbon presentation $\R$. This means we have available the Frobeniusator 2\-morphisms $\phiN$ and $\phiM$, which are invertible. We investigate their actions on internal string diagrams.
\begin{proposition}
For a linear representation $Z$ of a presentation that 2\-extends the ribbon presentation $\R$, the Frobeniusators $\phiN$ and $\phiM$ act on internal string diagrams in the following ways:
\mediumbordisms
\begin{align}
\label{eq:phiaction}
\begin{tz}
    \node[Pants, bot, top, height scale=1.2] (A) at (0,0) {};
    \node[Cyl, bot, height scale=1.2] [anchor=top] at (A.leftleg) (B) {};
    \node[Copants, bot, anchor=leftleg, height scale=1.2] at (A.rightleg) (D) {};
    \node[Cyl, bot, top, anchor=bottom, height scale=1.2] at (D.rightleg) (C) {};
    \node (f) [tiny label] at ([yshift=1pt] A.center) {$f$};
    \node (g) [tiny label] at ([yshift=-2pt] D.center) {$g$};
    \begin{scope}[internal string scope]
        \node (i) at ([yshift=\toff] A.belt) [above] {};
        \node (j) at ([yshift=\toff] C.top) [above] {};
        \node (k) at ([yshift=-\boff] B.bottom) [below] {};
        \node (l) at ([yshift=-\boff] D.belt) [below] {};
        \draw (i.south) -- (f.center);
        \draw [out=210, in=up] (f.center) to (B.top);
        \draw (f.center)
            to [out=-45, in=100] (D.leftleg)
            to [out=-80, in=135] (g.center);
        \draw [out=45, in=down] (g.center) to (C.bottom);
        \draw [out=up, in=down] (C.bottom) to (j.south);
        \draw [out=down, in=up] (g.center) to (l.north);
        \draw [out=down, in=up] (A.leftleg) to (k.north);
    \end{scope}
\end{tz}
&\,\,\xmapsto{\textstyle \isd(\phiN)}
\begin{tz}
    \node[Copants, bot, top, belt scale=2, height scale=1.2] (A) at (0,0) {};
    \node[Pants, bot, anchor=belt, belt scale=2, height scale=1.2] (B) at (A.belt) {};
    \begin{scope}[internal string scope]
        \node (i) at ([yshift=\toff] A.leftleg) [above] {};
        \node (j) at ([yshift=\toff] A.rightleg) [above] {};
        \node (k) at ([yshift=-\boff] B.leftleg) [below] {};
        \node (l) at ([yshift=-\boff] B.rightleg) [below] {};
        \node[tiny label] (p) at ([xshift=-0.35\cobwidth, yshift=-0.2\cobheight] A.center) {$f$};
        \node[tiny label] (q) at ([xshift=0.35\cobwidth, yshift=0.1\cobheight] B.center) {$g$};
        \draw (i.south)
            to[out=down, in=up, out looseness=1] (p.north)
            to (p.center)
            to[out=-125, in=up, looseness=0.8] (k.north);
        \draw (j.south)
            to [out=down, in=55, looseness=0.8] (q.center);
        \draw (q.center)
            to[out=down, in=up] (l.north);
        \draw (p.center)
            to[out=-45, in=135] (q.center);                
    \end{scope} 
\end{tz}
&
\begin{tz}[xscale=-1]
    \node[Pants, bot, top, height scale=1.2] (A) at (0,0) {};
    \node[Cyl, bot, height scale=1.2] [anchor=top] at (A.rightleg) (B) {};
    \node[Copants, bot, anchor=rightleg, height scale=1.2] at (A.leftleg) (D) {};
    \node[Cyl, bot, top, anchor=bottom, height scale=1.2] at (D.leftleg) (C) {};
    \node (f) [tiny label] at ([yshift=2pt] A.center) {$\tiny f$};
    \node (g) [tiny label] at ([yshift=-2pt] D.center) {$g$};
    \begin{scope}[internal string scope]
        \node (i) at ([yshift=\toff] A.belt) [above] {};
        \node (j) at ([yshift=\toff] C.top) [above] {};
        \node (k) at ([yshift=-\boff] B.bottom) [below] {};
        \node (l) at ([yshift=-\boff] D.belt) [below] {};
        \draw (i.south) -- (f.center);
        \draw [out=210, in=up] (f.center) to (B.top);
        \draw (f.center)
            to [out=-45, in=100] (D.rightleg)
            to [out=-80, in=135] (g.center);
        \draw [out=45, in=down] (g.center) to (C.bottom);
        \draw [out=up, in=down] (C.bottom) to (j.south);
        \draw [out=down, in=up] (g.center) to (l.north);
        \draw [out=down, in=up] (A.rightleg) to (k.north);
    \end{scope}
\end{tz}
&\,\,\xmapsto{\textstyle \isd(\phiM)}
\begin{tz}[xscale=-1]
    \node[Copants, bot, top, belt scale=2, height scale=1.2] (A) at (0,0) {};
    \node[Pants, bot, anchor=belt, belt scale=2, height scale=1.2] (B) at (A.belt) {};
    \begin{scope}[internal string scope]
        \node (i) at ([yshift=\toff] A.rightleg) [above] {};
        \node (j) at ([yshift=\toff] A.leftleg) [above] {};
        \node (k) at ([yshift=-\boff] B.rightleg) [below] {};
        \node (l) at ([yshift=-\boff] B.leftleg) [below] {};
        \node[tiny label] (p) at ([xshift=-0.35\cobwidth, yshift=-0.2\cobheight] A.center) {$f$};
        \node[tiny label] (q) at ([xshift=0.35\cobwidth, yshift=0.1\cobheight] B.center) {$g$};
        \draw (i.south)
            to[out=down, in=up, out looseness=1] (p.north)
            to (p.center)
            to[out=-125, in=up, looseness=0.8] (k.north);
        \draw (j.south)
            to [out=down, in=55, looseness=0.8] (q.center);
        \draw (q.center)
            to[out=down, in=up] (l.north);
        \draw (p.center)
            to[out=-45, in=135] (q.center);                
    \end{scope} 
\end{tz}
\end{align}
\end{proposition}
\begin{proof}
We apply the definitions of the Frobeniusators in terms of $\eta$, $\alpha$, and $\epsilon$, whose actions on internal string diagrams have already been computed. For the case of $\phiN$, this gives the following result:
\mediumbordisms
\begin{gather}
\nonumber
\begin{tz}
    \node (A) [Pants, bot, top, height scale=1.2] at (0,0) {};
    \node (B) [Cyl, bot, height scale=1.2] [anchor=top] at (A.leftleg) {};
    \node (D) [Copants, bot, anchor=leftleg, height scale=1.2] at (A.rightleg) {};
    \node (C) [Cyl, bot, top, anchor=bottom, height scale=1.2] at (D.rightleg) {};
    \node (f) [tiny label] at ([yshift=2pt] A.center) {$\tiny f$};
    \node (g) [tiny label] at ([yshift=-2pt] D.center) {$g$};
    \begin{scope}[internal string scope]
        \node (i) at ([yshift=\toff] A.belt) [above] {};
        \node (j) at ([yshift=\toff] C.top) [above] {};
        \node (k) at ([yshift=-\boff] B.bottom) [below] {};
        \node (l) at ([yshift=-\boff] D.belt) [below] {};
        \draw (i.south) -- (f.center);
        \draw [out=210, in=up] (f.center) to (B.top);
        \draw (f.center)
            to [out=-45, in=100] (D.leftleg)
            to [out=-80, in=135] (g.center);
        \draw [out=45, in=down] (g.center) to (C.bottom);
        \draw [out=up, in=down] (C.bottom) to (j.south);
        \draw [out=down, in=up] (g.center) to (l.north);
        \draw [out=down, in=up] (A.leftleg) to (k.north);
    \end{scope}
    \selectpart[green, inner sep=1pt]{(A-belt) (C-top)};
\end{tz}
\xmapsto{\textstyle \isd(\eta)}
\begin{tz}
    \node[Pants, bot] (A) at (0,0) {};
    \node[Cyl, bot] [anchor=top] at (A.leftleg) (B) {};
    \node[Copants, bot, anchor=leftleg] at (A.rightleg) (D) {};
    \node[SwishL, bot, anchor=bottom] at (D.rightleg) (C) {};
    \node[Pants, bot, anchor=leftleg, belt scale=1.5] at (A.belt) (E) {};
    \node[Copants, bot, anchor=belt, top, belt scale=1.5] at (E.belt) (F) {};
    \node (f) [tiny label] at ([yshift=2pt] A.center) {$f$};
    \node (g) [tiny label] at ([yshift=-2pt] D.center) {$g$};
    \begin{scope}[internal string scope]
        \node (i) at ([yshift=\toff] F.leftleg) [above] {};
        \node (j) at ([yshift=\toff] F.rightleg) [above] {};
        \node (k) at ([yshift=-\boff] B.bottom) [below] {};
        \node (l) at ([yshift=-\boff] D.belt) [below] {};
        \draw [out=210, in=up] (f.center) to (B.top);
        \draw [out=-45, in=100] (f.center) to node[midway, right] {} (D.leftleg);
        \draw [out=-80, in=135] (A.rightleg) to (g.center);
        \draw [out=45, in=down] (g.center) to (C.bottom);
        \draw[out=down, in=up] (g.center) to (l.north);
        \draw[out=down, in=up] (A.leftleg) to (k.north);
        \draw (i.south) to[out=down, in=up]  (F-belt.in-leftthird) to[out=down, in=up] (A.belt) to[out=down, in=up] (f.center);
        \draw (j.south) to[out=down, in=up] (F-belt.in-rightthird) to[out=down, in=up] (C.top) to[out=down, in=up] (C.bottom);
    \end{scope}
    \selectpart[green]{(F-belt) (C-bottom) (A-leftleg)};
\end{tz}
\xmapsto{\textstyle \isd(\alpha)}
\begin{tz}
        \node[Copants, bot, top, belt scale=1.5] (A) at (0,0) {};
        \node[Pants, bot, anchor=belt, belt scale=1.5, wide] (B) at (A.belt) {};
        \node[Pants, bot, anchor=belt] (D) at (B.rightleg) {};
        \node[Cyl, bot, anchor=top, tall] (E) at (B.leftleg) {};
        \node[Copants, bot, anchor=leftleg] (F) at (D.leftleg) {};
        \node (f) [tiny label] at ([xshift=-0.2\cobwidth, yshift=0.05\cobheight] B.center) {$f$};
        \node (g) [tiny label] at ([yshift=-0.1\cobheight] F.center) {$g$};
        \begin{scope}[internal string scope]
                \node (i) at ([yshift=\toff] A.leftleg) [above] {};
                \node (j) at ([yshift=\toff] A.rightleg) [above] {};
                \node (k) at ([yshift=-\boff] E.bottom) [below] {};
                \node (l) at ([yshift=-\boff] F.belt) [below] {};
    \draw (i.south)
        to [out=down, in=up] (f.center)
        to [out=-135, in=up] (E.top)
        to [out=down, in=up] (E.bottom)
        to [out=down, in=up] (k.north);
    \draw (j.south)
        to [out=down, in=up] (A-belt.in-rightthird)
        to [out=down, in=up] (B-rightleg.in-rightthird)
        to [out=down, in=up] (D.rightleg)
        to [out=down, in=45] (g.center)
        to [out=down, in=up] (l.north);
    \draw (f.center)
        to [out=-45, in=up] (B-rightleg.in-leftthird)
        to [out=down, in=up]
            node[near start, xshift=-0.1cm] {}
            (D.leftleg)
        to [out=down, in=135] (g.center);
    \selectpart[green]{(D-rightleg) (D-leftleg) (B-rightleg) (F-belt)};
\end{scope}
\end{tz}
\xmapsto{\textstyle \isd(\epsilon)}
\begin{tz}
    \node[Copants, bot, top, belt scale=2, height scale=1.2] (A) at (0,0) {};
    \node[Pants, bot, anchor=belt, belt scale=2, height scale=1.2] (B) at (A.belt) {};
    \begin{scope}[internal string scope]
        \node (i) at ([yshift=\toff] A.leftleg) [above] {};
        \node (j) at ([yshift=\toff] A.rightleg) [above] {};
        \node (k) at ([yshift=-\boff] B.leftleg) [below] {};
        \node (l) at ([yshift=-\boff] B.rightleg) [below] {};
        \node[tiny label] (p) at ([xshift=-0.35\cobwidth, yshift=-0.1\cobheight] A.center) {$f$};
        \node[tiny label] (q) at ([xshift=0.35\cobwidth, yshift=0.1\cobheight] B.center) {$g$};
        \draw (i.south)
            to[out=down, in=up, out looseness=1] (p.north)
            to (p.center)
            to[out=-125, in=up, looseness=0.8] (k.north);
        \draw (j.south)
            to [out=down, in=55, looseness=0.8] (q.center);
        \draw (q.center)
            to[out=down, in=up] (l.north);
        \draw (p.center)
            to[out=-45, in=135] (q.center);                
    \end{scope} 
\end{tz}
\end{gather}
The computation for $\phiM$ is similar.
\end{proof}

We now show that in this case, the balanced braided monoidal category is a ribbon category. The first step is to demonstrate that the balanced monoidal category $\cat{S}$ is rigid. A similar result was stated without proof by L\'opez-Franco, Street, and Wood~\cite{lsw11-di}.
\begin{proposition}\label{prop:Frob_to_rigid}
For a linear representation $Z$ of a presentation that 2\-extends the ribbon presentation $\R$, the associated balanced braided monoidal category of \autoref{lem:balancedrep} is rigid.
\end{proposition}
\begin{proof}
We consider the following composite involving the Frobeniusator isomorphism~\eqref{defn_of_phiright} and its inverse; in the third picture, we use \autoref{genericformsnake} to express the resulting string diagram in terms of some $f:I \to A' \otimes A$ and \mbox{$g:A \otimes A' \to I$}:
\mediumbordisms
\begin{gather}
\label{eq:fgidentity}
\begin{tz}
    \node (S) [Cyl, bot, top, anchor=top, height scale=3.4] at (0,0) {};
    \begin{scope}[internal string scope]
        \node (j) at ([yshift=\toff] S.top) [above] {};
        \node (k) at ([yshift=-\boff] S.bot) [below] {};
        \draw (j.south)
            to (k.north);
    \end{scope}
\end{tz}
\xmapsto{\begin{array}{@{}l@{}}\isd(\check\lambda {}^{\inv})\\\isd(\rho ^{\inv})\end{array}}
\begin{tz}[xscale=-1]
    \node [Copants, bot, belt scale=1.3] (A) at (0,0) {};
    \node [Pants, bot, anchor=belt, belt scale=1.3] (B) at (A.belt) {};
    \node [Cap, bot] at (A.rightleg) {};
    \node (C) [Cup] at (B.leftleg) {};
    \node (S) [Cyl, bot, anchor=top, height scale=0.7] at (B.rightleg) {};
    \node (T) [Cyl, bot, top, anchor=bot, height scale=0.7] at (A.leftleg) {};
    \begin{scope}[internal string scope]
        \node (j) at ([yshift=\toff] T.top) [above] {};
        \node (k) at ([yshift=-\boff] S.bot) [below] {};
        \draw (j.south)
            to (T.bot)
            to [out=down, in=up, looseness=0.8] (S.top) to (k.north);
    \end{scope}
    \selectpart[green]{(T-bottom) (B-rightleg)};
\end{tz}
\xmapsto{\displaystyle\isd(\phiM ^{\inv})}
\begin{tz}[xscale=-1]
    \node[Pants, bot] (A) at (0,0) {};
    \node[Cyl, bot, height scale=1.7] [anchor=top] at (A.rightleg) (B) {};
    \node[Copants, bot, anchor=rightleg] at (A.leftleg) (D) {};
    \node[Cyl, bot, top, height scale=1.7, anchor=bottom] at (D.leftleg) (C) {};
    \node (f) [tiny label] at ([yshift=2pt] A.center) {$\tiny f$};
    \node (g) [tiny label] at ([yshift=-2pt] D.center) {$g$};
    \begin{scope}[internal string scope]
        \node (i) at ([yshift=\toff] A.belt) [above] {};
        \node (j) at ([yshift=\toff] C.top) [above] {};
        \node (k) at ([yshift=-\boff] B.bottom) [below] {};
        \node (l) at ([yshift=-\boff] D.belt) [below] {};
        \draw [out=210, in=up] (f.center) to (B.top);
        \draw (f.center)
            to [out=-45, in=100] (D.rightleg)
            to [out=-80, in=135] (g.center);
        \draw [out=45, in=down] (g.center) to (C.bottom);
        \draw [out=up, in=down] (C.bottom) to (j.south);
        \draw [out=down, in=up] (A.rightleg) to (k.north);
    \end{scope}
    \node (CAP) [Cap, bot=false] at (A.belt) {};
    \node (CUP) [Cup] at (D.belt) {};
    \node (X) [Cobordism Bottom End 3D] at (A.belt) {};
    \selectpart[green]{(A-rightleg) (X) (C-bottom) (D-belt)};
\end{tz}
\xmapsto{\displaystyle\isd(\phiM)}
\begin{tz}[xscale=-1]
    \node[Copants, bot, belt scale=1.7] (A) at (0,0) {};
    \node[Pants, bot, anchor=belt, belt scale=1.7] (B) at (A.belt) {};
    \node [Cap, bot] at (A.rightleg) {};
    \node [Cup] at (B.leftleg) {};
    \node (S) [Cyl, bot, anchor=top, height scale=0.7] at (B.rightleg) {};
    \node (T) [Cyl, bot, top, anchor=bot, height scale=0.7] at (A.leftleg) {};
    \begin{scope}[internal string scope]
        \node (j) at ([yshift=\toff] T.top) [above] {};
        \node (k) at ([yshift=-\boff] S.bot) [below] {};
        \node[tiny label] (p) at ([xshift=-0.2\cobwidth, yshift=-0.1\cobheight] A.center) {$f$};
        \node[tiny label] (q) at ([xshift=0.2\cobwidth, yshift=0.1\cobheight] B.center) {$g$};
        \draw (p.center)
            to[out=-125, in=up, looseness=0.8] (k.north);
        \draw (j.south)
            to [out=down, in=55, looseness=0.8] (q.center);
        \draw (p.center)
            to[out=-45, in=135] (q.center);                
    \end{scope}
\end{tz}
\xmapsto{\displaystyle\begin{array}{@{}l@{}}\isd(\check\rho)\\\isd(\lambda)\end{array}}
\begin{tz}[xscale=-1]
    \node (S) [Cyl, top, anchor=top, height scale=1.7, bottom scale=1.5, bot=false] at (0,0) {};
    \node (T) [Cyl, bot, anchor=top, height scale=1.7, top scale=1.5] at (S.bot) {};
    \begin{scope}[internal string scope]
        \node (j) at ([yshift=\toff] S.top) [above] {};
        \node (k) at ([yshift=-\boff] T.bot) [below] {};
        \node[tiny label] (p) at ([xshift=-0.2\cobwidth, yshift=0.3\cobheight] S.bot) {$f$};
        \node[tiny label] (q) at ([xshift=0.2\cobwidth, yshift=-0.3\cobheight] S.bot) {$g$};
        \draw (p.center)
            to[out=-125, in=up, in looseness=1.5, out looseness=0.8](k.north);
        \draw (j.south)
            to [out=down, in=55, out looseness=1.5, in looseness=0.8] (q.center);
        \draw (p.center)
            to[out=-45, in=135] (q.center);                
    \end{scope}
\end{tz}
\end{gather}
\newcommand\sxto[1]{\mathop{\smash{\xto{#1}}}}%
Since this cobordism is diffeomorphic to the identity, the following equation must hold:
\begin{equation}
\label{eq:gfidentityseparate}
\begin{tz}[xscale=0.6]
\draw [internal string] (1,0.5) to [out=up, in=-45] (0,2) node [tiny label] {$f$} to [out=-135, in=45] node [red, below left, inner sep=1pt] {} (-1,1) node [tiny label] {$g$} to [out=135, in=down] (-2,2.5);
\end{tz}
\quad=\quad
\begin{tz}
\draw [internal string] (0,0) to (0,2);
\end{tz}
\end{equation}

Now consider the following composite:
\tikzset{every picture/.style={scale=0.8}}
\begin{equation}
\label{eq:idempotentcomposite}
\begin{tz}[xscale=-0.6]
\draw [internal string] (1,0.5) node [below, red] {$A'$} to [out=up, in=-45] (0,2) node [tiny label] {$f$} to [out=-135, in=45] node [red, below left, inner sep=1pt] {$A$} (-1,1) node [tiny label] {$g$} to [out=135, in=down] (-2,2.5) node [above, red] {$A'$};
\end{tz}
\end{equation}
By equation~\eqref{eq:gfidentityseparate}, this is an idempotent. Since we are taking representations in \bimod, our target category is Cauchy-complete, and we define the object $A^*$ to be the idempotent splitting of the composite~\eqref{eq:idempotentcomposite} via maps $A' \sxto u A^*$ and $A^* \sxto v A'$. We therefore have $u \circ v = \id _{A^*}$.

We will now show that the object $A^*$ is the right dual of the object $A$. We define the unit $\eta : I \to A^* \otimes A$ and counit $\epsilon : A \otimes A^* \to I$ as follows:
\begin{calign}
\begin{tz}[scale=0.6, yscale=-1]
\draw [red] (0,2.5) node [below] {$A^*$}
    to (0,1.8)
    to [out=down, in=135] (1,0) node [tiny label] {$\eta$} to [out=45, in=down] (2,1.8) to (2,2.5) node [below] {$A$};
\end{tz}
\quad:=\quad
\begin{tz}[scale=0.6, yscale=-1]
\draw [red] (0,2.5) node [below] {$A^*$} to (0,1.5) node [tiny label] {$u$} to [out=down, in=135] (1,0) node [tiny label] {$f$}  to [out=45, in=down] (2,1.8) to (2,2.5) node [below] {$A$};
\end{tz}
&
\begin{tz}[scale=0.6]
\draw [red] (0,2.5) node [above] {$A$}
    to (0,1.8)
    to [out=down, in=135] (1,0) node [tiny label] {$\epsilon$} to [out=45, in=down] (2,1.8) to (2,2.5) node [above] {$A^*$};
\end{tz}
\quad:=\quad
\begin{tz}[scale=0.6]
\draw [red] (0,2.5) node [above] {$A$} to (0,1.8) to [out=down, in=135] (1,0) node [tiny label] {$g$}  to [out=45, in=down] (2,1.5) node [tiny label] {$v$} to (2,2.5) node [above] {$A^*$};
\end{tz}
\end{calign}
We can now show directly that the duality equations are satisfied:
{\def\quad{\hspace{0.2cm}}
\begin{align}
\begin{tz}[scale=0.6, xscale=-1]
\draw [internal string] (1,-0.5) node [below, red] {$A^*$} to (1,0.2) to [out=up, in=-45] (0,2) node [tiny label] {$\eta$} to [out=-135, in=45] node [red, below left, inner sep=1pt] {$A$} (-1,0.5) node [tiny label] {$\epsilon$} to [out=135, in=down] (-2,2.3) to (-2,3) node [above, red] {$A^*$};
\end{tz}
\quad&=\quad
\begin{tz}[scale=0.6, xscale=-1]
\draw [internal string] (1,-0.5) node [below, red] {$A^*$} to (1,0.5) node [tiny label] {$u$} to [out=up, in=-45] (0,2) node [tiny label] {$f$} to [out=-135, in=45] node [red, below left, inner sep=1pt] {$A$} (-1,0.5) node [tiny label] {$g$} to [out=135, in=down] (-2,2.0) node [tiny label] {$v$} to (-2,3) node [above, red] {$A^*$};
\end{tz}
\quad=\quad
\begin{tz}[scale=0.6]
\draw [red] (0,-0.5) node [below] {$A^*$}
    to (0,0.2) node [tiny label] {$u$}
    to (0,0.9) node [tiny label] {$v$}
    to (0,1.6) node [tiny label] {$u$}
    to (0,2.3) node [tiny label] {$v$}
    to (0,3) node [above] {$A^*$};
\end{tz}
\quad=\quad
\begin{tz}[scale=0.6]
\draw [red] (0,-0.5) node [below] {$A^*$}
    to (0,3) node [above] {$A^*$};
\end{tz}
\\
\begin{tz}[scale=0.6, xscale=1]
\draw [internal string] (1,-0.5) node [below, red] {$A$} to (1,0.2) to [out=up, in=-45] (0,2) node [tiny label] {$\eta$} to [out=-135, in=45] node [red, below right, inner sep=0pt] {$A ^*$} (-1,0.5) node [tiny label] {$\epsilon$} to [out=135, in=down] (-2,2.3) to (-2,3) node [above, red] {$A$};
\end{tz}
\quad&=\quad
\begin{tz}[scale=0.6, xscale=1]
\draw [internal string] (1,-0.5) node [below, red] {$A$} to (1,0.5) to [out=up, in=-45] (0,2.3) node [tiny label] {$f$} to [out=-135, in=up] (-0.5,1.6) node [tiny label] {$u$} to (-0.5,0.9)  node [tiny label] {$v$} to [out=down, in=45] (-1,0.2) node [tiny label] {$g$} to [out=135, in=down] (-2,2.0) to (-2,3) node [above, red] {$A$};
\end{tz}
\quad=\quad
\begin{tz}[scale=0.6, xscale=-0.8]
\draw [red] (0,-0.5) node [below] {$A$}
    to [out=up, in=-135] (1,2) node [tiny label] {$f$}
    to (2,0.5) node [tiny label] {$g$}
    to (3,2) node [tiny label] {$f$}
    to (4,0.5) node [tiny label] {$g$}
    to [out=45, in=down] (5,3) node [above] {$A$};
\end{tz}
\quad=\quad
\begin{tz}[scale=0.6]
\draw [red] (0,-0.5) node [below] {$A$}
    to (0,3) node [above] {$A$};
\end{tz}
\end{align}}%
The final equality here uses the identity~\eqref{eq:gfidentityseparate}.
\end{proof}

Having established that $\cat{S}$ is rigid, we can now demonstrate that $\cat{S}$ is a ribbon category.
\begin{theorem} \label{thm:repofribisribbon}
        For a linear representation $Z$ of a presentation that 2\-extends the ribbon presentation $\R$, the twist is compatible, and hence the associated rigid balanced braided monoidal category of \autoref{prop:Frob_to_rigid} is a ribbon category.
\end{theorem}
\begin{proof}
We need to verify the equation $\theta_{A} {}^* = \theta_{A^*}$ where $A \in \Ob(\cat{S})$ and $\theta_A {}^*$ refers to the dual of $\theta_A$ built using the unit and counit maps.  (We make an arbitrary choice of duals; the compatibility of the twist does not depend on the choice.) This is equivalent to verifying the equality of applying $\theta$ on the left leg or the right leg of the counit map $A \otimes A^* \xto{} 1$. Applying $\theta$ on the left leg has the following effect on string diagrams:
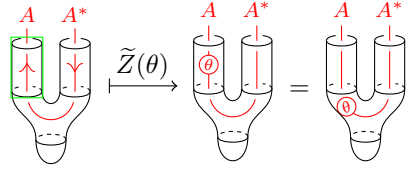
\mediumbordisms
\[
\begin{tz}
        \node[Cyl, top, bot] (A) at (0,0) {};
        \node[Copants, bot, anchor=leftleg] (B) at (A.bottom) {};
        \node[Cyl, top, bot, anchor=bottom] (C) at (B.rightleg) {};
        \node[Cup] (D) at (B.belt) {};
        \draw[red strand] ([yshift=\toff] A.top) node[above strand label, red] {$A$}
                to (A.top)
                to (A.bottom)
                to[out=down, in=down, looseness=2] (C.bottom)
                to
                (C.top)
                to ([yshift=\toff] C.top) node[above strand label, red] {$A^*$};
\node [red] at (A.center) {$\arrownode$};
\node [red, rotate=180] at (C.center) {$\arrownode$};
    \selectpart[green]{(A-top) (A-bottom)}
\end{tz}
\, \xmapsto{\isd(\theta)} \,
\begin{tztiny}
        \node[Cyl, top, bot] (A) at (0,0) {};
        \node[Copants, bot, anchor=leftleg] (B) at (A.bottom) {};
        \node[Cyl, top, bot, anchor=bottom] (C) at (B.rightleg) {};
        \node[Cup] (D) at (B.belt) {};
        \draw[red strand] ([yshift=\toff] A.top) node[above strand label, red] {$A$}
                to (A.top)
                to node[draw=red, text=red, tiny label] {$\theta$} (A.bottom)
                to[out=down, in=down, looseness=2] (C.bottom)
                to (C.top)
                to ([yshift=\toff] C.top) node[above strand label, red] {$A^*$};
\end{tztiny}
\, = \,
\begin{tztiny}
        \node[Cyl, top, bot] (A) at (0,0) {};
        \node[Copants, bot, anchor=leftleg] (B) at (A.bottom) {};
        \node[Cyl, top, bot, anchor=bottom] (C) at (B.rightleg) {};
        \node[Cup] (D) at (B.belt) {};
        \draw[red strand] ([yshift=\toff] A.top) node[above strand label, red] {$A$}
                to (A.top)
                to (A.bottom)
                to[out=down, in=down, looseness=2] 
                node[pos=0.2, draw=red, text=red, tiny label, rotate=30] {$\theta$}  (C.bottom)
                to (C.top)
                to ([yshift=\toff] C.top) node[above strand label, red] {$A^*$};
\end{tztiny}
\]
Applying $\theta$ on the right leg, we get the following result:
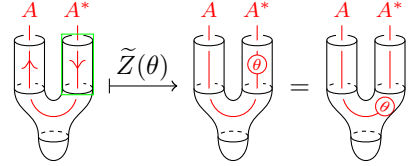
\[
\begin{tztiny}
        \node[Cyl, top, bot] (A) at (0,0) {};
        \node[Copants, bot, anchor=leftleg] (B) at (A.bottom) {};
        \node[Cyl, top, bot, anchor=bottom] (C) at (B.rightleg) {};
        \node[Cup] (D) at (B.belt) {};
        \draw[red strand] ([yshift=\toff] A.top) node[above strand label, red] {$A$}
                to (A.top)
                to (A.bottom)
                to[out=down, in=down, looseness=2] (C.bottom)
                to (C.top)
                to ([yshift=\toff] C.top) node[above strand label, red] {$A^*$};
    \selectpart[green]{(C-top) (C-bottom)}
\node [red] at (A.center) {$\arrownode$};
\node [red, rotate=180] at (C.center) {$\arrownode$};
\end{tztiny}
\, \xmapsto{\isd(\theta)} \,
\begin{tztiny}
        \node[Cyl, top, bot] (A) at (0,0) {};
        \node[Copants, bot, anchor=leftleg] (B) at (A.bottom) {};
        \node[Cyl, top, bot, anchor=bottom] (C) at (B.rightleg) {};
        \node[Cup] (D) at (B.belt) {};
        \draw[red strand] ([yshift=\toff] A.top) node[above strand label, red] {$A$}
                to (A.top)
                to (A.bottom)
                to[out=down, in=down, looseness=2] (C.bottom)
                to node[draw=red, text=red, tiny label] {$\theta$} (C.top)
                to ([yshift=\toff] C.top) node[above strand label, red] {$A^*$};
\end{tztiny}
\, = \,
\begin{tztiny}
        \node[Cyl, top, bot] (A) at (0,0) {};
        \node[Copants, bot, anchor=leftleg] (B) at (A.bottom) {};
        \node[Cyl, top, bot, anchor=bottom] (C) at (B.rightleg) {};
        \node[Cup] (D) at (B.belt) {};
        \draw[red strand] ([yshift=\toff] A.top) node[above strand label, red] {$A$}
                to (A.top)
                to (A.bottom)
                to[out=down, in=down, looseness=2] 
                node[pos=0.8, draw=red, text=red, tiny label, rotate=-20] {$\theta$}  (C.bottom)
                to (C.top)
                to ([yshift=\toff] C.top) node[above strand label, red] {$A^*$};
\end{tztiny}
\]
By relation~\eqref{tortile}, these two ways of applying $\theta$ are equal for any representation of the ribbon presentation, and hence the condition for a ribbon category, from \autoref{def:ribboncategory}, is satisfied.
\end{proof}

Recall that every ribbon category has a canonical (up to equivalence) pivotal structure (in fact, due to~\eqref{tortile}, this structure is spherical), built as a composite of a Reidemeister-I `loop' (using any choice of duals and unit and counit maps) with the inverse twist map.  We will make use of that structure in the next section on modularity.

\section{Representations of modular structures}
\label{sec:modcatmodobj}

We now investigate the linear representation theory of the modular presentation \M given in \autoref{defn_presentation_of_B}, and of the presentations given later in \autoref{sec:geometricalpresentations} which 2\-extend \M. We call such representations \textit{modular structures} in \twovect.  Since these presentations 2\-extend the ribbon presentation, all the results of \autoref{sec:internal} apply, and in particular a modular structure in \twovect has an associated  ribbon category.

By the results in the Appendix, a modular structure in \twovect necessarily takes values in the subcategory of semisimple 2\_vector spaces, and we  make heavy use of this fact throughout this section.

In \autoref{sec:simplesufficient} we show that when using the internal string diagram formalism, one can restrict labels of internal boundary circles to simple objects without loss of generality. We also show that every modular structure in \twovect factors as a product of simple ones, in which the tensor unit of the associated semisimple ribbon category has no proper subobjects. In \autoref{sec:actionsofdaggers} we determine, up to a constant $p$, the actions of the generators $\mu^\dag$, $\nu^\dag$, $\epsilon^\dag$, and $\eta^\dag$. We then make contact with the notion of modular tensor category, proving in \autoref{sec:mod} that for a modular structure, the associated ribbon category is modular. In \autoref{sec:anomalyvalue} we investigate the \textit{anomaly} $p^+/p^-$ of the modular category, where $p^+$ and $p^-$ are standard constants obtained from the modular tensor category, and show that $p^+/p^- = p^2$, fixing the value of the undetermined constant up to a sign. 

\subsection{Semisimplicity and simplicity}
\label{sec:simplesufficient}

By the analysis of the Appendix, in particular \autoref{cor:modissemisimple}, a linear representation (of one of the geometrical presentations extending the modular presentation) necessarily takes values in the subcategory of semisimple 2\_vector spaces. Without loss of generality, we can restrict the labelings of internal boundary circles to be simple objects, by means of the following equation:
\begin{equation}
\setlength\ymult{1pt}
\setlength{\cobwidth}{2.3\cobwidth}
\setlength{\ellipseheight}{1.5\ellipseheight}
\label{eq:simpleobjectsonly}
\begin{tz}
    \cobpartup 0 0
    \cobpartdown 0 0
    \node (A) at (0cm,2cm) [coordinate] {};
    \node (B) at (0cm,-2cm) [coordinate] {};
    \begin{scope}[curvein]
        \draw [dash pattern=on 1pt off 1pt] (A) to (0,1.7cm);
        \draw [dash pattern=on 1pt off 1pt] (B) to (0,-1.7cm);
        \draw (0,1.7cm) to (0,-1.7cm);
        \node at (A) [above] {$A$};
        \node at (B) [below] {$A$};
        \node at (0.5\cobwidth,0) [right] {$A$};
    \end{scope}
\end{tz}
\hspace{10pt}
=
\hspace{10pt}
\redscalar{\sum_{i,n}}
\begin{tz}
    \cobpartup 0 0
    \cobpartdown 0 0
    \node (A) at (0cm,2cm) [coordinate] {};
    \node (B) at (0cm,-2cm) [coordinate] {};
    \begin{scope}[curvein]
        \draw [dash pattern=on 1pt off 1pt] (A) to (0,1.7cm);
        \draw [dash pattern=on 1pt off 1pt] (B) to (0,-1.7cm);
        \draw (0,1.7cm) to (0,-1.7cm);
        \node at (A) [above] {$A$};
        \node at (B) [below] {$A$};
        \node [tiny label, minimum width=16pt] at (0,1.2cm) {$p_{i,n}$};
        \node [tiny label, minimum width=16pt] at (0,0.35cm) {$q_{i,n}$};
        \node at (0.5\cobwidth,0) [right] {$A$};
        \node at (0,0.76cm) [right=-1pt] {$\scriptstyle i$};
    \end{scope}
\end{tz}
\hspace{10pt}
\sim
\hspace{10pt}
\redscalar{\sum_{i,n}}
\begin{tz}
    \cobpartup 0 0
    \cobpartdown 0 0
    \node (A) at (0cm,2cm) [coordinate] {};
    \node (B) at (0cm,-2cm) [coordinate] {};
    \begin{scope}[curvein]
        \draw [dash pattern=on 1pt off 1pt] (A) to (0,1.7cm);
        \draw [dash pattern=on 1pt off 1pt] (B) to (0,-1.7cm);
        \draw (0,1.7cm) to (0,-1.7cm);
        \node at (A) [above] {$A$};
        \node at (B) [below] {$A$};
        \node [tiny label, minimum width=16pt] at (0,0.7cm) {$p_{i,n}$};
        \node [tiny label, minimum width=16pt] at (0,-0.7cm) {$q_{i,n}$};
        \node at (0.5\cobwidth,0) [right] {$S_i$};
        \node at (0,0.1cm) [right=-1pt] {$\scriptstyle i$};
    \end{scope}
\end{tz}
\end{equation}
We begin by rearranging our internal string diagram locally so that only the identity morphism of some object $A$, not necessarily simple, passes through our internal boundary circle. We then rewrite the identity on $A$ as a sum of projections $p_{i,n}$ onto and injections $q_{i,n}$ from simple objects $S_i$, using the semisimplicity of the category $\cat{S}$ associated to the circle. Finally, we use our equivalence relation to slide the injection morphisms $q_{i,n}$ into the lower cobordism.

We can use the internal string diagram formalism to show that every modular structure breaks up into a direct sum of simple modular structures.

\begin{defn}
A modular structure $Z$ in $\twovect$ is \emph{simple} if $Z(
\tikztinysphere) \simeq k$.
\end{defn}

\begin{lemma}
A modular structure in $\twovect$ is simple if and only if the unit object in the associated ribbon category has no proper subobjects.
\end{lemma}
\begin{proof}
If $C$ is a semisimple $k$-linear category, a linear functor $F: \vect \to \cat C$ picks out a simple object just when $\Hom _\twovect (F, F) \simeq k$. Given a modular structure $Z$, the linear functor $Z(\tikztinycup)$ picks out the unit object for the linear ribbon category associated to $Z$, and by the adjunction $\tikztinycup \dashv \tikztinycap$ we have $\Hom _\twovect \big( Z(\tikztinycup), Z(\tikztinycup) \big) \simeq \Hom _\twovect \big(\id_{\vect}, Z(\tikztinysphere) \big)$.  Hence if the unit object is simple, we must have $Z(\tikztinysphere) \simeq (-) \otimes k : \vect \to \vect$, and vice versa. 
\end{proof}

\begin{lemma}
\label{directsumofsimples}
Every modular structure $Z$ in $\twovect$ is equivalent to a direct sum of simple modular structures.
\end{lemma}
\begin{proof}
By \autoref{cor:modissemisimple}, we know that the ribbon category $C$ associated to $Z$ is semisimple. A semisimple braided monoidal category always decomposes, as a braided monoidal category, into a direct sum of braided monoidal categories with simple unit. Indeed, write $I =  \bigoplus_{j \in J} I_j$ for a decomposition of the unit object $I$ into simple objects $I_j$. Then for each simple object $X$ we must have $X \otimes I_{j_X} \cong X$ for a unique $j_X \in J$, and $X \otimes I_j = 0$ for $j \neq j_X$. Using the braiding, this means $j_X$ is also the unique index such that $I_{j_X} \otimes X \cong X$. Hence, as a linear category, $C = \sum_{j \in J} C_j$ where $C_j$ is the full subcategory (the $j$'th `factor') generated by the simple objects $X$ with $1_j \otimes X \cong X$. Since $X \otimes Y \cong (X \otimes 1_{j_X}) \otimes Y \cong X \otimes (1_{j_X} \otimes Y)$ we see that the tensor product of simple objects is zero if $X$ and $Y$ are from different factors, and produces a direct sum of simple objects in the same factor otherwise.

Now let $\Sigma$ be a connected surface, and consider the vector space $\isd \big( \Sigma  \big)^A _B$ for some choices of labelings $A$ and $B$ of the input and output boundary circles by simple objects. If these simple objects are not all associated to the same factor $j \in J$ then we will have $\isd \big( \Sigma  \big)^A _B = 0$, as the string diagram could be topologically deformed to bring the two simple objects from different factors into proximity, and the argument from the previous paragraph then tells us that we are tensoring by the zero object.

For each generating 2\-morphism of a modular structure, at least one of its domain or codomain is a connected surface, and so any modular structure must assign to it the zero linear map when its boundaries are labeled by simple objects from different factors. It follows that the modular structure is a direct sum of simple modular structures.
\end{proof}

\noindent
Given this result, we can consider our modular structures to be simple without loss of generality, and we will make this assumption for the rest of this section.

\subsection{Actions of the generators $\mu^\dag$, $\nu^\dag$, $\epsilon^\dag$, and $\eta^\dag$}
\label{sec:actionsofdaggers}

In this Section we study the actions of the generators $\mu^\dag$, $\nu^\dag$, $\epsilon^\dag$, and $\eta^\dag$ of a modular structure on internal string diagrams. Our strategy is to propose arbitrary linear maps for the values of these generators, and then see what restrictions we can obtain on these maps by considering equations we know must hold. Initially, we will find that the values can be determined up  to a constant $p$. In \autoref{lem:psquared} we will see that $p^2=p^+ p^-$, where $p^+$ and $p^-$ are the anomaly constants in a modular tensor category with simple tensor unit.

\begin{proposition}
\label{thm:mndaginternalstring}
For a modular structure $Z$ in \twovect, the generators $\mu ^\dag$ and $\nu ^\dag$ act in the following way on internal string diagrams, where $p$ is a nonzero constant yet to be determined:
\mediumbordisms
\begin{align}
\begin{tz}
\node (A) [Cyl, top, bot=false] at (0,0) {};
\node (B) [Cyl, bot] at (0, -0.5\cobheight) {};
    \begin{scope}[curvein]
        \node (top) at (A.top) [above=0.1cm] {$I$};
        \node (bot) at (B.bot) [below=0.15cm] {$I$};
    \end{scope}
\end{tz}
\quad
&\xmapsto{\textstyle \isd(\mu ^\dagger)}
\quad
{\red p}
\,\,
\begin{tz}
\node (B) [Cap, bot] at (0,0) {};
\node (A) [Cup, top] at (0,1.5\cobheight) {};
    \begin{scope}[curvein]
        \node (top) at (A.center) [above=0.1cm] {$I$};
        \node (bot) at (B.center) [below=0.15cm] {$I$};
    \end{scope}
\end{tz}
&\begin{tz}
\node [Cup] at (0,0) {};
\node [Cap, bot] at (0,0) {};
\begin{scope}[curvein]
\end{scope}
\end{tz}
\quad
&\xmapsto{\textstyle \isd(\nu ^\dagger)}
\quad
{\red \frac{1}{p}}
\,
\begin{tz}
\begin{scope}[curvein]
\end{scope}
\end{tz}
\end{align}
For other labelings of the boundary circles of the cylinder by simple objects, $\isd (\mu ^\dagger)$ is zero.
\end{proposition}
\begin{proof}
Suppose we label the boundary circles for the $\mu^\dagger$ action with simple objects not all isomorphic to $I$. Then by the internal string diagram formalism the target of $\isd(\mu^\dagger)$ is a zero-dimensional vector space. So the corresponding linear map is the zero map. The only nontrivial content lies in the case where boundary circles are all labeled by~$I$.

Given this, the following candidate actions are fully general:
\mediumbordisms
\begin{align}
\label{eq:muopaction}
\begin{tz}
\node (A) [Cyl, top, bot=false] at (0,0) {};
\node (B) [Cyl, bot] at (0, -0.5\cobheight) {};
    \begin{scope}[curvein]
        \node (top) at (A.top) [above=0.1cm] {$I$};
        \node (bot) at (B.bot) [below=0.15cm] {$I$};
    \end{scope}
\end{tz}
\quad
&\xmapsto{\textstyle Z(\mu ^\dag)}
\quad
{\red p}
\,\,
\begin{tz}
\node (B) [Cap, bot] at (0,0) {};
\node (A) [Cup, top] at (0,1.5\cobheight) {};
    \begin{scope}[curvein]
        \node (top) at (A.center) [above=0.1cm] {$I$};
        \node (bot) at (B.center) [below=0.15cm] {$I$};
    \end{scope}
\end{tz}
&
\begin{tz}
\node [Cup] at (0,0) {};
\node [Cap, bot] at (0,0) {};
\begin{scope}[curvein]
\end{scope}
\end{tz}
\quad
&\xmapsto{\textstyle Z(\nu ^\dag)}
\quad
{\red q}
\hspace{4.5pt}
\begin{tz}
\end{tz}
\end{align}
In the action~\eqref{eq:muopaction} the cobordisms are left blank to denote the identity string diagram in each connected component. Applying one of the equations witnessing the adjunction $\tikztinycap \dashv \tikztinycup$, we see that the following composite must be the identity:
\begin{equation}
\begin{tz}
    \node (A) [Cyl, bot, slightlytall] at (0,0) {};
    \node [Cap, bot] at (A.top) {};
    \begin{scope}[curvein]
        \node (bot) at (A.bot) [below=0.25cm] {$I$};
    \end{scope}
\end{tz}
\quad
\xmapsto{\textstyle \isd(\mu ^\dag)}
\quad
{\red p}\,
\begin{tz}
    \node (A) [Cap, bot] at (0,0) {};
    \node (B) [Cup] at (0,0) {};
    \node (C) [Cap, bot] at (0,-1.5\cobheight) {};
    \begin{scope}[curvein]
        \node (bot) at (C.center) [below=0.25cm] {$I$};
    \end{scope}
\end{tz}
\quad
\xmapsto{\textstyle \isd(\nu ^\dag)}
\quad
{\red p  q}\,
\begin{tz}
    \node (A) [Cap, invisible] at (0,0) {};
    \node (C) [Cap, bot] at (0,-1.5\cobheight) {};
    \begin{scope}[curvein]
        \node (bot) at (C.center) [below=0.25cm] {$I$};
    \end{scope}
\end{tz}
\end{equation}
As a result we obtain $q=1/p$.
\end{proof}

By a pivotal tensor category we mean one with chosen right duals for all objects and a chosen monoidal isomorphism from the identity to the right double dual functor.\footnote{Note that nothing does or can depend on the choice of duals, but the choice of pivotal structure is definitely nontrivial.  The fact that the notion of pivotal structure, as presented, depends on first making the irrelevant choice of duals is at root a sign that the standard definition of pivotality is not the most perspicacious one.}  Note that we always use the canonical pivotal structure associated to the ribbon category, described at the end of the previous section; indeed, the subsequent string diagrams would not even be well-defined with an alternate choice of pivotal structure.  

We recall the standard convention that in the graphical calculus for a semisimple ribbon tensor category, a closed unlabeled loop (possibly braided nontrivially with other strings, and in our internal string diagram formalism, possibly embedded nontrivially in the interior of the surface) denotes a sum over isomorphism classes of simple object, each weighted by the quantum dimension of that object~\cite{bk01-ltc}:
\begin{equation}
\begin{tz}
\draw (0,0) to [out=up, in=up, looseness=1.8] (1,0) to [out=down, in=down, looseness=1.8] (0,0);
\end{tz}
\quad=\quad
\sum_i d_i\,\,\,
\begin{tz}
\draw (0,0) to [out=up, in=up, looseness=1.8] (1,0) node [right] {$i$} to [out=down, in=down, looseness=1.8] (0,0);
\node at (1,0) {\arrownode[black]};
\end{tz}
\end{equation}
Note that here, and henceforth, cups and caps are the (arbitrarily) chosen units and counits for the dualities $j \dashv j^*$, and there is an implicit use of the canonical pivotal structure ${}^* i \cong i^{*}$ (which, note well, depends (only up to equivalence) on the same chosen units and counits).  The quantum dimension $d_i \in k$ is defined to be the value (as a string diagram, not as an internal string diagram) of a closed loop (not braided with any other strings) labelled by the simple object $S_i$:
\begin{equation}
d_i \quad:=\quad
\begin{tz}
\draw (0,0) to [out=up, in=up, looseness=1.8] (1,0) node [right] {$i$} to [out=down, in=down, looseness=1.8] (0,0);
\node at (1,0) {\arrownode[black]};
\end{tz}
\end{equation}

\begin{proposition}
\label{thm:epsilondagaction}
For a modular structure $Z$ in \twovect, the generator $\epsilon ^\dag$ acts in the following way on internal string diagrams:
\mediumbordisms
\begin{align}
\label{eq:epsilonopaction}
\begin{tz}
    \node (A) [Cyl, tall, bot, top] at (0,0) {};
    \begin{scope}[curvein]
        \draw ([yshift=\toff] A.top) to ([yshift=-\boff] A.bot);
    \end{scope}
\end{tz}
\quad
& \xmapsto{\displaystyle \isd (\epsilon ^\dag)}
\quad
{ \color{red} \frac{1}{p} }
\begin{tz}
    \node (A) [Pants, bot, top] at (0,0) {};
    \node (B) [Copants, bot, anchor=leftleg] at (A.leftleg) {};
    \begin{scope}[curvein]
        \draw ([yshift=\toff] A.belt)
            to [out=down, in=up, out looseness=1.5] (A-leftleg.in-leftthird)
            to [out=down, in=up, in looseness=1.7] ([yshift=-\boff] B.belt);
        \draw[red strand] (A-leftleg.in-rightthird)
            to [out=up,in=up,looseness=1.7]
                (A-rightleg.in-leftthird)
            to [out=down,in=down, looseness=1.7] (A-leftleg.in-rightthird);
    \end{scope}
\end{tz}
\end{align}
\end{proposition}
\begin{proof} 
We begin by noting that $\epsilon ^\dag$ is completely determined by its action on a 2\-sphere, as the following commutative diagram demonstrates:
\smallbordisms
\begin{equation}
\label{eq:epsilonsimplification}
\begin{tz}[scale=0.6]
\node (1) at (0,0)
{
    $\begin{tz}
        \node [Cyl, tall, top, bot] at (0,0) {};
    \end{tz}$
};
\node (2) at (4,0)
{
    $\begin{tz}
        \node (P) [Pants, top, bot, anchor=belt] at (0,0) {};
        \node [Copants, bot, anchor=leftleg] at (P.leftleg) {};
    \end{tz}$
};
\node (3) at (-4,-2)
{
    $\begin{tz}
        \node [Cyl, tall, bot, top] at (0,0) {};
        \node [Cap, bot] at (\cobwidth+\cobgap,0) {};
        \node [Cup] at (\cobwidth+\cobgap,0) {};
    \end{tz}$
};
\node (4) at (0,-4)
{
    $\begin{tz}
        \node [Cyl, veryveryverytall, bot, top] at (0,0) {};
        \node (P) [Pants, bot, anchor=leftleg] at (\cobgap+\cobwidth, 0) {};
        \node [Cap, bot] at (P.belt) {};
        \node (Q) [Copants, bot, anchor=leftleg] at (P.leftleg) {};
        \node [Cup] at (Q.belt) {};
    \end{tz}$
};
\node (5) at (4,-4)
{
    $\begin{tz}
        \node (P) [Pants, bot, anchor=leftleg] at (\cobgap+\cobwidth, 0) {};
        \node [Cap, bot] at (P.belt) {};
        \node (Q) [Copants, bot, anchor=rightleg] at (P.leftleg) {};
        \node (P2) [Pants, bot, anchor=belt] at (Q.belt) {};
        \node (Q2) [Copants, bot, anchor=leftleg] at (P2.rightleg) {};
        \node [Cyl, bot, tall, anchor=top] at (P.rightleg) {};
        \node [Cup] at (Q2.belt) {};
        \node [Cyl, slightlytall, bot, anchor=top] at (P2.leftleg) {};
        \node [Cyl, slightlytall, bot, top, anchor=bot] at (Q.leftleg) {};
    \end{tz}$
};
\node (6) at (8,-2)
{
    $\begin{tz}
        \node (Q) [Copants, bot, anchor=rightleg] at (0,0) {};
        \node (P) [Pants, bot, anchor=belt] at (Q.belt) {};
        \node (Q2) [Copants, bot, anchor=leftleg] at (P.leftleg) {};
        \node (P2) [Pants, bot, anchor=belt] at (Q2.belt) {};
        \node [Cyl, bot, anchor=top] at (P2.leftleg) {};
        \node [Cyl, bot, top, anchor=bot] at (Q.leftleg) {};
        \node [Cap, bot] at (Q.rightleg) {};
        \node [Cup] at (P2.rightleg) {};
    \end{tz}$
};
\begin{scope}[double arrow scope]
\draw (1) to node [auto] {$\epsilon^\dag$} (2);
\draw (1) to node [auto,swap] {$\nu$} (3);
\draw (3) to node [auto,swap] {$\epsilon ^\dag$} (4);
\draw (4) to node [auto,swap] {$\eta$} (5);
\draw (5) to node [below=5pt] {$\phiN,\phiM$} (6);
\draw (6) to node [above=5pt,pos=0.35] {$\rho, \check{\rho}$} (2);
\end{scope}
\end{tz}
\end{equation}
That this equation follows from the relations of a modular structure was demonstrated in~\mbox{\cite[\autoref{PII_locality_epsilon_dagger}]{PaperII}.}

We already understand how the other components in this diagram act on internal string diagrams, so to understand the general action of $\epsilon ^\dag$ we need only calculate how it acts on the 2\-sphere. Using the semisimplicity property established in \autoref{cor:modissemisimple}, we postulate the following completely-general form for this action, where the $c_i$ are constants to be determined:
\mediumbordisms
\begin{equation}
\begin{tz}
    \node [Cup] at (0,0) {};
    \node [Cap, bot] at (0,0) {};
\end{tz}
\gap
\xmapsto{\textstyle \isd(\epsilon^\dag)}
\gap
{\color{red} \sum_i c_i}\,
\begin{tz}
    \node (A) [Cap, bot] at (0,0) {};
    \node (B) [Pants, bot, anchor=belt] at (A) {};
    \node (C) [Copants, bot, anchor=leftleg] at (B.leftleg) {};
    \node [Cup] at (C.belt) {};
    \begin{scope}[curvein]
        \draw [red strand] (B.leftleg)
            to [out=up,in=up,looseness=1.7]
                node [right=-2pt,pos=0.85] {$\scriptstyle i$}
                node [pos=0.9, rotate=20] {$\arrownode$}
                (B.rightleg)
            to [out=down,in=down, looseness=1.7] (B.leftleg);
    \end{scope}
\end{tz}
\end{equation}

\noindent
Equation~\eqref{eq:epsilonsimplification} then gives the following action for $\epsilon ^\dag$ on the cylinder:
\begin{equation}
\label{eq:generalepsilonop}
\begin{tz}
    \node (A) [Cyl, tall, bot, top] at (0,0) {};
    \begin{scope}[curvein]
        \node (top) at (A.top) [above=0.2cm] {};
        \node (bot) at (A.bot) [below=0.25cm] {};
        \draw (top.south) to (bot.north);
    \end{scope}
\end{tz}
\gap\xmapsto{\textstyle \isd(\epsilon ^\dag)}\gap
{\color{red} \sum_i c_i}\,
\begin{tz}
    \node (A) [Pants, bot, top] at (0,0) {};
    \node (B) [Copants, bot, anchor=leftleg] at (A.leftleg) {};
    \begin{scope}[curvein]
        \node (top) at (A.belt) [above=0.2cm] {};
        \node (bot) at (B.belt) [below=0.25cm] {};
        \draw (top.south)
            to [out=down, in=up, out looseness=1.5] (A-leftleg.in-leftthird)
            to [out=down, in=up, in looseness=1.5] (bot.north);
        \draw[red] [red] (A-leftleg.in-rightthird)
            to [out=up,in=up,looseness=1.7]
                node [right=-2pt,pos=0.85] {$\scriptstyle i$}
                node [pos=0.9, rotate=20] {$\arrownode$}
                (A-rightleg.in-leftthird)
            to [out=down,in=down, looseness=1.7] (A-leftleg.in-rightthird);
    \end{scope}
\end{tz}
\end{equation}
To fix the values of the coefficients $c_i$, we use the fact that the following composite is equal to the identity:
\smallbordisms
\begin{equation}
\begin{aligned}
\begin{tikzpicture}
\node [Cyl, top, bot, tall] at (0,0) {};
\end{tikzpicture}
\end{aligned}
\longxdoubleto{\textstyle \epsilon ^\dag}
\begin{aligned}
\begin{tikzpicture}
\node (P) [Pants, bot, top] at (0,0) {};
\node [Copants, bot, anchor=leftleg] at (P.leftleg) {};
\selectpart[green, inner sep=1pt] {(P-leftleg)};
\end{tikzpicture}
\end{aligned}
\longxdoubleto{\textstyle \mu ^\dag}
\setlength\cupheight{0.8\cupheight}
\begin{aligned}
\begin{tikzpicture}
\node (P) [Pants, bot, top] at (0,0) {};
\node [Cup] at (P.leftleg) {};
\node (C) [Cyl, bot, anchor=top] at (P.rightleg) {};
\node (CP) [Copants, bot, anchor=rightleg] at (C.bottom) {};
\node [Cap, bot] at (CP.leftleg) {};
\end{tikzpicture}
\end{aligned}
\longxdoubleto{\textstyle \lambda, \check \lambda}
\begin{aligned}
\begin{tikzpicture}
\node [Cyl, verytall, top, bot] at (0,0) {};
\end{tikzpicture}
\end{aligned}
\end{equation}
This is a consequence of the equation $\lambda^* = {}^* \! \lambda$, which is shown to hold from the axioms of a modular structure in~\cite[\autoref{PII_rhostarstarlem}]{PaperII}. Analyzing this composite via internal string diagrams gives the following:
\mediumbordisms
\begin{equation}
\begin{aligned}
\begin{tikzpicture}
\node (C) [Cyl, top, bot, tall] at (0,0) {};
    \begin{scope}[curvein]
        \node (top) at (C.top) [above=0.2cm] {};
        \node (bot) at (C.bot) [below=0.25cm] {$i$};
        \draw (top.south) to (bot.north);
    \end{scope}
\end{tikzpicture}
\end{aligned}
\gap\xmapsto{\textstyle \isd (\epsilon ^\dag)}\gap
{\color{red} \sum_j c_j}
\begin{aligned}
\begin{tikzpicture}
    \node (A) [Pants, bot, top] at (0,0) {};
    \node (B) [Copants, bot, anchor=leftleg] at (A.leftleg) {};
    \begin{scope}[curvein]
        \node (top) at (A.belt) [above=0.2cm] {};
        \node (bot) at (B.belt) [below=0.25cm] {$i$};
        \draw (top.south)
            to [out=down, in=up, out looseness=1.5] (A-leftleg.in-leftthird)
            to [out=down, in=up, in looseness=1.5] (bot.north);
        \draw[red] [arrow data={0.45}{>}, red] (A-leftleg.in-rightthird)
            to [out=up,in=up,looseness=1.7]
                node [right=-2pt,pos=0.85] {$\scriptstyle j$}
                (A-rightleg.in-leftthird)
            to [out=down,in=down, looseness=1.7] (A-leftleg.in-rightthird);
    \end{scope}
\selectpart[green, inner sep=1pt] {(A-leftleg)};
\end{tikzpicture}
\end{aligned}
\gap\xmapsto{\textstyle \isd (\mu ^\dag)}\gap
{\red \frac{c_i p}{d_i}}
\setlength\cupheight{0.8\cupheight}
\begin{aligned}
\begin{tikzpicture}
\node (P) [Pants, bot, top] at (0,0) {};
\node [Cup] at (P.leftleg) {};
\node (C) [Cyl, bot, anchor=top] at (P.rightleg) {};
\node (CP) [Copants, bot, anchor=rightleg] at (C.bottom) {};
\node (CAP) [Cap, bot] at (CP.leftleg) {};
\begin{scope}[curvein]
    \node (top) at (P.belt) [above=0.2cm] {};
    \node (bot) at (CP.belt) [below=0.25cm] {$i$};
    \draw (top.south)
        to [out=down, in=up, out looseness=1.5] (P-leftleg.in-leftthird)
        to [out=down, in=down, looseness=4] (P-leftleg.in-rightthird)
        to [out=up, in=up, looseness=2] (P-rightleg.in-leftthird)
        to (C-bottom.in-leftthird)
        to [out=down, in=down, looseness=2] (CAP-center.in-rightthird)
        to [out=up, in=up, looseness=3] (CAP-center.in-leftthird)
        to [out=down, in=up, in looseness=1.5] (bot.north);
\end{scope}
\end{tikzpicture}
\end{aligned}
\gap\xmapsto{\stackrel{\textstyle \isd (\lambda)}{\textstyle \isd (\check \lambda)}}\gap
{\red \frac{c_i p}{d_i}}
\begin{aligned}
\begin{tikzpicture}
\node (C) [Cyl, verytall, top, bot] at (0,0) {};
\begin{scope}[curvein]
    \node (top) at (C.top) [above=0.2cm] {};
    \node (bot) at (C.bottom) [below=0.25cm] {$i$};
    \draw (top.south) to (bot.north);
\end{scope}
\end{tikzpicture}
\end{aligned}
\end{equation}
We conclude that $c_i = d_i/p$.
\end{proof}

\begin{proposition}
\label{thm:etadagaction}
For a modular structure in \twovect, the generator $\eta ^\dag$ acts in the following way on internal string diagrams:
\mediumbordisms
$$\begin{tz}
        \node[Pants, bot, belt scale=1.5] (A) at (0,0) {};
        \node[Copants, bot, anchor=belt, belt scale=1.5, top] (B) at (A.belt) {}; 
        \node (f) [tiny label] at ([yshift=5pt] B.belt) {$f$};
        \begin{scope}[internal string scope]
                \node (i) at ([yshift=\toff] B.leftleg) [above] {$k$};
                \node (j) at ([yshift=\toff] B.rightleg) [above] {$l$};
                \node (i2) at ([yshift=-\boff] A.leftleg) [below] {$i$};
                \node (j2) at ([yshift=-\boff] A.rightleg) [below] {$j$};
                \draw (i.south)
                        to[out=down, in=135] (f.center)
                        to[out=-135, in=up] (i2.north);
                \draw (j.south)
                        to[out=down, in=45] (f.center)
                        to[out=-45, in=up] (j2.north);
        \end{scope}
\end{tz}
\gap
\longxmapsto{\textstyle \isd (\eta ^\dag)}{40pt}
\gap
\frac{p \delta _{i,k} \delta _{j,l}}{d_i d_j}
\begin{aligned}
\begin{tikzpicture}
\node (f) [tiny label] at ([yshift=5pt] B.belt) {$f$};
\draw [red strand] (f.center)
        to [out=55, in=up, out looseness=3] +(10pt,0pt)
        to [out=down, in=-55, in looseness=3] (f.center);
\draw [red strand] (f.center)
        to [out=125, in=up, out looseness=3] +(-10pt,0pt)
        to [out=down, in=-125, in looseness=3] (f.center);
\end{tikzpicture}
\end{aligned}
 \begin{tz}
        \node[Cyl, tall, top, bot] (A) at (0,0) {};
        \node[Cyl, tall, top, bot] (B) at (\cobgap + \cobwidth, 0) {};
        \begin{scope}[internal string scope]
                \draw ([yshift=\toff] A.top) node[above]{$k$} -- ([yshift=-\boff] A.bot) node[below]{$i$};
                \draw ([yshift=\toff] B.top) node[above]{$l$} -- ([yshift=-\boff] B.bot) node[below]{$j$};
        \end{scope}
\end{tz}
$$
\end{proposition}
\begin{proof}
This follows from the action on internal string diagrams of the following decomposition of $\eta ^\dag$:
\smallbordisms
\begin{equation}
\begin{aligned}
\begin{tikzpicture}
\node [Pants] (P) at (0,0) {};
\node [Copants, anchor=belt, top] at (P.belt) {};
\end{tikzpicture}
\end{aligned}
\longxdoubleto{\textstyle\phiM}
\begin{aligned}
\begin{tikzpicture}
\node (C) [Copants] at (0,0) {};
\node (P) [Pants, anchor=leftleg, top] at (C.rightleg) {};
\node [Cyl, anchor=bottom, top] at (C.leftleg) {};
\node [Cyl, anchor=top] at (P.rightleg) {};
\selectpart[green, inner sep=1pt]{(P-leftleg)}
\end{tikzpicture}
\end{aligned}
\longxdoubleto{\textstyle \mu ^\dag}
\begin{aligned}
\begin{tikzpicture}
\node (C) [Copants] at (0,0) {};
\node (P) [Pants, anchor=leftleg, top] at ([yshift=1.5*\cobheight] C.rightleg) {};
\node (D) [Cyl, anchor=bottom, top, height scale=2.5] at (C.leftleg) {};
\node [Cyl, anchor=top, height scale=2.5] at (P.rightleg) {};
\node (A) [Cap] at (C.rightleg) {};
\node (B) [Cup] at (P.leftleg) {};
\selectpart[green]{(P-leftleg) (P-belt) (B) (P-rightleg)}
\selectpart[red]{(D-bottom) (C-belt) (A)}
\end{tikzpicture}
\end{aligned}
\longxdoubleto{{\color{red}\check\rho}, {\color{green}\lambda}}
\begin{aligned}
\begin{tikzpicture}
\node [Cyl, anchor=bottom, top, height scale=2.5] at (0,0) {};
\node [Cyl, anchor=bottom, top, height scale=2.5] at (\cobwidth+\cobgap,0) {};
\end{tikzpicture}
\end{aligned}
\end{equation}
This decomposition follows from~\cite[\autoref{PII_etadaglem}]{PaperII}.
\end{proof}

\subsection{Modularity}
\label{sec:mod}

In this section, we show that the ribbon category associated to a modular structure is a modular tensor category.  We begin by defining the following composite 2\-morphisms.
\smallbordisms
\begin{defn}[See~\cite{PaperII}, Definition~\ref{PII_defcomposites}]
\label{defn_II_III}
For a modular structure, we define the following composite 2\-morphisms:
\begin{align} \label{defn_of_II}
\II \quad&:=\quad
\begin{aligned}  \begin{tikzpicture}[xscale=2]
\node (1) at (0,0)
{
$\begin{tikzpicture}
        \node[Copants, top] (A) at (0,0) {};
        \node[Pants, anchor=belt] (B) at (A.belt) {};
\end{tikzpicture}$
};
\node (2) at (1,0)
{
$\begin{tikzpicture}
        \node[Pants, top] (A) at (0,0) {};
        \node[Cyl, anchor=top] (B) at (A.leftleg) {};
        \node[Copants, anchor=leftleg] (C) at (A.rightleg) {};
        \node[Cyl, anchor=bottom, top] (D) at (C.rightleg) {};
        \selectpart[green, inner sep=1pt] {(A-rightleg)};
\end{tikzpicture}$
};
\node (3) at (2,0)
{
$\begin{tikzpicture}
        \node[Pants, top] (A) at (0,0) {};
        \node[Cyl, anchor=top] (B) at (A.leftleg) {};
        \node[Copants, anchor=leftleg] (C) at (A.rightleg) {};
        \node[Cyl, anchor=bottom, top] (D) at (C.rightleg) {};
\end{tikzpicture}$
};
\node (4) at (3,0)
{
$\begin{tikzpicture}
        \node[Copants, top] (A) at (0,0) {};
        \node[Pants, anchor=belt] (B) at (A.belt) {};
\end{tikzpicture}$
};
\begin{scope}[double arrow scope]
    \draw (1) --  node[above]{$\phiN^\inv$} (2);
    \draw (2) --  node[above]{$\theta$} (3);
    \draw (3) --  node[above]{$\phiN$} (4);
\end{scope} \end{tikzpicture} \end{aligned}
\\
\label{defn_of_A}
A \quad&:=\quad
\begin{tz}[xscale=2, yscale=2]
\node (1) at (0,0)
{
$\begin{tikzpicture}
        \node[Pants, top] (A) at (0,0) {};
        \node[Copants, anchor=leftleg] (B) at (A.leftleg) {};
        \selectpart[green, inner sep=1pt] {(B-belt)};
\end{tikzpicture}$
};
\node [inner sep=0pt] (2) at (1,0)
{
$\begin{tikzpicture}
        \node[Pants, top] (A) at (0,0) {};
        \node[Copants, anchor=leftleg] (B) at (A.leftleg) {};
        \node[Pants, anchor=belt] (C) at (B.belt) {};
        \node[Copants, anchor=leftleg] (D) at (C.leftleg) {};
    \selectpart[green] {(A-leftleg) (A-rightleg) (C-leftleg) (C-rightleg)};
\end{tikzpicture}$
};
\node [inner sep=0pt] (3) at (2,0)
{
$\begin{tikzpicture}
        \node[Pants, top] (A) at (0,0) {};
        \node[Copants, anchor=leftleg] (B) at (A.leftleg) {};
        \node[Pants, anchor=belt] (C) at (B.belt) {};
        \node[Copants, anchor=leftleg] (D) at (C.leftleg) {};
        \selectpart[green] {(A-belt) (A-leftleg) (A-rightleg) (B-belt)};
\end{tikzpicture}$
};
\node (4) at (3,0)
{
$\begin{tikzpicture}
        \node[Pants, top] (A) at (0,0) {};
        \node[Copants, anchor=leftleg] (B) at (A.leftleg) {};
\end{tikzpicture}$
};
\begin{scope}[double arrow scope]
    \draw (1) -- node[above] {$\epsilon^\dagger$} (2);
    \draw (2) -- node[above] {$\II^\inv$} (3);
    \draw (3) -- node[above] {$\epsilon$} (4);
\end{scope}
\end{tz}
\\
\label{defn_of_III}
\III \quad&:=\quad
\begin{tz}[xscale=2, yscale=2]
\node (1) at (0,0)
{
$\begin{tikzpicture}
        \node[Pants, top] (A) at (0,0) {};
        \node[Copants, anchor=leftleg] (B) at (A.leftleg) {};
        \selectpart[green, inner sep=1pt] {(A-leftleg)};
\end{tikzpicture}$
};
\node (2) at (1,0)
{
$\begin{tikzpicture}
        \node[Pants, top] (A) at (0,0) {};
        \node[Copants, anchor=leftleg] (B) at (A.leftleg) {};
\end{tikzpicture}$
};
\node (3) at (2,0)
{
$\begin{tikzpicture}
        \node[Pants, top] (A) at (0,0) {};
        \node[Copants, anchor=leftleg] (B) at (A.leftleg) {};
        \selectpart[green, inner sep=1pt] {(A-leftleg)};
\end{tikzpicture}$
};
\node (4) at (3,0)
{
$\begin{tikzpicture}
        \node[Pants, top] (A) at (0,0) {};
        \node[Copants, anchor=leftleg] (B) at (A.leftleg) {};
\end{tikzpicture}$
};
\begin{scope}[double arrow scope]
    \draw (1) -- node[above] {$\theta$} (2);
    \draw (2) -- node[above] {$A$} (3);
    \draw (3) -- node[above] {$\theta$} (4);
\end{scope}
\end{tz}
\end{align}
\end{defn}

\begin{proposition}[Proposition \ref{PII_different_mods} in~\cite{PaperII}]
For a modular structure, the composites $A$ and $\III$ are invertible.
\qed
\end{proposition}

\setlength\obscurewidth{3pt}
\tikzset{knot/clip radius=\obscurewidth}
\tikzset{knot/clip width=0.1*\obscurewidth, knot/end tolerance=2pt}

We now investigate the actions of $\II$, $A$, and $\III$ on internal string diagrams. We write $\overline \theta$ as a shorthand for $\theta ^\inv$.
\begin{proposition}
Given a modular structure in \twovect, the composite $\II$ and its inverse have the following actions on internal string diagrams:
\normalbordisms
\begin{align}
\begin{tz}
\node (A) [Pants, anchor=belt, belt scale=1.5] at (0,0) {};
\node (B) [Copants, anchor=belt, top, belt scale=1.5] at (0,0) {};
\node (f) [tiny label] at (0,0.1) {$f$};
\strand [red strand] (f.center) to [out=-125, in=up] (A.leftleg) to +(0,-\boff);
\strand [red strand] (f.center) to [out=-55, in=up] (A.rightleg) to +(0,-\boff);
\strand [red strand] (f.center) to [out=125, in=down] (B.leftleg) to +(0,\toff);
\strand [red strand] (f.center) to [out=55, in=down] (B.rightleg) to +(0,\toff);
\end{tz}
\, &\xmapsto{\textstyle \isd(\II)} \,
\begin{tz}
\node (A) [Pants, anchor=belt, belt scale=1.7] at (0,0) {};
\node (B) [Copants, anchor=belt, top, belt scale=1.7] at (0,0) {};
\node (f) [tiny label] at (-0.45\cobwidth,0.15) {$f$};
\begin{knot}
\strand [red strand] (f.center) to [out=-125, in=up] (A.leftleg) to +(0,-\boff);
\strand [red strand] (f.center) to [out=125, in=down] (B.leftleg) to +(0,\toff);
\def\IIshift{0.3cm}
\strand [red strand] (f.center)
  to [out=55, in=left] (-0.1,0.3) node (p) {}
  to [out=right, in=left] (0.2,0) node (q) {}
  to [out=right, in=down, out looseness=0.4]
    node (t2) [pos=0.7] {}
    (B.rightleg)
  to +(0,\toff);
\strand [red strand] (f.center)
  to [out=-55, in=left] ([yshift=-\IIshift] p.center) {}
  to [out=right, in=left] ([yshift=\IIshift] q.center) {}
  to [out=right, in=up, out looseness=0.4]
    node (t1) [pos=0.8] {}
    (A.rightleg)
  to +(0,-\boff);
\flipcrossings{1}
\end{knot}
\node [tiny label, draw=red, text=red, rotate=5] at (t1.center) {$\theta$};
\node [tiny label, draw=red, text=red, rotate=-5] at (t2.center) {$\theta$};
\end{tz}
\,=\,
\begin{tz}
\node (A) [Pants, anchor=belt, belt scale=1.7] at (0,0) {};
\node (B) [Copants, anchor=belt, top, belt scale=1.7] at (0,0) {};
\node (f) [tiny label] at (0.45\cobwidth,0.15) {$f$};
\begin{knot}
\strand [red strand] (f.center) to [out=-55, in=up] (A.rightleg) to +(0,-\boff);
\strand [red strand] (f.center) to [out=55, in=down] (B.rightleg) to +(0,\toff);
\def\IIshift{0.3cm}
\strand [red strand] (f.center)
  to [out=135, in=right] (0.1,0.3) node (p) {}
  to [out=left, in=right] (-0.2,0) node (q) {}
  to [out=left, in=down, out looseness=0.4]
    node (t2) [pos=0.7] {}
    (B.leftleg)
  to +(0,\toff);
\strand [red strand] (f.center)
  to [out=-135, in=right] ([yshift=-\IIshift] p.center) {}
  to [out=left, in=right] ([yshift=\IIshift] q.center) {}
  to [out=left, in=up, out looseness=0.4]
    node (t1) [pos=0.8] {}
    (A.leftleg)
  to +(0,-\boff);
\flipcrossings{2}
\end{knot}
\node [tiny label, draw=red, text=red, rotate=5] at (t1.center) {$\theta$};
\node [tiny label, draw=red, text=red, rotate=-5] at (t2.center) {$\theta$};
\end{tz}
\\
\begin{tz}
\node (A) [Pants, anchor=belt, belt scale=1.5] at (0,0) {};
\node (B) [Copants, anchor=belt, top, belt scale=1.5] at (0,0) {};
\node (f) [tiny label] at (0,0.1) {$f$};
\strand [red strand] (f.center) to [out=-125, in=up] (A.leftleg) to +(0,-\boff);
\strand [red strand] (f.center) to [out=-55, in=up] (A.rightleg) to +(0,-\boff);
\strand [red strand] (f.center) to [out=125, in=down] (B.leftleg) to +(0,\toff);
\strand [red strand] (f.center) to [out=55, in=down] (B.rightleg) to +(0,\toff);
\end{tz}
\, &\xmapsto{\textstyle \isd(\II ^\inv)} \,
\begin{tz}
\node (A) [Pants, anchor=belt, belt scale=1.7] at (0,0) {};
\node (B) [Copants, anchor=belt, top, belt scale=1.7] at (0,0) {};
\node (f) [tiny label] at (-0.45\cobwidth,0.15) {$f$};
\begin{knot}
\strand [red strand] (f.center) to [out=-125, in=up] (A.leftleg) to +(0,-\boff);
\strand [red strand] (f.center) to [out=125, in=down] (B.leftleg) to +(0,\toff);
\def\IIshift{0.3cm}
\strand [red strand] (f.center)
  to [out=55, in=left] (-0.1,0.3) node (p) {}
  to [out=right, in=left] (0.2,0) node (q) {}
  to [out=right, in=down, out looseness=0.4]
    node (t2) [pos=0.7] {}
    (B.rightleg)
  to +(0,\toff);
\strand [red strand] (f.center)
  to [out=-55, in=left] ([yshift=-\IIshift] p.center) {}
  to [out=right, in=left] ([yshift=\IIshift] q.center) {}
  to [out=right, in=up, out looseness=0.4]
    node (t1) [pos=0.8] {}
    (A.rightleg)
  to +(0,-\boff);
\flipcrossings{2}
\end{knot}
\node [tiny label, draw=red, text=red, rotate=5] at (t1.center) {$\overline\theta$};
\node [tiny label, draw=red, text=red, rotate=-5] at (t2.center) {$\overline\theta$};
\end{tz}
\,=\,
\begin{tz}
\node (A) [Pants, anchor=belt, belt scale=1.7] at (0,0) {};
\node (B) [Copants, anchor=belt, top, belt scale=1.7] at (0,0) {};
\node (f) [tiny label] at (0.45\cobwidth,0.15) {$f$};
\begin{knot}
\strand [red strand] (f.center) to [out=-55, in=up] (A.rightleg) to +(0,-\boff);
\strand [red strand] (f.center) to [out=55, in=down] (B.rightleg) to +(0,\toff);
\def\IIshift{0.3cm}
\strand [red strand] (f.center)
  to [out=135, in=right] (0.1,0.3) node (p) {}
  to [out=left, in=right] (-0.2,0) node (q) {}
  to [out=left, in=down, out looseness=0.4]
    node (t2) [pos=0.7] {}
    (B.leftleg)
  to +(0,\toff);
\strand [red strand] (f.center)
  to [out=-135, in=right] ([yshift=-\IIshift] p.center) {}
  to [out=left, in=right] ([yshift=\IIshift] q.center) {}
  to [out=left, in=up, out looseness=0.4]
    node (t1) [pos=0.8] {}
    (A.leftleg)
  to +(0,-\boff);
\flipcrossings{1}
\end{knot}
\node [tiny label, draw=red, text=red, rotate=-5] at (t1.center) {$\overline\theta$};
\node [tiny label, draw=red, text=red, rotate=5] at (t2.center) {$\overline\theta$};
\end{tz}
\end{align}
\end{proposition}
\begin{proof}
This follows from \autoref{defn_II_III} and the actions of $\phiN$, $\theta$, and $\phiN ^\inv$ already established.
\end{proof}

\begin{proposition}
\label{lem:A_action}
Given a modular structure in \twovect, the composite $A$ acts in the following way on internal string diagrams:
\normalbordisms
\begin{equation}
\label{Aaction}
\def\quad{\hspace{5pt}}
\begin{tz}
\node (A) [Pants, top, belt scale=1.5] at (0,0) {};
\node (B) [Copants, anchor=leftleg, belt scale=1.5] at (A.leftleg) {};
\begin{scope}[internal string scope]
\draw ([yshift=\toff] A-belt.in-leftthird)
    to +(0,-\boff)
    to [out=down, in=up] (A.leftleg)
    to [out=down, in=up] (B-belt.in-leftthird)
    to +(0,-\boff);
\draw ([yshift=\toff] A-belt.in-rightthird)
    to +(0,-\boff)
    to [out=down, in=up] (A.rightleg)
    to [out=down, in=up] (B-belt.in-rightthird)
    to +(0,-\boff);
\end{scope}
\end{tz}
\quad\xmapsto{\isd(A)}\quad
{\color{red} \frac 1 p}
\begin{tz}
\node (A) [Pants, top, belt scale=1.5] at (0,0) {};
\node (B) [Copants, anchor=leftleg, belt scale=1.5] at (A.leftleg) {};
\begin{knot}
\strand [red strand] ([yshift=\toff] A-belt.in-leftthird)
    to +(0,-\toff)
    to [out=down, in=up] (A-rightleg.in-leftquarter)
    to [out=down, in=up] node [pos=0.7] (t1) {} (B-belt.in-leftthird)
    to +(0,-\boff);
\strand [red strand] ([yshift=\toff] A-belt.in-rightthird)
    to +(0,-\toff)
    to [out=down, in=up] (A-rightleg.in-rightquarter)
    to [out=down, in=up] (B-belt.in-rightthird)
    to +(0,-\boff);
\strand [red strand] (A.leftleg) to [out=down, in=down, looseness=1.5] (A.rightleg) to [out=up, in=up, looseness=1.5] node (t2) [pos=0.9] {} (A.leftleg);
\flipcrossings{2}
\end{knot}
\node [tiny label, draw=red, text=red, rotate=-20] at (t1.center) {$\overline\theta$};
\node [tiny label, draw=red, text=red, rotate=-10] at (t2.center) {$\overline\theta$};
\end{tz}
\quad=\quad
{\color{red} \frac 1 p}
\begin{tz}
\node (A) [Pants, top, belt scale=1.5] at (0,0) {};
\node (B) [Copants, anchor=leftleg, belt scale=1.5] at (A.leftleg) {};
\strand [red strand] ([yshift=\toff] A-belt.in-leftthird)
    to +(0,-\toff)
    to [out=down, in=up] (A-leftleg.in-leftthird)
    to [out=down, in=up] node [pos=0.7] (t1) {} (B-belt.in-leftthird)
    to +(0,-\boff);
\strand [red strand] ([yshift=\toff] A-belt.in-rightthird)
    to +(0,-\toff)
    to [out=down, in=up] (A-rightleg.in-rightthird)
    to [out=down, in=up] (B-belt.in-rightthird)
    to +(0,-\boff);
\strand [red strand] (A-leftleg.in-rightthird) to [out=down, in=down, looseness=1.8] (A-rightleg.in-leftthird) to [out=up, in=up, looseness=1.8] node (t2) [pos=0.9] {} (A-leftleg.in-rightthird);
\node [tiny label, draw=red, text=red, rotate=-10] at (t2.center) {$\overline\theta$};
\end{tz}
\end{equation}
\end{proposition}
\begin{proof}
We perform the following calculation:
\normalbordisms
\begin{equation}
\label{A_action_proof}
\def\quad{\hspace{5pt}}
\begin{tz}
\node (A) [Pants, top, belt scale=1.5] at (0,0) {};
\node (B) [Copants, anchor=leftleg, belt scale=1.5] at (A.leftleg) {};
\begin{scope}[internal string scope]
\draw ([yshift=\toff] A-belt.in-leftthird)
    to +(0,-\toff)
    to [out=down, in=up] (A.leftleg)
    to [out=down, in=up] (B-belt.in-leftthird)
    to +(0,-\boff);
\draw ([yshift=\toff] A-belt.in-rightthird)
    to +(0,-\toff)
    to [out=down, in=up] (A.rightleg)
    to [out=down, in=up] (B-belt.in-rightthird)
    to +(0,-\boff);
\end{scope}
\selectpart[green, inner sep=1pt] {(B-belt)};
\end{tz}
\quad\xmapsto{\isd(\epsilon ^\dag)}\quad
{\color{red} \frac 1 p}\,
\begin{tz}
\node (A) [Pants, top, belt scale=1.5] at (0,0) {};
\node (B) [Copants, anchor=leftleg, belt scale=1.5] at (A.leftleg) {};
\node (C) [Pants, top, anchor=belt, right leg scale=1.5, left leg scale=0.8, belt scale=1.5] at (B.belt) {};
\node (D) [Copants, anchor=leftleg, belt scale=1.5, right leg scale=1.5, left leg scale=0.8] at (C.leftleg) {};
\begin{scope}[internal string scope]
\draw ([yshift=\toff] A-belt.in-leftthird)
    to +(0,-\toff)
    to [out=down, in=up] (A.leftleg)
    to [out=down, in=up, in looseness=0.8, out looseness=1] (C.rightleg)
    to [out=down, in=up] (D-belt.in-leftthird)
    to +(0,-\boff);
\draw ([yshift=\toff] A-belt.in-rightthird)
    to +(0,-\toff)
    to [out=down, in=up] (A.rightleg)
    to [out=down, in=up] ([xshift=1pt] B-belt.in-rightquarter)
    to [out=down, in=up] (C-rightleg.in-rightquarter)
    to [out=down, in=up] (D-belt.in-rightthird)
    to +(0,-\boff);
\draw (C.leftleg) to [out=up, in=up, looseness=1.7] (C-rightleg.in-leftquarter) to [out=down, in=down, looseness=1.7] (C.leftleg);
\end{scope}
\selectpart[green] {(A-leftleg) (A-rightleg) (C-leftleg) (C-rightleg)};
\end{tz}
\quad\xmapsto{\isd(\II ^\inv)}\quad
{\color{red} \frac 1 p}\,
\begin{tz}
\node (B) [Copants, anchor=belt, belt scale=1.7, right leg scale=1.5] at (0,0) {};
\node (A) [Pants, top, right leg scale=1.5, left leg scale=1, anchor=leftleg, belt scale=1.5] at (B.leftleg) {};
\node (C) [Pants, anchor=belt, belt scale=1.7, right leg scale=1.5, left leg scale=1] at (B.belt) {};
\node (D) [Copants, anchor=leftleg, belt scale=1.7, right leg scale=1.5] at (C.leftleg) {};
\begin{knot}
\def\IIshift{0.3cm}
\strand [red strand] (C.leftleg) to [out=down, in=down, looseness=1.5] (C-rightleg.in-leftquarter)
  to [out=up, in=right, in looseness=1] ([xshift=-1pt, yshift=5pt] C.center)
  to [out=left, in=up, out looseness=0.5]
    node (t2) [pos=0.7] {}
    (C.leftleg);
\strand [red strand] ([yshift=-\boff] D-belt.in-leftthird)
    to +(0,\boff)
    to [out=up, in=down] (C.rightleg)
    to [out=up, in=-75] (-0.2,-0.3) {}
    to [out=85, in=down, out looseness=1.5]
        node (t1) [pos=0.6] {}
        (A.leftleg) to [out=up, in=down] (A-belt.in-leftthird)
    to +(0,\toff);
\strand [red strand] ([yshift=\toff] A-belt.in-rightthird)
    to +(0,-\toff)
    to [out=down, in=up] (A.rightleg)
    to [out=down, in=up] ([xshift=1pt] B-belt.in-rightquarter)
    to [out=down, in=up] (C-rightleg.in-rightquarter)
    to [out=down, in=up] (D-belt.in-rightthird)
    to +(0,-\boff);
\flipcrossings{2}
\end{knot}
\node [tiny label, draw=red, text=red, rotate=5] at (t1.center) {$\overline\theta$};
\node [tiny label, draw=red, text=red, rotate=-5] at (t2.center) {$\overline\theta$};
\selectpart[green] {(A-belt) (A-rightleg) (A-leftleg) (B-belt)};
\end{tz}
\quad\xmapsto{\isd(\epsilon)}\quad
{\color{red} \frac 1 p}
\begin{tz}
\node (A) [Pants, top, right leg scale=1.5, belt scale=1.5] at (0,0) {};
\node (B) [Copants, anchor=leftleg, right leg scale=1.5, belt scale=1.5] at (A.leftleg) {};
\begin{knot}
\strand [red strand] ([yshift=\toff] A-belt.in-leftthird)
    to +(0,-\toff)
    to [out=down, in=up] (A-rightleg.in-leftquarter)
    to [out=down, in=up] node [pos=0.8] (t1) {} (B-belt.in-leftthird)
    to +(0,-\boff);
\strand [red strand] ([yshift=\toff] A-belt.in-rightthird)
    to +(0,-\toff)
    to [out=down, in=up] (A-rightleg.in-rightquarter)
    to [out=down, in=up] (B-belt.in-rightthird)
    to +(0,-\boff);
\strand [red strand] (A.leftleg) to [out=down, in=down, looseness=1.3] (A.rightleg) to [out=up, in=up, looseness=1.3] node (t2) [pos=0.85] {} (A.leftleg);
\flipcrossings{2}
\end{knot}
\node [tiny label, draw=red, text=red, rotate=-10] at (t1.center) {$\overline\theta$};
\node [tiny label, draw=red, text=red, rotate=-25] at (t2.center) {$\overline\theta$};
\end{tz}
\end{equation}
By rearranging the internal strings in the second diagram here using \autoref{lem:cloaking}, established below, the alternative form on the right of expression~\eqref{Aaction} can be obtained.
\end{proof}

\begin{corollary}
\label{lem:III_action}
For a modular structure $Z$ in \twovect, the composite $\III$ has the following action on internal string diagrams:
\normalbordisms
\[
\begin{tz}[every to/.style={out=down,in=up}]
        \node[Pants, top, bot, belt scale=1.5] (A) at (0,0) {};
        \node[Copants,  bot, anchor=leftleg, belt scale=1.5] (B) at (A.leftleg) {};
        \strand[red strand] ([yshift=\toff] A-belt.in-leftthird)
                to (A-belt.in-leftthird)
                to (A.leftleg)
                to (B-belt.in-leftthird)
                to ([yshift=-\boff] B-belt.in-leftthird);
        \strand[red strand] ([yshift=\toff] A-belt.in-rightthird)
                to (A-belt.in-rightthird)
                to (A.rightleg)
                to (B-belt.in-rightthird)
                to ([yshift=-\boff] B-belt.in-rightthird);
\end{tz}
\quad \xmapsto{\textstyle \isd(\III)} \quad \frac{1}{p}\,
\begin{tz}[every to/.style={out=down,in=up}]
        \node[Pants, top, bot, right leg scale=1.5, belt scale=1.5] (A) at (0,0) {};
        \node[Copants,  bot, anchor=leftleg, right leg scale=1.5, belt scale=1.5] (B) at (A.leftleg) {};
        \strand[red strand] ([yshift=\toff] A-belt.in-leftthird)
                to (A-belt.in-leftthird)
                to (A-rightleg.in-leftquarter);
        \strand[red strand] ([yshift=\toff] A-belt.in-rightthird)
                to (A-belt.in-rightthird)
                to (A-rightleg.in-rightquarter)
                to (B-belt.in-rightthird)
                to ([yshift=-\boff] B-belt.in-rightthird);
        \strand[red strand] (A.leftleg) 
                to[out=down, in=down, looseness=1.5] (A.rightleg)
                to[out=up, in=up, looseness=1.5] (A.leftleg);
        \strand[red strand] (A-rightleg.in-leftquarter)
                to (B-belt.in-leftthird)
                to ([yshift=-\boff] B-belt.in-leftthird);
\end{tz}
\]
\end{corollary}
\begin{proof}
This follows from \autoref{defn_II_III} and \autoref{lem:A_action}.
\end{proof}
\begin{defn}
Given a modular structure $Z$ in \twovect, we define the {\em $s$-matrix} as the following collection of numbers, expressed as string diagrams in the associated semisimple ribbon category:
\[
s_{ij} \gap= \gap \frac 1 p
\begin{tz}[xscale=0.5]
        \draw[black strand] (0,0) to[out=up, in=up] (2,0) node[right, blue] {$j$} to[out=down, in=down] (0,0);
        \draw[black strand] (1,0) to[out=down, in=down] (-1,0) node[left] {$i$} to[out=up, in=up] (1,0);
        \draw[black strand] (0,0) to[out=up, in=up] (2,0);
        \node at (-1,0) {\arrownode[black]};
        \node [rotate=180] at (2,0) {\arrownode[black]};
\end{tz}
\]
In \autoref{lem:psquared} we will see that $p$ is a square root of the global dimension, so this definition agrees with~\mbox{\cite[Equation~(3.1.16)]{bk01-ltc}}.
\end{defn}

This is closely related to the $S$-matrix which forms part of the definition of a modular tensor category.
\begin{defn}
A \textit{modular tensor category} is a semisimple ribbon category over an algebraically-closed field $k$\footnote{A more general definition can be given in the non algebraically-closed case, but we do not need this extra generality.}, such that the S-matrix $S_{ij}$, with entries defined as follows, is invertible:
\[
S_{ij}
\gap = \gap
\begin{tz}[xscale=0.5]
        \draw[black strand] (0,0) to[out=up, in=up] (2,0) node[right, blue] {$j$} to[out=down, in=down] (0,0);
        \draw[black strand] (1,0) to[out=down, in=down] (-1,0) node[left] {$i$} to[out=up, in=up] (1,0);
        \draw[black strand] (0,0) to[out=up, in=up] (2,0);
        \node at (-1,0) {\arrownode[black]};
        \node [rotate=180] at (2,0) {\arrownode[black]};
\end{tz}
\]
\end{defn}
\noindent
Note that for a modular structure $Z$ in \twovect, its $s$-matrix will be invertible just when the  $S$-matrix of its associated semisimple ribbon category is invertible, because $p$ is an invertible constant. Recall also the following identity in a modular tensor category~\cite[Lemma 3.1.4]{bk01-ltc}:
\begin{equation}
\label{Smatrixidentity}
\begin{tz}[xscale=0.6]
\draw [black strand] (-1,1) to [out=up, in=up, looseness=0.9] (1,1);
\draw [black strand] (0,0) node [below] {$i$} to (0,2);
\draw [black strand] (-1,1) node [left] {$j$} to [out=down, in=down, looseness=0.9] (1,1);
\node at (-1,1) {\arrownode[black]};
\node at (0,1) {\arrownode[black]};
\end{tz}
\gap=\gap
\frac{S_{ji}}{d_i}
\begin{tz}
\draw (0,0) node [below] {$i$} to (0,2);
\node at (0,1) {\arrownode[black]};
\end{tz}
\end{equation}

\begin{proposition}
Given a modular structure $Z$ in \twovect, the composite $\III$ acts as the  $s$-matrix on the torus:
\mediumbordisms
\[
\begin{tz}
        \node[Cap, bot] (A) at (0,0) {};
        \node[Pants, bot, anchor=belt] (B) at (A.center) {};
        \node[Copants, bot, anchor=leftleg] (C) at (B.leftleg) {};
        \node[Cup] (D) at (C.belt) {};
        \draw[red strand] (B.rightleg) to[looseness=1.7, out=down, in=down] (B.leftleg) 
                to[looseness=1.7, out=up, in=up]
                node [left=-2pt, pos=0.25, red, font=\tiny] {$i$}
                node [pos=0.1, rotate=-20] {\arrownode}
                (B.rightleg);
\end{tz}
\, \xmapsto{\textstyle \isd(\III)}\, \sum_j s_{ij} \,
\begin{tz}
        \node[Cap, bot] (A) at (0,0) {};
        \node[Pants, bot, anchor=belt] (B) at (A.center) {};
        \node[Copants, bot, anchor=leftleg] (C) at (B.leftleg) {};
        \node[Cup] (D) at (C.belt) {};
        \draw[red strand] (B.rightleg) to[looseness=1.7, out=down, in=down] (B.leftleg) 
                to[looseness=1.7, out=up, in=up]
                node [left=-2pt, pos=0.25, red, font=\tiny] {$j$}
                node [pos=0.1, rotate=-20] {\arrownode}
                (B.rightleg);
\end{tz}
\]
\end{proposition}
\begin{proof}
We apply the action of $\III$ on internal string diagrams as given in \autoref{lem:III_action}, as well as the identity~\eqref{Smatrixidentity}:
\mediumbordisms
\[
\begin{tz}
        \node[Cap, bot] (A) at (0,0) {};
        \node[Pants, bot, anchor=belt] (B) at (A.center) {};
        \node[Copants, bot, anchor=leftleg] (C) at (B.leftleg) {};
        \node[Cup] (D) at (C.belt) {};
        \draw[red strand] (B.rightleg) to[looseness=1.7, out=down, in=down] (B.leftleg) 
                to[looseness=1.7, out=up, in=up]
                node [left=-2pt, pos=0.25, red, font=\tiny] {$i$}
                node [pos=0.1, rotate=-20] {\arrownode}
                (B.rightleg);
\end{tz}
\gap \xmapsto{\textstyle \isd(\III)} \gap
\frac{1}{p} \sum_j d_j \,
\begin{tz}
        \node[Cap, bot] (A) at (0,0) {};
        \node[Pants, bot, anchor=belt, right leg scale=2, height scale=1.5] (B) at (A.center) {};
        \node[Copants, bot, anchor=leftleg, right leg scale=2, height scale=1.5] (C) at (B.leftleg) {};
        \node[Cup] (D) at (C.belt) {};
        \draw[red strand] (B.rightleg) to[looseness=1.7, out=down, in=down]  (B.leftleg) node [left=-5pt, font=\tiny, yshift=8pt, red] {$j$} to[looseness=1.7, out=up, in=up] (B.rightleg);
        \draw[red strand] (B-rightleg.in-leftquarter) to[out=up, in=up, looseness=2] node[pos=0.2, right=7pt, red, font=\tiny] {$i$} node[pos=0.8, rotate=-150] {\arrownode} (B-rightleg.in-rightquarter)
                to[out=down, in=down, looseness=2.5] (B-rightleg.in-leftquarter);
        \draw[red strand] (B.rightleg) to[looseness=1.7, out=up, in=up] node[pos=0.9, rotate=-20] {\arrownode} (B.leftleg); 
\end{tz}
\gap = \gap \sum_j \frac 
{S _{ij}}{p}
\,
\begin{tz}
        \node[Cap, bot] (A) at (0,0) {};
        \node[Pants, bot, anchor=belt] (B) at (A.center) {};
        \node[Copants, bot, anchor=leftleg] (C) at (B.leftleg) {};
        \node[Cup] (D) at (C.belt) {};
        \draw[red strand] (B.rightleg) to[looseness=1.7, out=down, in=down] (B.leftleg) 
                to[looseness=1.7, out=up, in=up]
                node [left=-2pt, pos=0.25, red, font=\tiny] {$j$}
                node [pos=0.1, rotate=-20] {\arrownode}
                (B.rightleg);
\end{tz}
\]
\end{proof}

\noindent
Since $\III$ is invertible, it follows that the $s$-matrix is invertible, and hence that the associated ribbon category is modular.
\begin{corollary}
\label{ribbonismodular}
Given a modular structure in \twovect, the  associated semisimple ribbon category is modular.
\qed
\end{corollary}

\subsection{The anomaly}
\label{sec:anomalyvalue}

Here we investigate the value of the nonzero constant $p$ that arises from a modular structure in \twovect, and we show that $p^2 = p^+ p^-$, where $p^+$ and $p^-$ are constants obtained from the modular tensor category.

We begin by defining the anomaly as a composite.
\smallbordisms
\begin{defn}[Definition \ref{PII_anomalydef} from~\cite{PaperII}] For a modular structure, the \textit{anomaly} $x$ is defined as the following composite:
\[
x\quad:=\quad
\begin{tz}
    \node [Cyl, top, tall] (A) at (0,0) {};
\end{tz}
\longxdoubleto{\epsilon ^\dag}
\begin{tz}
        \node[Pants, top] (A) at (0,0) {};
        \node[Copants, anchor=leftleg] (B) at (A.leftleg) {};
        \selectpart[inner sep=1pt, green] {(A-leftleg)};
\end{tz}
\longxdoubleto{\theta}
\begin{tz}
        \node[Pants, top] (A) at (0,0) {};
        \node[Copants, anchor=leftleg] (B) at (A.leftleg) {};
\end{tz}
\longxdoubleto{\epsilon}
\begin{tz}
    \node [Cyl, top, tall] (A) at (0,0) {};
\end{tz}
\]
\end{defn}
\begin{lemma}[\autoref{PII_invertibility_of_x} from~\cite{PaperII}]
\label{lem:anomalyinvertible}
For a modular structure, the anomaly is invertible, with $x ^{\inv} = x ^{\dag}$.
\qed
\end{lemma}

Recall that in a modular tensor category with simple unit, there exist $p^+, p^- \in k$ satisfying the following equations~\cite[Lemma 3.1.5]{bk01-ltc}:
\mediumbordisms
\begin{align}
\label{eq:defppluspminus}
p^+
\begin{tz}
    \draw[black strand] (0,1) to (0,-1);
    \node [tiny label, text=black, draw=black] at (0,0) {$\overline\theta$};
\end{tz}
&\quad= \quad
\begin{tz}
    \draw[black strand] (0.5,-1) to (0.5,0);
    \draw[black strand] (0,0) to[out=up, in=up, looseness=1.8] (1,0) node[tiny label, text=black, draw=black] {$\theta$} to[out=down, in=down, looseness=1.8] (0,0);
    \draw[black strand] (0.5,1) to (0.5,0);
\end{tz}
&
p^-
\begin{tz}
    \draw[black strand] (0,1) to (0,-1);
    \node [tiny label, text=black, draw=black] at (0,0) {$\theta$};
\end{tz}
&\quad= \quad
\begin{tz}
    \draw[black strand] (0.5,-1) to (0.5,0);
    \draw[black strand] (0,0) to[out=up, in=up, looseness=1.8] (1,0) node[tiny label, text=black, draw=black] {$\overline \theta$} to[out=down, in=down, looseness=1.8] (0,0);
    \draw[black strand] (0.5,1) to (0.5,0);
\end{tz}
\end{align}
Explicit formulas for $p^+$ and $p^-$ can be obtained by choosing the open black strand to be the identity on the tensor unit object. The \emph{anomaly} of the modular tensor category is defined as the ratio $p^+/p^-$. If $p^+/p^- = 1$, we say that the modular tensor category is {\em anomaly-free}. 

We now consider the action of the anomaly morphism $x$ on internal string diagrams.
\begin{lemma}
\label{lem:anomalyaction}
For a simple modular structure $Z$ in \twovect, the anomaly and its dagger have the following actions on internal string diagrams:
\mediumbordisms
\begin{align*}
\begin{tz}
\node (A) [Cyl, top, tall] at (0,0) {};
\draw [red strand] ([yshift=-0.25cm] A.bot) to ([yshift=0.2cm] A.top);
\end{tz}
&
\longxmapsto{\displaystyle\isd (x)}{40pt}
{\color{red}\frac{p^+}{p}}
\begin{tz}
\node (A) [Cyl, top, tall] at (0,0) {};
\draw [red strand] ([yshift=-0.25cm] A.bot) to ([yshift=0.2cm] A.top);
\end{tz}
&
\begin{tz}
\node (A) [Cyl, top, tall] at (0,0) {};
\draw [red strand] ([yshift=-0.25cm] A.bot) to ([yshift=0.2cm] A.top);
\end{tz}
&
\longxmapsto{\displaystyle\isd (x^\dag)}{40pt}
{\color{red}\frac{p^-}{p}}
\begin{tz}
\node (A) [Cyl, top, tall] at (0,0) {};
\draw [red strand] ([yshift=-0.25cm] A.bot) to ([yshift=0.2cm] A.top);
\end{tz}
\end{align*}
\end{lemma}
\begin{proof}
This follows from the actions of $\epsilon^\dag$, $\theta$, and $\epsilon$ on internal string diagrams as already established, using \autoref{lem:cloaking}.
\end{proof}

\begin{lemma}
\label{lem:psquared}
For a simple modular structure $Z$ in \twovect, we have $p^2 = p^+ p^-$.
\end{lemma}
\begin{proof}
By \autoref{lem:anomalyinvertible} we have $\isd(x) \circ \isd(x ^\dag) = \id$, and hence by \autoref{lem:anomalyaction} we have $p^+ p^-/p^2 = 1$.
\end{proof}

\begin{remark}
\label{remark:roots}
Note the simple algebraic fact that there is a canonical bijection between square roots of $p^+ p^-$, the \textit{global dimension},  and square roots of $p^+/p^-$, the anomaly, by multiplying or dividing by $p^-$. So to choose a square root of the global dimension is equivalent to choosing a square root of the anomaly.
\end{remark}

\section{Field theories and modular categories}
\label{sec:classification}

In this section we classify the representations of the presentations of the bordism categories given in \autoref{sec:geometricalpresentations}. In \autoref{sec:modcatmodobj}, we showed that for a linear modular structure (that is a modular structure in $\twovect$), the associated monoidal category is a modular tensor category. In \autoref{sec:buildmodularstructure}, we show how a linear modular structure can be constructed from the data of a modular tensor category equipped in each factor with a square root of the global dimension. We analyze the space of equivalence classes of modular tensor categories and of linear modular structures, and prove the classification result for componentwise signature TQFTs. In \autoref{sec:classificationother}, we prove the classification results for TQFTs with oriented, signature, and $p_1$-structure.

\subsection{Componentwise-signature field theories}
\label{sec:buildmodularstructure}

We now construct a modular structure in $\bimod$ from a modular tensor category equipped with a square root of the anomaly in each factor. Note that by \autoref{cor:modissemisimple}, modular structures in $\bimod$ and $\twovect$ are equivalent.

\def\ZC{\ensuremath{Z_ \cat C ^p}}
\begin{proposition}
\label{constructtqft}
Given a simple modular tensor category \cat C, equipped with a constant $p$ satisfying $p^2 = p^+/p^-$, the following data defines a modular structure $\ZC$ in $\bimod$:
\begin{itemize}
\item On the generating object, $\ZC$ gives the category \cat C.
\item On generating 1\-morphisms, $\ZC$ acts as follows (the functors are specified by their structure vector spaces):
\begin{align}
\ZC (\tikztinypants) ^W _{X,Y} &:= \Hom _\cat C (W, X \otimes Y)
\\
\ZC (\tikztinycopants) _Y ^{W,X} &:= \Hom _\cat C (W \otimes X, Y)
\\
\ZC (\tikztinycup) ^W  &:= \Hom _\cat C (W, I)
\\
\ZC (\tikztinycap) _W &:= \Hom _\cat C (I, W)
\end{align}

\item On generating 2\-morphisms, $\ZC$ acts as follows (using the internal string diagram graphical notation):
\end{itemize}
\def\arrlen{1.5cm}
\def\arrgap{\hspace{0.2cm}}
\mediumbordisms
\begin{equation}
    \begin{tz}
    \node[Pants, bot, top, belt scale=1.5] (B) at (0,0) {};
    \node[Pants, bot, anchor=belt] (A) at (B.leftleg) {};
    \node[SwishL, bot, anchor=top] (C) at (B.rightleg) {};
    \begin{scope}[internal string scope]
        \node (j) at (A.leftleg) [below=\boff] {};
        \node (k) at (A.rightleg) [below=\boff] {};
        \node (l) at (C.bot) [below=\boff] {};
        \draw (j.north)
            to (A.leftleg)
            to [out=up, in=-90] (B-leftleg.in-leftthird)
            to [out=up, in=-90] (B-belt.in-leftquarter)
            to [out=up, in=-90] ([yshift=\toff] B-belt.in-leftquarter);
        \draw (k.north)
            to (A.rightleg)
            to [out=up, in=-90] (B-leftleg.in-rightthird)
            to [out=up, in=-90] (B-belt)
            to [out=up, in=-90] ([yshift=\toff] B.belt);
        \draw (l.north)
            to (C.bot)
            to [out=up, in=-90] (B.rightleg)
            to [out=up, in=-90] (B-belt.in-rightquarter)
            to [out=up, in=-90] ([yshift=\toff] B-belt.in-rightquarter);
    \end{scope}
    \end{tz}
\arrgap
\begin{tz}
\draw [|->] (0,0) to node [above] {$\ZC (\alpha)$} (1.7,0);
\draw [|->] (1.7,-0.5) to node [below] {$\ZC(\alpha ^{\inv})$} (0,-0.5);
\end{tz}
\arrgap
    \begin{tz}
    \node[Pants, bot, top, belt scale=1.5] (B) at (0,0) {};
    \node[Pants, bot, anchor=belt] (A) at (B.rightleg) {};
    \node[SwishR, bot, anchor=top] (C) at (B.leftleg) {};
    \begin{scope}[internal string scope]
        \node (i) at (B.belt) [above=\toff] {};
        \node (j) at (C.bot) [below=\boff] {};
        \node (k) at (A.leftleg) [below=\boff] {};
        \node (l) at (A.rightleg) [below=\boff] {};
        \draw (j.north)
            to (C.bot)
            to [out=up, in=-90] (B-leftleg)
            to [out=up, in=-90] (B-belt.in-leftquarter)
            to [out=up, in=-90] ([yshift=\toff] B-belt.in-leftquarter);
        \draw (k.north)
            to (A.leftleg)
            to [out=up, in=-90] (B-rightleg.in-leftthird)
            to [out=up, in=-90] (B-belt)
            to [out=up, in=-90] ([yshift=\toff] B.belt);
        \draw (l.north)
            to (A.rightleg)
            to [out=up, in=-90] (B-rightleg.in-rightthird)
            to [out=up, in=-90] (B-belt.in-rightquarter)
            to [out=up, in=-90] ([yshift=\toff] B-belt.in-rightquarter);
    \end{scope}
    \end{tz}
\end{equation}
\begin{equation}
\begin{tz}
    \node[Pants, top, bot] (A) at (0,0) {};
    \node[Cup] (B) at (A.leftleg) {};
    \node[Cyl, bot, anchor=top] (C) at (A.rightleg) {};
    \begin{scope}[internal string scope]
        \node (i) at ([yshift=\toff] A.belt) [above] {};
        \node (i2) at ([yshift=-\boff] C.bottom) [below] {};
        \draw (i2) to (C.top) to [out=up, in=-90] (A.belt) to (i);
    \end{scope}
\end{tz}
\arrgap
\begin{tz}
\draw [|->] (0,0) to node [above] {$\ZC(\lambda)$} (1.7,0);
\draw [|->] (1.7,-0.5) to node [below] {$\ZC(\lambda ^{\inv})$} (0,-0.5);
\end{tz}
\arrgap
\begin{tz}
    \node[Cyl, tall, bot, top] (A) at (0,0) {};
    \begin{scope}[internal string scope]
        \node (i) at ([yshift=\toff] A.top) [above] {};
        \node (i2) at ([yshift=-\boff] A.bot) [below] {};
        \draw (i2) to (i);
    \end{scope}
\end{tz}
\arrgap
\begin{tz}
\draw [<-|] (0,0) to node [above] {$\ZC(\rho)$} (1.7,0);
\draw [<-|] (1.7,-0.5) to node [below] {$\ZC(\rho ^{\inv})$} (0,-0.5);
\end{tz}
\arrgap
\begin{tz}
    \node[Pants, top, bot] (A) at (0,0) {};
    \node[Cup] (B) at (A.rightleg) {};
    \node[Cyl, bot, anchor=top] (C) at (A.leftleg) {};
    \begin{scope}[internal string scope]
        \node (i) at ([yshift=\toff] A.belt) [above] {};
        \node (i2) at ([yshift=-\boff] C.bottom) [below] {};
        \draw (i2) to (C.top) to [out=up, in=-90] (A.belt) to (i);
    \end{scope}
\end{tz}
\end{equation}
\begin{align}
\begin{tz}
    \node[Pants, bot, top, height scale=1.1] (A) at (0,0) {};
    \begin{scope}[internal string scope]
        \node (i) at ([yshift=\toff] A.belt) [above] {};
        \node (j) at ([yshift=-\boff] A.leftleg) [below] {};
        \node (k) at ([yshift=-\boff] A.rightleg) [below] {};
        \draw ([yshift=-\boff] A.leftleg) to (A.leftleg)
            to [out=up, in=down] (A-belt.in-leftthird) to +(0,0.3);
        \draw ([yshift=-\boff] A.rightleg) to (A.rightleg)
            to [out=up, in=down] (A-belt.in-rightthird) to +(0,0.3);
    \end{scope}
\end{tz}
\arrgap&
\begin{tz}
\path [use as bounding box] (0,0) rectangle (\arrlen,-0.5);
\draw [|->] (0,0) to node [above] {$\ZC(\beta)$} (\arrlen,0);
\draw [|->] (\arrlen,-0.5) to node [below] {$\ZC(\beta ^{\inv})$} (0,-0.5);
\end{tz}
\arrgap
\begin{tz}
    \node[Pants, bot, top, height scale=1.3] (A) at (0,0) {};
    \node[BraidA, bot, anchor=topleft] (B) at (A.leftleg) {};
    \draw[internal string] ([yshift=\toff] A-belt.in-rightthird)
        to (A-belt.in-rightthird)
        to [out=down, in=up, out looseness=1.7, in looseness=0.7] (B.topleft);
    \obscureA{B}
    {
\draw [internal string] (B.topleft) to [out=down, in=up, looseness=0.6] (B.bottomright);
    }
    \draw [internal string] (B.bottomright) to +(0,-0.3);
    \draw[internal string] ([yshift=\toff] A-belt.in-leftthird)
        to +(0,-0.3)
        to [out=down, in=up, out looseness=1.7, in looseness=0.7] (A.rightleg)
        to [out=down, in=up, looseness=0.6] (B.bottomleft)
        to +(0,-0.3);
\end{tz}
&
\begin{tz}
    \node[Cyl, bot, top] (A) at (0,0) {};
    \begin{scope}[internal string scope]
        \node (i) at ([yshift=\toff] A.top) [above] {};
        \node (i2) at ([yshift=-\boff] A.bot) [below] {};
        \draw (i.south) -- (i2.north);
    \end{scope}
\end{tz}
\arrgap&
\begin{tz}
\draw [|->] (0,0) to node [above] {$\ZC(\theta)$} (1.7,0);
\draw [|->] (1.7,-0.5) to node [below] {$\ZC(\theta ^{\inv})$} (0,-0.5);
\end{tz}
\arrgap
\begin{tz}
    \node[Cyl, bot, top] (A) at (0,0) {};
    \begin{scope}[internal string scope]
        \node (i) at ([yshift=\toff] A.top) [above] {};
        \node (i2) at ([yshift=-\boff] A.bot) [below] {};
        \draw (i.south) -- (i2.north);
        \node [tiny label] at (A.center) {$\theta$};
    \end{scope}
\end{tz}
\end{align}
\begin{align}
\begin{tz}
        \node[Cyl, tall, top, bot] (A) at (0,0) {};
        \node[Cyl, tall, top, bot] (B) at (\cobgap + \cobwidth, 0) {};
        \begin{scope}[internal string scope]
                \draw ([yshift=\toff] A.top) node[above]{} -- ([yshift=-\boff] A.bot) node[below]{};
                \draw ([yshift=\toff] B.top) node[above]{} -- ([yshift=-\boff] B.bot) node[below]{};
        \end{scope}
\end{tz}
\arrgap
&\longxmapsto{\textstyle \ZC (\eta)}{\arrlen}
\arrgap
\begin{tz}
        \node[Pants, bot, belt scale=1.5] (A) at (0,0) {};
        \node[Copants, bot, anchor=belt, belt scale=1.5, top] (B) at (A.belt) {}; 
        \begin{scope}[internal string scope]
                \node (i) at ([yshift=\toff] B.leftleg) [above] {};
                \node (j) at ([yshift=\toff] B.rightleg) [above] {};
                \node (i2) at ([yshift=-\boff] A.leftleg) [below] {};
                \node (j2) at ([yshift=-\boff] A.rightleg) [below] {};
                \draw (i.south) to[out=down, in=up] (B-belt.in-leftthird) to[out=down, in=up] (i2.north);
                \draw (j.south) to[out=down, in=up] (B-belt.in-rightthird) to[out=down, in=up] (j2.north);
        \end{scope}
\end{tz}
\\
\begin{tz}
        \node[Pants, bot, belt scale=1.5] (A) at (0,0) {};
        \node[Copants, bot, anchor=belt, belt scale=1.5, top] (B) at (A.belt) {}; 
        \node (f) [tiny label] at ([yshift=5pt] B.belt) {$f$};
        \begin{scope}[internal string scope]
                \node (i) at ([yshift=\toff] B.leftleg) [above] {$k$};
                \node (j) at ([yshift=\toff] B.rightleg) [above] {$l$};
                \node (i2) at ([yshift=-\boff] A.leftleg) [below] {$i$};
                \node (j2) at ([yshift=-\boff] A.rightleg) [below] {$j$};
                \draw (i.south)
                        to[out=down, in=135] (f.center)
                        to[out=-135, in=up] (i2.north);
                \draw (j.south)
                        to[out=down, in=45] (f.center)
                        to[out=-45, in=up] (j2.north);
        \end{scope}
\end{tz}
\arrgap
&\longxmapsto{\textstyle \ZC (\eta ^\dag)}{\arrlen}
\arrgap
{\color{red} \frac{p \, \delta _{i,k} \delta _{j,l}}{d_i d_j}}
\begin{aligned}
\begin{tikzpicture}
\node (f) [tiny label] at ([yshift=5pt] B.belt) {$f$};
\draw [red strand] (f.center)
        to [out=55, in=up, out looseness=3] +(10pt,0pt)
        to [out=down, in=-55, in looseness=3] (f.center);
\draw [red strand] (f.center)
        to [out=125, in=up, out looseness=3] +(-10pt,0pt)
        to [out=down, in=-125, in looseness=3] (f.center);
\end{tikzpicture}
\end{aligned}
 \begin{tz}
        \node[Cyl, tall, top, bot] (A) at (0,0) {};
        \node[Cyl, tall, top, bot] (B) at (\cobgap + \cobwidth, 0) {};
        \begin{scope}[internal string scope]
                \draw ([yshift=\toff] A.top) node[above]{$k$} -- ([yshift=-\boff] A.bot) node[below]{$i$};
                \draw ([yshift=\toff] B.top) node[above]{$l$} -- ([yshift=-\boff] B.bot) node[below]{$j$};
        \end{scope}
\end{tz}
\end{align}
\begin{align}
\begin{tz}
        \node[Pants, bot, top, belt scale=1.5] (A) at (0,0) {};
        \node[Copants, bot, anchor=leftleg, belt scale=1.5] (B) at (A.leftleg) {};
        \begin{scope}[internal string scope]
                \node (i) at ([yshift=\toff] A.belt) [above] {};
                \node (i2) at ([yshift=-\boff] B.belt) [below] {};
                \draw ([yshift=\toff] A-belt.in-leftthird)
                  to (A-belt.in-leftthird)
                  to [out=down, in=up] (A.leftleg)
                  to [out=down, in=up] (B-belt.in-leftthird)
                  to +(0,-0.3);
                \draw ([yshift=\toff] A-belt.in-rightthird)
                  to (A-belt.in-rightthird)
                  to [out=down, in=up] (A.rightleg)
                  to [out=down, in=up] (B-belt.in-rightthird)
                  to +(0,-0.3);
        \end{scope}
\end{tz}
\arrgap&\longxmapsto{\textstyle \ZC (\epsilon)}{\arrlen} \arrgap
\begin{tz}
        \node[Cyl, top, bot=false] (A) at (0,0) {};
        \node[Cyl, bot] (A2) at (0,-\cobheight) {};
        \begin{scope}[internal string scope]
        \draw ([yshift=\toff] A-top.in-leftthird) to ([yshift=-\boff] A2-bottom.in-leftthird);
        \draw ([yshift=\toff] A-top.in-rightthird) to ([yshift=-\boff] A2-bottom.in-rightthird);
        \end{scope}
\end{tz}
&
\begin{tz}
    \node (A) [Cyl, tall, bot, top] at (0,0) {};
    \begin{scope}[curvein]
        \node (top) at (A.top) [above=\toff] {};
        \node (bot) at (A.bot) [below=\boff] {};
        \draw (top.south) to (bot.north);
    \end{scope}
\end{tz}
\arrgap
& \longxmapsto{\displaystyle \ZC (\epsilon ^\dag)}{\arrlen}
\arrgap
{ \color{red} \frac{1}{p} }
\begin{tz}
    \node (A) [Pants, bot, top] at (0,0) {};
    \node (B) [Copants, bot, anchor=leftleg] at (A.leftleg) {};
    \begin{scope}[curvein]
        \node (top) at (A.belt) [above=0.2cm] {};
        \node (bot) at (B.belt) [below=0.25cm] {};
        \draw (top.south)
            to [out=down, in=up, out looseness=1.5] (A-leftleg.in-leftthird)
            to [out=down, in=up, in looseness=1.5] (bot.north);
        \draw[red strand] (A-leftleg.in-rightthird)
            to [out=up,in=up,looseness=1.7]
                (A-rightleg.in-leftthird)
            to [out=down,in=down, looseness=1.7] (A-leftleg.in-rightthird);
    \end{scope}
\end{tz}
\\
\begin{tz}
    \setlength\cupheight{1.3\cupheight}
    \node[Cup, top] (A) at (0,0) {};
    \node[Cap, bot] (B) at (0,-1.5\cobheight) {};
    \node [tiny label, text=red] (f) at ([yshift=-0.56\cupheight] A) {$\shrunkenf$};
    \node [tiny label, text=red] (g) at ([yshift=0.56\cupheight] B) {$g$};
    \begin{scope}[internal string scope]
        \node (i) at (A) [above=\toff] {};
        \node (j) at (B) [below=\boff] {};
        \draw (f.center) to (i);
        \draw (g.center) to (j);
    \end{scope}
\end{tz}
\arrgap&\longxmapsto{\textstyle \ZC (\mu)}{\arrlen} \arrgap
\begin{tz}
    \setlength\cupheight{1.3\cupheight}
    \node[Cyl, top, bot=false] (A) at (0,0) {};
    \node[Cyl, bot] (B) at (0,-0.5\cobheight) {};
    \node [tiny label, text=red] (f) at ([yshift=-0.66\cupheight] A.top) {$f$};
    \node [tiny label, text=red] (g) at ([yshift=0.56\cupheight] B.bot) {$g$};
    \begin{scope}[internal string scope]
        \node (i) at (A.top) [above=\toff] {};
        \node (j) at (B.bot) [below=\boff] {};
        \draw (f.center) to (i);
        \draw (g.center) to (j);
    \end{scope}
\end{tz}
&
\begin{tz}
\node (A) [Cyl, top, bot=false] at (0,0) {};
\node (B) [Cyl, bot] at (0, -0.5\cobheight) {};
    \draw [red strand] ([below=\boff] B.bot) node [below, text=red] {$i$} to ([above=\toff] A.top) node [above] [red] {$i$};
\end{tz}
\arrgap
&\longxmapsto{\textstyle \ZC(\mu ^\dagger)}{\arrlen}
\arrgap
{\red p \, \delta_{i,0}}
\,\,
\begin{tz}
\node (B) [Cap, bot] at (0,0) {};
\node (A) [Cup, top] at (0,1.5\cobheight) {};
    \begin{scope}[curvein]
        \node (top) at (A.center) [above=0.1cm] {$i$};
        \node (bot) at (B.center) [below=0.15cm] {$i$};
    \end{scope}
\end{tz}
\\
\begin{tz}
\end{tz}
\arrgap&\longxmapsto{\textstyle \ZC (\nu)}{\arrlen}\arrgap
\begin{tz}
    \node[Cup] (A) at (0,0) {};
    \node[Cap, bot] (B) at (0,0) {};
\end{tz}
&
\begin{tz}
\node [Cup] at (0,0) {};
\node [Cap, bot] at (0,0) {};
\begin{scope}[curvein]
\end{scope}
\end{tz}
\arrgap
&\longxmapsto{\textstyle \ZC (\nu ^\dagger)}{\arrlen}
\arrgap
{\red \frac{1}{p}}
\,
\begin{tz}
\begin{scope}[curvein]
\end{scope}
\end{tz}
\end{align}
Here and below, the index 0 refers to the simple unit object.
\end{proposition}

\begin{proof}
We must check that each axiom of a modular structure is satisfied. The majority of the axioms are immediate, as they are also modular tensor category axioms: this covers the pentagon, triangle, hexagon, twist, and ribbon relations. Using internal string diagrams, invertibility of  $\phiN$ and $\phiM$ is straightforward to verify, using their actions on internal string diagrams given in~\eqref{eq:phiaction} and rigidity of the modular tensor category. The 8 adjunction equations involving $\epsilon$, $\eta$, $\mu$, and $\nu$ and their opposites can also be straightforwardly verified by direct calculation.

The left side of the pivotality relation~\eqref{piv_on_sphere} has the following effect on internal string diagrams:
\mediumbordisms
\[
\begin{tz}
    \node[Cap, bot] (A) at (0,0) {};
    \node[Cup] (B) at (0,0) {};
    \selectpart[green, inner sep=1pt] {(A-center)};
\end{tz}
\gap\xmapsto{\ZC (\epsilon ^\dag)}\gap
{\red \frac 1 {p}}
\begin{tz}
    \node [Pants] (A) at (0,0) {};
        \node[Cap] at (A.belt) {};
    \node[Copants, anchor=leftleg] (B) at (A.leftleg) {};
    \node[Cup] at (B.belt) {};
    \selectpart[green, inner sep=1pt] {(A-rightleg)};
    \draw [red strand] (B.rightleg) 
        to [looseness=1.7, out=down, in=down] (A.leftleg)
        to [out=up, in=up, looseness=1.7] (B.rightleg);
\end{tz}
\gap\xmapsto{\ZC (\mu ^\dag)}\gap
{\red \frac {p} {p}}
\begin{tz}
    \node [Pants] (A) at (0,0) {};
        \node[Cap] at (A.belt) {};
    \node[Cyl, height scale=1.5, anchor=top] (B) at (A.leftleg) {};
    \node[Cup] (C) at (A.rightleg) {};
    \node[Copants, anchor=leftleg] (D) at (B.bottom) {};
    \node[Cap] (E) at (D.rightleg) {};
    \node[Cup] at (D.belt) {};
    \selectpart[green] {(A-rightleg) (D.rightleg south)};
\end{tz}
\gap\xmapsto{\ZC (\mu)}\gap
\begin{tz}
    \node [Pants] (A) at (0,0) {};
        \node[Cap] at (A.belt) {};
    \node[Copants, anchor=leftleg] (B) at (A.leftleg) {};
    \node[Cup] at (B.belt) {};
\end{tz}
\gap\xmapsto{\ZC (\epsilon)}\gap
\begin{tz}
    \node [Cap] at (0,0) {};
    \node[Cup] at (0,0) {};
\end{tz}
\]
Overall this gives the identity, and so the pivotality axiom is satisfied.

The top path of the modularity axiom~\eqref{MOD} has the following effect on internal string diagrams:
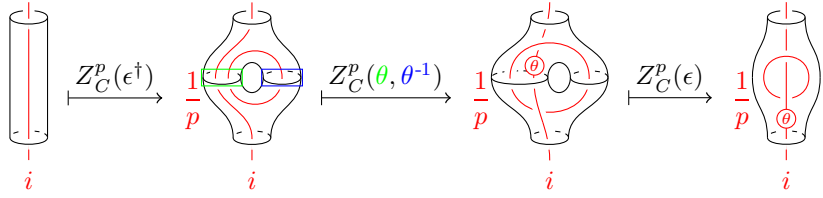
\normalbordisms
\begin{equation}
\nonumber
\begin{tz}
    \node[Cyl, top, tall] (A) at (0,0) {};
    \draw [red strand] ([yshift=-\boff] A.bot) node [below, red] {$i$} to ([yshift=\toff] A.top);
\end{tz}
\gap\xmapsto{\ZC(\epsilon ^\dag)}\gap
{\color{red} \frac{1}{p}}
\begin{tz}
        \node[Pants, top] (A) at (0,0) {};
        \node[Copants, anchor=leftleg] (B) at (A.leftleg) {};
        \draw[red strand] (B.rightleg) 
                to[looseness=1.7, out=down, in=down]        
          (A-leftleg.in-rightthird)
                to[out=up, in=up, looseness=1.7] (B.rightleg);
\draw [red strand] ([yshift=-\boff] B.belt) node [below, red] {$i$} to (B.belt) to [out=up, in=down] (A-leftleg.in-leftthird) to [out=up, in=down] (A.belt) to ([yshift=\toff] A.belt);
        \selectpart[green, inner sep=0.6pt]{(A-leftleg)};
        \selectpart[color=blue, inner sep=0.6pt] {(A-rightleg)};
\end{tz}
\gap\xmapsto{\ZC({\color{green} \theta}, {\color{blue} \theta ^{\inv}})}\gap
{\red \frac 1 {p}}
\begin{tz}
    \node [Pants, top, bot, left leg scale=1.5] (A) at (0,0) {};
    \node [Copants, anchor=leftleg, bot, left leg scale=1.5] (B) at (A.leftleg) {};
        \begin{knot}[draft mode=off]
    \strand [red strand] (B.rightleg) 
        to [looseness=1.5, out=down, in=down]         
          (A-leftleg.in-leftthird)
        to [out=up, in=up, looseness=1.5]
        (B.rightleg);
\strand [red strand] ([yshift=-\boff] B.belt) node [below, red] {$i$}
  to [out=up, in=down] (A-leftleg.in-rightquarter)
  to [out=up, in=down] ([yshift=\toff] A.belt);
  \flipcrossings{2}
\end{knot}
\draw [invisible strand, on layer=foreground] (B.rightleg)
    to [looseness=1.5, out=up, in=up]         
       node[draw=red, pos=0.1, text=red, tiny label] {$\overline\theta$}
       node[draw=red, pos=0.9, text=red, tiny label] {$\theta$}
    (A-leftleg.in-leftthird);
\draw [invisible strand] 
  (A-leftleg.in-rightquarter)
  to [out=up, in=down] 
    node[draw=red, pos=0.14, text=red, tiny label] {$\theta$}
  ([yshift=\toff] A.belt);
\end{tz}
\gap\xmapsto{\ZC(\epsilon)}\gap
{\red \frac 1 {p}}
\begin{tz}
    \node[Cyl, top, bottom scale=1.8, bot=false, anchor=bot] (X) at (0,0) {};
    \node [Cyl, bot, anchor=top, top scale=1.8] (Y) at (X.bot) {};
    \begin{knot}[draft mode=off]
    \strand [red strand] (0.3,0) 
      to [looseness=1.6, out=down, in=down]         
          (-0.3,0)
      to [out=up, in=up, looseness=1.6] (0.3,0);
    \strand [red strand] ([yshift=-\boff] Y.bot) node [below, red] {$i$}
      to [out=up, in=down] (X.bot)
      to [out=up, in=down] ([yshift=\toff] X.top);
  \flipcrossings{2}
\end{knot}
\node[draw=red, text=red, tiny label] at (0,-1.1\cobwidth) {$\theta$};
\end{tz}
\end{equation}
The bottom path has the following effect:
\begin{equation}
\begin{tz}
    \node[Cyl, top, tall] (A) at (0,0) {};
    \draw [red strand] ([yshift=-\boff] A.bot) node [below, red] {$i$} to ([yshift=\toff] A.top) node [above, red] {};
\end{tz}
\,\xmapsto{\ZC(\mu ^\dag)}\,
{\color{red} \delta _{i,0} p}
\begin{tz}
\node [Cap] at (0,0) {};
\node (A) [Cup, top] at (0,2*\cobheight) {};
\node at (0,-0.25cm) [below, red, white] {$i$};
\node at ([yshift=0.2cm] A.center) [above, red] {};
\end{tz}
\,\xmapsto{\ZC(\mu)}\,
{\color{red} \delta _{i,0} p}\,
\begin{tz}
    \node[Cyl, top, tall] (A) at (0,0) {};
\node at ([yshift=-0.25cm] A.bot) [below, red, white] {$i$};
\node at ([yshift=0.2cm] A.top) [above, red] {};
\end{tz}
\end{equation}
Thus the final diagrams in each sequence are equal if and only if the following identity holds in \cat{S}:
\begin{equation}
\label{cutting_eqn}
\begin{tz}[yscale=0.7]
    \begin{knot}[draft mode=off]
    \strand [black strand] (0.3,0) 
      to [looseness=2.3, out=down, in=down]         
          (-0.3,0) 
      to [out=up, in=up, looseness=2.3] (0.3,0);
    \strand [black strand] (0, -0.9) node[below, black] {$i$} to (0, 0.9);
  \flipcrossings{2}
\end{knot}
\end{tz}
\gap = \gap \delta_{i,0 \,} p ^2 \gap=\gap \delta_{i, 0} p ^+ p ^-
\end{equation}
This is a standard identity for modular tensor categories over an algebraically-closed field~\cite[Corollary 3.1.11]{bk01-ltc}, which is equivalent to the $S$-matrix being invertible~\cite{ker99-blk, PaperII}.
\end{proof}

\begin{proposition}
\label{equivalentMstructures}
Given a braided monoidal equivalence $C \simeq C'$ of simple modular tensor categories that preserves the twist, and a square root $p$ of their global dimensions (which must be equal), there is an equivalence of modular structures $Z_C ^p \simeq Z_{C'} ^{p}$ in $\bimod$.
\end{proposition}
\begin{proof}
A braided monoidal equivalence that preserves the twist is exactly an equivalence of $\B$-structures in \twovect. By \autoref{pro:addingrightadjointsforrepinprof}, this gives rise to an equivalence of $\B^\L$-structures in \bimod. The modular presentation $\M$ is a 2\-extension of $\B^\L$, with the following additional generating 2\-morphisms:
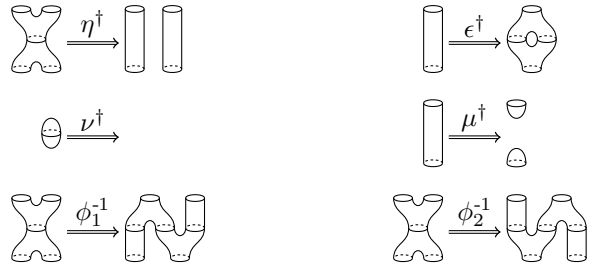
\smallbordisms
\begin{align*}
\begin{tz} 
 \node[Pants, bot] (A) at (0,0) {};
 \node[Copants, top, bot, anchor=belt] at (A.belt) {};
\end{tz}
&\longxdoubleto{\eta ^\dag}
\begin{tz} 
 \node[Cyl, tall, top, bot] (A) at (0,0) {};
 \node[Cyl, tall, top, bot] (B) at (2*\cobwidth, 0) {};
\end{tz}
&
\begin{tz} 
 \node[Cyl, top, bot, tall] (A) at (0,0) {};
\end{tz}
&\longxdoubleto{\epsilon ^\dag}
\begin{tz} 
 \node[Pants, top, bot] (A) at (0,0) {};
 \node[Copants, bot, anchor=leftleg] at (A.leftleg) {};
\end{tz}
\\[5pt]
\begin{tz}
        \node[Cap, bot] (A) at (0,0) {};
        \node[Cup] at (0,0) {};
\end{tz}
&\longxdoubleto{\nu ^\dag}{}
&
\begin{tz}
        \node[Cyl, top, bot, tall] (A) at (0,0) {};
\end{tz}
&\longxdoubleto{\mu ^\dag}
\begin{tz}
        \node[Cup, top] (A) at (0,0) {};
        \node[Cap, bot] (B) at (0,-2*\cobheight) {};
\end{tz}
\\[5pt]
        \begin{tz}
                \node[Copants, top, bot] (A) at (0,0) {};
                \node[Pants, bot, anchor=belt] (B) at (A.belt) {};
        \end{tz}
        &\longxdoubleto{\phiN^\inv}
        \begin{tz}
         \node[Pants, top, bot] (A) at (0,0) {};
         \node[Cyl, bot, anchor=top] (B) at (A.leftleg) {};
         \node[Copants, bot, anchor=leftleg] (C) at (A.rightleg) {};
         \node[Cyl, top, bot, anchor=bottom] (D) at (C.rightleg) {}; 
        \end{tz}
&
        \begin{tz}
                \node[Copants, top, bot] (A) at (0,0) {};
                \node[Pants, bot, anchor=belt] (B) at (A.belt) {};
        \end{tz}
         &
         \longxdoubleto{\phiM^\inv}
        \begin{tz}
         \node[Pants, top, bot] (A) at (0,0) {};
         \node[Cyl, bot, anchor=top] (B) at (A.rightleg) {};
         \node[Copants, bot, anchor=rightleg] (C) at (A.leftleg) {};
         \node[Cyl, top, bot, anchor=bottom] (D) at (C.leftleg) {}; 
         \end{tz}        
\end{align*}
To lift an equivalence of $\B^\L$\-structures $C \simeq C'$ to an equivalence of $\M$\-structures $Z_C ^p \simeq Z _{C'} ^p$, there is no extra data to be specified. One must verify that for each additional 2\-generator listed above, a compatibility equation is satisfied.

We can verify these equations using internal string diagrams. An equivalence $F : C \to C'$ gives rise to a bimodule $F_* : C \proarrow C'$, which is represented using internal string diagrams as follows: 
\normalbordisms
\begin{equation}
\begin{tz}
    \node (A) [Cyl, blue, top, whiteskirt] at (0,0) {};
    \node (B) [Cyl, anchor=top] at (A.bottom) {};
    \draw [red strand] ([yshift=-\boff] B.bottom) node [below] {$A$} to (A-bottom.south);
    \draw [red strand blue back] (A-bottom.south) to (A.top);
    \draw [red strand no back] (A.top) to +(0,0.3) node [above] {$F(A)$};
    \node [right=6pt] at (A.bottom) {$F$};
\end{tz}
\end{equation}
The white and blue parts of the cobordism contain internal strings valued in the ribbon categories $C$ and $C'$ respectively. The bimodule $F_* : C \proarrow C'$ is the interface between the blue and white parts. If a string passing through corresponds below the interface to the object $A$ of $C$, then above the interface it corresponds to $F(A)$ in $C'$.

Using this notation, the generators of the equivalence act as follows, where we are making implicit use of the canonical isomorphisms $F(A) \otimes_{C'} F(B) \simeq F(A \otimes_C B)$ and $F(I_C) \simeq I_{C'}$:
\mediumbordisms
\begin{align}
\label{psi_pants_and_psi_copants}
\begin{tz}
        \node at (0,0) {};
       \node (A) [Cyl, blue, top, whiteskirt, height scale=0.7] at (0,0) {};
       \node (A2) [Cyl, anchor=top, height scale=0.7] at (A.bottom) {};
       \node (B) [Pants, anchor=belt] at (A2.bottom) {};
       \draw [red strand] ([yshift=-\boff] B.leftleg) node [below] {$A$} to +(0,0.3) to [out=up, in=down] ([yshift=-1pt] A2-bottom.in-leftthird) to (A-bottom.in-leftthird);
        \draw [red strand blue back] ([yshift=-1pt] A-bottom.in-leftthird) to (A-top.in-leftthird);
        \draw [red strand no back] (A-top.in-leftthird) to +(0,0.3);
        \draw [red strand] ([yshift=-\boff] B.rightleg) node [below] {$B$} to +(0,0.3) to [out=up, in=down] ([yshift=-1pt] A2-bottom.in-rightthird) to (A-bottom.in-rightthird);
        \draw [red strand blue back] ([yshift=-1pt] A-bottom.in-rightthird) to (A-top.in-rightthird);
        \draw [red strand no back] (A-top.in-rightthird) to +(0,0.3);
        \node at ([yshift=\toff] A.top) [above] {\vphantom{$F($}};
        \fixboundingbox
        \node at ([yshift=\toff] A.top) [above] {$F(A \otimes B)$};
\end{tz}
&\longxmapsto{\psi_{\tikzverytinypants}}{1cm}
\begin{tz}
    \node (B) [Pants, blue, top] at (0,0) {};
    \node (C) [Cyl, anchor=top, blue, whiteskirt, height scale=0.7] at (B.leftleg) {};
    \node (D) [Cyl, anchor=top, blue, whiteskirt, height scale=0.7] at (B.rightleg) {};
    \node (C2) [Cyl, anchor=top, height scale=0.7] at (C.bottom) {};
    \node (D2) [Cyl, anchor=top, height scale=0.7] at (D.bottom) {};
    \draw [red strand] ([yshift=-\boff] C2.bottom) node [below] {$A$} to ([yshift=-1pt] C.bottom);
    \draw [red strand no back] (C.bottom) to  (C.top);
    \draw [red strand no back] (D.bottom) to  (D.top);
    \draw [red strand blue back] ([yshift=-1pt] B.leftleg) to [out=up, in=down] (B-belt.in-leftthird);
    \draw [red strand no back] (B-belt.in-leftthird) to +(0,0.3);
    \draw [red strand] ([yshift=-\boff] D2.bottom) node [below] {$B$} to ([yshift=-1pt] D.bottom);
    \draw [red strand blue back] ([yshift=-1pt] B.rightleg) to [out=up, in=down] (B-belt.in-rightthird);
    \draw [red strand no back] (B-belt.in-rightthird) to +(0,0.3);
    \node at ([yshift=\toff] B.belt) [above] {\vphantom{$F($}};
    \fixboundingbox
    \node at ([yshift=\toff] B.belt) [above] {$F(A \otimes B)$};
\end{tz}
&
\begin{tz}
\node (A) [Copants] at (0,0) {};
\node (B2) [Cyl, anchor=bottom, height scale=0.7] at (A.leftleg) {};
\node (C2) [Cyl, anchor=bottom, height scale=0.7] at (A.rightleg) {};
\node (B) [Cyl, blue, whiteskirt, top, anchor=bottom, height scale=0.7] at (B2.top) {};
\node (C) [Cyl, blue, whiteskirt, top, anchor=bottom, height scale=0.7] at (C2.top) {};
\draw [red strand] ([yshift=-\boff] A-belt.in-leftthird) to +(0,0.3) to [out=up, in=down] ([yshift=-1pt] B2.bottom) to (B2.top);
\draw [red strand blue back] ([yshift=-1pt] B.bottom) to (B.top);
\draw [red strand] ([yshift=-\boff] A-belt.in-rightthird) to +(0,0.3) to [out=up, in=down] ([yshift=-1pt] C2.bottom) to (C2.top);
\draw [red strand blue back] ([yshift=-1pt] C.bottom) to (C.top);
\node at ([yshift=-\boff] A.belt) [below] {$A \otimes B$};
\node [above] at ([yshift=\toff] B.top) {$\vphantom{(F}$};
\fixboundingbox
\draw [red strand no back] (B.top) to +(0,0.3) node [above] {$F(A)\,\,\,\,$};
\draw [red strand no back] (C.top) to +(0,0.3) node [above] {$\,\,\,\,F(B)$};
\end{tz}
&\longxmapsto{\psi_{\tikzverytinycopants}}{1cm}
\begin{tz}
\node (A) [Copants, blue, top] at (0,0) {};
\node (B2) [Cyl, anchor=top, height scale=0.7, blue, whiteskirt] at (A.belt) {};
\node (B) [Cyl, anchor=top, height scale=0.7] at (B2.bottom) {};
\draw [red strand] ([yshift=-\boff] B-bottom.in-leftthird) to ([yshift=-1pt] B2-bottom.in-leftthird);
\draw [red strand blue back] (B2-bottom.in-leftthird) to ([yshift=-1pt] A-belt.in-leftthird) to [out=up, in=down] (A.leftleg);
\draw [red strand] ([yshift=-\boff] B-bottom.in-rightthird) to ([yshift=-1pt] B2-bottom.in-rightthird);
\draw [red strand blue back] (B2-bottom.in-rightthird) to ([yshift=-1pt] A-belt.in-rightthird) to [out=up, in=down] (A.rightleg);
\node at ([yshift=-\boff] B.bottom) [below] {$A \otimes B$};
\node at ([yshift=\toff] A.leftleg) [above] {$\vphantom($};
\fixboundingbox
\draw [red strand no back] (A.leftleg) to +(0,0.3) node [above] {$F(A)\,\,\,\,$};
\draw [red strand no back] (A.rightleg) to +(0,0.3) node [above] {$\,\,\,\,F(B)$};
\end{tz}
\\
\begin{tz}
\node (A) [Cup, height scale=1] at (0,0) {};
\node (C) [Cyl, anchor=bottom, top, height scale=0.7] at (0,0) {};
\node (B) [Cyl, blue, whiteskirt, anchor=bottom, top, height scale=0.7] at (C.top) {};
\end{tz}
&\longxmapsto{\psi_{\tikzverytinycup}}{1cm}
\begin{tz}
\node (A) [Cup, blue, height scale=1, top] at (0,0) {};
\node [Cup, invisible] at (0,-1.4\cobheight) {};
\end{tz}
&
\begin{tz}
\node (A) [Cap, height scale=1] at (0,0) {};
\node (A) [Cap, invisible] at (0,1.4\cobheight) {};
\end{tz}
&\longxmapsto{\psi_{\tikzverytinycap}}{1cm}
\begin{tz}
\node (A) [Cap, blue, height scale=1] at (0,0) {};
\node (B) [Cyl, anchor=top, height scale=0.7, blue, whiteskirt] at (0,0) {};
\node (C) [Cyl, anchor=top, height scale=0.7] at (B.bottom) {};
\end{tz}
\end{align}
Using this technique, the required equations can be shown to hold. For example, for $\epsilon^\dagger$, we must verify the following equation: 
\mediumbordisms
\begin{equation} 
\begin{tz}[scale=3.5, yscale=0.95]
\node (1) at (0,1) {$\begin{tz}
        \node (A) [Cyl, blue, top, height scale=0.7, whiteskirt] at (0,0) {};
        \node (B) [Cyl, anchor=top, height scale=0.7] at (A.bottom) {};
        \selectpart[green]{(A)};
        \selectpart[red]{(B)};
\end{tz}$};
\node (2) at (0,0) {$
\begin{tz}
     \node (A) [Copants] at (0,0) {};
     \node (B) [Pants, anchor=leftleg] at (A.leftleg) {};
     \node (CA) [Cyl, anchor=bottom, height scale=0.7] at (B.belt) {};
     \node (C) [Cyl, anchor=bottom, blue, top, whiteskirt, height scale=0.7] at (CA.top) {};
     \selectpart[green]{(B-leftleg) (B-rightleg) (C-top)};
\end{tz}$};
\node (3) at (1,0) {$
\begin{tz}
     \node (A) [Copants] at (0,0) {};
     \node (AA) [Cyl, anchor=bottom, height scale=0.7] at (A.leftleg) {};
     \node (BB) [Cyl, anchor=bottom, height scale=0.7] at (A.rightleg) {};
     \node (AB) [Cyl, anchor=bottom, blue, whiteskirt, height scale=0.7] at (AA.top) {};
     \node (BB) [Cyl, anchor=bottom, blue, whiteskirt, height scale=0.7] at (BB.top) {};      
     \node (B) [Pants, anchor=leftleg, blue, top] at (AB.top) {};
     \selectpart[green]{(B-leftleg) (B-rightleg) (A-belt)};
\end{tz}$};
\node (4) at (1,1) {$
        \begin{tz}
     \node (A) [Copants, blue] at (0,0) {};
     \node (B) [Pants, anchor=leftleg, blue, top] at (A.leftleg) {};
     \node (C) [Cyl, anchor=top, blue, whiteskirt, height scale=0.7] at (A.belt) {};
     \node (D) [Cyl, anchor=top, height scale=0.7] at (C.bottom) {};

\end{tz}$};

\begin{scope}[double arrow scope]
\draw  (1) to node [auto, swap] {$\color{red} Z ^p _{C} (\epsilon ^\dag)$} (2);
\draw (2) to node [auto, swap] {$\psi _{\tikzverytinypants}$} (3);
\draw  (3) to node [auto, swap] {$\psi _{\tikzverytinycopants}$} (4);
\draw (1) to node [above] {$\color{green} Z ^p _{C'}(\epsilon ^\dag)$} (4);
\end{scope}
\end{tz}
\end{equation}
In terms of internal string diagrams, if we insert the formulae \eqref{psi_pants_and_psi_copants} for $\psi_{\tikzverytinypants}$ and $\psi_{\tikzverytinycopants}$, we need to verify the following:
\begin{equation}
\scalecobordisms{1.7}
\label{need_to_verify}
\frac{1}{p} \sum_{i \in I'} \dim_{C'}(S'_i) \, \,
\begin{tz}[scale=3.5]
     \node (A) [Copants, blue, whiteskirt, belt scale=1.5, right leg scale=0.7] at (0,0) {};
     \node (B) [Pants, blue, anchor=leftleg, top, belt scale=1.5, right leg scale=0.7] at (A.leftleg) {};
     \draw [red strand no back] ([yshift=-3pt] A-belt.in-leftthird) to[out=up, in=down] (B-leftleg.in-leftthird) to[out=up, in=down] ([yshift=3pt] B-belt.in-leftthird);
    \node (1) [tiny label, rectangle, inner sep=2pt, minimum width=20pt, minimum height=11pt] at ([yshift=7pt] $(B-leftleg.west)!0.5!(B-rightleg.east)$) {$\eta'_i$};
    \node (2) [tiny label, rectangle, inner sep=2pt, minimum width=20pt, minimum height=11pt] at ([yshift=-7pt] $(B-leftleg.west)!0.5!(B-rightleg.east)$) {$\epsilon'_i$};
\draw [red strand blue back] (B-leftleg.in-rightthird) to [out=up, in=down] (1.-140);
\draw [red strand blue back] (B.rightleg) to [out=up, in=down] node [font=\tiny, red, right] {$S' _i$} (1.-40);
\draw [red strand blue back] (B-leftleg.in-rightthird) to [out=down, in=up] (2.140);
\draw [red strand blue back] (B.rightleg) to [out=down, in=up] (2.40);
\node [rotate=180] at (B.rightleg) {\arrownode};
\node [rotate=0] at (B-leftleg.in-rightthird) {\arrownode};
\end{tz}
      \quad\stackrel{}{=} \quad 
\frac{1}{p} \sum_{i \in I} \dim_C (S_i) \, \,
\begin{tz}[scale=3.5]
     \node (A) [Copants, blue, whiteskirt, belt scale=1.5, right leg scale=0.7] at (0,0) {};
     \node (B) [Pants, blue, anchor=leftleg, top, belt scale=1.5, right leg scale=0.7] at (A.leftleg) {};
     \draw [red strand no back] ([yshift=-3pt] A-belt.in-leftthird) to[out=up, in=down] (B-leftleg.in-leftthird) to[out=up, in=down] ([yshift=3pt] B-belt.in-leftthird);
    \node (1) [tiny label, rectangle, inner sep=2pt, minimum width=20pt, minimum height=11pt] at ([yshift=7pt] $(B-leftleg.west)!0.5!(B-rightleg.east)$) {$F(\eta_i)$};
    \node (2) [tiny label, rectangle, inner sep=2pt, minimum width=20pt, minimum height=11pt] at ([yshift=-7pt] $(B-leftleg.west)!0.5!(B-rightleg.east)$) {$F(\epsilon_i)$};
\draw [red strand blue back] (B-leftleg.in-rightthird) to [out=up, in=down] (1.-140);
\draw [red strand blue back] (B.rightleg) to [out=up, in=down] node [font=\tiny, red, right] {$F(S_i)$} (1.-40);
\draw [red strand blue back] (B-leftleg.in-rightthird) to [out=down, in=up] (2.140);
\draw [red strand blue back] (B.rightleg) to [out=down, in=up] (2.40);
\node [rotate=180] at (B.rightleg) {\arrownode};
\node [rotate=0] at (B-leftleg.in-rightthird) {\arrownode};
\end{tz}
\end{equation}
The left hand side of \eqref{need_to_verify} involves a sum over loops labeled by the representative simple objects $S_i'$ of $C'$, with $\epsilon'_i : I' \rightarrow S'_i{}^* \otimes S'_i$ and $\eta'_i : S'_i{}^* \otimes S'_i \rightarrow I'$ the cup and cap maps specified by combining the chosen dual and pivotal structures on $C'$. The right hand side involves a sum over loops labeled by $F(S_i)$, where the $S_i$ are the representative simple objects of $C$, and $\eta_i : I \rightarrow S_i^* \otimes S_i$ and $\epsilon_i : S_i^* \otimes S_i \rightarrow I$ the cup and cap maps specified by combining the chosen dual and pivotal structures on $C$.
 Since $F$ is an equivalence, we can take the representative simple objects $S_i'$ to be $F(S_i)$. Also, since the pivotal structure is determined by the braided monoidal and twist structure, a braided monoidal equivalence which preserves the twist will preserve the pivotal structure. Hence $\dim_{C'} (F(X_i)) = \dim_C (X_i)$, so \eqref{need_to_verify} holds. The other five equations can be verified in a similar way. 
\end{proof}

We now prove our main classification theorem. We now do not assume that a TQFT $Z$ satisfies the simplicity condition $Z(\tikztinysphere) \simeq k$. We allow the unit object of an MTC to have simple proper subobjects $s \hookrightarrow I$, choosing  one representative from each of the isomorphism classes,  which we call the \textit{factors}. We say an object $T$ is \textit{preserved} by a factor $s$ if  $s \otimes t \simeq t$, and we write $[s]$ for a maximal set of non-isomorphic simple objects preserved by $s$. For each factor $s$ the \textit{Gauss sums} $p_s^+,p_s^- \in k$ are defined by $p_s^+ = \sum_{i \in [S]} \theta _i ^{} d_i ^2$ and $\smash{p_s^- = \sum _{i \in [S]} \theta _i ^\inv d_i ^2}$, where $d_i$ is the quantum dimension of the simple object $i \in [s]$.  The ratio $\smash{p _s^+/p_s ^-}$ is the \emph{anomaly} of the factor, and the product $p _s^+ p _s^-$ is the \emph{global dimension} of the factor.

\begin{customthm}{\ref{thm:introcsig}}
Linear representations of $\Bordcsig$ are classified by modular tensor categories equipped with a square root of the anomaly in each factor.
\end{customthm}

\begin{proof}
We  establish a bijection between:
\begin{itemize}
\item the collection of modular tensor categories equipped in each factor with a square root of the anomaly, up to braided balanced monoidal equivalence preserving the extra data; and
\item the collection of modular structures in \twovect, up to equivalence of modular structures.
\end{itemize}
By \autoref{Miscsig}, this will establish the desired result. To demonstrate this bijection, we construct functions between the two collections, show that these functions are well-defined on equivalence classes, and show that if we compose these functions in either direction, what is obtained is equivalent to the identity.

By \autoref{directsumofsimples}, a TQFT factors as a direct sum of simple TQFTs. The same is true for MTCs. \autoref{constructtqft} extends to the non-simple case by taking direct sums, as does \autoref{equivalentMstructures}.

For a modular structure $Z$ in \twovect, the linear category $Z(\perspectivecircle)$ is a modular tensor category, as established by \autoref{ribbonismodular}. Conversely, given a modular tensor category $C$ equipped with a root of the anomaly in each factor, we can construct a modular structure in $\bimod$ by \autoref{constructtqft}; under the symmetric monoidal equivalence between \twovect and \bimod when restricted to the fully dualizable objects (see \autoref{thm:equivalenttargets}), this yields a modular structure in \twovect.

Next, we show that these functions between the collections are well-defined on equivalence classes. Given a braided monoidal equivalence $C \simeq C'$ of modular tensor categories, \autoref{equivalentMstructures} constructs an equivalence $Z_C ^p \simeq Z_{C'} ^p$ of linear modular structures, for any consistent extra data $p$. Conversely, if $Z \simeq Z'$ as linear modular structures, then certainly they are also equivalent as $\B^\L$\-structures, since $\M$ 2\-extends $\B^\L$.

Finally, we verify that these functions are inverse bijections on equivalence classes. Let $C$ be a modular tensor category, and let $p$ be a choice of root for the anomaly in each factor. Then by construction, the modular tensor category associated to $Z_C ^p(\perspectivecircle)$ is identical to $C$. Conversely, suppose $Z$ is a linear modular structure; write $C$ for its associated modular tensor category, and $p$ for its associated square root data. Then $Z_C ^p = \isd \simeq Z_*$ by \autoref{def:stringfunctor}, and by the results of \Autoref{sec:internal,sec:modcatmodobj} showing how $\isd$ acts on the generating 2\-morphisms of \M. Since $(-)_*$ is an equivalence on the full subcategory of semisimple categories, this establishes the result.

\end{proof}

\noindent
Also recall \autoref{remark:roots}, according to which it is equivalent to specify a square root of the product $p_s^+ p_s^-$ in each factor, rather than a square root of the anomaly $p_s^+/p_s^-$.

\subsection{Oriented, signature, and $p_1$ field theories}
\label{sec:classificationother}

With the classification of modular structures completed, we can use the results of~\cite{PaperII, PaperIII} to also give classifications of oriented, signature, and $p_1$-structure TQFTs.   
\begin{customthm}{\ref{thm:introoriented}}
Linear representations of $\Bord$ are classified by MTCs for which the anomaly is the identity in each factor.
\end{customthm}
\begin{proof}
The presentation of \Bord given in \autoref{afpresentation} is a 2\-extension of the modular structure presentation by a single relation, which states that the following composite must equal the identity:
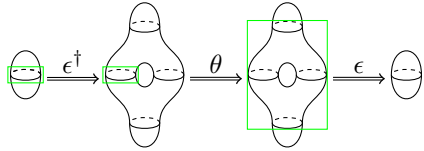
\begin{equation}
\nonumber
\begin{tz}
        \node[Cap] (A) at (0,0) {};
        \node[Cup] (B) at (0,0) {};
        \node [Cobordism Bottom End 3D] (C) at (0,0) {};
        \selectpart[green, inner sep=1pt] {(C)};
\end{tz}
\longxdoubleto{\epsilon^\dagger}
\begin{tz}
        \node[Cap] (A) at (0,0) {};
        \node[Pants, anchor=belt] (B) at (A.center) {};
        \node[Copants, anchor=leftleg] (C) at (B.leftleg) {};
        \node[Cup] (D) at (C.belt) {};
        \selectpart[green, inner sep=1pt] {(B-leftleg)};
\end{tz}
\longxdoubleto{\theta}
\begin{tz}
        \node[Cap] (A) at (0,0) {};
        \node[Pants, anchor=belt] (B) at (A.center) {};
        \node[Copants, anchor=leftleg] (C) at (B.leftleg) {};
        \node[Cup] (D) at (C.belt) {};
        \selectpart[green] {(A-center) (B-leftleg) (B-rightleg) (C-belt)};
\end{tz}
\longxdoubleto{\epsilon}
\begin{tz}
        \node[Cap] (A) at (0,0) {};
        \node[Cup] (B) at (0,0) {};
\end{tz}
\end{equation}
We consider the effect of this composite under the internal string diagram construction $\isd$, under the assumption that the TQFT is simple:
\begin{equation}
\label{eq:afisd}
\mediumbordisms
\begin{tz}
        \node[Cap] (A) at (0,0) {};
        \node[Cup] (B) at (0,0) {};
        \node [Cobordism Bottom End 3D] (C) at (0,0) {};
        \selectpart[green, inner sep=1pt] {(C)};
\end{tz}
\longxmapsto{\isd(\epsilon^\dagger)}{30pt}
{\red \frac 1 {p}}
\begin{tz}
    \node [Pants] (A) at (0,0) {};
        \node[Cap] at (A.belt) {};
    \node[Copants, anchor=leftleg] (B) at (A.leftleg) {};
    \node[Cup] at (B.belt) {};
    \selectpart[green, inner sep=1pt] {(A-leftleg)};
    \draw [red strand] (B.rightleg) 
        to [looseness=1.7, out=down, in=down] (A.leftleg)
        to [out=up, in=up, looseness=1.7] (B.rightleg);
\end{tz}
\longxmapsto{\isd(\theta)}{20pt}
{\red \frac 1 {p}}
\begin{tz}
    \node [Pants] (A) at (0,0) {};
    \node[Cap] (C) at (A.belt) {};
    \node[Copants, anchor=leftleg] (B) at (A.leftleg) {};
    \node[Cup] (D) at (B.belt) {};
    \draw [red strand] (B.rightleg) 
        to [looseness=1.7, out=down, in=down] (A.leftleg)
        to [out=up, in=up, looseness=1.7] node[tiny label, text=red, draw=red, pos=0.2, rotate=-10] {$\theta$} (B.rightleg);
    \selectpart[green] {(C-center) (A-leftleg) (A-rightleg) (B-belt)};
\end{tz}
\longxmapsto{\isd(\epsilon)}{20pt}
{\red \frac {p^+} {p}}
\begin{tz}
        \node[Cap] (A) at (0,0) {};
        \node[Cup] (B) at (0,0) {};
\end{tz}
\end{equation}
It is clear that the composite \eqref{eq:afisd} is the identity if and only if $p = p^+$; by \autoref{lem:psquared}, this is equivalent to $p = p^- = p^+$. So the effect of the anomaly-freeness axiom is to enforce that the anomaly $p^+/p^-$ takes the value 1 in each factor, and to eliminate the extra square-root degree of freedom.
\end{proof}

\begin{customthm}{\ref{thm:introsig}}
Linear representations of $\Bordsig$ are classified by MTCs for which the anomaly is the same in each factor, equipped with a single choice of square root of this anomaly.
\end{customthm}

\begin{proof}
\autoref{thm:sigpresentation} tells us that classifying linear representations of \Bordsig is the same as classifying linear representations of the global modular presentation of degree~1. This presentation extends the modular presentation by a single automorphism $\zeta$ of the empty surface, which must take some nonzero value $a\in k$ in any linear representation. This endomorphism satisfies the following relation:
\begin{equation}
\nonumber
\smallbordisms
\begin{tz}
    \node[Cap] (A) at (0,0) {};
    \node[Cup] (B) at (0,0) {};
    \draw[green] (1*\cobwidth, -\cupheight) rectangle +(\cobwidth, 2*\cupheight);
\end{tz}
\longxdoubleto{\zeta}
\begin{tz}
    \node[Cap] (A) at (0,0) {};
    \node[Cup] at (0,0) {};
\end{tz}
\quad = \quad
\begin{tz}
        \node[Cap] (A) at (0,0) {};
        \node[Cup] (B) at (0,0) {};
        \selectpart[green, inner sep=1pt] {(A-center)};
\end{tz}
\longxdoubleto{\epsilon ^\dag}
\begin{tz}
        \node[Cap] (A) at (0,0) {};
        \node[Pants, anchor=belt] (B) at (A) {};
        \node[Copants, anchor=leftleg] (C) at (B.leftleg) {};
        \node[Cup] (D) at (C.belt) {};
        \selectpart[green, inner sep=1pt] {(B-rightleg)};
\end{tz}
\longxdoubleto{\theta}
\begin{tz}
        \node[Cap] (A) at (0,0) {};
        \node[Pants, anchor=belt] (B) at (A) {};
        \node[Copants, anchor=leftleg] (C) at (B.leftleg) {};
        \node[Cup] (D) at (C.belt) {};
        \selectpart[green] {(B-rightleg) (B-leftleg) (A-center) (C-belt)};
\end{tz}
\longxdoubleto{\epsilon}
\begin{tz}
        \node[Cap] (A) at (0,0) {};
        \node[Cup] (B) at (0,0) {};
\end{tz}
\end{equation}
From calculation~\eqref{eq:afisd} above, we see that $a = p_s^+ / p_s$. Note that the value of $a$ does not depend on the factor. Together with the condition $p_s ^+ p_s ^- = p_s ^2$ inherited from the modular structure classification, we obtain the system of equations $\{p^+ _s / p_s = a, p_s ^+ p_s ^- = p_s ^2\}$ in the unknowns $a$ and $p_s$. This system is equivalent to the system $\{ a^2 = p^+_s /p^-_s, p_s = p _s^+/a\}$; solutions occur exactly when every factor has the same anomaly, and there is a single chosen square root $a$ of this anomaly.
\end{proof}

\begin{customthm}{\ref{thm:introp1}}
Linear representations of $\Bordp$ are classified by MTCs equipped with a sixth root of the anomaly in each factor.
\end{customthm}

\begin{proof}
The $p_1$-presentation given in \autoref{p1presentation} extends the modular presentation by an endomorphism $y'$ of the 2\-sphere, whose cube is the anomaly on the 2\-sphere. Since $Z(\tikztinycup)$ is the linear functor picking out the monoidal unit object, it follows that $Z(\tikztinysphere) \simeq \isd( \tikztinysphere) = \Hom_{Z(\raisebox{-1pt}{$\perspectivecircle$})}(I,I)$ is a vector space with basis given by projectors $\pi_s : I \to I$ onto each subobject $s$ of the unit object. These projectors annihilate one another: $\pi_s \circ \pi_{s'} = \delta_{s,s'} \pi_s$. Let us write $Z(y')$ as a matrix $M$ in this projector basis. We now consider one of the equations that~$y'$ is required to satisfy:
\smallbordisms
\begin{gather*}
\begin{tz}[xscale=1.5, yscale=1.5]
\node (1) at (0,0) {$\begin{tz}
    \node (A) [Cap] at (0,0) {};
    \node (B) [Cup] at (0,0) {};
    \node (C) [Cap] at (0,-1.5\cobheight) {};
    \node [Cup] at (0,-1.5\cobheight) {};
    \node (b1) [Cobordism Bottom End 3D] at (0,0) {};
    \node (b2) [Cobordism Bottom End 3D] at (0,-1.5\cobheight) {};
    \selectpart[green, inner sep=0pt, inner ysep=0.7pt]{(A) (B)};
    \selectpart[red, inner sep=0.7pt]{(B) (C) (b1) (b2)};
\end{tz}$};
\node (2) at (1,0) {$\begin{tz}
    \node [Cap] at (0,0) {};
    \node (B) [Cup] at (0,0) {};
    \node (C) [Cap] at (0,-1.5\cobheight) {};
    \node [Cup] at (0,-1.5\cobheight) {};
    \node (b1) [Cobordism Bottom End 3D, invisible] at (0,0) {};
    \node (b2) [Cobordism Bottom End 3D, invisible] at (0,-1.5\cobheight) {};
    \selectpart[green]{(B) (C) (b1) (b2)};
\end{tz}$};
\node (3) at (1,-1) {$\begin{tz}
    \node [Cap] at (0,0) {};
    \node [Cup] at (0,0) {};
\end{tz}$};
\node (4) at (0,-1) {$\begin{tz}
    \node [Cap] at (0,0) {};
    \node [Cup] at (0,0) {};
\end{tz}$};
\begin{scope}[double arrow scope]
\draw (1) to node [above, green] {$y'$} (2);
\draw (2) to node [right] {$\mu$} (3);
\draw (1) to node [left, red] {$\mu$} (4);
\draw (4) to node [below] {$y'$} (3);
\end{scope}
\end{tz}
\end{gather*}
Given a pair of elements $(\pi_s, \pi_{s'})$ in $\isd (\tikztinysphere)$, the linear map $\mu$ acts as $(\pi_s, \pi_{s'}) \mapsto \pi_s \circ \pi_{s'}$. Suppose we begin in the top-left diagram with the element $(\pi_s, \pi_{s'})$: going clockwise we obtain $M(\pi_s) \circ \pi_{s'}$  in the bottom-right diagram; going anticlockwise, we obtain $M(\pi_s \circ \pi_{s'}) = \delta _{s,s'} M(\pi_s)$. Equating these values, we conclude that $M$ is diagonal in our basis of projectors, so we write $M(\pi_s) = a_s \cdot \pi_s$.

We now consider the other defining equation for $y'$:
\begin{gather*}
\begin{tz}
        \node[Cap] (A) at (0,0) {};
        \node[Cup] (B) at (0,0) {};
\end{tz}
\xdoubleto{y' {}^3}
\begin{tz}
        \node[Cap] (A) at (0,0) {};
                \node[Cup] at (0,0) {};
\end{tz}
\quad = \quad
\begin{tz}
        \node[Cap] (A) at (0,0) {};
        \node[Cup] (B) at (0,0) {};
        \selectpart[green, inner sep=1pt] {(A-center)};
\end{tz}
\longxdoubleto{\epsilon ^\dag}
\begin{tz}
        \node[Cap] (A) at (0,0) {};
        \node[Pants, anchor=belt] (B) at (A) {};
        \node[Copants, anchor=leftleg] (C) at (B.leftleg) {};
        \node[Cup] (D) at (C.belt) {};
        \selectpart[green, inner sep=1pt] {(B-rightleg)};
\end{tz}
\longxdoubleto{\theta}
\begin{tz}
        \node[Cap] (A) at (0,0) {};
        \node[Pants, anchor=belt] (B) at (A) {};
        \node[Copants, anchor=leftleg] (C) at (B.leftleg) {};
        \node[Cup] (D) at (C.belt) {};
        \selectpart[green] {(B-rightleg) (B-leftleg) (A-center) (C-belt)};
\end{tz}
\longxdoubleto{\epsilon}
\begin{tz}
        \node[Cap] (A) at (0,0) {};
        \node[Cup] (B) at (0,0) {};
\end{tz}
\end{gather*}
The left-hand side yields the map $\pi_s \mapsto a ^3 _s \cdot \pi_s$; the right-hand side yields $\pi_s \mapsto (p_s ^- / p_s) \cdot \pi_s$, as derived in calculation~\eqref{eq:afisd}. Together with the condition on $p_s$ inherited from the modular structure classification, we therefore obtain the constraints $\{a_s ^3 = p_s ^- / p ^{}_s, \, p_s ^2 = p_s ^+ p_s ^-$\} for the unknowns $a_s$, $p_s$ in each factor. This is algebraically equivalent to the system $\{a_s ^6 = p_s ^+/p_s^-, \, p_s = p_s ^+ / a_s ^3 \}$, and so the extra structure to be specified is a sixth root of the anomaly $p_s^+/p_s^-$ in each factor; $p_s$ is then determined.
\end{proof}

\section{Dehn twist calculations}
\label{sec:examples}

In this section we work through some examples that illustrate calculations with internal string diagrams, in the case of oriented manifolds, so in particular $p=p^+=p^-$. In \autoref{sec:cloaking} we investigate a `cloaking' property of string diagrams in genus-one surfaces, in \autoref{sec:dehn} we show how Dehn twists about arbitrary curves act on internal string diagrams, in \autoref{mappingactions} we verify a braid relation for Dehn twists on a genus-one surface, in \autoref{sec:lens} we calculate the field theory invariant for a lens space, and in \autoref{sec:torusbundle} we calculate the invariant for torus bundles.

\tikzset{bot=true}
\def\off{5pt}

\subsection{Cloaking}
\label{sec:cloaking}

Note that for a linear representation $Z$ of a modular structure, the following are generally unequal elements of the vector space $\isd(\Sigma) _A ^A$, where $\Sigma = \tikztinypants \circ \tikztinycopants$ and \mbox{$A \in \Ob(\cat{S})$}:
\normalbordisms
\[
\begin{tz}
        \node [Pants, bot, top] (A) at (0,0) {};
    \node [Copants, bot, anchor=leftleg] (B) at (A.leftleg) {};
    \draw[red strand] ([yshift=\toff] A.belt) node[above strand label, red] {$A$}
        to[out=down, in=up] (A.belt)
        to [out=down, in=up, out looseness=1.2] (A.leftleg)
        to [out=down, in=up, in looseness=1.2] (B.belt)
        to[out=down, in=up] ([yshift=-\boff] B.belt) node[below strand label, red] {$A$};
\end{tz}
\quad\neq\quad
\begin{tz}
        \node [Pants, bot, top] (A) at (0,0) {};
    \node [Copants, bot, anchor=leftleg] (B) at (A.leftleg) {};
    \draw[red strand] ([yshift=\toff] A.belt) node[above strand label, red] {$A$}
        to[out=down, in=up] (A.belt)
        to [out=down, in=up, out looseness=1.2] (A.rightleg)
        to [out=down, in=up, in looseness=1.2] (B.belt)
        to[out=down, in=up] ([yshift=-\boff] B.belt) node[below strand label, red] {$A$};
\end{tz}
\]
However,  we have the following `cloaking' phenomenon. Recall that a closed unlabeled strand refers to a sum over the simple objects each weighted by their quantum dimension. This cloaking phenomenon is well-recognized in the literature~\cite{Kirby, walker-notes}.
\begin{lemma} \label{lem:cloaking} For a linear representation $Z$ of the modular presentation, the following are equal elements of $\isd(\Sigma) _A ^A$:
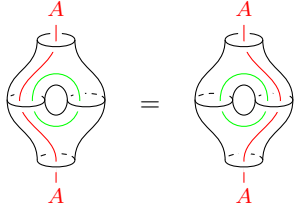
\[
\begin{tz}
        \node [Pants, bot, top] (A) at (0,0) {};
    \node [Copants, bot, anchor=leftleg] (B) at (A.leftleg) {};
    \draw[red strand] ([yshift=\toff] A.belt) node[above strand label, red] {$A$}
        to[out=down, in=up] (A.belt)
        to [out=down, in=up, looseness=0.9] (A-leftleg.in-leftthird)
        to [out=down, in=up, looseness=0.9] (B.belt)
        to[out=down, in=up] ([yshift=-\boff] B.belt) node[below strand label, red] {$A$};
    \draw[green strand] (A-leftleg.in-rightthird)
            to [out=up,in=up,looseness=1.9]
                (A-rightleg.in-leftthird)
            to [out=down,in=down, looseness=1.9] (A-leftleg.in-rightthird);
\end{tz}
\quad = \quad
\begin{tz}
        \node [Pants, bot, top] (A) at (0,0) {};
    \node [Copants, bot, anchor=leftleg] (B) at (A.leftleg) {};
    \draw[red strand] ([yshift=\toff] A.belt) node[above strand label, red] {$A$}
        to[out=down, in=up] (A.belt)
        to [out=down, in=up, looseness=0.9] (A-rightleg.in-rightthird)
        to [out=down, in=up, looseness=0.9] (B.belt)
        to[out=down, in=up] ([yshift=-\boff] B.belt) node[below strand label, red] {$A$};
    \draw[green strand] (A-leftleg.in-rightthird)
            to [out=up,in=up,looseness=1.9]
                (A-rightleg.in-leftthird)
            to [out=down,in=down, looseness=1.9] (A-leftleg.in-rightthird);
\end{tz}
\]
\end{lemma} 

\begin{proof}
It suffices to check this equation when the open strand is a simple object $S_i$. Suppose a product of simple objects $S_i \otimes S_j$ is expressed as a direct sum by families of projections $p: S_i \otimes S_j \to S_k$ and injections $p: S_k \to S_i \otimes S_j$, where for each $i,j,k$ there are $\dim(\Hom(S_i \otimes S_j, S_k))$ elements in the family. Then $S_k ^* \otimes S_i ^{}$ has a direct sum decomposition given by maps $q: S_k ^* \otimes S_i ^{} \to S_j$ and $q: S_j \to S_k^* \otimes S_i ^{}$ as follows:
\begin{align}
\begin{tz}[yscale=0.6, xscale=0.4]
    \node [tiny label, draw=\myblack, text=black] (p) at (0,0) {$q$};
    \draw [black strand] (-1,1.5) node [above] {$k^*$} to (-1,1) node [rotate=0] {\arrownode[black]} to [out=down, in=155] (p.center);
    \draw [black strand] ([yshift=0.5cm] 1,1) node [above] {$i$} to (1,1) node {\arrownode[black, rotate=180]} to [out=down, in=25] (p.center);
    \draw (p.center) to (0,-1) node {\arrownode[black, rotate=0]} to (0,-1.5) node [below] {$j^*$};
\end{tz}
&:=
\begin{tz}[yscale=0.6, xscale=0.4]
    \node [tiny label, draw=\myblack, text=black] (p) at (0,0) {$p$};
    \draw [black strand] (p.center) to [out=45, in=up, out looseness=2.5, in looseness=0.4] (1.5,0) to (1.5,-1) node [rotate=0] {\arrownode[black]} to (1.5,-1.5) node [below] {$j^*$};
    \draw (-1,1.5) node [above] {$i$} to (-1,1) node {\arrownode[black, rotate=180]} to [out=down, in=155] (0,0);
    \draw (0,0) to [out=down, in=down] (-2,0) to (-2,1) node [rotate=0] {\arrownode[black]} to (-2,1.5) node [above] {$k^*$};
\end{tz}
&
\begin{tz}[yscale=0.6, xscale=0.4, yscale=-1]
    \node [tiny label, draw=\myblack, text=black] (p) at (0,0) {$q$};
    \draw[black strand] ([yshift=0.5cm] -1,1) node [below] {$k^*$} to (-1,1) node [rotate=0] {\arrownode[black]} to [out=down, in=155] (p.center);
    \draw[black strand] ([yshift=0.5cm] 1,1) node [below] {$i$} to (1,1) node {\arrownode[black, rotate=180]} to [out=down, in=25] (p.center);
    \draw (p.center) to (0,-1) node {\arrownode[black]} to (0,-1.5) node [above] {$j^*$};
\end{tz}
&:=
\frac{d_j}{d_k}
\begin{tz}[yscale=0.6, xscale=0.4, yscale=-1]
    \node [tiny label, draw=\myblack, text=black] (p) at (0,0) {$p$};
    \draw [black strand] (p.center) to [out=45, in=up, out looseness=2.5, in looseness=0.4] (1.5,0) to (1.5,-1) node [rotate=0] {\arrownode[black]} to (1.5,-1.5) node [above] {$j^*$};
    \draw (-1,1.5) node [below] {$i$} to (-1,1) node {\arrownode[black, rotate=180]} to [out=down, in=155] (0,0);
    \draw (0,0) to [out=down, in=down] (-2,0) to (-2,1) node [rotate=0] {\arrownode[black]} to (-2,1.5) node [below] {$k^*$};
\end{tz}
\end{align}
We now perform the following calculation in $\isd (\Sigma) _A ^A$:
\begin{multline*}
\sum_j d_j \,
\begin{tz}
    \node [Pants, bot, top, belt scale=1.5] (A) at (0,0) {};
    \node [Copants, bot, anchor=leftleg, belt scale=1.5] (B) at (A.leftleg) {};
    \draw [red strand] ([yshift=\toff] A.belt) node[above strand label, red] {$i$}
        to [out=down, in=up] (A.belt)
        to [out=down, in=up, looseness=0.9] 
            (A-leftleg.in-leftthird)
        to [out=down, in=up, looseness=0.9] (B.belt)
        to [out=down, in=up] ([yshift=-\boff] B.belt)
            node[below strand label, red] {$i$};
    \draw[green strand] (A-leftleg.in-rightthird)
        to [out=up,in=up,looseness=1.7]
        (A-rightleg.in-leftthird)
        to [out=down,in=down, looseness=1.7] 
        node[pos=0.25, arrow style, rotate=180, \mygreen] {\arrownode[green, rotate=180]}
        node[below right, font=\tiny, green, yshift=0.2\cobwidth] {$j$}
        (A-leftleg.in-rightthird);
\end{tz}
\gap=\gap
\sum_{k,p,j} d_j \,
\begin{tz}
    \node [Pants, bot, top, belt scale=1.5] (A) at (0,0) {};
    \node [Copants, bot, anchor=leftleg, belt scale=1.5] (B) at (A.leftleg) {};
    \node[tiny label, draw=black, text=black] (p1) at ([xshift=0.2\cobwidth, yshift=0.6\cobwidth] A.leftleg) {$p$};
    \node[tiny label, draw=black, text=black] (p2) at ([xshift=0.2\cobwidth, yshift=-0.6\cobwidth] A.leftleg) {$p$};
    \draw[red strand] ([yshift=\toff] A.belt) node[above strand label, red] {$i$}
        to (A.belt)
        to [out=down, in=145, in looseness=2, out looseness=1] (p1.center);
    \draw[blue strand] (p1.center) to
        node [pos=0.4, arrow style, \myblue, rotate=180] {\arrownode[color=blue, rotate=180]}
        node [yshift=3pt, left, text=blue, font=\tiny] {$k$}
        (p2.center);
    \draw[red strand] (p2.center)
        to [out=-145, in=up, out looseness=2] (B.belt)
        to [out=down, in=up] ([yshift=-\boff] B.belt) node[below strand label, red] {$i$};
    \draw[green strand] (p1.center) to[out=60, in=up, out looseness=2] (A-rightleg.in-leftthird)
        to[out=down, in=-60, in looseness=2] 
        node[pos=0.25, arrow style, rotate=180, \mygreen] {\arrownode[green, rotate=180]}
        node[right, font=\tiny, green, yshift=0pt] {$j$}
        (p2.center);
\end{tz}
\gap = \gap
\sum_{k,p,j} d_j \,
\begin{tz}
    \node [Pants, bot, top, belt scale=1.5] (A) at (0,0) {};
    \node [Copants, bot, anchor=leftleg, belt scale=1.5] (B) at (A.leftleg) {};
    \node[tiny label, draw=\myblack, text=black] (p1) at ([xshift=-0.2\cobwidth, yshift=0.6\cobwidth] A.rightleg) {$p$};
    \node[tiny label, draw=\myblack, text=black] (p2) at ([xshift=-0.2\cobwidth, yshift=-0.6\cobwidth] A.rightleg) {$p$};
    \draw[red strand] ([yshift=\toff] A.belt) node[above strand label, red] {$i$}
        to [out=down, in=up] (A.belt)
        to [out=down, in=130, in looseness=1.2] (p1.center);
    \draw [blue strand] (p1.center)
        to [out=down, in=right, out looseness=2, in looseness=1]
            ([xshift=-5pt] A.center)
        to [out=left, in=up]
        node [pos=0.75, arrow style, rotate=0] {\arrownode[color=blue, rotate=180]}
            node [pos=0.75, left=-2pt, text=blue, font=\tiny] {$k$}
        (A.leftleg)
        to [out=down, in=left] ([xshift=-5pt] B.center)
        to [out=right, in=up, out looseness=1, in looseness=2] (p2.center);
    \draw [red strand] (p2.center)
        to [out=-135, in=up, looseness=0.9] (B.belt)
        to [out=down, in=up] ([yshift=-\boff] B.belt) node[below strand label, red] {$i$};
    \draw[green strand] (p1.center) to[out=60, in=up, out looseness=3]
        node[pos=0.9, arrow style, rotate=180] {\arrownode[color=green, rotate=180]}
        node[pos=0.8, right, font=\tiny, green, xshift=-0.1\cobwidth, yshift=0.1\cobwidth] {$j$}
        (A-rightleg.in-rightthird)
        to[out=down, in=-60, in looseness=3] (p2.center);
\end{tz}
\\ 
= \gap \sum_{k,p,j} d_k \,
\begin{tz}
    \node [Pants, bot, top, belt scale=1.5] (A) at (0,0) {};
    \node [Copants, bot, anchor=leftleg, belt scale=1.5] (B) at (A.leftleg) {};
    \node[tiny label, draw=black, text=black] (p1) at ([xshift=-0.2\cobwidth, yshift=0.6\cobwidth] A.rightleg) {$q$};
    \node[tiny label, draw=black, text=black] (p2) at ([xshift=-0.2\cobwidth, yshift=-0.6\cobwidth] A.rightleg) {$q$};
    \draw[red strand] ([yshift=\toff] A.belt) node[above strand label, red] {$i$}
        to (A.belt)
        to [out=down, in=35, in looseness=2, out looseness=1] (p1.center);
    \draw[green strand] (p1.center) to
        node [pos=0.4, arrow style, \myblue, rotate=180] {\arrownode[color=green]}  
        node [yshift=3pt, right, text=green, font=\tiny] {$j$}
        (p2.center);
    \draw[red strand] (p2.center)
        to [out=-35, in=up, out looseness=2] (B.belt)
        to [out=down, in=up] ([yshift=-\boff] B.belt) node[below strand label, red] {$i$};
    \draw[blue strand] (p1.center) to[out=120, in=up, out looseness=2, in looseness=1.5] (A-leftleg.in-rightthird)
        to[out=down, in=-120, in looseness=2, out looseness=1.5] 
        node[pos=0.25, arrow style, rotate=0] {\arrownode[color=blue]}
        (p2.center);
    \node[font=\tiny, blue, yshift=2pt, text=blue] at (A.leftleg) {$k$};
\end{tz}
\gap = \gap
\sum_k d_k \,
\begin{tz}
        \node [Pants, bot, top, belt scale=1.5] (A) at (0,0) {};
    \node [Copants, bot, anchor=leftleg, belt scale=1.5] (B) at (A.leftleg) {};
    \draw[red strand] ([yshift=\toff] A.belt) node[above strand label, red] {$i$}
        to[out=down, in=up] (A.belt)
        to [out=down, in=up, looseness=0.9] 
                (A-rightleg.in-rightthird)
        to [out=down, in=up, looseness=0.9] (B.belt)
        to[out=down, in=up] ([yshift=-\boff] B.belt) node[below strand label, red] {$i$};
    \draw[blue strand] (A-leftleg.in-rightthird)
            to [out=up,in=up,looseness=1.7]
            (A-rightleg.in-leftthird)
            to [out=down,in=down, looseness=1.7] 
            node[pos=0.75, arrow style, \mybrown, rotate=0] {\arrownode[color=blue]}
            node[left=2pt, font=\tiny, text=blue, yshift=-0.1\cobwidth] {$k$}
            (A-leftleg.in-rightthird);
\end{tz}
\end{multline*}
This completes the proof.
\end{proof}

\subsection{Dehn twists on oriented surfaces}
\label{sec:dehn}

In this section we show that for an oriented TQFT, the action on internal string diagrams of a Dehn twist of degree $n$ about any curve is given by allowing that curve to `descend' below the skin of the cobordism, decorated by a twist of degree $-n$, and multiplied by $1/p$. For example, a right-handed Dehn twist about the curve $c$ indicated below acts as follows on internal string diagrams: 
\begin{calign}
\begin{tz}
\node (P) [Pants, anchor=belt, belt scale=1.8] at (0,0) {};
\node (C) [Copants, anchor=belt, top, belt scale=1.8] at (0,0) {};
\draw [blue, densely dotted] (C-belt.west) to (C.seam middle);
\draw [blue, densely dotted] (P.seam middle) to (C-belt.east);
\begin{pgfonlayer}{selectionbox}
\draw [blue] (P.seam middle) to (C-belt.west);
\draw [blue] (C-belt.east) to node [below right=-3pt] {$c$} (C.seam middle);
\end{pgfonlayer}
\end{tz}
&
\begin{tz}
\node (A) [Pants, anchor=belt, belt scale=1.8] at (0,0) {};
\node (B) [Copants, anchor=belt, top, belt scale=1.8] at (0,0) {};
\node (f) [tiny label] at (0,0.1) {$f$};
\strand [red strand] (f.center) to [out=-125, in=up] (A.leftleg) to +(0,-\off);
\strand [red strand] (f.center) to [out=-55, in=up] (A.rightleg) to +(0,-\off);
\strand [red strand] (f.center) to [out=125, in=down] (B.leftleg) to +(0,\off);
\strand [red strand] (f.center) to [out=55, in=down] (B.rightleg) to +(0,\off);
\end{tz}
\xmapsto{\textstyle \isd(\theta _c)}
{\color{red}\frac 1 {p}}
\begin{tz}
\node (A) [Pants, anchor=belt, belt scale=1.8] at (0,0) {};
\node (B) [Copants, anchor=belt, top, belt scale=1.8] at (0,0) {};
\node (f) [tiny label] at (0,0.1) {$f$};
\begin{knot}
\strand [red strand] (f.center) to [out=-125, in=up] (A.leftleg) to +(0,-\off);
\strand [red strand] (f.center) to [out=-55, in=up] (A.rightleg) to +(0,-\off);
\strand [red strand] (f.center) to [out=125, in=down] (B.leftleg) to +(0,\off);
\strand [red strand] (f.center) to [out=55, in=down] (B.rightleg) to +(0,\off);
\strand [red strand] ([xshift=3pt] B-belt.east) to [out=up, in=left, looseness=0.5] ([yshift=-3pt] B.seam middle) to [out=right, in=up, looseness=0.5] node (t1) [pos=0.7] {} ([xshift=-3pt] B-belt.west) to [out=down, in=right, looseness=0.5] ([yshift=3pt] A.seam middle) to [out=left, in=down, looseness=0.5] ([xshift=3pt] B-belt.east);
\flipcrossings{1,4}
\node [tiny label, draw=red, text=red, rotate=10] at (t1.center) {$\overline\theta$};
\end{knot}
\end{tz}
\end{calign}
The proof method is by first establishing this for a generating family of Dehn twists, and then showing that this property is closed under composition.
\begin{proposition}
\label{basicdehntwists} For a linear representation $Z$ of \Bord, Dehn twists about curves $a$, $b$, and $c$ given below act in the following ways on internal string diagrams:
\normalbordisms
\scalecobordisms{1.2}
\begin{align}
\begin{tz}
\path (-\cobwidth-0.5*\cobgap,0) rectangle +(2*\cobwidth+\cobgap,0);
\node (C) [Cyl, top] at (0,0) {};
\draw [blue, densely dotted] (C.west) to [out=30, in=150] (C.east);
\begin{pgfonlayer}{selectionbox}
\draw [blue] (C.east) to [out=-150, in=-30] node [below, pos=0.5] {$a$} (C.west);
\end{pgfonlayer}
\end{tz}
&&
\begin{tz}
\node (C) [Cyl, top, top scale=1.2, bottom scale=1.2] at (0,0) {};
\begin{knot}
\strand[red strand] ([yshift=\toff] C.top) to ([yshift=-\boff] C.bot);
\end{knot}
\end{tz}
\,&\xmapsto{\textstyle \isd (\theta _a)}\,
{\color{red}\frac 1 {p}}
\begin{tz}
\node (C) [Cyl, top, top scale=1.2, bottom scale=1.2] at (0,0) {};
\begin{knot}
\strand[red strand] ([yshift=\toff] C.top) to ([yshift=-\boff] C.bot);
\strand[red strand] ([xshift=3pt] C.west) to [out=down, in=down, looseness=1.5] ([xshift=-3pt] C.east) to [out=up, in=up, looseness=1.5] node [pos=0] (t1) {} ([xshift=3pt] C.west);
\flipcrossings{2}
\end{knot}
\node [tiny label, draw=red, text=red, rotate=0] at (t1.center) {$\overline\theta$};
\end{tz}
\\
\begin{tz}
\node (P) [Pants, belt scale=1.5] at (0,0) {};
\node (C) [Copants, anchor=belt, top, belt scale=1.5] at (P.belt) {};
\draw [blue, densely dotted] (P.seam middle) to [out=110, in=-110] (C.seam middle);
\begin{pgfonlayer}{selectionbox}
\draw [blue] (P.seam middle) to [out=70, in=-70] node [right, pos=0.7] {$b$} (C.seam middle);
\end{pgfonlayer}
\end{tz}
&&
\begin{tz}
\node (A) [Pants, anchor=belt, belt scale=1.8] at (0,0) {};
\node (B) [Copants, anchor=belt, top, belt scale=1.8] at (0,0) {};
\node (f) [tiny label] at (0,0.1) {$f$};
\strand [red strand] (f.center) to [out=-125, in=up] (A.leftleg) to +(0,-\boff);
\strand [red strand] (f.center) to [out=-55, in=up] (A.rightleg) to +(0,-\boff);
\strand [red strand] (f.center) to [out=125, in=down] (B.leftleg) to +(0,\toff);
\strand [red strand] (f.center) to [out=55, in=down] (B.rightleg) to +(0,\toff);
\end{tz}
&\xmapsto{\textstyle \isd(\theta_b)}
{\color{red}\frac 1 {p}}
\begin{tz}
\node (A) [Pants, anchor=belt, belt scale=1.8] at (0,0) {};
\node (B) [Copants, anchor=belt, top, belt scale=1.8] at (0,0) {};
\node (f) [tiny label] at (0,0.1) {$f$};
\begin{knot}
\strand [red strand] (f.center) to [out=-125, in=up] (A.leftleg) to +(0,-\boff);
\strand [red strand] (f.center) to [out=-55, in=up] (A.rightleg) to +(0,-\boff);
\strand [red strand] (f.center) to [out=125, in=down] (B.leftleg) to +(0,\toff);
\strand [red strand] (f.center) to [out=55, in=down] (B.rightleg) to +(0,\toff);
\strand [red strand] ([xshift=0pt] B-belt.in-leftquarter) to [out=up, in=left, in looseness=0.5] ([yshift=-3pt] B.seam middle) to [out=right, in=up, out looseness=0.5] node (t1) [pos=0.84] {} ([xshift=0pt] B-belt.in-rightquarter) to [out=down, in=right, in looseness=0.5] ([yshift=3pt] A.seam middle) to [out=left, in=down, out looseness=0.5] ([xshift=0pt] B-belt.in-leftquarter);
\flipcrossings{1,3}
\node [tiny label, draw=red, text=red, rotate=10] at (t1.center) {$\overline\theta$};
\end{knot}
\end{tz}
\\
\label{Atwist}
\begin{tz}
\node (P) [Pants, top] at (0,0) {};
\node (C) [Copants, anchor=leftleg] at (P.leftleg) {};
\begin{pgfonlayer}{selectionbox}
\draw [blue] (P.leftleg) to [out=up, in=up, looseness=1.5] (P.rightleg) to [out=down, in=down, looseness=1.5] node [below] {$c$} (P.leftleg);
\end{pgfonlayer}
\end{tz}
&&
\begin{tz}
\node (P) [Pants, top] at (0,0) {};
\node (C) [Copants, anchor=leftleg] at (P.leftleg) {};
\draw [red strand] ([yshift=\toff] P-belt.in-leftthird) to [out=down, in=up] +(0,-\toff) [out=down, in=up] to (P.leftleg) to [out=down, in=up] (C-belt.in-leftthird) to +(0,-\boff);
\draw [red strand] ([yshift=\toff] P-belt.in-rightthird) to [out=down, in=up] +(0,-\toff) [out=down, in=up] to (P.rightleg) to [out=down, in=up] (C-belt.in-rightthird) to +(0,-\boff);
\end{tz}
\,&\xmapsto{\displaystyle \isd (\theta_c)}\,
{\color{red}\frac 1 {p}}
\begin{tz}
\node (P) [Pants, top] at (0,0) {};
\node (C) [Copants, anchor=leftleg] at (P.leftleg) {};
\draw [red strand] (P-leftleg.in-rightthird) to [out=up, in=up, looseness=1.8] node [pos=0.92] (t1) {} (P-rightleg.in-leftthird) to [out=down, in=down, looseness=1.8] (P-leftleg.in-rightthird);
\draw [red strand] ([yshift=\toff] P-belt.in-leftthird) to [out=down, in=up] +(0,-\toff) [out=down, in=up] to (P-leftleg.in-leftthird) to [out=down, in=up] (C-belt.in-leftthird) to +(0,-\boff);
\draw [red strand] ([yshift=\toff] P-belt.in-rightthird) to [out=down, in=up] +(0,-\toff) [out=down, in=up] to (P-rightleg.in-rightthird) to [out=down, in=up] (C-belt.in-rightthird) to +(0,-\boff);
\node [tiny label, draw=red, text=red, rotate=10] at (t1.center) {$\overline\theta$};
\end{tz}
\end{align}
\end{proposition}
\begin{proof}
The action of $\theta_a$ follows from a standard identity in a semisimple ribbon category~\cite[equation (3.1.6)]{bk01-ltc}. The Dehn twist $\theta_b$ is homotopic to $\II$, and its action follows from the action of $\theta_a$ and the definition of $\II$ given in expression~\eqref{defn_of_II}. The mapping cylinder of Dehn twist $\theta _c$ is diffeomorphic to the composite $A$ given in expression~\eqref{defn_of_A}, and so we use the action of $A$ as computed in \autoref{lem:A_action}. Recall that for an oriented theory, $p = p^+ = p^-$.
\end{proof}

\begin{theorem}
For a linear representation $Z$ of \Bord, for an arbitrary curve embedded on a surface with boundary, a Dehn twist about that curve acts on internal string diagrams by embedding the curve below the surface, decorated with an inverse Dehn twist and a scalar factor of $1/p$.
\end{theorem}
\begin{proof}
Given a surface with boundary, choose a family of generating curves forming a basis for its homology, with each curve locally homotopic to one of the standard curves presented in \autoref{basicdehntwists}. Any curve $C$ can be obtained up to homotopy by starting with a generating curve,  and taking its image under successive Dehn twists about generating curves. Write $N_C$ for the least possible number of Dehn twists required to construct $C$ in this manner.

We will prove our statement by induction on $N_C$. Suppose $N_C = 0$; then our curve is homotopic to a generating curve, and the result follows by \autoref{basicdehntwists}. Otherwise, suppose $N_C > 0$. Then there must exist some curve $C'$ and generating curve $a$ such that $N_{C'} < N_C$, and such that $C = \theta_a(C')$. We must show that the theorem holds for the Dehn twist $\theta _C = \theta _{\theta _a (C')}$.

Dehn twists about curves which are themselves in the image of a Dehn twist are controlled by a standard relation~\cite[Fact 3.7]{fm2011-primer}, which specializes to our case as follows:
\begin{equation}
\label{composite_dehn}
\theta _{\theta_a(C')} \circ \theta_a = \theta _a \circ \theta _{C'}
\end{equation}
By induction, the actions of $\isd(\theta _a)$ and $\isd(\theta _{C'})$ on internal string diagrams are given by the `embed below the surface' procedure stated in this theorem. Since $\theta_a$ is invertible, this equation uniquely determines $\theta _C = \theta _{\theta_a(C')}$ as an element of the mapping class group; likewise, since $\isd$ is a functor, the image of this equation under $\isd$ uniquely determines $\isd(\theta_C)$.  To prove the theorem for the Dehn twist $\theta_C$ it therefore suffices to see that equation~\eqref{composite_dehn} holds when each of the four Dehn twists is interpreted to mean, ``apply the `embed below the surface, label with a $\theta^{\inv}$, and divide by $p$' procedure".

If curves $a$ and $C'$ do not intersect, the desired equation holds clearly, so we assume that they do intersect.   In this case, we must check the `embed below the surface' implementation of~\eqref{composite_dehn}, namely:
\begin{equation}
\label{intersecting_Dehn_twists}
{\color{red}\frac{1}{p^2}}
\begin{tz}
\strand[red strand] (-0.75,0) to [out=up, in=-135] (-0.25,0.89) to [out=-55, in=left, in looseness=2, out looseness=0.5] (1,0.25) to [out=right, in=down] (2.25,1) to [out=up, in=right] (1,1.75) to [out=left, in=135, out looseness=1, in looseness=1] (-0.32,1.1) to [out=35, in=down] +(0.15,0.16) node (q) {};
\draw [red strand] (q.center) to [out=up, in=down] node [pos=0.75] (t1) {} (-0.5,2) to (-0.5,2.5);
\strand[red strand] (0,1) to [out=up, in=up] node (t2) [pos=1] {} (2,1) to [out=down, in=down] (0,1);
\node [tiny label, draw=red, text=red, rotate=0] at (t2.center) {$\overline\theta$};
\node at (1,1) {$\cdots$};
\node at ([xshift=4pt] -0.75,-0.1) [anchor=south west, inner sep=0pt] {${\theta _a(C')}$};
\node at (2,1) [anchor=east, inner sep=5pt] {$a$};
\node [tiny label, draw=red, text=red, rotate=0] at (-0.5, 2.1) {$\overline\theta$};
\end{tz}
\quad=\quad
{\color{red}\frac{1}{p^2}}
\begin{tz}
\begin{knot}
\strand[red strand] (0,0) to [out=up, in=down] (0.25,1) to [out=up, in=down] node (t1) [pos=0.7] {} (0,2) to (0,2.5);
\strand[red strand] (0,1) to [out=up, in=up] node (t2) [pos=1] {} (2,1) to [out=down, in=down] (0,1);
\flipcrossings{2}
\end{knot}
\node [tiny label, draw=red, text=red, rotate=0] at (t2.center) {$\overline\theta$};
\node at (1,1) {$\cdots$};
\node at ([xshift=4pt] 0,0) [anchor=south west, inner sep=0pt] {$C'$};
\node at (2,1) [anchor=east, inner sep=5pt] {$a$};
\node [tiny label, draw=red, text=red, rotate=0] at (0, 2.1) {$\overline\theta$};
\end{tz}
\end{equation}
On the left-hand side we embed (below the surface) a loop along the curve $a$, and then along the curve $\theta _a(b)$; on the right-hand side we embed (below the surface) along $b$, and then along $a$.

By applying an inverse Dehn twist to the left-hand cylinder of both sides of the cloaking identity of \autoref{lem:cloaking}, we obtain exactly equation~\eqref{intersecting_Dehn_twists}. This demonstrates that Dehn twists act in the necessary manner for any oriented structure. The ellipses in these expressions emphasize that the proof of this equation is local to any neighbourhood of the strands, and applies regardless of any surrounding string diagram, which may in general link the displayed strands nontrivially.
\end{proof}

\subsection{Mapping class group actions}
\label{mappingactions}
\stringscalecobordisms{1}
In this section we illustrate the use of internal string diagrams by verifying a braid relation in the mapping class group of a twice-punctured genus-one oriented surface (see~\cite[Section~3.5.1]{fm2011-primer}), namely that $\theta_a \theta_b \theta_a = \theta_b \theta_a \theta_b$ where $a$ and $b$ are longitudinal and meridional curves (in fact this relation holds for any two curves intersecting once on any surface). The relation is:
\begin{equation}
\label{mcg}
\smallbordisms
\begin{tz}[xscale=2]
\node (1) at (0,0) {$\begin{tz}
        \node (A) [Pants, top] at (0,0) {};
        \node (B) [Copants, anchor=leftleg] at (A.leftleg) {};
        \selectpart[green, inner sep=1pt]{(A-leftleg)};
        \selectpart[red]{(A-belt) (A-leftleg) (A-rightleg) (B-belt)};
\end{tz}$};
\node (2) at (1,1) {$\begin{tz}
        \node (A) [Pants, top] at (0,0) {};
        \node [Copants, anchor=leftleg] at (A.leftleg) {};
        \selectpart[green, inner sep=1pt]{(A-leftleg)};
\end{tz}$};
\node (3) at (2,1) {$\begin{tz}
        \node (A) [Pants, top] at (0,0) {};
        \node [Copants, anchor=leftleg] at (A.leftleg) {};
\end{tz}$};
\node (4) at (3,0) {$\begin{tz}
        \node (A) [Pants, top] at (0,0) {};
        \node [Copants, anchor=leftleg] at (A.leftleg) {};
\end{tz}$};
\node (5) at (1,-1) {$\begin{tz}
        \node (A) [Pants, top] at (0,0) {};
        \node [Copants, anchor=leftleg] at (A.leftleg) {};
\end{tz}$};
\node (6) at (2,-1) {$\begin{tz}
        \node (A) [Pants, top] at (0,0) {};
        \node [Copants, anchor=leftleg] at (A.leftleg) {};
        \selectpart[green, inner sep=1pt]{(A-leftleg)};
\end{tz}$};
\begin{scope}[double arrow scope]
\draw (1) to node [auto, red] {$A$} (2);
\draw (2) to node [auto] {$\theta$} (3);
\draw (3) to node [auto] {$A$} (4);
\draw (1) to node [auto, swap, green] {$\theta$} (5);
\draw (5) to node [auto, swap] {$A$} (6);
\draw (6) to node [auto, swap] {$\theta$} (4);
\end{scope}
\end{tz}
\end{equation}
Geometrically, the 2\-morphism $A$ represents the Dehn twist illustrated in expression~\eqref{Atwist}, and the 2\-morphism $\theta$ is the Dehn twist about the left circle.

We begin with the following lemma.
\begin{lemma}
\label{lem:crosstwist}
In a modular tensor category, the following identity holds:
\begin{equation}
\begin{tz}
\draw [draw=none, use as bounding box] (-0.5,0) rectangle (1,2);
\begin{knot}
\strand [red strand] (0,0) to (0,2);
\strand [red strand] (0.5,0) to (0.5,2);
\strand [red strand] (-0.5,1) to [out=up, in=up] (1,1) node (t1) {} to [out=down, in=down] (-0.5,1);
\flipcrossings{1,4}
\end{knot}
\node [tiny label, draw=red, text=red, rotate=0] at (t1.center) {$\overline\theta$};
\end{tz}
\quad=\quad
{\color{red} p^-}
\begin{tz}
\draw [draw=none, use as bounding box] (0,0) rectangle (0.5,2);
\begin{knot}
\strand [red strand] (0.5,0) to (0.5,0.25) to [out=up, in=down] (0,0.875) to [out=up, in=down] (0.5,1.5) to node [pos=0.4] (t1) {} (0.5,2);
\strand [red strand] (0,0) to (0,0.25) to [out=up, in=down] (0.5,0.875) to [out=up, in=down] (0,1.5) to node [pos=0.4] (t2) {} (0,2);
\flipcrossings{1}
\end{knot}
\node [tiny label, draw=red, text=red, rotate=0] at (t1.center) {$\theta$};
\node [tiny label, draw=red, text=red, rotate=0] at (t2.center) {$\theta$};
\end{tz}
\end{equation}
\end{lemma}
\begin{proof}
We prove this with the following sequence, in which the first and third equalities use properties of the ribbon calculus, and the middle equality uses equation~\eqref{eq:defppluspminus} defining the constant $p^-$:
\begin{equation}
\begin{tz}
\draw [draw=none, use as bounding box] (-0.5,0) rectangle (1,2);
\begin{knot}
\strand [red strand] (0,0) to (0,2);
\strand [red strand] (0.5,0) to (0.5,2);
\strand [red strand] (-0.5,1) to [out=up, in=up] (1,1) node (t1) {} to [out=down, in=down] (-0.5,1);
\flipcrossings{1,4}
\end{knot}
\node [tiny label, draw=red, text=red, rotate=0] at (t1.center) {$\overline\theta$};
\end{tz}
\quad=\quad
\begin{tz}[xscale=-1, yscale=-1]
\draw [draw=none, use as bounding box] (-1,0) rectangle (1.5,2);
\begin{knot}
\strand [red strand] (0.5,0) to (0.5,2);
\strand [red strand] (-0.5,1) to [out=up, in=up] (1,1) node (t1) {} to [out=down, in=down] (-0.5,1);
\strand [red strand] (-0.5,2) to [out=down, in=up] (-1,1) to [out=down, in=down] (0,0.5) to (0,1.5) to [out=up, in=up] (1.5,1) to [out=down, in=up, out looseness=1.5] (0,0);
\flipcrossings{2,3,4,5}
\end{knot}
\node [tiny label, draw=red, text=red, rotate=0] at (t1.center) {$\overline\theta$};
\end{tz}
\quad=\quad
{\color{red} p^-}
\begin{tz}[xscale=-1, yscale=-1]
\draw [draw=none, use as bounding box] (-0.5,0) rectangle (1,2);
\begin{knot}
\strand [red strand] (0.5,0) to (0.5,0.5) to [out=up, in=down] (0,1) to [out=up, in=down] (0.5,1.5) node (t1) {} to (0.5,2);
\strand [red strand] (-0.5,2) to [out=down, in=up] (-0.5,1) to (-0.5,0.5) to [out=down, in=down, looseness=1.5] (0,0.5) to [out=up, in=down] (0.5,1) to [out=up, in=down] (0,1.5) node (t2) {} to [out=up, in=up, in looseness=2.5] (1.0,1) to [out=down, in=up, out looseness=1.5] (0,0);
\flipcrossings{1,3,4}
\end{knot}
\node [tiny label, draw=red, text=red, rotate=0] at (t1.center) {$\theta$};
\node [tiny label, draw=red, text=red, rotate=0] at (t2.center) {$\theta$};
\end{tz}
\quad=\quad
{\color{red} p^-}
\begin{tz}
\draw [draw=none, use as bounding box] (0,0) rectangle (0.5,2);
\begin{knot}
\strand [red strand] (0.5,0) to (0.5,0.25) to [out=up, in=down] (0,0.875) to [out=up, in=down] (0.5,1.5) to node [pos=0.4] (t1) {} (0.5,2);
\strand [red strand] (0,0) to (0,0.25) to [out=up, in=down] (0.5,0.875) to [out=up, in=down] (0,1.5) to node [pos=0.4] (t2) {} (0,2);
\flipcrossings{1}
\end{knot}
\node [tiny label, draw=red, text=red, rotate=0] at (t1.center) {$\theta$};
\node [tiny label, draw=red, text=red, rotate=0] at (t2.center) {$\theta$};
\end{tz}
\end{equation}
\end{proof}

We now verify the mapping class group relation~\eqref{mcg} as follows. The composite $\theta A \theta$ has the following effect on internal string diagrams:
\def\quad{\hspace{5pt}}
\begin{equation}
\begin{tz}
\node (A) [Pants, top, belt scale=1.5] at (0,0) {};
\node (B) [Copants, anchor=leftleg, belt scale=1.5] at (A.leftleg) {};
\begin{scope}[internal string scope]
\draw ([yshift=\toff] A-belt.in-leftthird)
    to +(0,-\toff)
    to [out=down, in=up] (A.leftleg)
    to [out=down, in=up] (B-belt.in-leftthird)
    to +(0,-\boff);
\draw ([yshift=\toff] A-belt.in-rightthird)
    to +(0,-\toff)
    to [out=down, in=up] (A.rightleg)
    to [out=down, in=up] (B-belt.in-rightthird)
    to +(0,-\boff);
\end{scope}
\selectpart[green, inner sep=1pt] {(A-leftleg)};
\end{tz}
\quad\xmapsto{\isd(\theta)}\quad
\begin{tz}
\node (A) [Pants, top, belt scale=1.5] at (0,0) {};
\node (B) [Copants, anchor=leftleg, belt scale=1.5] at (A.leftleg) {};
\begin{scope}[internal string scope]
\draw ([yshift=\toff] A-belt.in-leftthird)
    to +(0,-\toff)
    to [out=down, in=up] node [pos=0.8] (t1) {} (A.leftleg)
    to [out=down, in=up] (B-belt.in-leftthird)
    to +(0,-\boff);
\draw ([yshift=\toff] A-belt.in-rightthird)
    to +(0,-\toff)
    to [out=down, in=up] (A.rightleg)
    to [out=down, in=up] (B-belt.in-rightthird)
    to +(0,-\boff);
\node [tiny label, draw=red, text=red, rotate=-10] at (t1.center) {$\theta$};
\end{scope}
\end{tz}
\quad\xmapsto{\isd(A)}\quad
{\color{red} \frac 1 p}
\begin{tz}
\node (A) [Pants, top, right leg scale=1.5, belt scale=1.5] at (0,0) {};
\node (B) [Copants, anchor=leftleg, right leg scale=1.5, belt scale=1.5] at (A.leftleg) {};
\begin{knot}
\strand [red strand] ([yshift=\toff] A-belt.in-leftthird)
    to +(0,-\toff)
    to [out=down, in=up] (A-rightleg.in-leftquarter)
    to [out=down, in=up] node [pos=0.8] (t1) {} (B-belt.in-leftthird)
    to +(0,-\boff);
\strand [red strand] ([yshift=\toff] A-belt.in-rightthird)
    to +(0,-\toff)
    to [out=down, in=up] (A-rightleg.in-rightquarter)
    to [out=down, in=up] (B-belt.in-rightthird)
    to +(0,-\boff);
\strand [red strand] (A.leftleg) to [out=down, in=down, looseness=1.3] (A.rightleg) to [out=up, in=up, looseness=1.3] node (t2) [pos=0.85] {} (A.leftleg);
\flipcrossings{2}
\end{knot}
\node [tiny label, draw=red, text=red, rotate=-25] at (t2.center) {$\overline\theta$};
\selectpart[green, inner sep=1pt] {(A-leftleg)};
\end{tz}
\quad\xmapsto{\isd(\theta)}\quad
{\color{red} \frac 1 p}
\begin{tz}
\node (A) [Pants, top, right leg scale=1.5, belt scale=1.5] at (0,0) {};
\node (B) [Copants, anchor=leftleg, right leg scale=1.5, belt scale=1.5] at (A.leftleg) {};
\begin{knot}
\strand [red strand] ([yshift=\toff] A-belt.in-leftthird)
    to +(0,-\toff)
    to [out=down, in=up] (A-rightleg.in-leftquarter)
    to [out=down, in=up] node [pos=0.8] (t1) {} (B-belt.in-leftthird)
    to +(0,-\boff);
\strand [red strand] ([yshift=\toff] A-belt.in-rightthird)
    to +(0,-\toff)
    to [out=down, in=up] (A-rightleg.in-rightquarter)
    to [out=down, in=up] (B-belt.in-rightthird)
    to +(0,-\boff);
\strand [red strand] (A.leftleg) to [out=down, in=down, looseness=1.3] (A.rightleg) to [out=up, in=up, looseness=1.3] node (t2) [pos=0.85] {} (A.leftleg);
\flipcrossings{2}
\end{knot}
\end{tz}
\end{equation}
We now investigate the action of $A \theta A$, where in the final step we apply \autoref{lem:crosstwist}:
\begin{gather}\nonumber
\begin{tz}
\node (A) [Pants, top, belt scale=1.5] at (0,0) {};
\node (B) [Copants, anchor=leftleg, belt scale=1.5] at (A.leftleg) {};
\begin{scope}[internal string scope]
\draw ([yshift=\toff] A-belt.in-leftthird)
    to +(0,-\toff)
    to [out=down, in=up] (A.leftleg)
    to [out=down, in=up] (B-belt.in-leftthird)
    to +(0,-\boff);
\draw ([yshift=\toff] A-belt.in-rightthird)
    to +(0,-\toff)
    to [out=down, in=up] (A.rightleg)
    to [out=down, in=up] (B-belt.in-rightthird)
    to +(0,-\boff);
\end{scope}
\end{tz}
\quad\xmapsto{\isd(A)}\quad
{\color{red} \frac 1 p}
\begin{tz}
\node (A) [Pants, top, belt scale=1.5] at (0,0) {};
\node (B) [Copants, anchor=leftleg, belt scale=1.5] at (A.leftleg) {};
\begin{knot}
\strand [red strand] ([yshift=\toff] A-belt.in-leftthird)
    to +(0,-\toff)
    to [out=down, in=up] (A-rightleg.in-leftquarter)
    to [out=down, in=up] node [pos=0.7] (t1) {} (B-belt.in-leftthird)
    to +(0,-\boff);
\strand [red strand] ([yshift=\toff] A-belt.in-rightthird)
    to +(0,-\toff)
    to [out=down, in=up] (A-rightleg.in-rightquarter)
    to [out=down, in=up] (B-belt.in-rightthird)
    to +(0,-\boff);
\strand [red strand] (A.leftleg) to [out=down, in=down, looseness=1.5] (A.rightleg) to [out=up, in=up, looseness=1.5] node (t2) [pos=0.9] {} (A.leftleg);
\flipcrossings{2}
\end{knot}
\node [tiny label, draw=red, text=red, rotate=-20] at (t1.center) {$\overline\theta$};
\node [tiny label, draw=red, text=red, rotate=-10] at (t2.center) {$\overline\theta$};
\selectpart[green, inner sep=1pt] {(A-leftleg)};
\end{tz}
\quad\xmapsto{\isd(\theta)}\quad
{\color{red} \frac 1 p}
\begin{tz}
\node (A) [Pants, top, belt scale=1.5] at (0,0) {};
\node (B) [Copants, anchor=leftleg, belt scale=1.5] at (A.leftleg) {};
\begin{knot}
\strand [red strand] ([yshift=\toff] A-belt.in-leftthird)
    to +(0,-\toff)
    to [out=down, in=up] (A-rightleg.in-leftquarter)
    to [out=down, in=up] node [pos=0.7] (t1) {} (B-belt.in-leftthird)
    to +(0,-\boff);
\strand [red strand] ([yshift=\toff] A-belt.in-rightthird)
    to +(0,-\toff)
    to [out=down, in=up] (A-rightleg.in-rightquarter)
    to [out=down, in=up] (B-belt.in-rightthird)
    to +(0,-\boff);
\strand [red strand] (A.leftleg) to [out=down, in=down, looseness=1.5] (A.rightleg) to [out=up, in=up, looseness=1.5] node (t2) [pos=0.9] {} (A.leftleg);
\flipcrossings{2}
\end{knot}
\node [tiny label, draw=red, text=red, rotate=-20] at (t1.center) {$\overline\theta$};
\end{tz}
\\\quad\xmapsto{\isd(A)}\quad
{\color{red} \frac 1 {p^2}}
\def\ss{0.417*\cobwidth}
\begin{tz}
\node (A) [Pants, top, right leg scale=2.5, left leg scale=1, belt scale=2] at (0,0) {};
\node (B) [Copants, anchor=leftleg, right leg scale=2.5, left leg scale=1, belt scale=2] at (A.leftleg) {};
\begin{knot}
\strand [red strand] ([yshift=\toff] A-belt.in-leftthird)
    to +(0,-\toff)
    to [out=down, in=up] ([xshift=-3*\ss] A-rightleg.west)
    to [out=down, in=up] node [pos=0.7] (t1) {} (B-belt.in-leftthird)
    to +(0,-\boff);
\strand [red strand] ([yshift=\toff] A-belt.in-rightthird)
    to +(0,-\toff)
    to [out=down, in=up] ([xshift=-1*\ss] A-rightleg.west)
    to [out=down, in=up] (B-belt.in-rightthird)
    to +(0,-\boff);
\strand [red strand] ([xshift=-5*\ss] A-rightleg.west) to [out=down, in=down, looseness=1.5] ([xshift=-2*\ss] A-rightleg.west) to [out=up, in=up, looseness=1.5] node (t2) [pos=0.15] {} ([xshift=-5*\ss] A-rightleg.west);
\strand [red strand] (A.leftleg) to [out=down, in=down, looseness=1.3] ([xshift=-4*\ss] A-rightleg.west) to [out=up, in=up, looseness=1.3] node (t3) [pos=0.9] {} (A.leftleg);
\flipcrossings{2,3}
\end{knot}
\node [tiny label, draw=red, text=red, rotate=-20] at (t1.center) {$\overline\theta$};
\node [tiny label, draw=red, text=red, rotate=30] at (t2.center) {$\overline\theta$};
\node [tiny label, draw=red, text=red, rotate=-10] at (t3.center) {$\overline\theta$};
\end{tz}
\quad=\quad
{\color{red} \frac {p^-}{p^2}}
\begin{tz}
\node (A) [Pants, top, right leg scale=1.5] at (0,0) {};
\node (B) [Copants, anchor=leftleg, right leg scale=1.5] at (A.leftleg) {};
\begin{knot}
\strand [red strand] ([yshift=\toff] A-belt.in-leftthird)
    to +(0,-\toff)
    to [out=down, in=up] (A-rightleg.in-leftquarter)
    to [out=down, in=up] node [pos=0.8] (t1) {} (B-belt.in-leftthird)
    to +(0,-\boff);
\strand [red strand] ([yshift=\toff] A-belt.in-rightthird)
    to +(0,-\toff)
    to [out=down, in=up] (A-rightleg.in-rightquarter)
    to [out=down, in=up] (B-belt.in-rightthird)
    to +(0,-\boff);
\strand [red strand] (A.leftleg) to [out=down, in=down, looseness=1.3] (A.rightleg) to [out=up, in=up, looseness=1.3] node (t2) [pos=0.85] {} (A.leftleg);
\flipcrossings{2}
\end{knot}
\end{tz}
\end{gather}
For an anomaly-free TQFT $p^- = p$, so the mapping class group relation is verified.

\subsection{Lens spaces}
\label{sec:lens}
Given integers $p, q \in \mathbb{Z}$, the lens space $L(p,q)$ (see~\cite{PrasolovSossinsky}) is the 3\-dimensional manifold obtained by gluing two copies of the solid torus $T$ along its boundary $\Sigma$ along a diffeomorphism $h : \Sigma \rightarrow \Sigma$, $h \in SL(2, \mathbb{Z})$ such that 
\[
h (\alpha)= q \alpha + p \beta
\]
where $\alpha = (1,\, 0)^T$ and $\beta = (0, \, 1)^T$ are the standard meridian and longitude in $H_1(\Sigma, \mathbb{Z})$ respectively. In terms of this basis, the matrices representing $\theta$ and $\III$ look as follows (we will use the same symbols for the induced maps on homology):
\begin{align}
\theta &= \left( \begin{array}{cc} 1 & 1 \\ 0 & 1 \end{array} \right)
&
\III &= \left( \begin{array}{cc} 0 & 1 \\ -1 & 0 \end{array} \right)
\end{align}
Note that the diffeomorphism class of $L(p,q)$ only depends on the ratio $p/q$, and does not depend on~$h(\beta)$. 
\begin{lemma}
\label{lens_lemma}
For integers $p, q \in \ZZ$, the matrix 
\begin{equation} \label{lens1}
h = \III \theta^{\minus m_n} \III^{\inv} \cdots \theta^{\minus m_1} \III^{\inv}
\end{equation}
satisfies $h(\alpha) = q \alpha + p \beta$, where $m_i$ arise from a continued fraction expansion of $p/q$ as follows:
\begin{equation}
\label{lens2}
\begin{aligned}\frac{p}{q} = m_n - \cfrac{1}{m_{n-1} - \cfrac{1}{\cdots
- \cfrac{1}{m_2 - \cfrac{1}{m_1}}}}
\end{aligned}
\end{equation}
\end{lemma}
\begin{proof}
This is a reformulation of~\cite[Proposition~2.5]{jef92-csw}.
\end{proof}
It is well known that $L(p,q)$ can be obtained via surgery on a linear chain of linked circles with framing numbers given by the coefficients $m_i$~\cite{PrasolovSossinsky}.  
\begin{lemma} \label{lens_lemma2} The lens space $L(p,q)$ is given by the following composite of generators of $\O$:
\smallbordisms
\begin{equation}
\label{lens_eq}
L(p,q) \gap=\gap
\begin{tz}
\draw[green] (0,0) rectangle (\cobwidth, \cobwidth);
\end{tz}
\longxdoubleto{\nu}
\begin{tz}
    \node[Cap] (A) at (0,0) {};
    \node[Cup] (B) at (0,0) {};
    \selectpart[green, inner sep=1pt] {(A-center)};
\end{tz}
\longxdoubleto{\epsilon ^\dagger}
\begin{tz}
    \node[Cap] (A) at (0,0) {};
    \node[Pants, anchor=belt] (B) at (A) {};
    \node[Copants, anchor=leftleg] (C) at (B.leftleg) {};
    \node[Cup] (D) at (C.belt) {};
    \selectpart[green] {(A-center) (B-leftleg) (B-rightleg) (C-belt)};
\end{tz}
\longxdoubleto{\tilde h}
\begin{tz}
    \node[Cap, bot=false] (A) at (0,0) {};
    \node[Pants, anchor=belt] (B) at (A) {};
    \node[Copants, anchor=leftleg] (C) at (B.leftleg) {};
    \node[Cup] (D) at (C.belt) {};
    \node (E) [Cobordism Bottom End 3D] at (0,0) {};    
    \selectpart[green] {(E) (B-leftleg) (B-rightleg) (C-belt)};
\end{tz}
\longxdoubleto{\epsilon}
\begin{tz}
    \node[Cap] (A) at (0,0) {};
    \node[Cup] (B) at (0,0) {};
    \selectpart[green, inner sep=1pt] {(A-center)};
\end{tz}
\longxdoubleto{\nu ^\dagger}
\begin{tz}
\end{tz}
\end{equation}
Here $\tilde h$ is
\[
\tilde{h} = \theta^{m_n} \, \III \, \theta^{m_{n-1}} \cdots \III \, \theta^{m_1}
\]
where the $m_i$ are given by \eqref{lens2}.
\end{lemma}
\begin{proof} The lens space $L(p,q)$ is formed by gluing two copies of the standard solid torus $T$ to each other along its boundary $\Sigma$ via the explicit diffeomorphism $h$ from \autoref{lens_lemma}. Under the geometric realization functor $F \colon \bicat{F}(\O) \rightarrow \Bord$, the solid torus is represented as
\[
 \widetilde{T} = F(\epsilon^\dagger \circ \nu).
\]
Note that the interior of $\widetilde{T}$ is not the na\"{i}ve interior of the surface embedded in~$\mathbb{R}^3$:
\mediumbordisms
\begin{equation} \label{littlepic2}
\begin{tz}  \node[Cap] (A) at (0,0) {};
        \node[Pants, anchor=belt] (B) at (A) {};
        \node[Copants, anchor=leftleg] (C) at (B.leftleg) {};
        \node[Cup] (D) at (C.belt) {};  \end{tz}
\end{equation}
Rather, this picture is a two-dimensional time slice of $\widetilde{T}$; the interior of $\widetilde{T}$ lies `in the past'. Concretely, this amounts to making the substitutions $\theta \mapsto A^{\inv}$, $\III \mapsto \III^{\inv}$ in formula~\eqref{lens1}. Hence $L(p,q)$ is represented in $\bicat{F}(\O)$ as \eqref{lens_eq}, where
\begin{eqnarray*}
        \tilde{h} & = & \III^{\inv} (A^{m_n} \III) \cdots (A^{m_1} \III) \\
         & = & \III^{\inv} (\III \theta^{m_n}) \cdots (\III \theta^{m_1}) \\
         & = & \theta^{m_n} \III  \theta^{m_{n-1}} \cdots \III \theta^{m_1}.
\end{eqnarray*}
In the second line we have used the braid relation $A \theta A = \theta A \theta$ from Section \ref{mappingactions}, which after inserting the definition $\III = \theta A \theta$ implies that $A = \III \theta \III^{\inv}$.  
\end{proof}

\begin{corollary} The invariant of an oriented TQFT $Z$ on a lens space is the following:
\[
  Z(L(p,q)) = \frac{1}{p^{n+1}} \begin{tz}[scale=1.5]
                \begin{knot}
                        \strand [red strand] (0.8,0) to[out=up, in=left] (1.3, 0.5)  to[out=right, in=up] (1.8,0) to[out=down, in=right] (1.3, -0.5) node [tiny label, draw=red, text=red, inner sep=1pt] {$\theta^{m_{n}}$} to[out=left, in=down] (0.8,0);
                        \strand [red strand] (1.6, 0) to[out=up, in=left] (2.1, 0.5);
                        \strand [red strand] (1.6, 0) to[out=down, in=left] (2.1, -0.5);
                        \node at (2.5, 0) {$\cdots$};
                        \strand [red strand] (2.9, 0.5) to[out=right, in=up] (3.4, 0) to[out=down, in=right] (2.9, -0.5);
                        \strand [red strand] (3.2,0) to[out=up, in=left] (3.7, 0.5)  to[out=right, in=up] (4.2,0) to[out=down, in=right] (3.7, -0.5) node [tiny label, draw=red, text=red, inner sep=1pt] {$\theta^{m_2}$} to[out=left, in=down] (3.2,0);                
                        \strand [red strand] (4.0,0) to[out=up, in=left] (4.5, 0.5)  to[out=right, in=up] (5,0) to[out=down, in=right] (4.5, -0.5) node [tiny label, draw=red, text=red, inner sep=1pt] {$\theta^{m_1}$} to[out=left, in=down] (4,0);
                        \flipcrossings{1, 3, 5};
                \end{knot}
                
  \end{tz}
\]
\end{corollary}
\begin{proof} This follows from \autoref{lens_lemma2} and \autoref{lem:III_action}. 
\end{proof}

\subsection{Torus bundles}
\label{sec:torusbundle}
Let $\tau \in SL(2, \mathbb{Z})$ be a diffeomorphism of the torus up to isotopy, and let $M_\tau$ be the associated torus bundle over $S^1$:
\[
 M_\tau = M \times [0,1] \, / \, (x,0) \sim (\tau(x), 1)
\]
\begin{lemma} For an oriented TQFT $Z$, the invariant of the torus bundle $M_\tau$ is $Z(M_\tau) = \Tr(Z(\tau))$.
\end{lemma}
\begin{proof}
A presentation for $M_\tau$ is $M_\tau = X^\dagger \circ (\tau \sqcup \id) \circ X$, where $X$ is the composite whose image $Z(X)$ under the TQFT $Z$ is given as follows:
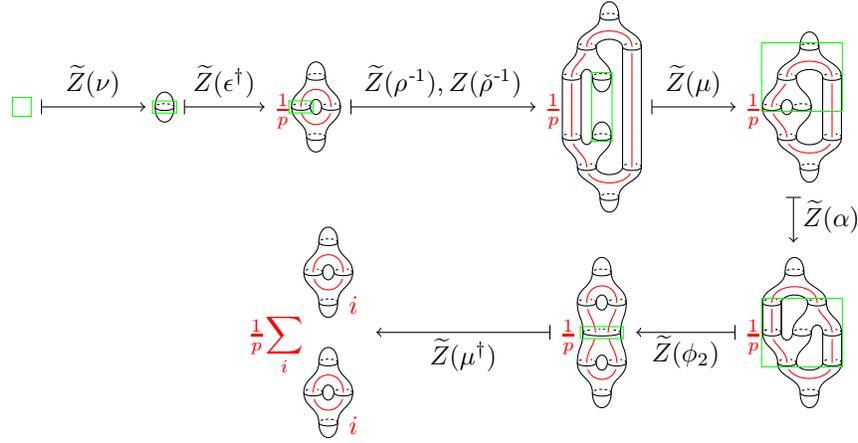
\smallbordisms
\setlength\obscurewidth{2.4pt}
\[
\begin{tz}[xscale=3.8, yscale=3]
\node (1) at (0,0)
{
$\begin{tikzpicture}
        \draw[green] (0,0) rectangle (\cobwidth, \cobwidth);
\end{tikzpicture}$
};
\node (2) at (0.5,0)
{
$\begin{tikzpicture}
        \node[Cap] (A) at (0,0) {};
        \node[Cup] (B) at (0,0) {};
        \selectpart[green, inner sep=1pt] {(A-center)};
\end{tikzpicture}$
};
\node (3) at (1,0)
{
$ {\color{red} \frac{1}{p}}
\begin{tz}
        \node[Cap] (A) at (0,0) {};
        \node[Pants, anchor=belt] (B) at (A) {};
        \node[Copants, anchor=leftleg] (C) at (B.leftleg) {};
        \node[Cup] (D) at (C.belt) {};
        \draw[red strand] (B.leftleg) to[looseness=2, out=up, in=up] (B.rightleg) to[looseness=2, out=down, in=down]  (B.leftleg);
        \selectpart[green, inner sep=1pt] {(B-leftleg)};
\end{tz}$
};

\node (4) at (2,0)
{
$ {\color{red} \frac{1}{p}}
\begin{tz}
        \node[Cap] (A) at (0,0) {};
        \node[Pants, anchor=belt, wide] (B) at (A) {};
        \node[Pants, anchor=belt] (C) at (B.leftleg) {};
        \node[Cup] (D) at (C.rightleg) {};
        \node[Cyl, tall, anchor=top] (E) at (C.leftleg) {};
        \node[Copants, anchor=leftleg] (F) at (E.bottom) {};
        \node[Cap] (G) at (F.rightleg) {};
        \node[Copants, anchor=leftleg, wide] (H) at (F.belt) {};
        \node[Cup] (I) at (H.belt) {};
        \node[Cyl, veryverytall, anchor=top] (J) at (B.rightleg) {};
        \draw[red strand] (B.leftleg) to[out=up, in=up, looseness=1.2] (B.rightleg) to (H.rightleg) to[out=down, in=down, looseness=1.2] (H.leftleg) to[out=up, in=down] (F.leftleg) to (C.leftleg) to[out=up, in=down] (B.leftleg);
        \selectpart[green] {(C-rightleg) (G-center)};
\end{tz}$
};

\node (5) at (2.7,0)
{
$ {\color{red} \frac{1}{p}}
\begin{tz}
        \node[Cap] (A) at (0,0) {};
        \node[Pants, anchor=belt, wide] (B) at (A) {};
        \node[Pants, anchor=belt] (C) at (B.leftleg) {};
        \node[Copants, anchor=leftleg] (D) at (C.leftleg) {};
        \node[Copants, wide, anchor=leftleg] (E) at (D.belt) {};
        \node[Cup] (F) at (E.belt) {};
        \node[Cyl, anchor=top, tall] (G) at (B.rightleg) {};
        \draw[red strand] (B.leftleg) to[looseness=1.2, out=up, in=up] (B.rightleg) to (G.bottom) to[out=down, in=down, looseness=1.2] (E.leftleg) to[out=up, in=down] (D.leftleg) to [out=up, in=down] (B.leftleg);
        \selectpart[green] {(A-center) (C-leftleg) (B-rightleg)};
\end{tz}$
};

\node (6) at (2.7,-1)
{
$ {\color{red} \frac{1}{p}}
\begin{tz}
        \node[Cap] (A) at (0,0) {};
        \node[Pants, anchor=belt, wide] (B) at (A) {};
        \node[Pants, anchor=belt] (C) at (B.rightleg) {};
        \node[Cyl, anchor=top] (D) at (B.leftleg) {};
        \node[Copants, anchor=leftleg] (E) at (D.bottom) {};
        \node[Cyl, anchor=top] (F) at (C.rightleg) {};
        \node[Copants, wide, anchor=rightleg] (G) at (F.bottom) {};
        \node[Cup] (H) at (G.belt) {};
        \node[Cyl, anchor=top] (G) at (B.leftleg) {};
        \draw[red strand] (B.leftleg) to[out=up, in=up, looseness=1.2] (B.rightleg) to[out=down, in=up] (C.rightleg) to (F.bottom) to[out=down, in=down, looseness=1.2] (E.belt) to[out=up, in=down] (E.leftleg) to (B.leftleg);
        \selectpart[green] {(B-leftleg) (B-rightleg) (E-belt) (F-bottom)};
\end{tz}$
};

\node (7) at (2,-1)
{
$ {\color{red} \frac{1}{p}}
\begin{tz}
        \node[Cap] (A) at (0,0) {};
        \node[Pants, anchor=belt] (B) at (A) {};
        \node[Copants, anchor=leftleg, belt scale=2] (C) at (B.leftleg) {};
        \node[Pants, anchor=belt, belt scale=2] (D) at (C.belt) {};
        \node[Copants, anchor=leftleg] (E) at (D.leftleg) {};
        \node[Cup] (F) at (E.belt) {};
        \selectpart[green, inner sep=1pt] {(C-belt)};
        \draw [red strand] (B.leftleg) to [out=up, in=up, looseness=1.9] (B.rightleg) to [out=down, in=up] (C-belt.in-rightthird) to [out=down, in=up] (D.rightleg) to [out=down, in=down, looseness=1.9] (D.leftleg) to [out=up, in=down] (C-belt.in-leftthird) to [out=up, in=down] (B.leftleg);
\end{tz}$
};

\node (8) at (1,-1)
{
$ {\color{red} \frac{1}{p} \displaystyle{\sum_i}}
 \begin{tz}
        \node[Cap] (A) at (0,0) {};
        \node[Pants, anchor=belt] (B) at (A) {};
        \node[Copants, anchor=leftleg] (C) at (B.leftleg) {};
        \node[Cup] (D) at (C.belt) {};
        \node[Cap] (E) at ([yshift=-2*\cobheight] D.center) {};
        \node[Pants, anchor=belt] (F) at (E) {};
        \node[Copants, anchor=leftleg] (G) at (F.leftleg) {};
        \node[Cup] (H) at (G.belt) {};
\draw[red strand] (B.leftleg)
  to[out=up, in=up, looseness=2] (B.rightleg)
  to[out=down, in=down, looseness=2]
  node[pos=0.25, below right, red] {$i$}
    (B.leftleg);
        \draw[red strand] (F.leftleg) to[out=up, in=up, looseness=2] (F.rightleg) to[out=down, in=down, looseness=2]
  node[pos=0.25, below right, red] {$i$}
    (F.leftleg);
\end{tz}$
};

\draw[|->] (1) -- node[above] {$\isd(\nu)$} (2);
\draw[|->] (2) -- node[above] {$\isd(\epsilon^\dagger)$} (3);
\draw[|->] (3) -- node[above] {$\isd(\rho^{\inv}), Z(\check{\rho}^{\inv})$} (4);
\draw[|->] (4) -- node[above] {$\isd(\mu)$} (5);
\draw[|->] (5) -- node[right] {$\isd(\alpha)$} (6);
\draw[|->] (6) -- node[below] {$\isd(\phiM)$} (7);
\draw[|->] (7) -- node[below] {$\isd(\mu^\dagger)$} (8);
\end{tz}
\]
Hence $\isd(M_\tau)$ is equal to the following composite:
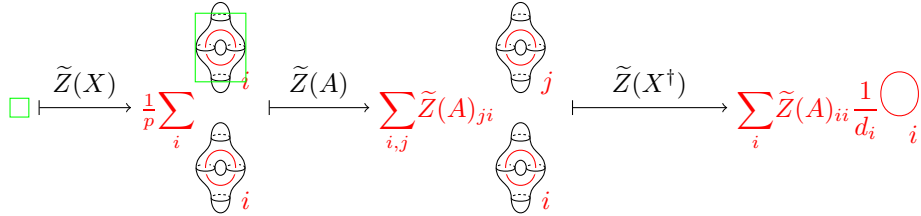
\[
\begin{tz}[xscale=6, yscale=4]
\node (1) at (0,0)
{$\begin{tikzpicture}
        \draw[green] (0,0) rectangle (\cobwidth, \cobwidth);
\end{tikzpicture}$
};
\node (2) at (0.4,0)
{${\color{red} \frac{1}{p} \displaystyle{\sum_i}}
\begin{tz}
        \node[Cap] (A) at (0,0) {};
        \node[Pants, anchor=belt] (B) at (A) {};
        \node[Copants, anchor=leftleg] (C) at (B.leftleg) {};
        \node[Cup] (D) at (C.belt) {};
        
        \node[Cap] (E) at ([yshift=-2*\cobheight] D.center) {};
        \node[Pants, anchor=belt] (F) at (E) {};
        \node[Copants, anchor=leftleg] (G) at (F.leftleg) {};
        \node[Cup] (H) at (G.belt) {};
        \draw[red strand] (B.leftleg) to[out=up, in=up, looseness=2] (B.rightleg) to[out=down, in=down, looseness=2]
  node[pos=0.25, below right, red] {$i$}
    (B.leftleg);
        \draw[red strand] (F.leftleg) to[out=up, in=up, looseness=2] (F.rightleg) to[out=down, in=down, looseness=2]
  node[pos=0.25, below right, red] {$i$}
    (F.leftleg);
        \selectpart[green] {(A-center) (B-leftleg) (B-rightleg)
         (C-belt)};
\end{tz}$};
\node (3) at (1,0)
{${\color{red}\displaystyle{\sum_{i,j}} \isd(A)_{ji}}
\begin{tz}
        \node[Cap] (A) at (0,0) {};
        \node[Pants, anchor=belt] (B) at (A) {};
        \node[Copants, anchor=leftleg] (C) at (B.leftleg) {};
        \node[Cup] (D) at (C.belt) {};
        \node[Cap] (E) at ([yshift=-2*\cobheight] D.center) {};
        \node[Pants, anchor=belt] (F) at (E) {};
        \node[Copants, anchor=leftleg] (G) at (F.leftleg) {};
        \node[Cup] (H) at (G.belt) {};
        \draw[red strand] (B.leftleg) to[out=up, in=up, looseness=2] (B.rightleg) to[out=down, in=down, looseness=2]
  node[pos=0.25, below right, red] {$j$}
  (B.leftleg);
        \draw[red strand] (F.leftleg) to[out=up, in=up, looseness=2] (F.rightleg) to[out=down, in=down, looseness=2]
  node[pos=0.25, below right, red] {$i$}
    (F.leftleg);
\end{tz}$};
\node (4) at (1.8,0)
{$ {\color{red} \displaystyle{\sum_{i}} \isd(A)_{ii} \frac{1}{d_i}}
\begin{tz}
        \draw[red] (0,0) to[out=up, in=up, looseness=2] ([xshift=2*\cobwidth] 0,0) to[out=down, in=down, looseness=2]
  node[below right]{$i$}
    (0,0);
\end{tz}$};
\draw[|->] (1) -- node[above] {$\isd(X)$} (2);
\draw[|->] (2) -- node[above] {$\isd(A)$} (3);
\draw[|->] (3.east) -- node[above] {$\isd(X^\dagger)$} (4);
\end{tz}
\]
Note that in the computation of $\isd(X^\dagger)$, the sum over $i$ and $j$ picks up a $\delta_{ij}$ term when $\isd(\mu^\dagger)$ is applied just after the $\isd(\alpha^{-1})$ step. The final result is precisely $\Tr(Z(\tau))$.
\end{proof}

\appendix

\section{A bestiary of 2--vector spaces}
\def\R{\mathrm{R}}
\def\hatboxtimes{\mathop{\widehat{\boxtimes}}}
\label{sec:bestiary}

There are numerous symmetric monoidal 2\-categories which can serve as potential targets for extended topological field theories. Let $\vect$ denote the symmetric monoidal category of $k$\-vector spaces and linear maps, with the symmetric monoidal structure given by the tensor product. When we say in this section that a structure is linear, we mean with respect to the field $k$, which is taken to be algebraically closed. We will not in general require our vector spaces to be finite-dimensional. We will consider the following target 2\-categories, each of which has objects given by some sort of `2\-vector space': 
\begin{itemize}
        \item $\kvtwovect$, Kapranov and Voevodsky's 2\-category of 2\_vector spaces~\cite{kv94-2categories};
        \item $\ALG$, the 2\-category of algebras, bimodules, and maps, and certain restrictions such as $\alg$ and $\ALG^\L$ defined below;
        \item $\lincat$, the 2\-category of finite abelian linear categories, right exact functors, and natural transformations;
        \item $\twovect$, the 2\-category of Cauchy-complete (i.e. additive and idempotent complete) $\vect$-enriched categories, linear functors, and natural transformations;
        \item $\bimod$, the 2\-category of all $\vect$-enriched categories, bimodules, and maps of bimodules. 
\end{itemize}
They are standard in the literature, and we do not give full definitions of them here; for more details, see~\cite{kv94-2categories, ds97-mbh, t98-sskl, borceux1}.
A notable definition of 2\_vector space which is not among this family is given by Baez and Crans~\cite{bc04-hda4}, with an `internal' rather than an `enriched' style; it would be interesting to consider the classification problem for that target.

For each case listed above, TQFTs valued in that target all factor through the simplest in the list, Kapranov and Voevodsky's 2\-category $\kvtwovect$ of 2\_vector spaces. We will prove this by considering dualizability properties of these 2\-categories. This builds on observations of Tillmann~\cite{t98-sskl}. Dualizability is a fundamental concept in the study and classification of TQFTs, especially extended TQFTs. The reason is two-fold. First, the bordism $n$\-category demonstrates a high degree of dualizablity. In all dimensions and for all topological structures it is \emph{fully-dualizable}. Second, duality is \emph{hereditary}, meaning that the images under any functor of morphisms and objects which are dualizable are themselves dualizable. This means that TQFTs with values in a given $n$\-category will always factor through the fully-dualizable sub-$n$\-category. Often this subcategory is vastly simpler.

In fact, in the cases considered here our bordism categories satisfy an even stronger \textit{ambidextrous} duality. Our analysis would be somewhat simpler if we restricted ourselves to just the ambidextrous case alone. However instead we will focus on  full-dualizablity, with an eye to the future and greater generality.

\subsection{A result on Cauchy completion}
\label{sec:cauchycompletionresult}

We give here the proof of the result on Cauchy-complete linear categories that was stated in \autoref{sec:linearcategories}.

\begin{customprop}{\ref{pro:CauchyComplete}}
For a linear category $\cat C$, the following are equivalent:
\begin{enumerate}
\item $\cat C$ is Cauchy-complete;
\item $\cat C$ admits finite direct sums and all idempotents split;
\label{item:idempotentssplit}
\item every bimodule $P: k \proarrow \cat C$ that is a left adjoint, is representable;
\item for all small $\vect$-enriched categories $\cat B$, every bimodule \mbox{$S: \cat B \proarrow \cat C$} that is a left adjoint, is representable.
\end{enumerate} 
\end{customprop}

\begin{proof}
The equivalence of (1) and (2) is a result of Street~\cite{street-absolutecolimits}.
Clearly (4) $\Rightarrow$ (3).  (Note that in item (3), the field $k$ is viewed as a one object category with hom space $k$.)  The collection of bimodules $P: k \proarrow \cat C$ that are left-adjoints is closed under forming (a) direct sums of bimodules, and (b) retracts. It follows that if $\cat C$ is Cauchy-complete (so that all such $P$ are represented by objects in $\cat C$) then $\cat C$ necessarily admits direct sums and is idempotent-complete, hence $(3) \Rightarrow (2)$.
                
The proofs in the other direction are adapted from~\cite[Theorem~7.9.3]{borceux1}, which treats the $\set$-enriched case. We will first show that $(2) \Rightarrow (3)$.
        Let $P: k \proarrow \cat C$ be a bimodule that is left-adjoint to the bimodule $Q: \cat C \proarrow k$. These bimodules are nothing more than functors of the following types:
        \begin{align*}
                P &: \cat C^\op \to \vect \\
                Q &: \cat C \to \vect
        \end{align*}
The composite $P \circ Q$ is the bimodule given by $(P \circ Q)(c, c') = P(c) \otimes Q(c')$. The composite $Q \circ P$ is a linear functor from $k$ to $\vect$, and thus picks out the single vector space
\begin{equation*}
        (Q \circ P)(*) = \bigoplus_{c \in C} Q(c) \otimes P(c) {\,/\,} {\sim} .
\end{equation*} 

The counit is a natural transformation with components given by
\begin{equation*}
        \varepsilon_{c,c'}: P(c) \otimes Q(c') \to \Hom_C(c,c').
\end{equation*}
The unit is specified by a map of vector spaces
\begin{equation*}
        \eta: k \to (Q \circ P)(*)
\end{equation*} 
and thus specifies an element of this vector space, which we will denote as 
\begin{equation*}
        \eta_* = \sum_{i=1}^n q_i \otimes p_i 
\end{equation*}  
where the $c_i$ are objects in $C$ and $q_i \in Q(c_i)$, $p_i \in P(c_i)$. 

Now suppose that $\cat C$ is Cauchy complete. Since it admits direct sums, we may form the object $x = \bigoplus_{i=1}^n c_i$. We have an endomorphism $e: x \to x$, which in components is given by
\begin{equation*}
        e_{ij} = \varepsilon_{c_i, c_j}(p_i, q_j) \in \Hom_C(c_i, c_j).
\end{equation*}
The triangle equations, which express the fact that $\eta$ and $\varepsilon$ are the unit and counit of an adjunction, show first that $e$ is an idempotent, and hence, since idempotents split in $\cat C$, that it is determined by $y$, a direct summand of $x$, and secondly that $P$ is representable by the object $y$.  

Finally we will show that (2) $\Rightarrow$ (4). Each object $b \in \cat B$ corresponds to a functor $b: k \to \cat B$, and hence to a left-adjoint bimodule $b_* : k \proarrow \cat B$. The composite $S \circ b_*: k \proarrow \cat C$ is a left-adjoint bimodule and hence, by assumption, it is represented by an object $F(b)$. In other words we have constructed objects $F(b) \in \cat C$ for each $b \in \cat B$, and isomorphisms 
\begin{equation*}
        S(c, b) \cong \Hom_{\cat C}(c, F(b))
\end{equation*}
which are natural in $c$. Moreover for each morphism $b \to b'$ in $\cat(B)$ we have a linear map
\begin{equation*}
        \Hom_{\cat C}(c, F(b)) \cong S(c, b) \to S(c, b') \cong \Hom_{\cat C}(c, F(b')),
\end{equation*} 
which is again natural in $c$. It follows from the full-faithfulness of the (enriched) Yoneda embedding that this linear map is induced by a map $F(b) \to F(b')$ in $\cat C$, and hence $F$ is the desired functor.        
\end{proof}

\subsection{2-categories of 2-vector spaces}
\label{sec:linearrep}

In \autoref{sec:Cauchycomplete} we considered the symmetric monoidal 2\-categories $\twovect$ and $\bimod$ in detail. We will not repeat that here, but refer the reader to the earlier section. 

The 2\-category of algebras, bimodules, and bimodule maps is one of the quintessential examples of a symmetric monoidal 2\-category.
\begin{defn}
        The symmetric monoidal 2\-category $\ALG$ is defined as follows:
        \begin{itemize}
        \item objects are $k$-algebras;
        \item 1\-morphisms from $A$ to $B$ are $A$-$B$-bimodules;
        \item 2\-morphisms are bimodule maps;
        \item the monoidal structure is the tensor product $\otimes_k$ over $k$.
        \end{itemize}
        The composition of bimodules ${}_A M_B$, ${}_B N_C$ is given by relative tensor product of bimodules $M \otimes_B N$. 
        
        We will also consider the sub-2\-categories $\alg$ and $\ALG^\L$. The former is the sub-2\-category consisting of finite-dimensional $k$-algebras and finite-dimensional bimodules. The latter is the sub-2\-category of all algebras, but with only the left-adjoint bimodules as morphisms. These are precisely those $A$-$B$-bimodules $M$ which are finitely generated and projective when considered as a right $B$-modules~\cite[Lemma~3.60]{CSPthesis}.
\end{defn}

There is a fully-faithful inclusion functor $\ALG \to \bimod$ which sends an algebra $A$ to a one\-object  $\vect$-enriched category with morphism space $A$.  Via this functor we can regard $\ALG$, and hence $\alg$ and $\ALG^\L$, as sub-2\-categories of $\bimod$. In particular an $A$-$B$-bimodule $M$ is the same thing as a bimodule $M: A \proarrow B$.  If $N$ is a $B$-$C$-bimodule, the composite bimodule $(N \circ M): A \proarrow C$ is given by the tensor product of bimodules $M \otimes_B N$. Note that by convention this translation is order-reversing.

We may also specialize our 2\-category to obtain Kapranov and Voevodsky's 2\-category. 
\begin{defn}[See~\cite{kv94-2categories}]
        The symmetric monoidal 2\-category $\kvtwovect$ is defined as follows:
        \begin{itemize}
        \item objects are the natural numbers $\NN$;
        \item 1\-morphisms are matrices of finite-dimensional vector spaces;
        \item 2\-morphisms are matrices of linear maps;
        \item the monoidal structure is induced by the addition of natural numbers and the tensor product of matrices. 
        \end{itemize}
        Composition is given by using the tensor product and direct sum of vector spaces in place of the usual multiplication and addition in the formula for matrix multiplication. 
\end{defn}

The 2\-category $\kvtwovect$ is equivalent to the full sub-2\-category of $\alg$ spanned by the objects of the form $\oplus_n k$. That is, the object $n \in \kvtwovect$ corresponds to the algebra given by the direct sum of $n$ copies of the ground field $k$. Since $k$ is algebraically closed, any finite-dimensional semisimple algebra over $k$ is a finite sum of finite-dimensional matrix algebras over $k$. As $M_n(k)$ is Morita equivalent to $k$, we conclude that $\kvtwovect$ is also equivalent to the full sub-2\-category of $\alg$ consisting of the finite-dimensional semi-simple $k$-algebras.

\begin{defn}
A linear category is a $\vect$-enriched category which is also additive and such that the abelian group structures on hom sets induced by the additive structure and induced by the $\vect$-enrichment agree. A linear functor is a functor which is simultaneously additive and enriched. 
\end{defn}

\begin{defn}[See~\cite{Douglas:2014aa, MR1106898}]
The symmetric monoidal 2\-category $\lincat$ is defined as follows:
        \begin{itemize}
        \item objects are \emph{finite} abelian linear categories; that is, they are abelian linear categories $\cat C$ which further satisfy:
        \begin{itemize}
                \item $\cat C$ has finite-dimensional hom spaces for morphisms;
                \item every object of $\cat C$ has finite length\footnote{See \autoref{def:finlength}};
                \item $\cat C$ has enough projectives;
                \item $\cat C$ has finitely many isomorphism classes of simple objects. 
        \end{itemize}
        \item 1\-morphisms are right-exact linear functors;
        \item 2\-morphisms are natural transformations;
        \item the monoidal structure is the Deligne tensor product $\boxtimes$.
        \end{itemize}
\end{defn}

\subsection{Comparison functors}

There are several symmetric monoidal functors comparing these various symmetric monoidal 2\-categories. These are summarized in Figure~\ref{fig:diagramof2Vects} below. Besides the functor $\Phi$, which we have already described in detail, and the obvious inclusion functors, we have the following:
\begin{itemize}
        \item $\mathit{Rep}: \alg \to \lincat$, which associates to an algebra $A$ the category of finite-dimensional right $A$-modules and to a bimodule ${}_A M_B$ the right exact functor $(-) \otimes_A M$. This is known to be an equivalence of symmetric monoidal 2\-categories (see~\cite{Douglas:2014aa}, \cite[Lemma~3.58]{CSPthesis}.)
        \item $(\hat{-}): \ALG^\L \to \twovect$, which associates to an algebra $A$ its category of finitely generated projective modules $\hat{A}$. Tensoring with an arbitrary bimodule ${}_AM_B$ will not generally preserve finitely generated projective modules, and so we do not get a functor defined on all of $\ALG$. To preserve projectives, $M$ must be finitely generated and projective as a $B$-module. This is precisely the condition that $M$ is a left-adjoint~\cite[Lemma~3.60]{CSPthesis}. 
        \item $(-)^\fgp: \lincat \to \bimod$ which sends a finite linear category to its subcategory of compact projective objects. A general right exact functor will not preserve projective objects and hence does not restrict to a functor between these subcategories, however it does give rise to a bimodule.  
\end{itemize}

\noindent
Each of these functors is symmetric monoidal. That $\mathit{Rep}$ is symmetric monoidal follows from standard properties of the Deligne tensor product~\cite{MR1106898} (see also~\cite{Douglas:2014aa}). The key to showing that $(\hat{-})$ is symmetric monoidal is the observation that $P \otimes Q$ is a projective $A \otimes B$-module whenever $P$ and $Q$ are projective.  Finally that $\Phi$ is symmetric monoidal essentially follows from Proposition~\ref{pro:Cauchycompletionisequiv}, which shows in particular that $C \otimes D$ is equivalent in $\bimod$ to $C \hat \otimes D$.

\begin{defn}
        A functor of symmetric monoidal 2\-categories $F: \bicat{B} \to \bicat{C}$ is called:
        \begin{itemize}
                \item \emph{essentially surjective} if every object of $\bicat{C}$ is equivalent to one in the image of~$F$;
                \item \emph{locally fully-faithful} if it is fully-faithful on 2\-morphisms, i.e. for all $x,y\in\bicat{B}$, $F:\hom_\bicat{B}(x,y) \to \hom_\bicat{C}(F(x), F(y))$ is fully-faithful;
                \item \emph{fully-faithful} if it is locally fully-faithful and essentially full on 1\-morphisms, i.e. for all $x,y\in\bicat{B}$, $F:\hom_\bicat{B}(x,y) \to \hom_\bicat{C}(Fx, Fy)$ is an equivalence of categories. 
        \end{itemize}
\end{defn} 

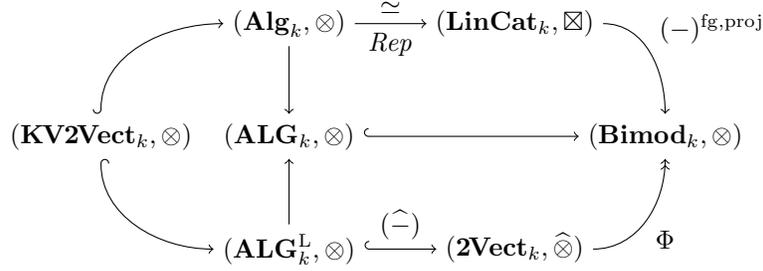
\begin{figure}[htbp]
        \begin{center}
        \begin{tikzpicture}
                        \node (KV) at (-0.5, 0 ) {$(\kvtwovect, \otimes)$};
                        \node (alg) at (2, 1.5) {$(\alg, \otimes)$};
                        \node (ALG) at (2, 0) {$(\ALG, \otimes)$};
                        \node (ALGe) at (2, -1.5) {$(\ALG^\L, \otimes)$};
                        \node (LINCAT) at (5, 1.5) {$(\lincat, \boxtimes)$};
                        \node (TWOVECT) at (5, -1.5) {$(\twovect, \hat \otimes)$};
                        \node (prof) at (7, 0) {$(\bimod, \otimes)$};
                        
                        \draw [right hook->] (KV) to [out = 90, in = 180] node [above left] {} (alg);
                        \draw [right hook->] (KV) to [out = -90, in = 180] node [above] {} (ALGe);
                        \draw [->] (alg) -- node [above] {} (ALG);
                        \draw [->] (ALGe) -- node [above] {} (ALG);
                        \draw [right hook->] (ALG) -- node [above] {} (prof);
                        \draw [->] (alg) -- node [above] {$\simeq$} node [below] {$\mathit{Rep}$} (LINCAT);
                        \draw [right hook->] (ALGe) -- node [above] {$(\hat{-})$} (TWOVECT);
                        \draw [->] (LINCAT) to [out= 0, in = 90] node [above right] {$(-)^\fgp$} (prof);
                        \draw [->>] (TWOVECT) to [out = 0, in = -90] node [below right] {$\Phi$} (prof);
        \end{tikzpicture}
        \end{center}
        \caption{The relation between various notions of 2\_vector space. All functors are locally fully-faithful. The fully-faithful functors are indicated by $\hookrightarrow$, and the essentially surjective functors are indicated by $\twoheadrightarrow$.}
        \label{fig:diagramof2Vects}
\end{figure}

All of the functors in Figure~\ref{fig:diagramof2Vects} are locally fully-faithful (i.e. fully-faithful on 2\-morphisms). Those that are fully-faithful have been indicated with a hook-arrow `$\hookrightarrow$'. The only fully-faithful functor which is not obviously so is $(\hat{-})$. We record this as a lemma.
\begin{lemma}\label{lem:Alttowovectfullyfaithful}
The functor $(\hat{-}): \ALG^\L \to \twovect$ is fully-faithful.
\end{lemma}

\begin{proof}
        Let $A$ and $B$ be algebras and let $\hat{A}$ and $\hat{B}$ be their Cauchy completions, i.e. their categories of finitely generated projective modules. Let $F: \hat{A} \to \hat{B}$ be a linear functor. Then $F(A_A) = M$ is an $A$-$B$ bimodule which is finitely generated and projective as a right $B$-module. We will see that $F(-) \cong (-) \otimes_A M_B$ as linear functors. This is clear on the subcategory of finite rank free $A$-modules since $F$ is additive:
        \begin{equation*}
                F( \oplus_n A) \cong \oplus_n F(A) \cong \oplus_n M_B \cong (\oplus_n A) \otimes_A M_B.
        \end{equation*}
        But every finitely generated projective module is a retract of a finite rank free module, and so it follows that $F(-) \cong (-) \otimes_A M_B$. Thus $(\hat{-})$ is essentially full on 1\-morphisms. A similar argument shows that natural transformations $(-) \otimes_A M_B \to (-) \otimes_A N_B$ are in bijection with $A$-$B$ bimodule maps ${}_AM_B \to {}_A N_B$. 
\end{proof}

Proposition~\ref{pro:Cauchycompletionisequiv} implies that $\Phi$ is essentially surjective, which we indicate with a double-headed arrow $\twoheadrightarrow$. Furthermore, the diagram in Figure~\ref{fig:diagramof2Vects} commutes up to canonical natural isomorphism. The only non-obvious commutativity occurs in the upper and lower right-hand squares, which both follow from \autoref{pro:Cauchycompletionisequiv}. In this case it is the fact that for an algebra $A$ the bimodule $\Phi(i): A \proarrow \hat{A} \simeq (A)^\fgp$ is an equivalence in $\bimod$.    

\begin{remark}
The composite $\kvtwovect \to \ALG^\L \to \twovect$ is fully-faithful on 2\-morphisms and essentially full on 1\-morphisms. It identifies $\kvtwovect$ with the Cauchy-complete categories equivalent to a finite product of $(\vect)^\fd$, finite-dimensional vector spaces. 
\end{remark}

\subsection{Preliminaries on fully-dualizable 2-categories}




Recall from Section \ref{sec:linearcategories} that for a 2\-category $\bicat B$ we write $\bicat B^\L$ for the largest sub-2\-category consisting of all the objects of $\bicat B$, but only those 1\-morphisms of $\bicat B$ which admit right-adjoints in $\bicat B$ (hence are left-adjoints). Similarly we will let $\bicat{B}^\R$ denote the largest sub-2\-category with only those 1\-morphisms which admit left-adjoints (hence are right-adjoints). Note that if a 1\-morphism is a left adjoint, while it will be a 1\-morphism in $\bicat B ^\L$, it will not necessarily be a left adjoint in $\bicat B ^\L$.

We can iterate this construction. The 2\-category $(\bicat{B}^\L)^\L$, which is a sub-2\-category of $\bicat{B}^\L$, has 1\-morphisms given by the 1\-morphisms of $\bicat B$ which admit right adjoints which themselves admit right adjoints. $\bicat{B}^{\L^n} := ((\bicat{B}^\L) \dots)^\L$ ($n$ times) has 1\-morphisms given by the 1\-morphisms of \bicat B which admit a length-$n$ chain of successive right adjoints. The intersection of these sub-2\-categories for all $n$ we will denote as $\bicat{B}^{\L^\infty}$. Its 1\-morphisms admit an infinite chain of successive right adjoints.  A similar analysis holds for $\bicat{B}^\R$, $\bicat{B}^{\R^n}$,  and $\bicat{B}^{\R^\infty}$.

\begin{defn}
        A 2\-category $\bicat{B}$ \emph{admits duals for 1\-morphisms} if every 1\-morphisms has both a left and right adjoint; this is the case precisely when the inclusion $(\bicat{B}^{\L^\infty})^{\R^\infty} \subseteq \bicat{B}$ is an equality. 
\end{defn}


\begin{remark}
        In the above definition we have made an arbitrary choice for the order of applying the functors $(-)^\L$ and $(-)^\R$, but in fact this choice is irrelevant since these functors commute, $( \bicat B ^\L)^\R = ( \bicat B ^\R)^\L$. Hence you obtain an equivalent definition by applying these functors in any order, provided both $(-)^\L$ and $(-)^\R$ each occur infinitely often.
\end{remark}

From the definition we see immediately that adjoints are hereditary in the following sense: if $f \in \bicat{B}$ is a left-adjoint 1\-morphism and $F: \bicat{B} \to \bicat{C}$ is a functor, then $F(f)$ is a left-adjoint 1\-morphism in $\bicat{C}$. It follows that $F$ restricts to a functor $F: \bicat{B}^\L \to \bicat{C}^\L$. This yields:

\begin{lemma}
        If $\bicat{B}$ is a 2\-category that admits duals for 1\-morphisms and $\bicat{C}$ is any 2\-category, then every functor $F: \bicat{B} \to \bicat{C}$ factors as a strict composite
        \begin{equation*}
                \bicat{B} \to (\bicat{C}^{\L^\infty})^{\R^\infty} \subseteq \bicat{C}, 
        \end{equation*}
        and moreover this factorization is unique. \qed
\end{lemma}

This lemma justifies calling $(\bicat{B}^{\L^\infty})^{\R^\infty}$ the maximal sub-2\-category of $\bicat{B}$ which admits duals for 1\-morphisms. We will use the more concise notation \mbox{$\bicat{B}^{\d_1} := (\bicat{B}^{\L^\infty})^{\R^\infty}$}.

\begin{lemma}
        If $(f:x \to y) \in \bicat{B}$ is a 1\-morphism which admits an ambidextrous adjoint, then $(f:x \to y) \in \bicat{B}^{\d_1}$.
\end{lemma}

\begin{proof}
        Let $g:y \to x$ be any ambidextrous adjoint to $f$. Let $\bicat{C}$ be the sub-2\-category of $\bicat{B}$ consisting of the two objects $x$ and $y$, the 1\-morphisms $f$ and $g$ as well as their various composites, and all 2\-morphisms coming from $\bicat{B}$. The 2\-category $\bicat{C}$ admits duals for 1\-morphisms by construction (since any 1\-morphism does fit into a doubly-infinite sequence of successive adjunctions). Thus we have $\bicat{C} = \bicat{C}^{\d_1} \subseteq \bicat{B}^{\d_1} \subseteq \bicat{B}$. 
\end{proof}

\begin{corollary}\label{cor:modulardualmorphism}
        Let $\mathcal Q$ be one of the geometrical presentations $\M$, $\O$, $\M_k$, $\N_k$ from \autoref{sec:geometricalpresentations}. Then the associated symmetric monoidal 2\-category $\bicat{F}(\mathcal Q)$ admits duals for 1\-morphisms. 
\end{corollary}

\begin{proof}
        Each of the generating 1\-morphisms of these presentations admits an ambidextrous adjoint. 
\end{proof}

\begin{lemma}\label{lem:ffrestrictstoffonduals}
        Let $F: \bicat{B} \to \bicat{C}$ be a fully-faithful functor between symmetric monoidal 2\-categories (so we may regard $\bicat{B}$ as a full sub-2\-category of $\bicat{C}$). Then the restriction $\bicat{B}^{\d_1} \to \bicat{C}^{\d_1}$ is again fully-faithful. \qed
\end{lemma}

Now we turn to duals for objects. 

\begin{defn} \label{duals_for_objects}
        Let $(\bicat{B}, \otimes, I)$ be a symmetric monoidal 2\-category. An object $x \in \bicat{B}$ is \emph{dualizable} if there exists an object $x^*$ and morphisms $e: x \otimes x^* \to I$ and $c: I \to x^* \otimes x$ satisfying the zig-zag equations up to isomorphism; that is, there exist isomorphisms of the following types:
        \begin{align}
                (e \otimes \id_{x}) \circ (\id_{x} \otimes c) &\cong \id_{x} \label{zig-zag1} \\
                (\id_{x^*} \otimes e) \circ (c \otimes \id_{x^*}) &\cong \id_{x^*} \label{zig-zag2}
        \end{align} 
        We will let $\bicat{B}^{\d_0}$ denote the full sub-2\-category spanned by all the objects of $\bicat{B}$ which are dualizable.
\end{defn}

\noindent
Note that we do not specify the isomorphisms which witness the zig-zag equations, nor require them to satisfy any coherence equations, though we could do so~\cite[Theorem~2.19]{pstragowski-thesis}. Moreover $x \otimes y$ is dualizable when $x$ and $y$ are, as is the unit object~$I$. It follows that $\bicat{B}^{\d_0}$ is a sub--symmetric monoidal 2\-category. 

We will also make no left/right distinction. The symmetric monoidal structure makes any, say, left dual canonically into an ambidextrous dual. This also has the consequence that $(\bicat{B}^{\d_0})^{\d_0} = \bicat{B}^{\d_0}$; and so, in contrast to the case of adjoints, there is no need to iterate the process of taking sub-2\-categories of dualizable objects. 

\begin{defn}
        A symmetric monoidal 2\-category \emph{admits duals for objects} if $\bicat{B} = \bicat{B}^{\d_0}$. 
\end{defn}

\begin{lemma}\label{lem:ribbondualobject}
        Let $\mathcal Q$ denote any presentation that 2\-extends the ribbon presentation $\Rib$ (for example any of the modular presentations $\M$, $\O$, $\M_k$, $\N_k$ from \autoref{sec:modularobjectdefinition}), then $\bicat{F}(\mathcal Q)$ admits duals for objects. 
\end{lemma}

\begin{proof}
        It is enough to show that the single generating object admits duals. In fact this object is self dual, and has the following evaluation and coevaluation morphisms $e$ and $c$:
\smallbordisms
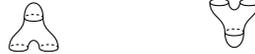
        \begin{equation*}
                \begin{tikzpicture}
                                \node[Pants] (A) at (0,0) {};
                                \node[Cap] at (A.belt) {};
                        \end{tikzpicture} \hspace{2cm}
                \begin{tikzpicture}
                                \node[Copants, top] (A) at (0,0) {};
                                \node[Cup] at (A.belt) {};
                        \end{tikzpicture} 
        \end{equation*}
That these satisfy the zig-zag equations is a easy calculation, utilizing the fact that $\phiN$ and $\phiM$ are invertible. 
\end{proof}

\noindent
As with adjoints, dualizability of objects is hereditary: if $\bicat{C}$ admits duals for objects then any functor $F:\bicat{C} \to \bicat{B}$ factors as a composite $\bicat{C} \to \bicat{B}^{\d_0} \subseteq \bicat{B}$. 

Combining the above operations we have inclusions:
\begin{equation*}
        (\bicat{B}^{\d_1})^{\d_0} \subseteq (\bicat{B}^{\d_0})^{\d_1} \subseteq \bicat{B}.
\end{equation*}
In general these are both proper inclusions. In particular the objects of $(\bicat{B}^{\d_0})^{\d_1}$ may no longer be dualizable in this sub-2\-category. 

\begin{defn}
        A symmetric monoidal 2\-category $\bicat{B}$ is \emph{fully-dualizable} if \mbox{$\bicat{B} = (\bicat{B}^{\d_1})^{\d_0} =: \bicat{B}^\fd$}. An object of $(\bicat{B}^{\d_1})^{\d_0}$ is called a \emph{fully-dualizable object}. 
\end{defn}

\begin{proposition}
        Let $\mathcal Q$ be one of the modular presentations $\M$, $\O$, $\M_k$, $\N_k$ from \autoref{sec:modularobjectdefinition}. Then the associated symmetric monoidal 2\-category $\bicat{F}(\mathcal Q)$ is fully-dualizable and hence any representation \mbox{$C: \bicat{F}(\mathcal Q) \to \bicat C$} factors uniquely through the fully-dualizable sub-2\-category $\bicat{C}^\fd$. 
\end{proposition}

\begin{proof}
        This is a direct consequence of \autoref{cor:modulardualmorphism}, \autoref{lem:ribbondualobject}, and the fact that full-dualizability is hereditary. 
\end{proof}

\subsection{Fully-dualizable 2--vector spaces are semisimple}

The following question is then natural to ask: what are the fully-dualizable sub-2\-categories of the various 2-categories of algebras and linear categories considered above? As we have seen in the last section any representation of (the 2\-categories generated by) the modular presentations must factor through these  sub-2\-categories. The 2-category $(\alg)^\fd$ corresponds to the full sub-2\-category spanned by the separable $k$-algebras \cite{CSPthesis, Davidovich-PhD, perry-spintfts}.  Since we assume $k$ to be an algebraically-closed field, every separable $k$-algebra is Morita equivalent to a finite direct sum $\oplus_n k$, which can be used to establish that the  fully-dualizable sub-2\-category, in this case, coincides with Kapranov and Voevodsky's 2\-category. In what follows we will not rely on those results, we only need the following special case
 (which may also be established by direct inspection).

\begin{lemma}
Kapranov and Voevodsky's 2\-category is fully-dualizable: $\kvtwovect = (\kvtwovect)^\fd$. \qed
\end{lemma}

\noindent In a different direction, Tillmann has considered the sub-2\-category of dualizable objects in $\twovect$:

\begin{proposition}[{\cite[Theorem~2.5]{t98-sskl}}]
\label{prop:equivalentalgebra}
An object $\cat C$ in $\twovect$ is dualizable if and only if there exists a finite-dimensional semi-simple $k$-algebra $A$ with \mbox{$\cat C \simeq \hat{A}$}, the category of finitely generated projective $A$-modules.
\end{proposition}

\noindent We will give an alternative proof of Tillmann's result in the next section. 

We can also relate $\twovect$ to $\smash{\bimod^\L}$ via the functor $\Phi: \twovect \rightarrow \bimod $ which sends a linear functor $F : \cat{C} \rightarrow D$ to its associated representable bimodule $F_* \colon \cat{C} \proarrow \cat{D}$. 

\begin{proposition}\label{pro:leftadjointbimodule}
        The functor $\Phi = (-)_*$ induces an equivalence of symmetric monoidal 2\-categories $\twovect \simeq \smash{\bimod^\L}$.
\end{proposition}

\begin{proof}
        By Proposition~\ref{pro:Cauchycompletionisequiv} the functor $(-)_*$ is essentially surjective. By Proposition~\ref{pro:CauchyComplete}(4) it is essentially full on 1\-morphisms. By the enriched Yoneda lemma, it is fully-faithful on 2\-morphisms, and hence an equivalence.  
\end{proof}

\noindent With these preliminaries established we are now ready for the main result of this appendix.

\begin{theorem}
\label{thm:equivalenttargets}
        For each of the symmetric monoidal 2\-categories $\bicat{C} =$ $\ALG$, $\alg$, $\ALG^\L$, $\lincat$, $\twovect$, or $\bimod$, the natural inclusion of $\kvtwovect$ into the fully-dualizable objects induces an equivalence
        \begin{equation*}
                \kvtwovect \simeq \bicat{C}^\fd
        \end{equation*}
        with Kapranov and Voevodsky's 2\-category of 2\_vector spaces.
\end{theorem}

\begin{proof}
        We begin with a few observations. First as a formal consequence of the definition we have $(\ALG)^{\mathrm d_1} \simeq(\ALG^\L)^{\mathrm d_1}$, and as a consequence of \autoref{pro:leftadjointbimodule} we have $(\twovect)^{\mathrm d_1} \simeq (\bimod)^{\mathrm d_1}$. Next we observe that if $B$ is a finite-dimensional algebra and ${}_AM_B$ is an $A$-$B$-bimodule which is a finitely generated projective $B$-module, then $M$ and its dual $\Hom_B(M_B, B_B)$ are also finite-dimensional. It follows that the inclusion $\alg^\L \to \ALG^\L$ is fully-faithful. Moreover, we have already observed in \autoref{lem:Alttowovectfullyfaithful} that $(\hat{-}): \ALG^\L \to \twovect$ is fully-faithful. By virtue of \autoref{lem:ffrestrictstoffonduals} we have established a chain of fully-faithful inclusions
        \begin{equation*}
                \kvtwovect \hookrightarrow (\alg)^{\mathrm d_1} \hookrightarrow (\ALG)^{\mathrm d_1} \hookrightarrow (\twovect)^{\mathrm d_1} \simeq (\bimod)^{\mathrm d_1},
        \end{equation*}
which induces another chain of fully-faithful inclusions
        \begin{equation*}
                \kvtwovect \hookrightarrow (\alg)^\text{fd} \hookrightarrow (\ALG)^\text{fd} \hookrightarrow (\twovect)^\text{fd} \simeq (\bimod)^\text{fd}.
        \end{equation*}
If the composite $\kvtwovect \to (\twovect)^\text{fd}$ were essentially surjective on objects, then each of these inclusions is an equivalence, and the theorem is established.         
        
By Tillmann's result given in Prop.~\ref{prop:equivalentalgebra}, the composite $\kvtwovect \to (\twovect)^\text{fd} \to (\twovect)^{\mathrm d_0}$ is essentially surjective, and since the second functor is injective on objects it follows that      
 $\kvtwovect \to (\twovect)^\text{fd}$ is essentially surjective, as desired. 
\end{proof}

\begin{corollary}\label{cor:modissemisimple}
        If $\mathcal Q$ is one of the modular presentations $\M$, $\O$, $\M_k$, $\N_k$ from \autoref{sec:geometricalpresentations}, then representations of $\mathcal Q$ in any of the symmetric monoidal 2\-categories $\ALG$, $\alg$, $\ALG^\L$, $\lincat$, $\twovect$, or $\bimod$, are equivalent and factor through the natural inclusion of $\kvtwovect$. \qed
\end{corollary}

\subsection{Dualizable 2-vector spaces are semisimple}
We now show that the dualizable objects in $\twovect$ are precisely the finite semisimple linear categories. We are grateful to Andr\'{e} Henriques for the following proof of \autoref{prop:equivalentalgebra}, which we accomplish in a series of lemmas.


Suppose that $\cat{S}$ is dualizable. Then there exists a Cauchy-complete linear category $\cat{D}$, together with unit and counit functors as follows, where $\widehat \boxtimes$ represents the Cauchy completion of the enriched tensor product:
\begin{align} \label{witnesssemi}
 N &\colon \Vect^\text{fd} \rightarrow D \hatboxtimes C
 &
 E &\colon C \hatboxtimes D \rightarrow \Vect^\text{fd}
\end{align}
Since every finite-dimensional vector space is a finite direct sum of the ground field $k$, the functor $N$ is determined by its value on $k$, which is the object $N(k) \in D \hatboxtimes C$.  
For each pair of objects $A \in \cat{C}$, $B \in \cat{D}$, consider $A \boxtimes B$ as an object in $\cat{C} \hatboxtimes \cat{D}$ using the canonical inclusion functor $\cat{C} \boxtimes {D} \hookrightarrow \cat{C} \hatboxtimes \cat{D}$.
Every object in $D \hatboxtimes C$ may be written as a formal retract of a finite direct sum of objects of the form $Y \boxtimes X$, with $Y \in D$ and $X \in C$. Thus write
\[
 N(k) =  \left( \bigoplus_{i=1}^n Y_i \boxtimes X_i, e \right)
\]
for some objects $X_i \in \cat{C}$, $Y_i \in \cat{D}$, $i=1 \ldots n$, and $e$ the idempotent which defines the retraction. Also  write 
\[
 E_{A,B} = E(A \boxtimes B)
\]
for the finite-dimensional vector spaces defined by $E$. 


\begin{lemma} For all $A, B \in \cat{C}$, the vector space $\Hom_{\cat{C}}(A, B)$ is finite-dimensional. \label{hom-finite}
\end{lemma}
\begin{proof}
Let $A, B \in \cat{C}$ be any two objects. The zig-zag equation \eqref{zig-zag1} implies that $\Hom_{\cat{C}}(A, B)$ is a retract of the finite-dimensional vector space
\begin{equation*}
        \bigoplus_{i = 1}^n \Hom(E_{A, Y_i}, E_{B, Y_i}) \otimes \mathrm{id}_{X_i},
\end{equation*}
a subspace of 
\begin{equation*}
        \Hom_{\cat{C}}\left( \bigoplus_{i = 1}^n E_{A, Y_i} \otimes X_i, \bigoplus_{i = 1}^n E_{B, Y_i} \otimes X_i  \right).
\end{equation*} 
Hence $\Hom_{\cat{C}}(A, B)$ is finite-dimensional. 
\end{proof}
On the other hand, the functors $N$ and $E$ from \eqref{witnesssemi} induce functors
\begin{equation} \label{witness_iso}
 I \colon \cat{C} \rightarrow \bar{\cat{D}}, \quad J \colon \bar{\cat{D}} \rightarrow \cat{C}
\end{equation}
where $\bar{\cat{D}} = \Fun_k (\cat{D}, \Vect^\text{fd})$. These functors are best understood in terms of wire diagrams for the symmetric monoidal 2\-category $\twovect$, where they are drawn as
\[
 \begin{tz}[yscale=1.5]
    \draw (0,0) node[left] {$C$} to node[fill=white, draw] {$I$} (0,-1) node[left]{$ \bar{\cat{D}}$};
 \end{tz}
 \qquad \text{and} \qquad
 \begin{tz}[yscale=1.5]
    \draw (0,0) node[left] {$\bar{\cat{D}}$} to node[fill=white, draw] {$J$} (0,-1) node[left]{$ \cat{C}$};
 \end{tz}
 \]
 respectively. Here composition of 1\-morphisms runs from top to bottom. The functor $I$ is defined as the unique functor such that
\[
\begin{tz}[yscale=1.3]
    \draw (0,0) node[left] {$C$} to node[draw, fill=white] {$I$} (0,-1) node[left] {$\bar{\cat{D}}$} to[out=down, in=down] node[draw, fill=white] {ev} (1,-1) to  (1,0) node[right] {$D$};
\end{tz}
\quad = \quad
\begin{tz}[yscale=1.3]
    \draw (0,0) node[left] {$C$} to (0,-1) to[out=down, in=down] node[draw, fill=white] {$E$} (1,-1) to  (1,0) node[right] {$D$};
\end{tz} \, ,
\]
where $ev \colon \bar{D} \mathop{\hat{\boxtimes}} {D} \rightarrow \Vect^\text{fd}$ is induced by the evaluation functor. The functor $J$ is defined by composing $N$ and $\text{ev}$ together,
\[
 \begin{tz}[yscale=2]
    \draw (0,0) node[left] {$\bar{\cat{D}}$} to node[fill=white, draw] {$J$} (0,-1) node[left]{$ \cat{C}$};
 \end{tz} \quad = \quad
\begin{tz}[yscale=1]
    \draw (0,0) node[left] {$\bar{\cat{D}}$} to (0,-1) to[out=down, in=down] node[draw, fill=white] {ev} (1,-1) to[out=up, in=up] node[draw, fill=white] {$N$} (2,-1) to (2,-2) node[right] {$C$};
\end{tz} \, .
\]
The zig-zag equations \eqref{zig-zag1} and \eqref{zig-zag2} then immediately establish the following.
\begin{lemma} The functors $I$ and $J$ furnish an equivalence of categories \mbox{$\cat{C} \simeq \bar{\cat{D}}$}. \qed
\end{lemma}
\begin{corollary}\label{cor:Cabelian} $\cat{C}$ is an abelian category. \qed
\end{corollary}

\begin{corollary}\label{cor:SEStoSES}
        For each $Y \in \cat{D}$ and short exact sequence $0 \to A \to B \to C \to 0$ in $\cat{C}$ we have a (necessarily split) short exact sequence of vector spaces,
        \begin{equation*}
                0 \to E_{A, Y} \to E_{B, Y} \to E_{C, Y} \to 0.
        \end{equation*}
\end{corollary}
\begin{proof}
        Since $I: \cat{C} \to \bar{\cat{D}}$ is an equivalence, 
        \begin{equation*}
                0 \to I(A) \to I(B) \to I(C) \to 0
        \end{equation*}
        is a short exact sequence in the functor category $\bar{\cat{D}}$. Since a sequence in a functor category is short exact if and only if it is pointwise,
         the corollary follows by evaluating the above sequence on the object $Y$.  
\end{proof}

Before proving Prop.~\ref{prop:equivalentalgebra} we would like to introduce some notation. 

\begin{defn} \label{def:finlength}
        An object $X$ in an abelian linear category $\cat{C}$ is called {\em simple} if there are precisely two subobjects, the zero object and $X$ itself. The category $\cat{C}$ is called {\em semisimple} if every object is a finite direct sum of simple objects. An object $X$ has {\em finite length} if every strictly decreasing chain of subobjects $X = X_0 \supsetneq X_1 \supsetneq X_2 \supsetneq  \cdots$ has finite length.
\end{defn}

By Schur's Lemma, an object $X$ is a simple object if and only if $\End_\cat{C}(X,X)$ is a division algebra over $k$. Since we are assuming that $k$ is algebraically closed, if $\cat{C}$ has finite-dimensional hom spaces, then this is the case if and only if $\End_\cat{C}(X,X) \cong k$ is the ground field. 


\begin{lemma}
        Let $\cat{C}$ be an abelian linear category which has finite-dimensional hom spaces. If every object is projective, then every object has finite length. 
\end{lemma}

\begin{proof}
        Let $Y \subsetneq X$ be a proper subobject, and let $C$ be the cokernel of the inclusion. Since $C$ is projective, we have a splitting $X \cong Y \oplus C$. Hence we have an ismorphism 
        \begin{equation*}
                \End_\cat{C}(X) \cong \End_\cat{C}(Y) \oplus \End_{\cat{C}}(C) \oplus \Hom_{\cat{C}}(C,Y) \oplus \Hom_{\cat{C}}(Y,C).
        \end{equation*} 
        Since $C \neq 0$, we have $\End_{\cat{C}}(C) \neq 0$, and hence $\dim \End_\cat{C}(X) > \dim \End_\cat{C}(Y)$. It follows that any strictly decreasing chain of proper subobjects of $X$ has length bounded by $\dim \End_\cat{C}(X)$.
\end{proof}

\begin{corollary}\label{cor:everynonzerohassimplesub}
        If $C$ is abelian linear category which has finite-dimensional hom spaces and where every object is projective, then every non-zero object contains a simple proper subobject. \qed
\end{corollary}

\begin{proposition}\label{prop:semisimpleiseveryobjprojective}
        Let $\cat{C}$ be an abelian linear category which has finite-dimensional hom spaces. Then $\cat{C}$ is a semisimple category if and only if every object is projective. 
\end{proposition}

\begin{proof}
        Suppose first that $\cat{C}$ is semisimple, that is that every object splits as a finite direct sum of simple objects. 
         Then we must show that every object $P$ is projective. Since projectives are closed under finite direct sum, it is enough to show that simple objects are projective. For every object $X$ and every simple object $S$, the decomposition of $X$ as a direct sum of simples induces a projection $X \to \Hom(S, X) \otimes S$, which as a projective is an epimorphism. Thus for $X \to Y$ we have a natural square:
         \begin{center}
         \begin{tikzpicture}
                        \node (LT) at (0, 1.5) {$X$};
                        \node (LB) at (0, 0) {$Y$};
                        \node (RT) at (3, 1.5) {$\Hom(S, X) \otimes S$};
                        \node (RB) at (3, 0) {$\Hom(S, Y) \otimes S$};
                        \draw [->] (LT) -- node [left] {$$} (LB);
                        \draw [->>] (LT) -- node [above] {$$} (RT);
                        \draw [->] (RT) -- node [right] {$$} (RB);
                        \draw [->>] (LB) -- node [below] {$$} (RB);
         \end{tikzpicture}
         \end{center}
         whose horizontal arrows are epimorphisms. It follows that if $X \to Y$ is an epimorphism, then so is $\Hom(S,X) \to \Hom(S,Y)$, and hence $S$ is projective. 
         
         Conversely, suppose that every object of $\cat{C}$ is projective. Let $\Lambda$ be the set of isomorphism classes of simple objects in $\cat{C}$. For each object $X$ and each finite subset $I \subseteq \Lambda$ we may form the subobject 
         \begin{equation*}
                X_I = \bigoplus_{i \in I} \Hom(S_i, X) \otimes S_i  \to X
         \end{equation*}
         where $S_i$ is any pre-chosen representative of the isomorphism class $i \in I$. To see that $X_I$ is really a subobject, we let $K_I$ be the kernel of $X_I \to X$. We first observe that for any simple object $S$ we have:
         \begin{equation*}
                \Hom(S, X_I) = \begin{cases}
                        \Hom(S,X) & \text{if } [S] \in I \\
                        0 & \text{else}
                \end{cases}
         \end{equation*}
         It follows that for all simple objects $S$ that $\Hom(S, K_I) = 0$, and hence that $K_I = 0$ by Corollary~\ref{cor:everynonzerohassimplesub}. Since every object is projective, $X$ splits as a direct sum $X \cong X_I \oplus C_I$, where $C_I$ is the cokernel of $X_I \to X$. Thus we have $\dim \End(X_I) \leq \dim \End(X)$. 
         
Let $P(\Lambda)$ denote the collection of all subsets of $\Lambda$ and $P^f(\Lambda)$      denote the collection of all finite subsets. We define a function:
\begin{align*}
        d: & P^f(\Lambda) \to \NN \\
        &I  \mapsto \dim \End(X_I)
\end{align*}     
This function has the properties that
\begin{enumerate}
        \item $d(I) \leq \dim \End(X)$, i.e., it is globally bounded, and
        \item $d(I) \leq d(J)$, whenever $I \subseteq J$. 
\end{enumerate}
It follows that we may extend $d$ to all of $P(\Lambda)$ by the formula $d(K) = \max_{I \subseteq K, I \in P^f(\Lambda)} d(I)$, and that this extension retains these two properties. In short the subobject $X_\Lambda = \bigoplus_{[S] \in \Lambda} \Hom(S, X) \otimes S$ is well defined and is a finite direct sum of simple objects. We will show that $X \cong X_\Lambda$. 

To this end consider the short exact sequence:
\begin{equation*}
        0 \to X_\Lambda \to X \to C_\Lambda \to 0.
\end{equation*}
Since every object is projective, the functor $\Hom(Y, -)$ is exact for every $Y$, and hence we have an exact sequence
\begin{equation*}
        0 \to \Hom(Y, X_\Lambda) \to \Hom(Y, X) \to \Hom(Y, C_\Lambda) \to 0.
\end{equation*}
However when $Y = S$ is a simple object, the first map is an isomorphism. Thus $\Hom(S, C_\Lambda) = 0$ for all simple objects $S$, which by Corollary~\ref{cor:everynonzerohassimplesub} shows that $C_\Lambda = 0$. 
\end{proof}

\begin{defn}
        In a $\vect$-enriched category, an object $Q$ is a {\em generator} if the functor \mbox{$\Hom_{\cat{C}}(Q, -): \cat{C} \to \vect$} is faithful.
\end{defn}

\noindent Under the assumption that $\cat{C}$ has finite-dimensional hom vector spaces,  the following three notions coincide for an object $Q \in \cat{C}$:
\begin{itemize}
        \item $Q$ is a generator;
        \item for all objects $X \in \cat{C}$, there exists a surjection $\oplus_n Q \to X$ from a finite direct sum of copies of $Q$;
        \item the evaluation map $\Hom_\cat{C}(Q, X) \otimes Q \to X$ is surjective. 
\end{itemize}
The equivalence of these three descriptions is well-known and not difficult to verify (see for example~\cite[Lemma~2.22]{Douglas:2014aa}).
\begin{lemma}\label{lem:projectivezerotogeniszero}
        Let $\cat{C}$ be an abelian linear category which has finite-dimensional hom spaces. Let
$X \in \cat{C}$ be a projective object and $Q \in \cat{C}$ be a generator. If $\Hom_{\cat{C}}(X, Q) = 0$, then $X =0$ is the zero object. 
\end{lemma}
\begin{proof}
        We have an exact sequence
        \begin{equation*}
                \Hom_{\cat{C}}(Q, X) \otimes Q \to X \to 0
        \end{equation*}
        to which we may apply the exact functor $\Hom_{\cat{C}}(X, -)$. The resulting exact sequence shows that $\Hom_{\cat{C}}(X,X) = 0$, and hence $X = 0$.
\end{proof}

\begin{proposition}
        Let $\cat{C}$ be an abelian linear category in which all hom spaces are finite-dimensional. Then the following are equivalent: 
        \begin{enumerate}
                \item $\cat{C}$ is semisimple (i.e., every object is a finite direct sum of simple objects), and there are only finitely many isomorphism classes of simple objects; 
                \item $\cat{C}$ is a  Kapranov-Voevodsky 2\_vector space; 
                \item $\cat{C} \simeq \hat{A}$, the category of finite-dimensional representations of a finite-dimensional semisimple $k$-algebra $A$;
                \item $\cat{C}$ is semisimple, and there exists a generator;
                \item every object is projective and there exists a generator;
                \item every short exact sequence in $C$ splits and there exists a generator;
        \end{enumerate}
\end{proposition}

\begin{proof}
        It is immediate (1) $\Leftrightarrow$ (2) $\Rightarrow$ (3), and that (5) $\Leftrightarrow$ (6). Proposition~\ref{prop:semisimpleiseveryobjprojective} shows that (4)  $\Leftrightarrow$ (5), and we deduce that (3) $\Rightarrow$ (4) since the generator is supplied by the algebra itself, viewed as a module. The proposition is established if we show that (4) $\Rightarrow$ (1), i.e., that if $\cat{C}$ is semisimple  and there exists a generator, then  there are only finitely many isomorphism classes of simple objects.
                
Let $Q \in \cat{C}$ be a generator. Then, since $\cat{C}$ is semisimple, $Q = \oplus_i S_i$ where the sum is finite and each $S_i$ is simple. There are finitely many simple objects which occur in this sum. Suppose that $S$ is a simple object which is not isomorphic to any of the $S_i$ factors of $Q$. Then it follows that $\Hom(S, Q) = 0$, but since $S$ is also projective, \autoref{lem:projectivezerotogeniszero} applies and shows that $S = 0$, contradicting the fact that $S$ is simple. Thus any simple object in $\cat{C}$ is isomorphic to one of the finitely many simple summands of $Q$. 
\end{proof}

\noindent
Returning to the proof of \autoref{prop:equivalentalgebra}, we have already seen that if $\cat{C}$ is a dualizable object in $\twovect$, then $\cat{C}$ has finite-dimensional hom\_vector spaces (\autoref{hom-finite}) and is abelian (\autoref{cor:Cabelian}).  \autoref{prop:equivalentalgebra} will be established by showing that there is a generator and that every short exact sequence splits. Recall that $N(k) \in D \mathop{\hat{\boxtimes}} C$ may be written as $(\oplus_i Y_i \boxtimes X_i, e)$.

\begin{lemma}   
        Every object in $\cat{C}$ is a retract of a finite direct sum of the objects $X_i$; in particular the object $Q = \oplus_i X_i$ is a generator. \label{finiteretract}
\end{lemma}
\begin{proof}
This is a consequence of the first zig-zag equation \eqref{zig-zag1}, which says precisely that $A$ is a retract of $\bigoplus_{i=1}^n E_{A, Y_i} \otimes X_i$ for each object $A \in \cat{C}$.
\end{proof}

\noindent We also have a version for morphisms.

\begin{lemma}
        Each morphism $f \colon X \to Y$ in $\cat{C}$ is a retract of the morphism $\bigoplus_{i=1}^n E_{f, id_{Y_i}} \otimes \id_{X_i} $. \qed
\end{lemma}

\noindent Our final lemma proves \autoref{prop:equivalentalgebra}.

\begin{lemma}
\label{splittinglem}
Every short exact sequence in $\cat{C}$ splits. \end{lemma}
\begin{proof}
By the previous lemma, the short exact sequence $0 \to A \xto f B \xto g C \to 0$ is a retract of the following sequence:
\begin{equation}
\label{eqn:SES}
0 \to \bigoplus_i E_{Y_i, A} \otimes X_i \xto{\bigoplus_i E_{\id, f} \otimes \id} \bigoplus_i E_{Y_i, B} \otimes X_i \xto{\bigoplus_i E_{\id, g} \otimes \id} \bigoplus_i E_{Y_i, C} \otimes X_i \to 0
\end{equation}
This is again a short exact sequence by \autoref{cor:SEStoSES}, and in particular it is determined by the following short exact sequences of vector spaces:
\[
0 \xrightarrow{} E_{Y_i, A} \xrightarrow{E_{\id, f}} E_{Y_i, B} \xrightarrow{E_{\id, g}} E_{Y_i, C} \xrightarrow{} 0
\]
Each of these splits, and hence the sequence~\eqref{eqn:SES} splits. Thus the original short exact sequence $0 \to A \to B \to C \to 0$ is a retract of a split short exact sequence, hence is itself split. 
\end{proof}


\bibliographystyle{plainurl}
\bibliography{references}

\end{document}